\documentclass[12pt]{amsart}
\usepackage{amsmath, amsthm, amssymb}
\usepackage{amsfonts}
\usepackage{fullpage}
\usepackage{color}
\usepackage{hyperref}
\usepackage{enumitem}
\usepackage{bm}
\usepackage{verbatim}
\usepackage{upgreek}

\usepackage[dvipsnames]{xcolor}
\usepackage{graphicx}  
\DeclareGraphicsExtensions{.pstex,.eps}

\usepackage{xcolor}

\renewcommand{\ell}{{l}}  

\newcommand{\bH}{\mathbf{H}}

\newcommand{\bb}{\mathbf{b}}

\newcommand{\sfH}{\mathsf{H}}
\newcommand{\sfL}{\mathsf{L}}

\newcommand{\R}{{\mathbb{R}}}
\newcommand{\Z}{{\mathbb Z}}

\newcommand{\N}{{\mathbb N}}

\renewcommand{\AA}{{\mathcal A}}

\newcommand{\EE}{{\mathcal E}}

\newcommand{\MM}{{\mathcal M}}
\newcommand{\FF}{{\mathcal F}}

\renewcommand{\SS}{{\mathcal S}}

\renewcommand{\FF}{{\mathfrak F}}

\newcommand{\tri}{|\!|\!|}

\newcommand{\tF}{{\tilde{F}}}

\newcommand{\tg}{{\tilde g}}

\newcommand{\tGa}{{\tilde \Ga}}

\newcommand{\tV}{{\tilde V}}
\newcommand{\tnu}{{\tilde \nu}}

\newcommand{\tr}{\operatorname{tr}}

\renewcommand{\Re}{\mathop{\rm Re}\nolimits}
\renewcommand{\Im}{\mathop{\rm Im}\nolimits}

\newcommand{\dH}{\dot H}

\DeclareMathOperator{\Ric}{Ric}

\theoremstyle{plain}

\newtheorem{thm}{Theorem}[section]
\newtheorem{prop}[thm]{Proposition}
\newtheorem{cor}[thm]{Corollary}
\newtheorem{lemma}[thm]{Lemma}

\theoremstyle{definition}

\newtheorem{rem}{Remark}[thm]
\newtheorem{defn}[thm]{Definition} 

\numberwithin{equation}{section}

\def\squarebox#1{\hbox to #1{\hfill\vbox to #1{\vfill}}}

\newcommand{\<}{\langle}
\renewcommand{\>}{\rangle}

\newcommand{\Fo}{F_\mathbf{1}}
\newcommand{\Sigo}{\Sigma_\mathbf{1}}
\newcommand{\go}{g^\mathbf{1}}
\newcommand{\nuo}{\nu^\mathbf{1}}
\newcommand{\nuoo}{\nu^\mathbf{1}_{1}}
\newcommand{\nut}{\nu^\mathbf{1}_{2}}
\newcommand{\Lamo}{\Lambda^\mathbf{1}}

\renewcommand{\d}{\partial}
\newcommand{\ep}{\epsilon}
\newcommand{\lV}{\lVert}
\newcommand{\rV}{\rVert}

\def\be{{\beta}}
\def\ga{\gamma}
\def\Ga{\Gamma}
\def\de{\delta}
\def\De{\Delta}
\def\ep{\epsilon}
\def\ka{\kappa}
\def\la{\lambda}
\def\La{\Lambda}
\def\si{\sigma}
\def\Si{\Sigma}
\def\om{\omega}

\def\ka{\kappa}
\def\nab{\nabla}

\def\al{\alpha}
\def\les{\lesssim}

\newcommand{\tSigma}{\tilde{\Sigma}}

\title[Skew mean curvature flow]
{Local well-posedness of the skew mean curvature flow for large data}

\author[J. Huang]
{Jiaxi Huang}

\author[D. Tataru]
{Daniel Tataru}

\address{School of Mathematics and Statistics, Beijing Institute of Technology, 
Beijing
100081, P.R. China}
\email{jiaxih@bit.edu.cn}

\address{Department of Mathematics, University of California, Berkeley \\
Berkeley, CA 94720, USA}
\email{tataru@math.berkeley.edu}

\subjclass[2020]{Primary: 35Q55; Secondary: 53E10.}

\keywords{Skew mean curvature flow, large data, local well-posedness, low regularity}

\begin{document}

\begin{abstract}
The skew mean curvature flow is an evolution equation for $d$ dimensional ma\-nifolds
embedded in $\mathbb{R}^{d+2}$ (or more generally, in a Riemannian manifold). It can be viewed as a Schr\"odinger analogue of the mean curvature 
flow, or alternatively  as a quasilinear version of the Schr\"odinger Map equation.  
In this article, we prove large data local well-posedness in low-regularity Sobolev spaces for the skew mean curvature flow in dimension $d\geq 2$. This is achieved by introducing several new ideas: (i) a time discretization method to establish
the existence of smooth solutions, (ii) constructing the orthonormal frame by a parallel transport method and a lifting criterion, (iii) introducing intrinsic fractional function spaces $X^s\subset H^s$ on a noncompact manifold for any $s>\frac{d}{2}$, such that the $X^s$-norm of the second fundamental form can be propagated well along the quasilinear Schr\"odinger flow, (iv) deriving a difference equation to prove the uniqueness result for solutions $F\in C^2$, which is independent in the choices of gauge. 
Our method turns out to be more robust for large data problem.
\end{abstract}

\date{\today}
\maketitle

\centerline{\today}

\setcounter{tocdepth}{1}
\pagenumbering{roman} \tableofcontents \newpage \pagenumbering{arabic}

\section{Introduction}

In this article we continue our study of the 
local well-posedness for the skew mean curvature flow (SMCF).
This  is a nonlinear Schr\"odinger-type flow 
modeling the evolution of a $d$ dimensional oriented manifold  embedded into a fixed oriented $d+2$ dimensional manifold. It can be seen as a Schr\"odinger analogue of the well studied mean curvature flow. 

In earlier works \cite{HT21,HT22}, we have proved the local well-posedness of (SMCF) flow for small initial data in low regularity Sobolev spaces.
This was achieved by developing a suitable gauge formulation of the  equations, which allowed us to reformulate the problem as a quasilinear Schr\"odinger evolution, and then by constructing the solutions via a Picard iteration.

In this article, we consider the local well-posedness of the skew mean curvature flow for large data, also for low regularity initial data.
As an iterative/fixed point construction method does not suffice for the large data problem, here we use a time discretization method (see \cite[Section 4.2]{IT-primer}) to construct our solution.
Also, since our earlier function spaces have issues for large data, here we introduce  new fractional function spaces $X^s \subset H^s$, in order to address these difficulties.

\subsection{The (SMCF) equations}
Let $\Sigma^d$ be a $d$-dimensional oriented manifold, and $(\mathcal{N}^{d+2},g_{\mathcal{N}})$ be a $d+2$-dimensional oriented Riemannian manifold. Let $I=[0,T]$ be an interval and $F:I\times \Sigma^d \rightarrow \mathcal{N}$ be a 
one parameter family of immersions. This induces a time dependent Riemannian structure 
on $\Sigma^d$. For each $t\in I$, we denote the submanifold by $\Sigma_t=F(t,\Sigma)$,  its tangent bundle by $T\Sigma_t$, and its normal bundle by $N\Sigma_t$ respectively. For an arbitrary vector $Z$ at $F$ we denote by $Z^\perp$
its orthogonal projection onto $N\Sigma_t$.
The mean curvature  $\mathbf{H}(F)$ of $\Sigma_t$ can be identified naturally with a section of the normal bundle $N\Sigma_t$.

The normal bundle $N\Sigma_t$ is a rank two vector bundle with a naturally induced complex structure $J(F)$ which simply rotates a vector in the normal space by $\pi/2$ positively. Namely, for any point $y=F(t,x)\in \Sigma_t$ and any normal vector $\nu\in N_{y}\Sigma_t$, we define $J(F)\in N_{y}\Sigma_t$ as the unique vector with the same length
so that 
\[
J(F)\nu\bot \nu, \qquad \omega(F_{\ast}(e_1), F_{\ast}(e_2),\cdots F_{\ast}(e_d), \nu,  J(F)\nu)>0,
\]
where $\om$ is the volume form of $\mathcal{N}$ and $\{e_1,\cdots,e_d\}$ is an oriented basis of $\Sigma^d$. The skew mean curvature flow (SMCF) is defined as the initial value problem
\begin{equation}           \label{Main-Sys}
\left\{\begin{aligned}
&(\d_t F)^{\perp}=J(F)\mathbf{H}(F),\\
&F(0,\cdot)=F_0,
\end{aligned}\right.
\end{equation}
which evolves a codimension two submanifold along its binormal direction with a speed given by its mean curvature.

The (SMCF) was derived both in physics and mathematics. 
The one-dimensional (SMCF) in the Euclidean space $\R^3$ is the well-known vortex filament equation (VFE)
\begin{align*}
\d_t \ga=\d_s \ga\times \d_s^2 \ga,
\end{align*}
where $\ga$ is a time-dependent space curve, $s$ is its arc-length parameter and $\times$ denotes the cross product in $\R^3$. The (VFE) was first discovered by Da Rios \cite{DaRios1906} in 1906 in the study of the free motion of a vortex filament.

The (SMCF) also arises in the study of asymptotic dynamics of vortices in the context of superfluidity and superconductivity. For the Gross-Pitaevskii equation, which models the wave function associated with a Bose-Einstein condensate, physics evidence indicates that the vortices would evolve along the (SMCF). An incomplete verification was attempted by Lin \cite{LinT00} for the vortex filaments in three space dimensions. For higher dimensions, Jerrard \cite{Je02} proved this conjecture when the initial singular set is a codimension two sphere with multiplicity one.

The other motivation is that the (SMCF) naturally arises in the study of the hydrodynamical Euler equation. A singular vortex in a fluid is called a vortex membrane in higher dimensions if it is supported on a codimension two subset. The law of locally induced motion of a vortex membrane can be deduced from the Euler equation by applying the Biot-Savart formula. Shashikanth \cite{Sh12} first investigated the motion of a vortex membrane in $\R^4$ and showed that it is governed by the two dimensional (SMCF), while Khesin \cite{Kh12} then generalized this conclusion to any dimensional vortex membranes in Euclidean spaces.

From a mathematical standpoint, the (SMCF) equation is a canonical geometric flow for codimension two submanifolds which can be viewed as the Schr\"odinger analogue of the well studied mean curvature flow. In fact, the infinite-dimensional space of codimension two immersions of a Riemannian manifold admits a generalized Marsden-Weinstein sympletic structure, and hence the Hamiltonian flow of the volume functional on this space is verified to be the (SMCF). Haller-Vizman \cite{HaVi04} noted this fact when they studied the nonlinear Grassmannians.
For a detailed mathematical derivation of these equations we refer the reader  to the article \cite[Section 2.1]{SoSun17}. 

The one dimensional case of this problem has been extensively studied. This is because the one dimensional (SMCF) flow
agrees with the classical  Schr\"odinger Map type equation, provided that one chooses suitable coordinates, i.e. the arclength
parametrization. As such, it exhibits many special properties 
(e.g. complete integrability) which are absent in higher 
dimensions. For more details we refer the readers to the articles \cite{BaVe20,Ve14}.

In contrast, the theory of higher-dimensional (SMCF) is far less developed. This is primarily because it falls into the class of quasilinear Schr\"odinger-type geometric flows, which present significant analytical challenges.
Song and Sun \cite{SoSun17} took an important first step towards establishing well-posedness. 
They explored the basic properties of (SMCF) and proved the first local existence result in two dimensions, taking a smooth, compact, oriented surface as the initial data. This result was later generalized by Song \cite{So19} to compact oriented manifolds of arbitrary dimension $d\geq 2$.
In \cite{So19}, Song also made a significant contribution to the earlier uniqueness result by introducing a geometrically intrinsic distance $\mathcal L(F,\widetilde{F})$, constructed via a parallel transport method, that exploits the underlying geometric structure of the (SMCF).
Subsequently, Li \cite{Li1,Li2} studied a class of transversal perturbations of Euclidean planes under the (SMCF), proving a global regularity result for small initial data and a local well-posedness for large data. 
The aforementioned works offer valuable insights for investigating (SMCF). Nevertheless, as pointed out in \cite{So19}, two key questions remain unresolved: local well-posedness for large data and global well-posedness for small data on non-compact manifolds with low regularity.

To study the well-posedness of (SMCF) on noncompact manifolds, a crucial step is to establish a rigorous and self-contained formulation. 
The first incomplete attempt in this direction was made by Gomez \cite{Go}, who derived a Schr\"odinger-type equation for the second fundamental form, along with a set of compatibility conditions.
Recently, the authors in \cite{HT21,HT22} further refined and improved Gomez's derivation.
In these works, we introduced harmonic/Coulomb and heat gauges in order to obtain a complete gauge formulation.
By combining this gauge framework with the local energy decay estimates, we established a Hadamard-style local well-posedness result in low-regularity Sobolev spaces for small initial data.
In subsequent work \cite{HLT}, together with Li, we applied Strichartz estimates and energy estimates to prove small-data global regularity for (SMCF) in dimensions 
$d\geq 4$, thereby extending the local existence result of \cite{HT21}.

In this article we continue our study of the local well-posedness for (SMCF) with large initial data.  
Precisely, we let $\Sigma^d=\R^d$ have trivial topology, and we restrict the target as the Euclidean space $\mathcal N^{d+2}=(\R^{d+2},g_{\R^{d+2}})$. Thus, the reader should visualize $\Sigma_t$ as an asymptotically flat codimension two submanifold of $\R^{d+2}$. 
A key role in both \cite{HT21,HT22,HLT} and in this article is played by our gauge choices, which are discussed next.

\subsection{Gauge choices for (SMCF)}

There are two components for the gauge choice, which are briefly discussed here and in full detail in Section~\ref{Sec-gauge}:

\begin{enumerate}
    \item The choice of coordinates on $I \times \Sigma$.
    \item The choice of an orthonormal frame on $I \times N\Sigma$.
\end{enumerate}

Indeed, as written above in \eqref{Main-Sys}, the (SMCF) equations  are independent of the choice of coordinates in $I \times \Sigma$; here we include 
the time interval $I$ to emphasize that coordinates may be chosen in a time dependent fashion. The manifold $\Sigma^d$ simply serves to provide a parametrization for the moving manifold $\Sigma_t$; it  determines the topology of $\Sigma_t$, but nothing else. Thus, the (SMCF) system written in the form \eqref{Main-Sys} should be seen as a geometric evolution, with a large gauge group, namely the group of time dependent changes of coordinates in $I \times \Sigma$.  One may think of the 
gauge choice here as having two components, (i) the choice of coordinates at the initial time,  and (ii) the time evolution 
of the coordinates. One way to describe the latter choice
is to rewrite the equations in the form 
\begin{equation}\label{Main-Sys-re}
\left\{\begin{aligned}
&(\d_t - V \partial_x) F =J(F)\mathbf{H}(F),\\
&F(0,\cdot)=F_0,
\end{aligned}\right.
\end{equation}
where the vector field $V$ can be freely chosen, and captures the time evolution of the  coordinates. 
Indeed, some of the earlier papers \cite{SoSun17} and \cite{So19} on (SMCF) use this formulation with $V = 0$,
which we will refer to as the \emph{temporal gauge}. This would seem to simplify the equations, however it introduces difficulties at the level of comparing solutions. This is because in this gauge 
the regularity of the map $F$ is no longer determined by the intrinsic regularity of the second fundamental form, and instead there is a loss of derivatives in the analysis. This loss is what prevents a complete low regularity theory in that approach.

Our ideas in \cite{HT21,HT22} were to use harmonic coordinates on $\Sigma$ at the initial time, while introducing heat coordinates for later times, i.e. a \emph{heat gauge}. This choice improves the regularity of the metric $g$ and also allows the metric to be propagated effectively. This propagation implicitly fixes $V$, which can be obtained as the solution to an appropriate parabolic equation.
The approach is robust and can even be applied to large data problems. In the present paper, however, we allow
for a more flexible choice of the initial coordinates, which is made relative to a reference, regularized manifold. Then the heat coordinates at later time 
are also chosen relative to the reference manifold.
This idea affords us greater flexibility in the choice of initial coordinates than \cite{HT22}, particularly in low dimension.

We now discuss the second component of the gauge choice, namely the orthonormal frame in the normal bundle. Such a choice is needed in order to fix the second fundamental form for $\Sigma$; indeed, the (SMCF) is most naturally interpreted as a nonlinear Schr\"odinger evolution for the second fundamental form of $\Sigma$.  
In our earlier papers \cite{HT21,HT22}, the orthonormal frame was easily constructed because the metric and the second fundamental form were small. However, this approach is no longer well-suited for the large data case. To address this, we first construct an orthonormal frame on a smooth background manifold by parallel transport method and imposing a modified Coulomb gauge. This gauge choice provides effective control over the frame. We then obtain the desired frame on $\Sigma$ via a perturbative method. At later times, we continue to use the heat gauge to propagate the frame.

\subsection{Scaling and function spaces}
To understand what are the natural thresholds for local well-posedness, it is interesting  
to consider the scaling properties of the solutions. As one might expect, a clean scaling law is obtained when $\Sigma^d = \R^d$ and $\mathcal{N}^{d+2} = \R^{d+2}$. Then we have the following scaling invariance:

\begin{prop}[Scale invariance for (SMCF)]
	Assume that $F$ is a solution of \eqref{Main-Sys} with initial data $F(0)=F_0$, then ${F}_\mu(t,x):=\mu^{-1}F(\mu^2 t,\mu x)$ is a solution of \eqref{Main-Sys} with initial data ${F}_\mu(0)=\mu^{-1}F_0(\mu x)$.
\end{prop}

The above scaling would suggest the critical Sobolev space for our moving surfaces 
$\Sigma_t$ to be $\dot H^{\frac{d}{2}+1}$. However, instead of working directly with the 
surfaces, it is far more convenient to track the regularity at the level of the curvature
$\mathbf{H}(\Sigma_t)$, which scales at the level of $\dot H^{\frac{d}2-1}$. 
For our main result we will use instead
inhomogeneous Sobolev spaces, and it will suffice to go one derivative above scaling.

\subsection{The main result}
Our objective in this paper is to establish the local well-posedness of skew mean curvature flow for large data at low regularity.

We begin with the ellipticity of metric and the volume form. Assume that the inverse of metric $g$ on the initial manifold $\Sigma_0$ is elliptic and $g$ is near $I$ at infinity, i.e.
\begin{align}   \label{Ell-cond}
    g^{\al\be}\xi_\al \xi_\be\geq C^{-1}|\xi|^2,\qquad \lim_{x\rightarrow \infty}(g_{\al\be})=I.
\end{align} 
This also implies that $g\leq CI$ is bounded from above. 
Moreover, the initial manifold $\Sigma_0=F_0(\R^d)$ is an immersion, so the kernel of $d F_0(x)$ is $\{0\}$, therefore using $\lim\limits_{x\rightarrow\infty}(g_{\al\be})=I$ for the exterior of a large ball $B_0(R)$ and Heine-Borel Theorem for the compact set $B_0(R)$, it follows that there exists $c>0$ such that
    \[    \inf_x \min_{\al\in\R^d,|\al|=1}\Big|\frac{\d F_0}{\d \al}\Big|^2\geq c^2,       \]
Hence, under the condition \eqref{Ell-cond} and the above analysis, there exists a $0<c_0:=\min\{c,C^{-1}\}$ such that
\begin{equation}   \label{MetBdd}
    c_0 I\leq g\leq c_0^{-1}I,\qquad c_0^d\leq \det g\leq c_0^{-d}.
\end{equation}
These have been discussed in \cite[p.35]{Li2}.

We are now ready to state our main result, which we split into three parts in a modular fashion. We begin with the case of regular data:

\begin{thm}[Existence of regular solutions]  \label{RegSol-thm}
Let $d \geq 2$ and $k>\frac{d}{2}+5$ be an integer. Let $(\Sigma_0,g_0)$ be a smooth, complete, immersed Riemannian submanifold of dimension $d$ with bounded second fundamental form
\begin{equation*}
\|\Lambda_0\|_{\sfH^k(\Sigma_0)}\leq M,
\end{equation*}
bounded Ricci curvature and bounded geometry, i.e. 
\begin{equation}
 \label{bdd-geom}   
|\Ric(\Sigma_0)|\leq  C_0, \qquad \inf_{x\in \Sigma_0} {\rm Vol}_{g(0)}(B_x(1))\geq v, \quad c_0 I\leq g_0\leq c_0^{-1}I,
\end{equation}
for some $C_0>0$, $v>0$ and $c_0>0$,
where ${\rm Vol}_{g(0)}(B_x(1),\Sigma_0)$ stands for the volume of ball $B_x(1)$ on $\Sigma_0$ with respect to $g(0)$. Then there exists a unique smooth solution $\Sigma(t)=F(t,\R^d)$ on a time interval $[0,T]$ depending on $M$, $C_0$, $c_0$ and $v$, such that
\begin{align*}
    \|\Lambda\|_{\sfH^k(\Sigma)}\les M.
\end{align*}
\end{thm}

\begin{rem}
The assumptions in \eqref{bdd-geom} are made to ensure that the Sobolev embeddings hold on a noncompact manifold.
The regularity $k>\frac{d}{2}+5$ for $\Lambda$ is chosen in order to more easily control errors in our construction of solutions via an Euler scheme.
\end{rem}

While extra regularity was used for our initial existence result, we are able to match this with an uniqueness result at a much lower regularity:

\begin{thm}[Uniqueness of solutions]  \label{Uniqueness-thm}
Let $d \geq 2$. Let $(\Sigma_0,g_0)$ be a smooth, complete, immersed Riemannian submanifold of dimension $d$ satisfying \eqref{bdd-geom}, admitting a uniform $C^2$
parametrization,
\begin{equation*}
\| \partial F_0\|_{C^1} \leq M, \qquad c_0I \leq g_0 \leq c_0^{-1}I,    
\end{equation*}
and with $L^2$ bounded second fundamental form
\begin{equation*}
	\|\Lambda_0\|_{L^2(\Sigma_0)}\leq M.
\end{equation*}
 Then there exists a unique local solution $\Sigma(t)=F(t,\R^d)$ in the class of functions $F$ preserving the above properties. 
\end{thm}

By Sobolev embeddings, this uniqueness result in particular suffices in order to  conclude that the 
solutions provided by Theorem~\ref{RegSol-thm} are unique.
This theorem can also be seen as a corollary of Proposition~\ref{Prop-DiffBound} in Section~\ref{Sec-lin},
which provides $L^2$ difference bounds for solutions 
with different initial data.

The existence result for regular solutions,
together with the uniqueness result in 
Theorem~\ref{Uniqueness-thm} and the energy estimates for linearized equations in Proposition~\ref{prop4.2}
serve as key stepping stones for our proof of local well-posedness in the rough data case.

\begin{thm}[Local well-posedness for rough data]   \label{LWP-thm}
    Let $d \geq 2$ and $s > \frac{d}{2}$. For a small parameter $0 < \delta \ll 1$, denote $\si_d=1-\de$ if $d=2$ or $\si_d=1$ if $d\geq 3$.
    Assume that the initial data $\Sigma_0$ with metric $g_0$  and mean curvature $ \mathbf{H}_0$ satisfies the condition \eqref{MetBdd} and 
\begin{equation} \label{small-datad3}
	\||D|^{\si_d}(g_0-I_d)\|_{H^{s+1-\si_d}}\leq M, \qquad \lV \mathbf{H}_0 \rV_{H^s(\Sigma_0)}\leq M,
\end{equation}	
relative to some parametrization of $\Sigma_0$. Then the skew mean curvature flow \eqref{Main-Sys} for maps from $\R^d$ to the Euclidean space $(\R^{d+2},g_{\R^{d+2}})$ is locally well-posed on the time interval $I=[0,T(M,c_0)]$ in a suitable gauge. 
\end{thm}

\begin{rem}
The parameter $\si_d$ is chosen such that we could bound  $g_0-I$ in $L^\infty$.
\end{rem}

In the next section we reformulate the (SMCF) equations as a quasilinear Schr\"odinger 
evolution for a good scalar complex variable $\la$, which is exactly the second fundamental form but represented in our chosen gauge.  There we provide a more complete, alternate formulation of the above result, as a well-posedness result for the $\la$ equation.
In the final section of the paper we 
close the circle and show the local well-posedness of (SMCF) for rough initial data as the limit of regular solutions.

Once our problem is rephrased as a nonlinear Schr\"odinger evolution,  one may compare its 
study with earlier results on general quasilinear Schr\"odinger evolutions. This story begins 
with the classical work of Kenig-Ponce-Vega \cite{KPV2,KPV3,KPV}, where local well-posedness is established for more regular and localized data. Lower regularity results in translation invariant Sobolev spaces were later established by Marzuola-Metcalfe-Tataru~\cite{MMT3,MMT4,MMT5}. Finally,
in the case of cubic nonlinearities
this theory was redeveloped and improved to the sharp regularity thresholds by Ifrim-Tataru~\cite{IT-qnls,IT-qnls2}.
The local energy decay properties of the Schr\"odinger equation, as developed earlier in \cite{CS,CKS,Doi,Doi1}, play a key role in these results. 
While here we are using some of the ideas in the above papers, the present problem is both more complex and exhibits additional structure. 
Because of this, new ideas and more work are required in order to close the estimates required for both the full problem and for its linearization.

In contrast to the previous works \cite{HT21,HT22} that relied on additional assumptions, our approach yields a clearer and more natural result.
(i) Our result eliminates the low-frequency assumption $\|g_0-I\|_{Y^{lo}_0}$ that was required in \cite{HT22}.
The difference arises from the method used to construct solutions.
The authors of \cite{HT22} employed an iterative method based on the coupled Schr\"odinger-parabolic system, which forced us to control the $Y$-norms of $g_0-I$ by solving an elliptic equation; this was particularly necessary in two dimensions due to the failure of the embedding $H^1 \subseteq L^\infty$. 
In this article, however, the existence theory is provided by Theorem~\ref{RegSol-thm}. Therefore, for rough solutions, it suffices to establish the uniform energy estimates in Sobolev spaces, for which the assumptions \eqref{MetBdd} and \eqref{small-datad3} are sufficient. 
(ii) The nontrapping condition introduced in \cite{MMT5} is not required for our results. Here, we introduce intrinsic fractional Sobolev spaces $X^s\subset H^s$, inspired by the favorable propagation properties of the intrinsic norm $\sfH^k$ for integer 
$k$. Using these spaces, we establish the energy estimates for the SMCF directly, without the need for an additional nontrapping condition.

The uniqueness for the SMCF is established under the assumption that solutions $F$ are merely of class $C^2$, and the proof is independent of the choice of gauge. In \cite{So19}, Song made significant progress towards this uniqueness result by employing a method of parallel transport in order to compare different solutions, one in the class $F\in C^4$ and another in the class $\tF\in C^5$.
In this article, we instead employ the general formulation~\eqref{Main-Sys-re} with the vector field $V$ left free. This allows the coordinates of the second solution $\tF$ to be chosen such that the difference $F-\tF$ is comparable to its normal component $\omega \in N\Sigma(F)$. Furthermore, we derive a Schr\"odinger-type equation for $\omega$, which enables us to establish a Gr\"onwall's inequality with a constant that depends only on the $C^2$ norms of $F$ and $\tF$.
This approach yields a two derivative improvement in the required regularity compared to the result in \cite{So19}.

\subsection{An overview of the paper}

Our first objective in this article will be to review the derivation of a self-contained formulation of the (SMCF) flow, interpreted as a nonlinear Schr\"odinger equation for a well chosen variable. This variable, denoted by $\la$, represents the second fundamental form on $\Sigma_t$, in complex notation. 
In addition, we will use several dependent variables, as follows:
\begin{itemize}
    \item The Riemannian metric $g$ on $\Sigma_t$.
    \item The magnetic potential $A$, associated to the natural connection on the 
    normal bundle $N \Sigma_t$.
\end{itemize}
These additional variables will be viewed as uniquely determined by our main variable $\la$, initial metric $g_0$ and connection $A_0$ in a dynamical fashion. 
This is first done at the initial time by retaining the original coordinates on $\Sigma_0$, while introducing a good orthonormal frame on $N \Sigma_0$ that is a small perturbation of the modified Coulomb gauge on background manifold. Finally, our dynamical gauge choice also has two
components:

\begin{enumerate}
    \item[(i)] The choice of coordinates on $\Sigma_t$; here we use heat coordinates, with 
    suitable boundary conditions at infinity.
    \item[(ii)] The choice of the orthonormal frame on $N\Sigma_t$; here we use the heat gauge, again assuming flatness at infinity.
\end{enumerate}

To begin this analysis, in the next section we describe the gauge choices, so that by the end we obtain 
\begin{enumerate}
    \item[(a)] A nonlinear Schr\"odinger equation for $\la$, see \eqref{mdf-Shr-sys-2}.  
    \item[(b)] A parabolic system \eqref{par-syst} for the dependent variables
    $\SS=(g,A)$, together with suitable compatibility conditions (constraints).
\end{enumerate}
 
Setting the stage to solve these equations, in Section~\ref{Sec3} we first introduce some notation and a range of inequalities on noncompact manifolds. These inequalities, particularly the Sobolev embeddings, will play a crucial role in the construction of regular solutions presented in Section~\ref{Sec-RegSol}. 
Then, we describe the function spaces for both $\la$ and $\SS$.
Our starting point is provided by the intrinsic Sobolev norms $\sfH^k$ of $\la$, which  are well propagated along (SMCF).  Based on these norms, we then define their fractional versions, namely the  $X^s$-norms, using a characterization which is akin to a Littlewood-Paley decomposition, or to a discretization of the $J$ method of interpolation.    
The $X^s$-norm of $\la$ for $s>\frac{d}{2}$ is almost equivalent to its $H^s$-norm and satisfies the embedding $X^s\subset H^s$. To keep consistency, we also introduce corresponding $Y^{s+1}$ and $Z^s$ norms for metric $g$ and connection $A$, respectively, which satisfy similar properties.

We organize the proofs in a modular
fashion as follows:

\medskip 
{\bf I. The linearized equations, difference estimates and the uniqueness result.} We begin our analysis in Section~\ref{Sec-lin}, where we focus on the linearized equations and the difference estimates for (SMCF). First, we derive the linearized equations for the normal and tangent components of a family of maps $F(t,x;s)$ parameterized by $s$. $L^2$-type energy estimates for these linearized variables are then readily obtained; these will be used later to construct rough solutions as limits of smooth solutions.
Second, we establish difference estimates in $L^2$ for $C^2$-solutions of (SMCF), which subsequently guarantees uniqueness. To achieve this, we compare two distinct solutions, $\Sigma(F)$ and $\tilde{\Sigma}(\tilde{F})$, in extrinsic form and define an $L^2$ distance between them.
For this we exploit the gauge freedom: the gauge for $\Sigma$ is left free, while the gauge for $\tilde{\Sigma}$ is chosen specifically so that the difference $|F - \tilde{F}|$ is controlled by its normal component $|\omega|$. Furthermore, motivated by the structure of the linearized equations, we show that the normal component $\omega$ itself satisfies a Schr\"odinger-type linearized equation with additional quadratic terms. This yields a favorable Gr\"onwall's inequality for the $L^2$ distance, from which we obtain the desired difference estimates and hence the uniqueness result.

\medskip

{\bf II. The orthonormal frame and regularized initial manifolds.} 
The Section~\ref{Sec-Ini} is devoted to an analysis of the initial data conditions. First, we fix the gauge for the normal bundle, which presents greater complexity than in the small data case of \cite{HT21,HT22}.
Here it suffices to construct a global normal frame on a smooth reference manifold
$\Sigma_\bb$, as $\Sigma_0$ can be regarded as a small perturbation of $\Sigma_\bb$. To achieve this, proceed in two parts.
Inside a large ball $B_{x_0}(R+1)$, we obtain an interior frame $\nu^{(int)}$ by parallel transport of an orthonormal frame from a fixed point $x_0$. Conversely, outside a large ball $B^c_{x_0}(R)$, we obtain an exterior frame $\nu^{(ext)}$ following the method in \cite{HT22}, which leverages the small $L^\infty$ variation of the tangent frame $\d_x F_\bb$.
The global orthonormal frame is then constructed by gluing $\nu^{(int)}$ and $\nu^{(ext)}$ and using the topology of $\Sigma_\bb$ together with an appropriate lifting criterion. Moreover, applying a rotation to this frame and imposing the modified Coulomb gauge condition yields a well-controlled smooth orthonormal frame $\nu_\bb$ in $N\Sigma_\bb$.
Second, we bound the initial data for $\lambda$, $g$, and $A$ in the spaces $X^s$, $Y^{s+1}$, and $Z^s$, respectively. Here, we construct a family of continuous regularized manifolds $\Sigma^{(h)}$ via Littlewood-Paley projection, with carefully chosen gauges. For these manifolds, we prove the norm equivalences $X^s \sim H^s$, $Y^{s+1} \sim H^{s+1}$, and $Z^s \sim H^s$ at initial time, and establish difference estimates and high-frequency bounds that are propagated by the (SMCF) flow.

\medskip

{\bf III. Energy estimates.}
In Sections~\ref{Sec-Para} and~\ref{Sec-EnEs}, we prove energy estimates for the coupled Schr\"odinger-parabolic system in low-regularity function spaces. Note that, since the second fundamental form $\la$ is propagated well in our intrinsic-type spaces $X^s$, the nontrapping condition required in \cite{MMT5} is not needed here. At the same time, in order to extend our solution later, a key step is to show that the Sobolev embedding conditions required for the estimates on noncompact manifolds remain valid. Furthermore, for a family of regularized solutions, we establish difference and high-frequency bounds, which are then used to establish convergence in the strong topology.

\medskip

{\bf IV. The existence of regular solutions.}
In Section~\ref{Sec-RegSol}, we construct the regular solutions using a time discretization method via Euler-type iterative scheme, which originally appeared in the context of semigroup theory, see e,g.  \cite{Barbu}. This method was then implemented in studying the compressible Euler equations in a physical vacuum by Ifrim-Tataru \cite{IFEuler} (see also the expository paper \cite{IT-primer} for an outline of the principle).
However, a naive implementation of Euler's method loses derivatives; to rectify this we precede the Euler step by a suitable regularization based on a Willmore-type heat flow, with spatial truncation frequency scale set to $\ep^{-1/4}$. This regularization scale is needed in order to be able to bound the error in the Euler step. In addition, we prove that the Sobolev embedding conditions are preserved throughout the construction, which allows us to establish the energy estimates. 

We note that our construction is very different from any other approaches previously
used in analyzing this problem; they all relied on parabolic regularizations.

\medskip

{\bf V. Rough solutions as limits of regular solutions.}
The last section of the paper aims
to construct rough solutions as strong limits of smooth solutions. This is achieved by considering a family of continuous regularizations of the initial data, which generates corresponding smooth solutions $F^{(h)}$ on a time interval $[0,T]$ that is independent of $h$. 
For these smooth solutions, we first control the $L^2$-type distance between consecutive ones, using the energy estimates for the linearized equations in Proposition~\ref{prop4.2}. This establishes the existence of a rough solution as the limit in $L^2$.
Second, we control the higher Sobolev norms $H^{N+2}$ using our energy estimates.
By combining these bounds with the frequency envelopes technique, we obtain the strong convergence in $H^{s+2}$.
A similar argument yields continuous dependence of the solutions in terms of the initial data also in the strong
topology, as well as our main continuation result in Theorem~\ref{LWP-thm}. We can also refer to \cite{ABITZ} for an abstract theory.

\bigskip

\section{The differentiated equations and gauge choices}
\label{Sec-gauge}

The goal of this section is to review the derivation of our differentiated equations under suitable gauge conditions, as in \cite[Section 2]{HT22}. These equations involve the main independent variable $\la$, which represents 
the second fundamental form in complex notation, as well as the 
following auxiliary variables: the metric $g$ and the connection coefficients $A$ for the normal bundle. 
Finally, we conclude the section with a gauge formulation of our main result, see Theorem~\ref{LWP-MSS-thm}.

\subsection{Notations and the compatibility conditions}\label{sec-2.1}

Let $(\Sigma^d,g)$ be a $d$-dimensional oriented manifold and let $(\R^{d+2},g_{\R^{d+2}})$ be $(d+2)$-dimensional Euclidean space. Let $\al,\be,\ga,\cdots \in\{1,2,\cdots,d\}$. Considering the immersion $F:\Sigma\rightarrow (\R^{d+2},g_{\R^{d+2}})$, we obtain the induced metric $g$, its inverse and the Christoffel symbols on $\Sigma$,
\begin{equation}        \label{g_metric}
	g_{\al\be}=\d_{x_{\al}} F\cdot \d_{x_{\be}} F,\quad (g^{\al\be})=(g_{\al\be})^{-1},\quad \Gamma^{\ga}_{\al\be}=\ g^{\ga\si}\Ga_{\al\be,\si}=\ g^{\ga\si}\d^2_{\al\be}F\cdot\d_\si F.
\end{equation}
Let $\nab$ be the cannonical Levi-Civita connection on $\Si$ associated with the induced metric $g$.

Next, we introduce a complex structure on the normal bundle $N\Sigma_t$. This is achieved 
by choosing $\{\nu_1,\nu_2\}$ to be an orthonormal basis of $N\Si_t$ such that
\[
J\nu_1=\nu_2,\quad J\nu_2=-\nu_1.  
\]
Such a choice is not unique; in making it we introduce a second component to our gauge group,  
namely the group of sections of an $SU(1)$ bundle over $I \times \R^d$. We  also complexify the normal frame $\{\nu_1,\nu_2\}$ as 
\[m=\nu_1+i\nu_2.\] 
Then the vectors $\{ F_1,\cdots, F_d,\nu_1,\nu_2\}$ form a frame at each point on the manifold $(\Sigma,g)$, where $F_\al$ is defined by
\[F_\al:=\d_\al F.\]

We define the tensors $\kappa_{\al\be},\ \tau_{\al\be}$, the connection coefficients $A_{\al}$ and the temporal component $B$ of the connection in the normal bundle by
\[
\kappa_{\al\be}:=\d^2_{\al\be} F\cdot\nu_1,\quad \tau_{\al\be}:=\d^2_{\al\be} F\cdot\nu_2,\quad A_{\al}=\d_{\al}\nu_1\cdot\nu_2,\quad B=\d_t \nu_1\cdot \nu_2.
\]
Then we obtain the complex second fundamental form $\lambda$ and the mean curvature $\psi$ by 
\begin{equation*} \lambda_{\al\be}=\kappa_{\al\be}+i\tau_{\al\be},\quad \psi:=\tr\la=g^{\al\be}\la_{\al\be}.
\end{equation*}
We remark that the action of sections of the $SU(1)$ bundle is given by 
\begin{equation}   \label{gauge-A}
\psi \to e^{i \theta}\psi, \quad \lambda \to e^{i \theta }\lambda, \quad m \to e^{i\theta} m,
\quad A_\alpha \to A_\alpha - \partial_\alpha \theta,
\end{equation}
for a real valued function $\theta$.
\medskip

Our first objective for this section will be to interpret the (SMCF) equation as 
a nonlinear Schr\"odinger evolution for $\la$, by making suitable gauge choices.

We begin by expressing the Ricci curvature and compatibility conditions in terms of $\la$. Precisely, if we differentiate the frame, we obtain a set of structure equations of the following type 
\begin{equation}\label{strsys-cpf}
\left\{\begin{aligned}
&\d^2_{\al\be} F=\Ga_{\al\be}^\ga F_\ga+\Re(\lambda_{\al\be}\bar{m})\,,\\
&\d_{\al}^A m=-\lambda^{\ga}_{\al}F_\ga\,,
\end{aligned}\right.
\end{equation}
where 
$\d_{\al}^A=\d_{\al}+iA_{\al}$.
The Ricci formula $[\nab_{\al},\nab_{\be}]\d_\ga F=R(\d_\al F,\d_\be F)\d_\ga F$, combined with structure equations \eqref{strsys-cpf}, yield
the Riemannian curvature and Ricci curvature
\begin{align}          \label{R-la}
R_{\si\ga\al\be}=\Re(\lambda_{\be\ga}\bar{\la}_{\al\si}-\la_{\al\ga}\bar{\la}_{\be\si}),\quad \Ric_{\ga\be}=\Re (\la_{\ga\be}\bar{\psi}-\la_{\ga\al}\bar{\la}_{\be}^{\al}),
\end{align}
and the compatibility condition
\begin{equation}         \label{la-commu}
\nab^A_{\al} \la_{\be\ga}=\nab^A_{\be}\la_{\al\ga}.
\end{equation}
From the relation $[\nab^A_{\al},\d^A_{\be}]m =i(\d_\al A_\be-\d_\be A_\al)m$, we could obtain \eqref{la-commu} again as well as
\begin{equation}\label{cpt-AiAj-2}
\nab_{\al} A_{\be}-\nab_{\be} A_{\al}=\Im(\la^{\ga}_{\al}\bar{\la}_{\be\ga}),
\end{equation}
where the latter can be seen as the complex form of the Ricci equations.

\subsection{The evolutions of metric  \texorpdfstring{$g$}{}, connection \texorpdfstring{$A$}{} and the second fundamental form \texorpdfstring{$\la$}{} under (SMCF)} 
Here we start with deriving the equations of motion for the frame, assuming that the immersion $F$ satisfying \eqref{Main-Sys}. Then this will yield the main Schr\"odinger equation for $\la$, as well as the evolutions of metric $g$ and the curvature relation. 

Under the frame $\{F_1,\cdots,F_d,m\}$, we rewrite the (SMCF) equations in the form
\begin{equation}   \label{sys-cpf}
\d_t F=J(F)\bH(F)+V^{\ga} F_{\ga}=-\Im (\psi\bar{m})+V^{\ga} F_\ga,
\end{equation}  
where $V^{\ga}$ is a vector field on the manifold $\Sigma$, which 
in general depends on the choice of coordinates. 
Then, applying $\d_\al$ to \eqref{sys-cpf}, by the structure equations \eqref{strsys-cpf} and the orthogonality relation $m\bot F_{\al}$ we obtain the following equations of motion for the frame
\begin{equation}              \label{mo-frame}
\left\{\begin{aligned}
&\d_t F_\al=-\Im (\d^A_{\al} \psi \bar{m}-i\la_{\al\ga}V^{\ga} \bar{m})+[\Im(\psi\bar{\la}^{\ga}_{\al})+\nab_{\al} V^{\ga}] F_\ga,\\
&\d^{B}_t m=-i(\d^{A,\al} \psi -i\la^{\al}_{\ga}V^{\ga} ) F_\al,
\end{aligned}\right.
\end{equation}
where we use the covariant time derivative $\d_t^B=\d_t +iB$.

From \eqref{mo-frame} we can derive the evolution equations for the metric $g$, the connection $A$ and the second fundamental form $\la$ directly. Indeed, by the definition of the induced metric $g$ \eqref{g_metric}  and \eqref{mo-frame}, we have
\begin{align}         \label{g_dt}
\d_t g_{\al\be}
=2\Im(\psi\bar{\la}_{\al\be})+\nab_{\al}V_{\be}+\nab_{\be}V_{\al}.
\end{align}
So far, the choice of $V$ has been unspecified; it depends on the choice of coordinates 
on our manifold as the time varies. 

Next, from the commutation relation
$[\d^{B}_t,\d^A_{\al}]m=i(\d_t A_{\al}-\d_{\al} B)m$,
by equating the tangential component we obtain the evolution equation for $\la$
\begin{equation*}
\d^{B}_t\la^{\si}_{\al}+\la^{\ga}_{\al}(\Im(\psi\bar{\la}^{\si}_{\ga})+\nab_{\ga} V^{\si})=i\nab^A_{\al}(\d^{A,\si} \psi -i\la^{\si}_{\ga}V^{\ga} ),
\end{equation*}
which yields the main Schr\"odinger equation for $\lambda$ by using the relations \eqref{la-commu} and \eqref{R-la},
\begin{equation}     \label{Sch-la}
\begin{aligned}
i(\d^{B}_t-V^\ga\nab^A_\ga)\la_{\al\be}
+\nab^A_\si\nab^{A,\si}\la_{\al\be}
&= i\la^{\ga}_{\al}\nab_{\be} V_{\ga}
+i\la_\be^\ga\nab_\al V_\ga
+\psi\Re(\la_{\al\de}\bar{\la}^\de_\be)\\
&\quad -\Re(\la_{\si\de}\bar{\la}_{\al\be}-\la_{\si\be}\bar{\la}_{\al\de})\la^{\si\de}-\la_{\al\mu}\bar{\la}^\mu_\si\la^\si_\be
.\end{aligned}
\end{equation}
By equating the normal components, we also obtain the compatibility condition (curvature relation) 
\begin{equation}\label{Cpt-A&B}
\d_t A_{\al}-\d_{\al} B = \Re(\la_{\al}^{\ga}\bar{\d}^A_{\ga}\bar{\psi})-\Im (\la^\ga_\al \bar{\la}_{\ga\si})V^\si.
\end{equation}

In addition, from \eqref{R-la}, \eqref{cpt-AiAj-2}, \eqref{g_dt} and \eqref{Cpt-A&B} we have the commutators
\begin{align}   \label{comm-nab}
&[\nab^A_\al,\nab^A_\be]=R+i\Im(\la_\al^\ga \bar{\la}_{\be\ga})\approx \la\ast \la,\\\label{comm-dt}
&[\nab^A,\d_t^B]=\nab \d_t g+i(\nab_\al B-\d_t A_\al)\approx \la\ast \nab^A \la +\nab^2 V+\la^2 V.
\end{align}

\subsection{The background manifold \texorpdfstring{$\Sigma_{\mathbf{b}}$}{}}
Here we introduce a smooth background manifold $\Sigma_\bb$,  which is a small perturbation of the initial manifold, so that for a short time the manifold $\Sigma_t$ can be seen as a small perturbation of this background manifold. This will be used later in order to construct the orthonormal frame in $\Sigma$.

Begin with the fixed initial map $F_0:\R^d\rightarrow \R^{d+2}$ with metric $g_0$ and the mean curvature $\bH_0$, and satisfying \eqref{MetBdd} and \eqref{small-datad3}. Let $N_1$ be chosen, depending on $M$, $c_0$, and $C_0$, to be sufficiently large so that $\ep_0:=2^{-N_1} \ll_M 1$. 
We decompose $F_0$ as $F_0=P_{\leq N_1}F_0+P_{>N_1}F_0$, where the frequency cutoff $N_1$ is a large parameter, to be chosen so that 
the second component is sufficiently small. 
We denote the background map $F_{\mathbf{b}}$ and corresponding background manifold $\Si_{\mathbf{b}}$ as 
\begin{equation*}
	F_{\mathbf{b}}=P_{\leq N_1}F_0,\qquad \Sigma_{\mathbf{b}}=F_{\mathbf{b}}(\R^d),
\end{equation*}
whose \emph{global coordinates} are fixed, given by $(x^1_{\mathbf{b}},\cdots,x^d_{\mathbf{b}})$.
Assuming $\ep_0$ is small enough, $F_{\mathbf{b}}$ is an immersion with  $\d^2_x F_{\mathbf{b}}\in \cap_{k=1}^\infty H^k$, the metric $g_\bb$ remains elliptic,
and the background manifold $\Sigma_{\mathbf{b}}$ is a smooth manifold.
From the above definition, we have the metric $g_\bb$ given by 
\begin{equation*}
g_{{\mathbf{b}},\al\be}=\d_{x^\al_{\mathbf{b}}}F_{\mathbf{b}}\cdot \d_{x^\be_{\mathbf{b}}}F_{\mathbf{b}},\qquad g_{\mathbf{b}}-I\in H^k.
\end{equation*}
We note that the bounds for $g_{\mathbf{b}}$ depend on the frequency cutoff $N_1$ and $k$.

On the smooth manifold $\Sigma_\bb$ we can construct a smooth orthonormal frame in $N\Sigma_\bb$. Then we obtain a fixed gauge by imposing the modified Coulomb gauge condition
\begin{align*}
	\d_\al A_{\bb,\al}=0.
\end{align*} 
The gauge condition will allow us to bound the Sobolev norms for connection $A_\bb$ and the second fundamental form $\la_\bb$ in terms of the initial data size $M$ and $\epsilon_0$, see Lemma \eqref{Lem-Gaugebb}.

\subsection{The gauge choices}    
Here we take the first step towards fixing the gauge, by choosing to work in the original coordinates at $t=0$ while using \emph{heat coordinates} for $t>0$. Precisely,
at the initial time $t=0$ we will not change the coordinates and instead adopt the original coordinates.
For later times $t>0$ we introduce the \emph{heat gauge}, where we
require the coordinate functions $\{x^\al,\al=1,\cdots,d\}$ to be global Lipschitz solutions of the  heat equations
\begin{equation*}
(\partial_t - \Delta_g - V^\gamma \partial_\gamma) x^\alpha = 0.
\end{equation*}
This can be expressed in terms of the Christoffel symbols $\Ga$, namely,
\begin{equation} \label{Heat-coordinate}
g^{\al\be}\Ga^\ga_{\al\be}=V^\ga.
\end{equation}
Once a choice of coordinates is made at the initial time, the coordinates will be uniquely determined later by this gauge condition.

With the advection field $V$ fixed via the heat coordinate condition \eqref{Heat-coordinate}, we can derive a parabolic equation for the metric $g$, see \cite[Lemma 2.4]{HT22}:
\begin{equation} \label{g-Heat-Eq}
\begin{aligned}  
\d_t g_{\mu\nu}-g^{\al\be}\d^2_{\al\be}g_{\mu\nu}=&\ 2\Re (\la_{\mu\nu}\bar{\psi}-\la_{\mu\si}\bar{\la}_{\nu}^{\si})+2\Im(\psi\bar{\la}_{\mu\nu})-2g^{\al\be}\Ga_{\mu\be,\si}\Ga^\si_{\al\nu}\\
&\ +\d_{\mu}g^{\al\be}\Ga_{\al\be,\nu}+\d_{\nu}g^{\al\be}\Ga_{\al\be,\mu}\,.
\end{aligned}
\end{equation}

Now we take the next step towards fixing the gauge, and consider the choices of the orthonormal frame in normal bundle $N\Sigma$. Our starting point is provided by the curvature relations \eqref{cpt-AiAj-2} at fixed time, 
respectively \eqref{Cpt-A&B} dynamically, together with the gauge group \eqref{gauge-A}. We will fix the gauge in two steps, first in a static, elliptic fashion at the initial time, and then dynamically, using a heat flow, for later times.  

At the initial time $t=0$ we fix the gauge for $A$ by imposing the generalized Coulomb gauge condition
\begin{equation}\label{Coulomb}
	\nab^\al A_\al=\nab^\al A_{\bb,\al}\ ,
\end{equation}
where $A_\bb$ are the connection coefficients on $N\Sigma_\bb$. We remark that the condition \eqref{Coulomb} is only used to obtain a good orthonormal frame on $N\Sigma_0$.

For later times $t>0$, we adopt the heat gauge to propagate the orthonormal frame,
\begin{equation} \label{heat-gauge}
	\nab^\al A_\al=B.
\end{equation} 
Then, as in \cite[Lemma 2.2]{HT22}, we obtain the parabolic equation for $A$
\begin{equation}\label{Heat-A-pre}
\begin{aligned}
(\d_t-\De_g)A_\al =&\ 
\nab^\si\Im(\la^\ga_\al \bar{\la}_{\si\ga})+\nab_\ga \Re(\la_{\al}^\ga \bar{\psi})-\frac{1}{2}\nab_\al |\psi|^2\\
&\ -\Re(\la_\al^\si\bar{\psi}-\la_{\al\be}\bar{\la}^{\be\si})A_\si-\Im (\la^\ga_\al \bar{\la}_{\ga\si})V^\si\,.
\end{aligned}
\end{equation}

\subsection{The modified Schr\"{o}dinger system}
Here we carry out the last step in our analysis of the equations, and 
state the main result in a suitable gauge. 
 
In conclusion, under the heat coordinate condition \eqref{Heat-coordinate} and heat gauge condition \eqref{heat-gauge}, by \eqref{Sch-la}, \eqref{g-Heat-Eq} and \eqref{Heat-A-pre}, we obtain the covariant Schr\"{o}dinger equation for the complex second fundamental form tensor $\la$
\begin{equation}        \label{mdf-Shr-sys-2}
\left\{
\begin{aligned}
&    \begin{aligned}
i(\d^{B}_t-V^\ga\nab^A_\ga)\la_{\al\be}
+\nab^A_\si\nab^{A,\si}\la_{\al\be}
&= 
i\la^{\ga}_{\al}\nab_{\be} V_{\ga}
+i\la_\be^\ga\nab_\al V_\ga
+\psi\Re(\la_{\al\de}\bar{\la}^\de_\be)\\
&\quad -\Re(\la_{\si\de}\bar{\la}_{\al\be}-\la_{\si\be}\bar{\la}_{\al\de})\la^{\si\de} 
-\la_{\al\mu}\bar{\la}^\mu_\si\la^\si_\be\,,
\end{aligned}\\
    & \la(0,x) = \la_0(x).
    \end{aligned}
\right.    
\end{equation}
These equations are fully covariant, and do not depend on the gauge choices
made earlier. On the other hand, our gauge choices imply that 
the advection field $V$ and the connection coefficient $B$ are determined by the metric $g$ and connection $A$ via \eqref{Heat-coordinate}, respectively, \eqref{heat-gauge}. In turn,
the metric $g$ and the connection coefficients $A$  are determined in a parabolic fashion via the following equations for $g_{\mu\nu}$ and $A_\al$
\begin{equation}           \label{par-syst}
\left\{\begin{aligned}
&\begin{aligned}  
\d_t g_{\mu\nu}-g^{\al\be}\d^2_{\al\be}g_{\mu\nu}=&\ 2\Re (\la_{\mu\nu}\bar{\psi}-\la_{\mu\si}\bar{\la}_{\nu}^{\si})+2\Im(\psi\bar{\la}_{\mu\nu})-2g^{\al\be}\Ga_{\mu\be,\si}\Ga^\si_{\al\nu}\\
&\ +\d_{\mu}g^{\al\be}\Ga_{\al\be,\nu}+\d_{\nu}g^{\al\be}\Ga_{\al\be,\mu}\,.
\end{aligned}\\
&\begin{aligned}
(\d_t-\De_g)A_\al &=
\nab^\si\Im(\la^\ga_\al \bar{\la}_{\si\ga})+\nab_\ga \Re(\la_{\al}^\ga \bar{\psi})-\frac{1}{2}\nab_\al |\psi|^2\\
&\quad -\Re(\la_\al^\si\bar{\psi}-\la_{\al\be}\bar{\la}^{\be\si})A_\si-\Im (\la^\ga_\al \bar{\la}_{\ga\si})V^\si\,.
\end{aligned}\\
&V^\ga=g^{\al\be}\Ga^\ga_{\al\be}\,,\quad B=\nab^\al A_\al\,,
\end{aligned}\right.
\end{equation}
with initial data 
\begin{equation}    \label{ini-gA}
    g(0,x)=g_0,\quad A_\al(0,x)=A_0.
\end{equation}
These are determined at the initial time by using the original coordinates on $\Sigma_0$,  respectively the generalized Coulomb gauge for $A_0$.

Fixing the remaining degrees of freedom (i.e. the affine group for the choice of the 
coordinates as well as the time dependence of the $SU(1)$ connection) 
 we can assume that the  following conditions hold at infinity in an averaged sense:
 \begin{equation*}  
g(\infty) = I_d,\quad A(\infty) = 0.     
 \end{equation*}
These are needed to insure  the unique solvability of the above parabolic equations
in a suitable class of functions.

We have now reduced the problem to the main Schr\"odinger-Parabolic system \eqref{mdf-Shr-sys-2}-\eqref{par-syst}. This system will be the key to proving the large-data solvability of the (SMCF) system in low-regularity Sobolev spaces, which is the primary objective of the rest of this paper.

Now we can restate here the large data local well-posedness result for the (SMCF) system in 
Theorem~\ref{LWP-thm} in terms of the above system:

\begin{thm}[Local well-posedness for large data in the good gauge]   \label{LWP-MSS-thm}
Let $d\geq 2$ and $s>\frac{d}{2}$. Assume that the initial manifold $\Sigma_0$ satisfies \eqref{MetBdd} and the bounds 
\begin{equation*}
    \lV \la_0\rV_{X^s}+ \lV g_0\rV_{Y^{s+1}}+ \lV A_0\rV_{Z^{s}} \leq M_1,
\end{equation*}
where $M_1=C(M)$ depends on $M$.
Then there exists $T=T(M,c_0)$ sufficiently small such that the (SMCF)
is locally well-posed in $X^s\times Y^{s+1}\times Z^s$ on the time interval $I=[0,T(M,c_0)]$. Moreover, the second fundamental form $\la$, the metric $g$ and the connection coefficients $A$ satisfy the bounds
\begin{equation}\label{psi-full-reg}
    \lV \la\rV_{L^\infty([0,T]; X^s)}\leq 2M_1,\qquad   \lV (g,A)\rV_{L^\infty([0,T]; Y^{s+1}\times Z^s)}\leq 2 M_1.
\end{equation}
\end{thm}

The function space $X^s$ appearing in the theorem is defined in the next
section, as a fractional counterpart of the intrinsic Sobolev norms  $\sfH^k$-norm of $\la$, which propagates well for any integer $k$. Thus here we use this property to define the fractional $X^s$-norm whose characterization is akin to Littlewood-Paley decomposition. This allows us to establish energy estimates more easily and without requiring a non-trapping condition.
In addition, the $X^s$-norm of $\la$ is equivalent to its $H^s$-norm at initial time and controls the $H^s$ norm of the solution for later times. Consequently, the $H^s$-norm of $\la$ is controlled by initial data.
For consistency, the corresponding $Y^{s+1}$-norms and $Z^s$-norms for $g$ and $A$ are also needed; these norms enjoy similar embedding and boundedness properties.

In the above theorem,  by well-posedness we mean a full Hadamard-type well-posedness, including the following properties:
\begin{enumerate}[label=\roman*)]
    \item Existence of solutions $\la \in C[0,1;H^s]$, with the additional regularity 
    properties \eqref{psi-full-reg}.
    \item Uniqueness in the same class.
    \item Continuous dependence of solutions with respect to the initial data 
    in the strong $H^s$ topology.
    \item Weak Lipschitz dependence  of solutions with respect to the initial data 
    in the weaker $L^2$ topology.
    \item Energy bounds and propagation of higher regularity.
\end{enumerate}

\bigskip

\section{Function spaces and notations}   \label{Sec3}
The goal of this section is twofold. First, we introduce some notations as well as  inequalities on non-compact manifolds. Second, we define the function spaces where we aim to solve 
the (SMCF) system in the good gauge, given by \eqref{mdf-Shr-sys-2} and \eqref{par-syst}.
In particular, we introduce a new function space $X^s$, different from the spaces introduced in \cite{MMT3,MMT4,MMT5}, so that we could propagate the regularity of the second fundamental form $\la$ along (SMCF) in some fractional Sobolev spaces.

\subsection{Notations and some properties on manifolds}
We begin with some constants. 
We denote $[a]$ as the largest integer such that $[a]\leq a\in \R$.
Let regularity index $s>d/2$ and $0<\de\ll 1$ be a small constant. Let the $\si_d$ and $\de_d$ be as
\begin{equation} \label{Constant}
\si_d=1-\delta_d,\qquad \de_d=\left\{\begin{aligned}
        \de,\quad d=2,\\
        0,\quad d\geq 3.
    \end{aligned}\right.
\end{equation}
Let $N_1>0$ and $h_0>0$ be sufficiently large such that 
\[\ep_0:= 2^{-N_1}\sim 2^{-h_0}\ll_M 1\,.\]

For a function $u(t,x)$ or $u(x)$, let $\hat{u}=\mathcal F u$ and $\check{u}=\mathcal F^{-1}u$ denote the Fourier transform and inverse Fourier transform in the spatial variable $x$, respectively. Fix a smooth radial function $\varphi:\R^d \rightarrow [0,1] $ supported in $\{x:|x|\leq 2\}$ and equal to 1 in $\{x:|x|\leq 1\}$, and for any $j\in \Z$, let
\begin{equation*}
	\varphi_j(x):=\varphi(x/2^j)-\varphi(x/2^{j-1}).
\end{equation*}
We then have the spatial Littlewood-Paley decomposition,
\begin{equation*}
	\sum_{j=-\infty}^{\infty}P_j (D)=1, \qquad \sum_{j=0}^{\infty}S_j (D)=1,
\end{equation*}
where $P_j$ localizes to frequency $2^j$ for $j\in \Z$ with $\mathcal F(P_j u)=\varphi_j(\xi)\hat{u}(\xi)$, 
$S_0(D)=\sum_{j\leq 0}P_j(D)$ and $S_j(D)=P_j(D)$ for $j>0$.

\begin{lemma}
Let $k\in\Z$, $1\leq q\leq p\leq \infty$ and $s>\frac{d}{2}$. We have
\begin{align}\nonumber
&\lV P_k f\rV_{L^p}\lesssim 2^{kd(\frac{1}{q}-\frac{1}{p})}\lV P_k f\rV_{L^q},\\ \label{ineqfg}
&\|fg\|_{H^s}\les \|f\|_{H^s}(\|g\|_{L^\infty}+\|P_{>0}g\|_{\dot H^s}).
\end{align}
\end{lemma}
\begin{proof}
The first one is Bernstein's inequality. The second one \eqref{ineqfg} is easily obtained using a paradifferential decomposition.
\end{proof}

Alternatively we will also use a continuous Littlewood-Paley decomposition
\begin{align}  \label{CLPD}
    1=\int_\R P_h\  dh= P_{<h_0} + \int_{h_0}^\infty P_h\  dh,
\end{align}
where the symbols $p_h(\xi)$ of $P_h$ are localized in the region ${2^{h-1} <|\xi|<
2^{h+1}}$ and coincide up to scaling,
\begin{align*}
    p_h(\xi)=p_0(2^{-h}\xi).
\end{align*}
We define
\begin{equation*}
    P_{<h}=\int_{-\infty}^h P_l\ dl,\quad P_{>h}=\int^{+\infty}_h P_l\ dl.
\end{equation*}
Then we have the equivalence
\begin{equation*}
\|u\|_{H^s}^2 \approx \|P_{<h_0}u\|_{H^s}^2+\int_{h_0}^\infty 2^{2hs}\|P_h u\|_{L^2}^2 dh.
\end{equation*}

Now we define the standard Sobolev spaces $H^s$ for any $s\in \R$, which is given by
\begin{align*}
    \|u\|_{H^s}=\|\<\xi\>^s \hat{u}(\xi)\|_{L^2}.
\end{align*} 
For the metric $g_{\al\be}$ and connection $A_\al$, we will use the function spaces
\begin{align*}
\|(g, A)\|_{\EE^s}&=\||D|^{\si_d}g\|_{L^\infty([0,T]; H^{s+1-\si_d})}+\||D|^{1+\si_d}g\|_{L^2([0,T]; H^{s+1-\si_d})}\\
&\quad +\||D|^{\de_d}A\|_{L^\infty([0,T]; H^{s-\de_d})}+\||D|^{1+\de_d} A\|_{L^2([0,T]; H^{s-\de_d})}.
\end{align*}
Ideally here one would like to set $\delta_d = 0$, but this is only possible in dimensions three and higher due to the construction of orthonormal frame in $N\Sigma$.

We also need the intrinsic Sobolev spaces on a smooth manifold $(\MM,g)$. Since the Schr\"odinger equation \eqref{mdf-Shr-sys-2} is a quasilinear equations with variable coefficients $g$, the intrinsic Sobolev spaces are effective to derive its energy estimates later. Let $A_\ga$ be a magnetic potential.
For any complex tensor $T=T^{\al_1\cdots\al_r}_{\be_1\cdots\be_s}dx^{\beta_1}\otimes...dx^{\beta_s}\otimes\frac{\partial }{\partial x^{\alpha_1}}\otimes...\otimes\frac{ \partial }{\partial x^{\alpha_r}}$, the covariant derivative is defined by
\[\nab^A_\ga T=\nab_\ga T+iA_\ga T,\]
where 
\begin{align*}
\nab_\ga T^{\al_1\cdots\al_r}_{\be_1\cdots\be_s}=\d_\ga T^{\al_1\cdots\al_r}_{\be_1\cdots\be_s}+\sum_{i=1}^r \Ga^{\al_i}_{\ga\si} T^{\al_1\cdots\al_{i-1}\si\al_{i+1}\cdots\al_r}_{\be_1\cdots\be_s}-\sum_{j=1}^s \Ga^{\si}_{\ga\be_{j}} T^{\al_1\cdots\al_r}_{\be_1\cdots\be_{j-1}\si\be_{j+1}\cdots\be_s}.
\end{align*}
We have
\begin{align*}
|\nab^A T|^2_g=g_{\al_1\al'_1}\cdots g_{\al_r\al'_r} g^{\be_1\be'_1}\cdots g^{\be_s\be'_s} \nab^A_\ga T^{\al_1\cdots\al_r}_{\be_1\cdots\be_s}\overline{\nab^{A,\ga} T^{\al'_1\cdots\al'_r}_{\be'_1\cdots\be'_s}}.
\end{align*}
Then the intrinsic Sobolev norm $\mathsf H^k$ for nonnegative integer $k\in \mathbb N$ is defined by
\begin{equation}     \label{IntSob}
\|T\|_{\mathsf H^k}^2=\sum_{l=0}^k \int_\MM|\nab^{A,l} T|_g^2 \ dvol\,,
\end{equation}
where volume form is $dvol=\sqrt{\det g}dx$ and $\nab^{A,l}$ is the $l$-th order covariant derivatives. For convenience, we also define the associated $\sfL^p$-norm and $\sfH^{k,p}$ as
\begin{equation*}
\|T\|_{\sfL^p}^p=\int_\MM|T|_g^p \ dvol,\qquad \|T\|_{\sfH^{k,p}}^p=\sum_{l=0}^k\int_\MM|\nab^{A,l}T|_g^p \ dvol\,.
\end{equation*}
We denote by $C^m_B(\MM)$ the space of $C^m$ functions $u:\MM\rightarrow \R$ equipped with the finite norm
\[\|u\|_{C^m}=\sum_{j=0}^m \sup_x |\nab^j u|_g.\] 

Next, we state some inequalities on Riemannian manifolds. Let us first recall the following interpolation inequality proved by Hamilton \cite[Section 12]{Ha1982}.

\begin{thm}[Theorem 12.1, p.291\cite{Ha1982}]
Let $(\MM,g)$ be a $C^2$-Riemannian manifold without boundary of dimension $d$ and let $T$ be any tensor on $\MM$. Suppose $\frac{1}{p}+\frac{1}{q}=\frac{1}{r}$ with $r\geq 1$. Then 
\begin{equation*}
    \|\nab T\|_{\sfL^{2r}}^2\leq (2r-2+d)\|\nab^2 T\|_{\sfL^p}\|T\|_{\sfL^q}.
\end{equation*}
\end{thm}
\begin{rem}
Note that the Theorem 12.1 in \cite[p.291]{Ha1982} assumes the manifold $\mathcal{M}$ is compact. However, since the proof only relies on integration by parts, Theorem 12.1 still holds for smooth manifolds without boundary.
\end{rem}

As corollaries of this theorem, we have the following inequalities.
\begin{cor}[Corollary 12.6, p.293 \cite{Ha1982}]  \label{Inter-lem}
    If T is any tensor on the smooth manifold $(\MM,g)$ without boundary and if $1 \leq  i \leq  l-1$, then with a constant $C=C(d,l)$ depending only on dimensions $d=dim \ \MM$ and $l$, which is independent of the metric $g$ and the connection $\Ga$, we have the estimate
    \begin{equation}   \label{HaInterpo}
        \int_{\R^d} |\nab^i T|^{\frac{2l}{i}} \ d\mu\leq C \max_{\MM} |T|_{g}^{2(\frac{l}{i}-1)}\int_{\R^d} |\nab^l T|^2 \ d\mu.
    \end{equation}
\end{cor}
\begin{cor}[Corollary 12.7, p.294 \cite{Ha1982}]
    If $T$ is any tensor on the smooth manifold $(\MM,g)$ without boundary then with a constant $C=C(n,d)$ depending only on $n$ and $d = dim \ \MM$  and independent of the metric $g$ and the connection $\Ga$ we have the estimate 
    \begin{equation}  \label{L2interpolation}
        \|\nab^i T\|_{\sfL^2}\leq C\|\nab^n T\|_{\sfL^2}^{\frac{i}{n}}\|T\|_{\sfL^2}^{1-\frac{i}{n}}, \qquad 0\leq i\leq n.
    \end{equation}
\end{cor}

We then state the Sobolev embedding theorem for noncompact manifolds, which play a crucial role in constructing regular solutions. 
\begin{thm}
[Theorem 3.4, p.63 \cite{Hebey1999}]
\label{thm-SobEm}
Let $(\MM,g)$ be a smooth, complete Riemannian manifold of dimension $d$ with Ricci
curvature bounded from below.
Assume that 
\[\inf_{x\in M} {\rm Vol}_g(B_x(1))>0,\] 
where ${\rm Vol}_g(B_x(1))$ stands for the volume of $B_x(1)$ with respect to $g$. Given $p\geq 1$ and $m<k-\frac{d}{p}$, we have that $\sfH^{k,p}(\MM)\subset C^m_B(\MM)$, and the embedding is continuous. 
\end{thm}

We also need the following estimates concerning volumes, which are a corollary of Gromov's volume comparison theorem in \cite[Theorem 1.1, p.11]{Hebey1999}.
\begin{lemma}[p.12 \cite{Hebey1999}]
Let $(\MM,g)$ be a smooth, complete Riemannian manifold of dimension $d$ with Ricci
curvature satisfying $\Ric_{(\MM,g)}\geq kg$ for some $k$ real, then for any $x\in \MM$ and any $0<r<R$,
\begin{equation} \label{VolComp}
    {\rm Vol}_g(B_x(r))\geq e^{-\sqrt{(d-1)|k|}R} \left(\frac{r}{R} \right)^d {\rm Vol}_g(B_x(R)).
\end{equation}
\end{lemma}

\subsection{Function spaces}
Since in the Hilbertian case all interpolation methods yield the same result, for the $X^s$ norm we will use a characterization which is akin to a Littlewood-Paley decomposition, or to a discretization of the $J$ method of interpolation. 

Using the continuous Littlewood-Paley decomposition \eqref{CLPD}, we can regularize an immersed manifold $\Sigma=F(\R^d)$ and its orthonormal frame $(\nu_1,\nu_2)$ by 
\[  \Sigma^{(h)}=P_{<h}F_0(\R^d) ,\quad (\tilde\nu_1^{(h)},\tilde \nu_2^{(h)})=(P_{<h}\nu_1,P_{<h}\nu_2),     \]
where $h>h_0>0$ with $2^{-h_0}\sim \ep_0\ll_M 1$ such that the metric of $\Sigma^{(h)}$ is elliptic.
Then $\Sigma^{(h)}$ and $\Sigma$ are small perturbations of $\Sigma^{(h_0)}$. The orthonormal frame $(\nu_1^{(h)},\nu_2^{(h)})$ can be constructed from $(\tilde\nu_1^{(h)},\tilde \nu_2^{(h)})$ using projection and Schmidt orthogonalization. 
Thus we obtain a family of regularized $\la^{(h)}$, $g^{(h)}$ and $A^{(h)}$ on $\Sigma^{(h)}$, which are denoted as $[\la^{(h)}]$, $[g^{(h)}]$ and $[A^{(h)}]$ respectively. The regularization for the initial manifold will be implemented in Section~\ref{Reg-sec}.

Motivated by the above regularizations, we collect all of the regularizations as a set and define the fractional $X^s$-norm for the second fundamental form $\la$ on $\Sigma$. This norm can be propagated as it evolves along the (SMCF). We define the set of regularizations of $\la$ as smoothness in $h$,
\begin{align*}
    \text{Reg}(\lambda)=\big\{ [\lambda^{(h)}]: &\text{ the second fundamental form of regularized manifold } \Sigma^{(h)} \text{ of }\Sigma\\
    &h\in[h_0,\infty),\quad \lim_{h\rightarrow \infty}\|\la-\la^{(h)}\|_{H^s}=0\big\},
\end{align*}
as well as define the sets of regularizations of $g$ and $A$ as
\begin{align*}
&\begin{aligned}
    \text{Reg}(g)=\big\{ [g^{(h)}]: &\text{ the metric of regularized manifold } \Sigma^{(h)} \text{ of }\Sigma\\
    & h\in[h_0,\infty),\quad \lim_{h\rightarrow \infty}\||D|^{\si_d}(g-g^{(h)})\|_{H^{s+1-\si_d}}=0\big\},
\end{aligned}\\
&\begin{aligned}
    \text{Reg}(A)=\big\{ [A^{(h)}]: &\text{ the connection of regularized manifold } \Sigma^{(h)} \text{ of }\Sigma\\
    & h\in[h_0,\infty),\quad \lim_{h\rightarrow \infty}\||D|^{\de}(A-A^{(h)})\|_{H^{s-\de}}=0\big\}.
\end{aligned}
\end{align*}
Linearizing around $\Sigma^{(h)}$, we can define the linearized variables $\mu^{(h)}$ as 
\[
(\mu^{(h)})_{\al\be}=\d_h(\la^{(h)})_{\al\be},\qquad (\mu^{(h)})^{\al\be}=g^{(h)\al\si}g^{(h)\be\de}(\mu^{(h)})_{\si\de}.
\]
Then we define the fractional $X^s$-norm of $\la$ as follows.

\begin{defn} \label{def-Xsla}
Let $s>\frac{d}{2}$. For any regularization $[\la^{(h)}]\in \text{Reg}(\lambda)$, we define the intrinsic and extrinsic $s$-norm of $\lambda^{(h)}$ as 
\begin{align} \label{def-int}
\begin{aligned}
    \tri[\la^{(h)}]\tri_{s,\, int}^2=&\  
    2^{2h_0(s-[s])} \| \lambda^{(h_0)} \|_{\sfH^{[s]}}^2+2^{2h_0(s-[s]-1)} \| \lambda^{(h_0)} \|_{\sfH^{[s]+1}}^2\\
    +&\ \max_{N\in\{[2s], [2s]+1\}}\int_{h_0}^{+\infty} 
    2^{2h(s-N)}\| \la^{(h)} \|_{\sfH^N}^2 \ dh+\int_{h_0}^{+\infty} 
    2^{2hs} \|\d_h\la^{(h)}\|_{\sfL^2}^2 dh,
\end{aligned}\\ \label{def-ext}
\begin{aligned}
    \tri[\la^{(h)}]\tri_{s,\,ext}^2=&\  
    2^{2h_0(s-[s])} \| \lambda^{(h_0)} \|_{H^{[s]}}^2+2^{2h_0(s-[s]-1)} \| \lambda^{(h_0)} \|_{H^{[s]+1}}^2\\
    +&\ \max_{N\in\{[2s], [2s]+1\}}\int_{h_0}^{+\infty} 
    2^{2h(s-N)}\| \la^{(h)} \|_{H^N}^2 \ dh+\int_{h_0}^{+\infty} 
    2^{2hs} \|\d_h\la^{(h)}\|_{L^2}^2 dh.
\end{aligned}
\end{align}
Then we define the $X^s$-norm of $\lambda$ as
\begin{equation*}
\|\la\|_{X^s_{int}}=\inf_{[\lambda^{(h)}]\in \text{Reg}(\lambda)} \tri [\la^{(h)}]\tri_{s,\, int}\,,\qquad \|\la\|_{X^s_{ext}}=\inf_{[\lambda^{(h)}]\in \text{Reg}(\lambda)} \tri [\la^{(h)}]\tri_{s,\,ext}\,.
\end{equation*}
\end{defn}

\begin{rem}
i) Naively, $\mu^{(h)}=\d_h\la^{(h)}$ and $\lambda^{(h)}$ can be regarded as $\mu^{(h)} \sim P_h \lambda$ and $\lambda^{(h)} \sim P_{< h} \lambda$, respectively. In this sense, the $X^s$-norm is almost identical to the $H^s$-norm.
However, the equivalence between $X^s$ and $H^s$ does not always hold due to gauge choices.
In fact, $X^s$ is a smaller function space than $H^s$, i.e., $X^s \subset H^s$. 
ii) Due to its quasilinear structure, the intrinsic norm 
$X^s_{int}$ has better propagation properties along the (SMCF) flow, despite the equivalence of the 
$X^s_{int}$- and $X^s_{ext}$-norms in an appropriate gauge.
\end{rem}

To keep consistency, the corresponding $Y^s$-norms of $g$ and $Z^s$-norms of $A$ are also needed.
\begin{defn}
    We define the $s$-norm for metric $g$ and connection coefficients $A$ as 
\begin{align*}
    \tri [g^{(h)}]&\tri_{s+1,\, g}^2:=  \||D|^{\si_d}g^{(h_0)}\|_{H^{s+1-\si_d}}^2+\int_0^t \||D|^{1+\si_d}g^{(h_0)}(\tau)\|_{H^{s+1-\si_d}}^2 d\tau \\
    +&\max_{N\in\{[2s], [2s]+1\}}\int_{h_0}^\infty 2^{2h(s-N)}\big(\||D|^{\si_d} g^{(h)}\|_{H^{N+1-\si_d}}^2+\int_0^t\||D|^{1+\si_d}g^{(h)}\|_{H^{N+1-\si_d}}^2d\tau\big) dh\\
    +&\int_{h_0}^{\infty}2^{2hs} \|\d_h g^{(h)}\|_{H^1}^2dh+\int_0^t \int_{h_0}^{\infty} 2^{2hs} \|\d\d_h g^{(h)}\|_{H^1}^2dhd\tau,\\
    \tri [A^{(h)}]&\tri_{s,\, A}^2:=   \||D|^{\de_d}A^{(h_0)}\|_{H^{s-\de_d}}^2+\int_0^t\||D|^{1+\de_d} A^{(h_0)}(\tau)\|_{H^{s-\de_d}}^2 d\tau\\
    +&\max_{N\in\{[2s], [2s]+1\}}\int_{h_0}^\infty 2^{2h(s-N)} \big(\| |D|^{\de_d}A^{(h)}\|_{H^{N-\de_d}}^2+\int_0^t \||D|^{1+\de_d} A^{(h)}\|_{H^{N-\de_d}}^2 d\tau \big) dh\\
    +&\int_{h_0}^{\infty} 2^{2hs} \|\d_h A^{(h)}\|_{L^2}dh+\int_0^t \int_{h_0}^{\infty} 2^{2hs} \|\d\d_h A^{(h)}\|_{L^2}dhd\tau.
\end{align*}
Then the $Y^{s+1}$ and $Z^s$-norms are given by
\begin{align*}
    \| g\|_{Y^{s+1}}=\inf_{[g^{(h)}]\in \text{Reg}(g)} \tri [g^{(h)}]\tri_{s+1,\, g}\,,\qquad \| A\|_{Z^{s}}=\inf_{[A^{(h)}]\in \text{Reg}(A)} \tri [A^{(h)}]\tri_{s,\, A}\,.
\end{align*}
\end{defn}

\begin{prop}[Embeddings]
    Let $s>\frac{d}{2}$.
    For the functions $g\in Y^{s+1}$, $A\in Z^s$ and $\la\in X^s$, we have the following properties:
    
    (i) Embeddings:
    \begin{align}  \label{Emb-g}
    \||D|^{\si_d}g\|_{H^{s+1-\si_d}}&\les\|g\|_{Y^{s+1}}\,,\\  \label{Emb-A}
    \|A\|_{H^s}&\les \|A\|_{Z^s}\,,\\ \label{Emb-la}
    \|\la\|_{H^s}&\les \|\la\|_{X^s_{ext}}\,.
    \end{align} 

    (ii) If $\|g\|_{Y^{s+1}},\ \|A\|_{Z^s}\les C(M)$, then we have the equivalence
    \begin{align}  \label{Equ-la}
        \|\la\|_{X^s_{ext}}\approx_M \|\la\|_{X^s_{int}}\,
    \end{align}
    with implicit constants depending on $M$.
\end{prop}
\begin{rem}
Due to the above equivalence property, we will not distinguish between the two function spaces $X^s_{int}$ and $X^s_{ext}$, and will simply denote them as $X^s$.
\end{rem}

\begin{proof}[Proof of \eqref{Emb-g}, \eqref{Emb-A} and \eqref{Emb-la}]
The proofs of the embeddings are similar, here we only focus on the bound \eqref{Emb-la}.

For any regularization $[\la^{(h)}]\in \text{Reg}(\la)$, as $\lim\limits_{h\rightarrow\infty}\|\la^{(h)}\|_{H^s}=\|\la\|_{H^s}$, the $H^s$-norm of $\la$ is expressed as
\begin{align*}
\|\la\|_{H^s}^2=\|\la^{(h_0)}\|_{H^s}^2+\int_{h_0}^\infty \frac{d}{dh}\|\la^{(h)}\|_{H^s}^2 dh.
\end{align*}
By interpolation, the first term is bounded by
\begin{align*}
    \|\la^{(h_0)}\|_{H^s}\les \|\la^{(h_0)}\|_{H^{[s]}}^{[s]+1-s} \|\la^{(h_0)}\|_{H^{[s]+1}}^{s-[s]}\les \tri[\la^{(h)}]\tri_{s,\, ext}.
\end{align*}
Using H\"older's inequality and interpolation, we can also bound the second term by
\begin{align*}
 & \ \int_{h_0}^\infty\frac{d}{dh}  \|\la^{(h)}\|_{H^s}^2 = 
 \int_{h_0}^\infty  \langle \la^{(h)}, \mu^{(h)}\rangle_{H^s} dh
 \lesssim \int_{h_0}^\infty  \| \la^{(h)}\|_{H^{2s}} \|\mu^{(h)}\|_{L^2} dh
 \\
 \lesssim & \ (\int_{h_0}^\infty  2^{-2hs}\| \la^{(h)}\|_{H^{2s}}^{2} dh)^\frac12
  (\int_0^\infty 2^{2hs}\|\mu^{(h)}\|_{L^2}^2 dh)^\frac12\\
  \les &\ (\int_{h_0}^\infty  2^{2h(s-(N-1))}\| \la^{(h)}\|_{H^{N-1}}^{2} dh)^\frac{\theta}{2} (\int_{h_0}^\infty  2^{2h(s-N)}\| \la^{(h)}\|_{H^{N}}^{2} dh)^\frac{1-\theta}{2}
  (\int_{h_0}^\infty 2^{2hs}\|\mu^{(h)}\|_{L^2}^2 dh)^\frac12\\
  \les &\ \tri[\la^{(h)}]\tri_{s,\, ext}^2\,,
\end{align*}
where $N=[2s]+1$ and $\theta=N-2s$. 
Then taking the infimum with respect to $[\la^{(h)}] \in \text{Reg}(\la)$
\begin{align*}
\|\la\|_{H^s}^2=\inf_{[\la^{(h)}]\in\text{Reg}(\la)}\bigg(\|\la^{(h_0)}\|_{H^s}^2+\int_{h_0}^\infty \frac{d}{dh}\|\la^{(h)}\|_{H^s}^2 dh\bigg)\les \inf_{[\la^{(h)}]\in\text{Reg}(\la)}\tri[\la^{(h)}]\tri_{s,\, ext}^2=\|\la\|_{X^s_{ext}}^2.
\end{align*}
Hence, the bound \eqref{Emb-la} follows.
\end{proof}

\begin{proof}[Proof of the equivalence \eqref{Equ-la}]

Firstly, we prove the bound 
\begin{align} \label{equles}
    \tri [\la^{(h)}]\tri_{s,int}\les \tri [\la^{(h)}]\tri_{s,ext}.
\end{align}

For the first two terms in \eqref{def-int}, by $\det g^{(h_0)}\sim 1$, it suffices to prove 
\begin{equation}  \label{Equ-low}
    \|\la^{(h_0)}\|_{\sfH^k}\les \|\la^{(h_0)}\|_{H^k},\qquad k=[s], [s]+1.
\end{equation}
By Sobolev embeddings and the formula
\begin{equation} \label{nabl-Form}
    \nab^{A,l} \la=\d^l\la+\sum_{j=1}^l \sum_{k_1+\cdots+k_j+k_{j+1}= l-j}\d^{k_1}(\Ga+A) \cdots \d^{k_j}(\Ga+A)\cdot \d^{k_{j+1}} \la,
\end{equation}
we have
\begin{align}  \label{Lam-EqKey}
    \|\nab^{A^{(h_0)},l} \la^{(h_0)}\|_{L^2}=\|\d^l \la^{(h_0)}\|_{L^2}+\|\Ga^{(h_0)}+A^{(h_0)}\|_{H^{\max\{s,l-1\}}}^l \|\la^{(h_0)}\|_{H^{l-1}}.
\end{align}
Then the bound \eqref{Equ-low} follows.

For the third term in \eqref{def-int}, by \eqref{nabl-Form} and interpolation \eqref{HaInterpo} we have
\begin{align}  \label{Lam-EqKey2}
\|\nab^{A^{(h)},l} \la^{(h)}\|_{L^2}&\les \|\d^l \la^{(h)}\|_{L^2}+\sum_{l-j<s}(\|\Ga^{(h)}\|_{H^s}+\||D|^{\de_d}A^{(h)}\|_{H^{s-\de_d}})^j\|\la^{(h)}\|_{H^s}\\\nonumber
&+\sum_{1\leq j\leq s}\sum_{\al\in\AA} \|\Ga^{(h)}+A^{(h)}\|_{L^\infty}^{j-1+\al}\|\la^{(h)}\|_{L^\infty}^{1-\alpha}\|\Ga^{(h)}+A^{(h)}\|_{\dot H^{l-j}}^{1-\al}\|\la^{(h)}\|_{H^{l-j}}^{\al}\\  \nonumber
&\les \|\d^l \la^{(h)}\|_{L^2}+\sum_{j}\sum_{\al}M^{j-1+\al}\|\la^{(h)}\|_{H^s}^{1-\al}\\\nonumber
&\cdot(\|\d g^{(h)}\|_{H^{l-1}}^{1-\al}+\||D|^{\de_d} A^{(h)}\|_{H^{l-1-\de_d}}^{1-\al}) \|\la^{(h)}\|_{H^{l-1}}^{\al}.
\end{align}
where $\AA=\{\frac{k}{l-j}:0\leq k\leq l-j\}$. Then we arrive at
\begin{align*}
    &\ \int_{h_0}^\infty 2^{2h(s-N)}\|\la^{(h)}\|_{\sfH^N}^2 dh\\
    &\les 
    \int_{h_0}^\infty 2^{2h(s-N)}\Big[\|\la^{(h)}\|_{H^N}^2+M^{2(N-1+\al)}\|\la^{(h)}\|_{H^s}^{2(1-\al)}\\
    &\qquad \cdot(\|\d g^{(h)}\|_{H^{N-1}}^{2(1-\al)}+\||D|^{\de_d} A^{(h)}\|_{H^{N-1-\de_d}}^{2(1-\al)}) \|\la^{(h)}\|_{H^{N-1}}^{2\al}\Big] dh\\
    &\les \tri [\la^{(h)}] \tri_{s,ext}^2+\tri[\la^{(h)}]\tri_{s,ext}^{2(1-\al)}(\tri [g^{(h)}]\tri_{s+1}+\tri [A^{(h)}]\tri_s)^{2(1-\al)} \tri [\la^{(h)}] \tri_{s,ext}^{2\al}
    \les \tri [\la^{(h)}] \tri_{s,\, ext}^2.
\end{align*}
Moreover, the last terms in \eqref{def-int} and \eqref{def-ext} are equivalent due to $cI\leq g^{(h)}\leq CI$ and $\det g^{(h)}\sim 1$,
\begin{align} \label{equ-mu}
    \int_{h_0}^\infty 2^{2hs}\|\mu^{(h)}\|_{L^2}^2 dh\sim \int_{h_0}^\infty 2^{2hs}\|\mu^{(h)}\|_{\sfL^2}^2 dh\,,
\end{align}
Hence, the estimate \eqref{equles} is obtained.

Secondly, we prove the bound 
\begin{align*} 
    \tri [\la^{(h)}]\tri_{s,int}\gtrsim \tri [\la^{(h)}]\tri_{s,ext}.
\end{align*}

By $cI\leq g^{(h)}\leq CI$ and $dvol_{g^{(h)}}\sim dx$ for any $h\geq h_0$, we have
\begin{align}  \label{BoundL2}
    \|\la^{(h)}\|_{L^2}^2\leq \int c^{-2}g^{(h)\al\mu}g^{(h_0)\be\nu}\la^{(h)}_{\al\be}\bar{\la}^{(h)}_{\mu\nu} dvol_{g^{(h)}}\les 
    \|\la^{(h)}\|_{\sfL^2(\Sigma^{(h)})}^2.
\end{align}
Then using \eqref{Lam-EqKey} and induction over $l$, we get
\begin{align*}
    \|\la^{(h_0)}\|_{H^l}\les \|\la^{(h_0)}\|_{\sfH^l}+\|\Ga^{(h_0)}+A^{(h_0)}\|_{H^{\max\{s,l-1\}}}^l \|\la^{(h_0)}\|_{H^{l-1}}\les \|\la^{(h_0)}\|_{\sfH^l}.
\end{align*}
Hence, we obtain that for $k=[s], [s]+1$
\begin{equation} \label{equ-lah0}
    \|\la^{(h_0)}\|_{H^k}\les \|\la^{(h_0)}\|_{\sfH^k}.
\end{equation}
Using a similar argument to \eqref{Emb-la}, and combined with \eqref{equ-lah0} and \eqref{equ-mu}, we also have
\begin{align}  \label{Emb-la2}
   \|\la\|_{H^s}^2\les \tri[\la^{(h)}]\tri_{s,\, ext}\tri[\la^{(h)}]\tri_{s,\, int}\,.
\end{align}

For the third term in \eqref{def-ext}, 
by \eqref{BoundL2}, \eqref{Lam-EqKey2}, \eqref{Emb-la2} and H\"older's inequality, we get
\begin{align*}
&\quad \int_{h_0}^\infty 2^{2h(s-N)}\|\la^{(h)}\|_{H^N}^2 dh\\
&\les 
\int_{h_0}^\infty 2^{2h(s-N)h}(\|\la^{(h)}\|_{\sfH^N}^2+M^{2(N-1+\al)}\|\la^{(h)}\|_{H^s}^{2(1-\al)}\|(\d g^{(h)}, A^{(h)})\|_{\dot H^{N-1}}^{2(1-\al)} \|\la^{(h)}\|_{H^{N-1}}^{2\al}) dh\\
&\les \tri [\la^{(h)}] \tri_{s,int}^2+
\int_{h_0}^\infty 2^{2h(s-N)}\|\la^{(h)}\|_{H^s}^{2(1-\al)}(\|\d g^{(h)}\|_{H^{N-1}}^{2(1-\al)}+\||D|^{\de_d} A^{(h)}\|_{H^{N-1-\de_d}}^{2(1-\al)})\\
&\quad \cdot\|\la^{(h)}\|_{\sfL^2}^{\frac{2\al} {N}}\|\la^{(h)}\|_{H^{N}}^{\frac{2\al(N-1)}{N}} dh\\
&\les \tri [\la^{(h)}] \tri_{s,int}^2+\tri[\la^{(h)}]\tri_{s,ext}^{1-\al}\tri [\la^{(h)}]\tri_{s,int}^{1-\al}(\tri [g^{(h)}]\tri_{s+1,g}+\tri [A^{(h)}]\tri_{s,A})^{2(1-\al)} \\
&\quad \cdot \tri [\la^{(h)}] \tri_{s,int}^{\frac{2\al}{N}}\tri [\la^{(h)}] \tri_{s,ext}^{\frac{2\al(N-1)}{N}}\\
&\leq \frac{1}{2}\tri [\la^{(h)}]\tri_{s,ext}^2+C\tri [\la^{(h)}] \tri_{s, int}^2.
\end{align*}
This, together \eqref{equ-mu} and \eqref{equ-lah0}, yields the bound $\tri [\la^{(h)}]\tri_{s,ext}\les \tri[\la^{(h)}]\tri_{s,int}$.
Hence, we obtain the equivalence $\|\la\|_{X^s_{int}}\approx \|\la\|_{X^s_{ext}}$.
\end{proof}

Next, following \cite{MMT3,MMT4,MMT5}, we define the frequency envelopes  which will be used in multilinear estimates. Consider a Sobolev-type space $U$ for which we have
\begin{equation*}
\lV u\rV_U^2=\sum_{k=0}^{\infty} \lV S_k u\rV_U^2.
\end{equation*}
A frequency envelope for a function $u\in U$ is a positive $l^2$-sequence, $\{ a_j\}$, with 
\begin{equation*}
\lV S_j u\rV_U\leq a_j.
\end{equation*}
We shall only allow slowly varying frequency envelopes. Thus, we require $a_0\approx \lV u\rV_U$ and 
\begin{equation*}
a_j\leq 2^{\delta|j-k|} a_k,\quad j,k\geq 0,\ 0<\delta\ll s-d/2.
\end{equation*}
The constant $\delta$ shall be chosen later and only depends on $s$ and the dimension $d$. We define the frequency envelopes $\{c_j \}_{j\geq h_0}$ for the initial manifold $F_0$ and its orthonormal frame $\nu_0$ as 
\begin{equation*}
\begin{aligned}
c_j &= 2^{-\de|j-h_0|}(\|P_{<h_0}\d^2 F\|_{H^s}+\|P_{<h_0}\d \nu\|_{H^s})\\
&\quad + \sum_{k
\geq h_0} 2^{-\delta|j-k|} \big(\int_k^{k+1}2^{2hs}(\| P_{h}\d^2 F\|_{L^2}^2+\|P_h \d \nu\|_{L^2}^2) dh\big)^{1/2},
\end{aligned}
\end{equation*}
which is slowly varying, i.e. $c_k\leq 2^{\de |k-j|}c_j$.
Then
$\| c_j\|_{\ell^2} \approx \|\d^2 F\|_{H^s}+\|\d \nu\|_{H^s}$.

\bigskip 
\section{The linearized equations and the uniqueness result} \label{Sec-lin}
In this section, we first derive the linearized equations for a family of maps $F(t,x;s)$, where the normal component $(\d_s F)^\perp$ and the tangent component $(\d_s F)^{\top}$ will be considered separately. 
Then we prove energy estimates for the linearized equations.
These will play a key role in constructing rough solutions. 

Next, we establish $L^2$ difference bounds for solutions, which could be viewed as difference versions of the estimates for the linearized equations. As a corollary, this will yield the uniqueness result in Theorem~\ref{Uniqueness-thm}.

\subsection{The Linearized equations}
Here we consider a family of maps $F(t,x;s)$ with parameter $s$, which evolves along the (SMCF). Let $(\nu_1,\nu_2)$ be the corresponding orthonormal frame in normal bundle. Assume that $\d_s F$ can be expressed as
\begin{align*}
\d_s F=\Xi+U^\ga \d_\ga F,\qquad \Xi\in N\Sigma\,,
\end{align*}
and we define the complex normal vector $\omega$ to be
\[\om=\Xi\cdot m,\qquad m=\nu_1+i\nu_2\,.\]
Then we obtain the following linearized equations.
\begin{lemma}  \label{Lin-Lem}
The normal component $\om$ and tangential component $U$ of $\d_s F$ satisfy
\begin{align}\label{LinEq}
&i(\d^B_t-V^\ga\nab^A_\ga)\om+\nab^{A,\al}\nab^A_\al \om=-\Re(\la^{\al\be}\bar{\om})\la_{\al\be},\\\label{dt-Tangent}
&\d_t U_{\al}=g_{\al\be}\d_s V^\be+\Im(\psi \overline{\nab_\al^A\om})-\Im(\d^A_\al \psi\bar{\om})+2\Im(\psi\bar{\la}^\ga_\al)U_{\ga}
+\nab_\al V^\ga U_{\ga}+V^\si \nab_\si U_{\al}.
\end{align}
\end{lemma}

We can now state the energy estimates for the linearized equations.
\begin{prop} \label{prop4.2}
If $\|\la\|_{H^s},\|A\|_{L^\infty},\|g\|_{W^{1,\infty}}\les M$ on $[0,T(M)]$, the normal component $\om$ and the tangent vector $U$ satisfy the estimates
\begin{align} \label{om-Energy0}
\frac{d}{dt}\|\om\|_{\sfL^2}^2&\leq C(M)\|\om\|_{\sfL^2}^2,\\ \label{om-Energy}
\frac{d}{dt}\|\om\|_{\sfH^1}^2&\leq C(M)\|\om\|_{\sfH^1}^2,\\\label{U-Energy0}
\frac{d}{dt}\|U\|_{\sfL^2}
&\leq  C(M)(\|\d_h g\|_{H^1}+\|\om\|_{\sfH^1}+\|U\|_{\sfL^2}).
\end{align}
\end{prop}

We begin with the derivations of \eqref{LinEq} and \eqref{dt-Tangent}, and then prove the estimates in Proposition~\ref{prop4.2}.

\begin{proof}[Proof of the formula \eqref{LinEq}]
Applying $\d^B_t$ to $\om$, then using \eqref{sys-cpf} and \eqref{mo-frame} we have
\begin{align*}
    &\ \d^B_t\om=\d^B_t(\d_s F\cdot m)
    =\d_s \d_t F\cdot m+\d_s F\cdot \d^B_t m\\
    &=\d_s (J(F)\bH(F)+V^\ga F_\ga)\cdot m-i U^\ga (\nab^{A}_\ga \psi-i\la_{\ga\si}V^\si)\\
    &=\d_s \big(  (\De_g F\cdot \nu_1)\nu_2-(\De_g F\cdot\nu_2)\nu_1\big)\cdot m+V^\ga \d_\ga \d_s F\cdot m-i U^\ga (\nab^{A}_\ga \psi-i\la_{\ga\si}V^\si)\\
    &=\d_s \big(  g^{\al\be}(\d_{\al\be}^2 F\cdot \nu_1)\nu_2-g^{\al\be}(\d_{\al\be}^2 F\cdot\nu_2)\nu_1\big)\cdot m+V^\ga \d_\ga \d_s F\cdot m-i U^\ga (\nab^{A}_\ga \psi-i\la_{\ga\si}V^\si).
\end{align*}
Next we calculate the right-hand side one by one.

Firstly, we consider the case that $\d_s$ is applied to $g^{\al\be}$. Since $\Xi\perp \d_\al F$, we have
\begin{align*}  
\d_s g_{\mu\nu}&=\d_s (\d_\mu F\cdot\d_\nu F)=\d_\mu (\Xi+U^\ga F_\ga)\cdot \d_\nu F+ \d_\mu F \cdot\d_\nu (\Xi+U^\ga F_\ga)\\
&= -2\Xi\cdot \d^2_{\mu\nu}F+\nab_\mu U_\nu+\nab_\nu U_\mu=-2\Re(\la_{\mu\nu}\bar{\om})+\nab_\mu U_\nu+\nab_\nu U_\mu.
\end{align*}
Then 
\begin{align*}
    \d_s g^{\al\be}=-g^{\al\mu}\d_s g_{\mu\nu} g^{\nu\be}=2\Re(\la^{\al\be}\bar{\om})-\nab^\mu U^\nu-\nab^\nu U^\mu,
\end{align*}
therefore
\begin{align*}
&\ \d_s g^{\al\be}\big( (\d_{\al\be}^2 F\cdot \nu_1)\nu_2-(\d_{\al\be}^2 F\cdot\nu_2)\nu_1\big)\cdot m\\
&=2(\Re(\la^{\al\be}\bar{\om})-\nab^\al U^\be) (\ka_{\al\be}\nu_2-\tau_{\al\be}\nu_1)\cdot m=i2\la_{\al\be} (\Re(\la^{\al\be}\bar{\om})-\nab^\al U^\be).
\end{align*}

Secondly, we consider the case that $\d_s$ is applied to $\d^2_{\al\be}F$. By the expression of $\d_s F$, we have
\begin{align}  \label{LinEq-keyT}
\big(g^{\al\be}(\d^2_{\al\be}\d_s F\cdot \nu_1)\nu_2-g^{\al\be}(\d^2_{\al\be}\d_s F\cdot \nu_2)\nu_1\big)\cdot m
=ig^{\al\be}\d^2_{\al\be}(\Xi+U^\ga F_\ga)\cdot m
\end{align}
where 
\begin{align*}
&\ ig^{\al\be}\d^2_{\al\be}\Xi\cdot m=ig^{\al\be} \big(\d^A_\al (\d_\be \Xi\cdot m)-\d_\be \Xi\cdot \d^A_\al m\big)\\
&=ig^{\al\be}\big( \d_\al^A(\d_\be \Xi\cdot m)+\d_\be \Xi\cdot\la_{\al}^\ga F_\ga  \big)
=ig^{\al\be}\d^A_\al\d^A_\be \om -i\la^{\be\ga} \Xi\cdot\d^2_{\be\ga}F\\
&=i\nab^A_\al\nab^{A,\al} \om +ig^{\al\be}\Ga_{\al\be}^\ga \d^A_\ga \om-i\la^{\be\ga} \Re(\la_{\be\ga}\bar{\om}),
\end{align*}
and 
\begin{align*}
    &\ ig^{\al\be}\d^2_{\al\be}(U^\ga F_\ga)\cdot m=ig^{\al\be}\nab_{\al}\d_\be(U^\ga F_\ga)\cdot m+ig^{\al\be}\Ga_{\al\be}^\si \d_\si(U^\ga F_\ga)\cdot m\\
    &= ig^{\al\be} \big(\nab_\al^A(U^\ga \la_{\be\ga})-\nab_\be U^\ga F_\ga\cdot \nab^A_\al m \big) +ig^{\al\be}\Ga_{\al\be}^\si U^\ga \la_{\si\ga}\\
    &= ig^{\al\be} \big(\nab_\al^A(U^\ga \la_{\be\ga})+\nab_\be U^\ga \la_{\al\ga} \big) +ig^{\al\be}\Ga_{\al\be}^\si U^\ga \la_{\si\ga}\\
    &= i(2\nab^\al U^\ga \la_{\al\ga}+U^\ga \nab^A_\ga \psi+g^{\al\be}U^\ga \Ga_{\al\be}^\de \la_{\de\si}).
\end{align*}

Thirdly, when $\d_s$ is applied to $\nu_i$, we get
\begin{align*}
    &\ \big(  g^{\al\be}(\d_{\al\be}^2 F\cdot \d_s\nu_1)\nu_2-g^{\al\be}(\d_{\al\be}^2 F\cdot\d_s\nu_2)\nu_1\big)\cdot m\\
    &= i g^{\al\be}(\d_{\al\be}^2 F\cdot \d_s m)
    =i g^{\al\be}\d_{\al\be}^2 F\cdot (-iA_0 m-(\d^{A,\ga}\om+U^\si \la_\si^\ga) F_\ga)\\
    &=A_0\psi-ig^{\al\be}\Ga_{\al\be}^\ga (\d^A_\ga\om+U^\si \la_{\ga\si}).
\end{align*}
and 
\begin{align*}
&\ \big(  g^{\al\be}(\d_{\al\be}^2 F\cdot \nu_1)\d_s\nu_2-g^{\al\be}(\d_{\al\be}^2 F\cdot\nu_2)\d_s\nu_1\big)\cdot m\\
&=\big(-g^{\al\be}\ka_{\al\be}A_0\nu_1-g^{\al\be}\tau_{\al\be}A_0\nu_2\big)\cdot m
=-A_0 g^{\al\be}\la_{\al\be}=-A_0\psi.
\end{align*}

Finally, by \eqref{strsys-cpf} we have
\begin{align*}
&V^\ga \d_\ga \d_s F\cdot m=V^\ga \big(\d^A_\ga (\Xi\cdot m)- \Xi\cdot \d^A_\ga m\big)+V^\ga U^\si \nab_\ga F_\si\cdot m
=V^\ga \d^A_\ga \om+V^\ga U^\si \la_{\ga\si}.
\end{align*}

Hence, collecting the above calculations yields
\begin{align*}
    \d^B_t\om =i\nab^{A,\al}\d_\al^A\om+i\la_{\al\be}\Re(\la^{\al\be}\bar{\om})+V^\ga \nab^A_\ga \om,
\end{align*}
which implies the linearized equation \eqref{LinEq}.
\end{proof}

\begin{proof}[Proof of the formula \eqref{dt-Tangent}]
    Apply $\d_t$ to $U_{\al}=\<\d_s F,\d_\al F\>$, we have
    \begin{align*}
        &\ \d_t U_{\al}=\<\d_t \d_s F,\d_\al F\>+\<\d_s F,\d_\al \d_t F\>\\
        &=\d_s \<\d_t F,\d_\al F\>-\<\d_t F,\d_\al \d_s F\>+\<\d_s F,\d_\al \d_t F\>=:I_1+I_2+I_3.
    \end{align*}
    Then by \eqref{sys-cpf}, the first term $I_1$ is written as
    \begin{align*}
        I_1&=\d_s \<\d_t F,\d_\al F\>=\d_s\<-\Im(\psi\bar{m})+V^\ga F_\ga,\d_\al F\>=\d_s V_\al\\
        &=g_{\al\be}\d_s V^\be +V^\be (-2\Re(\la_{\al\be}\bar{\om})+\nab_\al U_\be+\nab_\be U_\al).
    \end{align*}
    By the formula $\d_t F_\al$ in \eqref{mo-frame} and the formula $\d_s F$, we rewrite the third term $I_3$ as
    \begin{align*}
        I_3&=\<\d_s F,\d_\al \d_t F\>=\<\Xi+U^\ga F_\ga,-\Im(\d^A_\al \psi \bar{m}-i\la_{\al\ga}V^\ga\bar{m})+(\Im(\psi\bar{\la}^\ga_\al)+\nab_\al V^\ga)F_\ga\>\\
        &=-\Im(\d^A_\al \psi \bar{\om})+\Re(\la_{\al\ga}\bar{\om})V^\ga+(\Im(\psi\bar{\la}^\ga_\al)+\nab_\al V^\ga)U_\ga.
    \end{align*}
    Finally, we deal with the term $I_2$. Apply $\d_\al$ to $\d_s F$, we have
    \begin{align*}
        \d_\al \d_s F&=\d_\al \big(\Re(\om \bar{m})+U^\ga F_\ga\big)
        =\Re(\d^A_\al \om\bar{m}+\om\overline{\d^A_\al m})+\nab_\al U^\ga F_\ga+U^\ga \nab_\al F_\ga\\
        &=\Re(\d^A_\al \om \bar{m}+U_\ga \la_\al^\ga \bar{m})-\Re(\om \bar{\la}^\si_\al)F_\si+\nab_\al U^\ga F_\ga.
    \end{align*}
    This, together with $\d_t F$ in \eqref{mo-frame}, yields
    \begin{align*}
        I_2=&\ -\<-\Im(\psi\bar{m}),\Re(\d^A_\al \om \bar{m}+U_\ga \la_\al^\ga \bar{m})\>
        -\<V^\ga F_\ga,-\Re(\om \overline{\la}^\si_\al)F_\si+\nab_\al U^\ga F_\ga\>\\
        =&\ \Im(\psi \overline{\d^A_\al \om})+\Im(\psi \bar{\la}^\ga_\al)U_\ga+\Re(\om\bar{\la}^\si_\al)V_\si-V^\ga \nab_\al U_\ga.
    \end{align*}
    Inserting the expressions of $I_1,\ I_2$ and $I_3$ into $\d_t U_\al$, the formula \eqref{dt-Tangent} is obtained.
\end{proof}

\begin{proof}[Proof of \eqref{om-Energy0}]
From the linearized equation \eqref{LinEq} and \eqref{g_dt}, we have
\begin{align*}
    &\ \frac{d}{dt}\|\omega\|_{\sfL^2}^2=\int 2\Re(\d_t^B\omega\cdot \bar{\omega})+ |\omega|^2\frac{1}{2}g^{\al\be}\d_t g_{\al\be}\  dvol\\
    &=\int 2\Re\big[(V^\ga \nab^A_\ga \omega+i\nab^{A,\al}\nab^A_\al \omega+i\Re(\la^{\al\be}\bar{\om})\la_{\al\be})\bar{\om}\big]+|\omega^{(h)}|^2 \nab^\al V_\al \ dvol\\
    &=\int -2\Re i|\nab^A \om|^2 -2\Re(\la^{\al\be}\bar{\om})\Im(\la_{\al\be}\bar{\om}) \ dvol \\
    &\leq 2\|\la\|_{L^\infty}^2 \|\om\|_{\sfL^2}^2\leq C(M) \|\om\|_{\sfL^2}^2.
\end{align*}
Then the estimate \eqref{om-Energy0} follows.
\end{proof}
\begin{proof}[Proof of \eqref{om-Energy}]
We consider the covariant derivative of $\om$,
\begin{align*}
    &\ \frac{1}{2}\frac{d}{dt}\|\d^A\om\|_{\sfL^2}^2=\frac{1}{2}\frac{d}{dt}\int g^{\al\be}\d^A_\al \om \overline{\d^A_\be \om}\  dvol\\
    &=\int \Re(\d^A_\al \d_t^B \om \overline{\d^{A,\al}\om}) dvol+\int \Re([\d_t^B,\d^A_\al ] \om \overline{\d^{A,\al}\om}) \ dvol\\
    &\quad +\int \frac{1}{2}\d_t g^{\al\be}\d^A_\al \om \overline{\d^A_\be \om}+|\d^A\om|^2_g \frac{1}{4}g^{\al\be}\d_t g_{\al\be} dvol=:I_1+I_2+I_3.
\end{align*}
For the first integral $I_1$, by \eqref{LinEq} we have
\begin{align*}
I_1&=\Re \int \nab^A_\al(V^\ga \nab^A_\ga \om+i\De^A_g \om+i\Re(\la^{\mu\nu}\bar{\om})\la_{\mu\nu})\overline{\nab^{A,\al}\om}\  dvol\\
&= \Re\int \nab_\al V^\ga \nab^A_\ga \om \overline{\nab^{A,\al}\om}+V^\ga \nab^A_\ga \nab^A_\al \om \overline{\nab^{A,\al}\om}+V^\ga [\nab^A_\al,\nab^A_\ga] \om \overline{\nab^{A,\al}\om} \ dvol\\
&\quad +\Re \int -i\De^A_g \om \overline{\De^A_g\om} +i\nab^A_\al(\Re(\la^{\mu\nu}\bar{\om})\la_{\mu\nu})\overline{\nab^{A,\al}\om} \ dvol\\
&=\Re\int \nab_\al V^\ga \nab^A_\ga \om \overline{\nab^{A,\al}\om}-\frac{1}{2}\nab_\ga V^\ga |\nab^A \om|^2 +iV^\ga \Im(\la_{\al\de}\bar{\la}^\de_\ga)  \om \overline{\nab^{A,\al}\om} \ dvol\\
&\quad +\Re \int i\nab^A_\al(\Re(\la^{\mu\nu}\bar{\om})\la_{\mu\nu})\overline{\nab^{A,\al}\om} \ dvol.
\end{align*}
The terms $I_2$ and $I_3$ are written as
\begin{align*}
    I_2&=\Re\int i(\d_t A_\al -\d_\al B)\om \overline{\d^{A,\al}\om} \ dvol\\
    &=\Re \int \big(i\Re (\la_\al^\ga \overline{\d^A_\ga \psi})-i\Im(\la_\al^\ga \bar{\la}_{\ga\si})V^\si\big)\om \overline{\d^{A,\al}\om}\  dvol,\\
    I_3&=\int (-\Im(\psi \bar{\la}^{\al\be})-\nab^\al V^\be)\nab^A_\al \om \overline{\nab^A_\be \om}+|\nab^A \om|^2_g \frac{1}{2}\nab_\al V^\al \ dvol.
\end{align*}
Then we obtain
\begin{align*}
    \frac{1}{2}\frac{d}{dt}\|\d^A\om\|_{\sfL^2}^2 =&\ \int -\Im(\psi \bar{\la}^{\al\be})\nab^A_\al \om \overline{\nab^A_\be \om}-\Re(\la_{\al}^\ga \overline{\d^A_\ga \psi})\Im(\om \overline{\d^{A,\al}\om})\\
    &\ -\Im(\nab^A_\al (\Re(\la^{\mu\nu}\bar{\om})\la_{\mu\nu} )\overline{\nab^{A,\al}\om}) \ dvol .
\end{align*}
Thus 
\begin{align*}
    \frac{d}{dt}\|\om\|_{\sfH^1}^2 &\les \|\la\|_{L^\infty}^2\|\om\|_{\sfH^1}^2+\|\la\|_{L^\infty}\| \la\|_{H^{s}}(1+\|A\|_{L^\infty})\|\om\|_{H^1}\|\om\|_{\sfH^1}
    \les C(M)\|\om\|_{\sfH^1}^2,
\end{align*}
which further implies that $\|\om(t)\|_{\sfH^1}\les \|\om(0)\|_{\sfH^1}$.
\end{proof}

\begin{proof}[Proof of \eqref{U-Energy0}]
    By the formula \eqref{dt-Tangent} of $U_\al$, we derive
\begin{align*}
&\ \frac{1}{2}\frac{d}{dt}\|U\|_{\sfL^2}^2=\int \d_t U_\al U^\al +U_\al U_\be \frac{1}{2}\d_t g^{\al\be}+|U|_g^2 \frac{1}{4}g^{\al\be}\d_t g_{\al\be} \ dvol_g\\
&=\int \d_h V^\al U_\al+\big(\Im(\psi\overline{\d^A_\al \om})-\Im(\d^A_\al\psi \bar{\om})+2\Im(\psi\bar{\la}_\al^\ga)U_\ga\big)U^\al\ dvol_g\\
&\quad+\int (\nab_\al V^\ga U_\ga+V^\si \nab_\si U_\al)U^\al+ U_\al U_\be (-\Im(\psi\bar{\la}^{\al\be})-\nab^\al V^\be)+|U|_g^2 \frac{1}{2}\nab_\al V^\al dvol_g\\
&=\int \d_h V^\al U_\al+\big(\Im(\psi\overline{\d^A_\al \om})-\Im(\d^A_\al\psi \bar{\om})+\Im(\psi\bar{\la}_\al^\ga)U_\ga\big)U^\al\ dvol_g.
\end{align*}
Then we obtain
\begin{align*}
\frac{1}{2}\frac{d}{dt}\|U\|_{\sfL^2}^2&\leq (\|\d_h V_\al\|_{L^2}+C(M)\|\om\|_{\sfH^1}+\|\la\|_{L^\infty}^2\|U\|_{L^2})\|U\|_{L^2}\\
&\les C(M)(\|\d_h g\|_{H^1}+\|\om\|_{\sfH^1}+\|U\|_{\sfL^2})\|U\|_{\sfL^2} .
\end{align*}
This implies the inequality \eqref{U-Energy0}.
\end{proof}

\subsection{The difference bounds and the uniqueness result}
To compare two surfaces $\Sigma$, $\tSigma$ at fixed time we need some notion of $L^2$ distance between the two surfaces. 
One choice would be 
\begin{equation*}
d_{L^2}^2(\Sigma,\tSigma) = \int_{\Sigma} d(x,\tSigma)^2 dvol_{\Sigma}     
\end{equation*}
This definition is not perfect in that it is not symmetric
and possibly not a distance. However, under uniform $C^2$ bounds for the two surfaces and small $L^2$ distances,
these two properties can be seen to hold up to constants, which is all we need
in the sequel.

\begin{prop}[The difference bounds for (SMCF)]\label{Prop-DiffBound}
Suppose $\Sigma_t$, $\tSigma_t$  are $C^2$ solutions of (SMCF) in a time interval $[0,T]$,  with 
of size $\leq M$, in the sense that there exist parametrizations $F,\ \tF$ so that 
\begin{equation*}
\| \d F\|_{C^1},\|\d\tF\|_{C^1} \leq M, \qquad g, \tilde g \geq M^{-1} I.     
\end{equation*}
Assume in addition that the two surfaces are initially close,
\begin{equation*}
d_{L^2}(\Sigma_0,\tSigma_0)\leq \ep \ll_M 1. 
\end{equation*} 
Then within the time interval $[0,T]$ we have
\begin{align}   \label{DB-FtF}
 d_{L^2}(\Sigma_t,\tSigma_t)   \les d_{L^2}(\Sigma_0,\tSigma_0).
\end{align}
\end{prop}

Here the gauge of $\Sigma(F)$ is free, while the gauge of the solution $\tilde\Sigma(\tF)$ must be chosen such that we have a good Gr\"onwall's inequality. In the frame $(\d_1F,\cdots,\d_d F,\nu_1,\nu_2)$, the difference $\tF-F$ can be expressed as
\begin{align*}
    \de F=\tF-F=\Xi+U^\ga \d_\ga F,\qquad \om:=\Xi\cdot m.
\end{align*}
The first step of the proof is to favourably choose the 
gauge of $\tSigma$ in order to guarantee that $|\delta F| \lesssim |\omega|$:

\begin{lemma} \label{Lem4.4}
Under the assumptions of the Proposition~\ref{Prop-DiffBound}, we can choose the
parametrization $\tF$ for $\tSigma$ so that we still have the uniform $C^2$ bound
\begin{equation*}
 \| \tF\|_{C^2} \lesssim_M 1,   
\end{equation*}
and so that we have the pointwise equivalence 
\begin{equation*} 
|F(x)-\tF(x)| \approx d(F(x),\tSigma) \approx d(\tF(x),\Sigma).  
\end{equation*}
\end{lemma}
The last property guarantees that $|F-\tF| \lesssim |\omega|$,
which will allow us to simply estimate the time evolution of $\omega$.
\begin{proof}
First we localize the problem, covering $\Sigma$ with balls $B_j$ of size $\delta$, centered at $F(x_j)$ where $\delta$ is an intermediate scale so that 
\[
\epsilon \ll_M \delta \ll_M 1.
\]
Within each such ball, $\Sigma$ is nearly flat. Due to the $L^2$
closeness assumption, this collection of  balls must also cover $\tSigma$, and their intersection  with $\tSigma$ is also almost flat. Then by the implicit function theorem and $rank(\frac{\d F}{\d x})=rank(\frac{\d \tF}{\d x})=d$, in a well chosen orthonormal frame adapted to $B_j$ we may represent both surfaces as graphs,
\[
\Sigma \cap B_j = \{ (y, G_j(y))\}, \qquad  \tSigma \cap B_j = \{ (y, \tilde G_j(y))\}
\]
where $\|G_j\|_{C^2}, \| \tilde G_j\|_{C^2} \lesssim_M 1$, with small gradients 
\[
\|\partial_y G_j\|\lesssim_M \delta
\]
and the $L^2$ closeness condition is expressed as
\[
d_{L^2}^2(\Sigma,\tSigma) \approx \sum_j \|G_j -\tilde G_j\|_{L^2}^2 
\]
Within each $B_j$ we can simply define new $C^2$ coordinates $\tilde x_j = \tilde x_j(x)$ on $\tilde \Sigma$ via 
\[
F(x) =  (y, G_j(y)) \Longrightarrow \tilde F(\tilde x_j)= (y, \tilde G_j(y)),
\]
which have the desired properties in the Lemma. It remains to 
assemble these coordinates together, which is easily achieved 
using a partition of unit associated to the $B_j$ covering.
We note here that neighboring frames are at angle $\lesssim \delta$, which implies that we have $|\tilde x_j - \tilde x_k|\lesssim \delta d(x,\tSigma)$, allowing us to gain 
local smallness for the difference of $F$ and $\tilde F$ in the $C^1$ norm in the new coordinates.

The argument above applies not only at fixed time, but also uniformly on time intervals $O(\delta)$, where the same local covering and frames can be used.
\end{proof}

We now continue the proof of the Proposition~\ref{Prop-DiffBound}, using the matched coordinates on the two surfaces given by the above Lemma. Since $\de F\in C^2$ is also small on a short time interval, we can define the normal vectors by
\begin{align*}
    \bar{\nu}_{j}=\nu_{j}-\tg^{\al\be}\<\nu_{j},\d_\al \de F\>\d_\be \tF. 
\end{align*}
Then the orthonormal frame $(\tnu_{1},\tnu_{2})$ in $N\tilde\Sigma(\tF)$ is given by
\begin{align*} 
\tnu_{1}=\frac{\bar{\nu}_{1}}{|\bar{\nu}_{1}|},\quad \tnu_{2}=\frac{\bar{\bar{\nu}}_{2}}{|\bar{\bar{\nu}}_{2}|}, \qquad \text{with}\quad \bar{\bar{\nu}}_{2}=\bar{\nu}_{2}-\<\bar{\nu}_{2},\tnu_{1}\>\tnu_{1}.
\end{align*}
Now we have the following Lemma.

\begin{lemma}
The normal component $\om$ of difference $\de F$ satisfies the following formula
\begin{align}\label{dtom-Uniq}
&\begin{aligned}
    &\ i(\d_t^B -\tilde V^\ga\nab^A_\ga)\om+\tilde g^{\al\be}\tilde \nab^A_\al \d^A_\be \om\\
    =&\ -\de g^{\al\be}U^\ga \nab^A_\al \la_{\be\ga}-\Re(\la^{\al\be}\bar{\om})\la_{\al\be}+\de g^{\al\be}\la_\al^\si\Re(\la_{\be\si}\bar{\om})\\
	&\ +i\de V^\ga U^\si \la_{\ga\si}+\tg^{\al\be}\de\Ga_{\al\be}^\mu U^\ga\la_{\mu\ga}-2\de g^{\al\be}\nab_\al U^\ga \la_{\be\ga}+O(\d^2 F |\d \de F|_{\tg}^2)\,,
\end{aligned}
\end{align}
where $\de g^{\al\be}=\tg^{\al\be}-g^{\al\be}$, $\de V^\ga=\tV^\ga-V^\ga$ and $\de \Ga^{\mu}_{\al\be}=\tGa_{\al\be}^\mu-\Ga_{\al\be}^\mu$.
\end{lemma}

\begin{proof}
Applying $\d_t^B$ to $\om$ yields
\begin{align*}
&\ \d^B_t \om=\d_t^B \<\de F, m\>=\<\d_t(\tF-F), m\>+\<\tF-F, \d^B_t m\>\\
=&\ \<J(\tF)\bH(\tF)-J(F)\bH(F),m\>+\<\tV^\ga \tF_\ga-V^\ga F_\ga,m\>+\<\tF-F, \d^B_t m\>.
\end{align*}
By $F_\ga \perp m$, $\de F$ and \eqref{mo-frame}, we express the last two terms above as
\begin{align*}
\<\tV^\ga \tF_\ga-V^\ga F_\ga,m\>&=\tV^\ga \<\d_\ga (\tF-F),m\>=\tV^\ga\<\d_\ga(\Xi+U^\si F_\si),m\>
=\tV^\ga (\d^A_\ga \om +U^\si \la_{\ga\si}),\\
\<\tF-F, \d^B_t m\>&=\<\Xi+U^\si F_\si,-i(\d^{A,\al}\psi-i\la^\al_\ga V^\ga)F_\al\>
=-iU^\si (\d^A_\si\psi-i\la_{\ga\si}V^\ga).
\end{align*}
Next, we consider the first term. Using the expression for $J(F)\bH(F)$, this is written as
\begin{align*}
	&\ \<J(\tF)\bH(\tF)-J(F)\bH(F),m\>\\
	=&\ \<(\tg^{\al\be}\d^2_{\al\be}\tF\cdot \tnu_1\ \tnu_2-\tg^{\al\be}\d^2_{\al\be}\tF\cdot \tnu_2\ \tnu_1)-(g^{\al\be}\d^2_{\al\be}F\cdot \nu_1\ \nu_2-g^{\al\be}\d^2_{\al\be}F\cdot \nu_2\ \nu_1),m\>\\
	=&\ \<\tg^{\al\be}\d^2_{\al\be}\tF\cdot \tnu_1\ (\tnu_2-\nu_2)-\tg^{\al\be}\d^2_{\al\be}\tF\cdot \tnu_2\ (\tnu_1-\nu_1),m\>\\
	&\ + \<\tg^{\al\be}\d^2_{\al\be}\tF\cdot (\tnu_1-\nu_1)\ \nu_2-\tg^{\al\be}\d^2_{\al\be}\tF\cdot (\tnu_2-\nu_2)\ \nu_1,m\>\\
	&\ + \<\tg^{\al\be}\d^2_{\al\be}(\tF-F)\cdot \nu_1\ \nu_2-\tg^{\al\be}\d^2_{\al\be}(\tF-F)\cdot \nu_2\ \nu_1,m\>\\
	&\ + (\tg^{\al\be}-g^{\al\be})\<\d^2_{\al\be}F\cdot \nu_1\ \nu_2-\d^2_{\al\be}F\cdot \nu_2\ \nu_1,m\>\\
	=&:I_1+I_2+I_3+I_4.
\end{align*}

\emph{a) Estimates for $I_1$ and $I_2$.}
Since $1-|\bar{\nu}_1|^2=|\nu_1\cdot \d \de F|_{\tg}^2=O(|\d\de F|_{\tg}^2)$, it follows that
\begin{align*}
\tnu_1-\nu_1=\frac{1-|\bar{\nu}_1|^2}{|\bar{\nu}_1|(1+|\bar{\nu}_1|)}\bar{\nu}_1-\tg^{\al\be}\<\nu_1,\d_\al \de F\>\d_\be \tF=O(|\d\de F|_{\tg}^2)-\tg^{\al\be}\<\nu_1,\d_\al \de F\>\d_\be \tF.
\end{align*}
Since $\<\bar{\nu}_2,\tnu_1\>=|\bar{\nu}_1|^{-1}\<\bar{\nu}_2,\bar{\nu}_1\>=-|\bar{\nu}_1|^{-1}\tg^{\al\be}\<\nu_1,\d_\al \de F\>\<\nu_2,\d_\be \de F\>=O(|\d\de F|_{\tg})^2$ and $1-|\bar{\bar{\nu}}_2|^2=1-|\bar{\nu}_2|^2+|\bar{\nu}_2\cdot \tnu_1|^2=|\nu_2\cdot \d\de F|_\tg^2+|\bar{\nu}_2\cdot \tnu_1|^2=O(|\d\de F|_{\tg})^2$, then
\begin{align*}
	\tnu_2-\nu_2&=\frac{1-|\bar{\bar{\nu}}_2|^2}{|\bar{\bar{\nu}}_2|(1+|\bar{\bar{\nu}}_2|)}\bar{\bar{\nu}}_2+\bar{\bar{\nu}}_2-\nu_2=O(|\d\de F|_{\tg})^2-\tg^{\al\be}\<\nu_2,\d_\al \de F\>\d_\be \tF-\<\bar{\nu}_2,\tnu_1\>\tnu_1\\
	&=O(|\d\de F|_{\tg})^2-\tg^{\al\be}\<\nu_2,\d_\al \de F\>\d_\be \tF.
\end{align*}

Thus by $\d \tF\cdot m=\d \de F\cdot m$, we obtain
\begin{align*}
I_1&=\Re\tilde \psi(\tnu_2-\nu_2)\cdot m-\Im\tilde\psi (\tnu_1-\nu_1)\cdot m=O(\tilde \psi|\d\de F|_{\tg}^2).
\end{align*}
Further, by $\de F=\Xi+U^\ga F_\ga$, we arrive at
\begin{align*}
I_2&=i\tg^{\al\be}\d^2_{\al\be}\tF\cdot (\tilde m-m)=i\tg^{\al\be}\d^2_{\al\be}\tF\cdot \big(O(|\d\de F|_{\tg}^2)-\tg^{\mu\nu}\<m,\d_\mu \de F\>\d_\nu \tF\big)\\
&=O(\tg^{\al\be} \d^2_{\al\be} \tF |\d\de F|_{\tg}^2)-i\tg^{\al\be}\tGa_{\al\be}^\mu(\d^A_\mu\om+U^\si \la_{\mu\si}).
\end{align*}

\emph{b) Estimate for $I_3$.} This term $I_3$ is expressed in the same manner as \eqref{LinEq-keyT}. Then we also have
\begin{align*}
I_3=i\tg^{\al\be}\big(\d^A_\al\d^A_\be \om -\la_\al^\si \Re(\la_{\be\si}\om)\big)+i\tg^{\al\be}\big(2\nab_\al U^\ga \la_{\be\ga}+U^\ga \nab^A_\al \la_{\be\ga}+\Ga_{\al\be}^\si U^\ga \la_{\si\ga}\big).
\end{align*}

\emph{c) Estimate for $I_4$.}
By the expression of $\de F$, we have
\begin{align*}
\tg_{\mu\nu}-g_{\mu\nu}&=\<\d_\mu \de F,\d_\nu \tF\>+\<\d_\mu F,\d_\nu \de F\>
=\<\d_\mu\de F,\d_\nu F\>+\<\d_\mu F,\d_\nu\de F\>+\d_\mu \de F \d_\nu\de F\\
&=-2\Re(\la_{\mu\nu}\bar{\om})+\nab_\mu U_\nu+\nab_\nu U_\mu+\d_\mu \de F \d_\nu\de F.
\end{align*}
Then we obtain
\begin{align*}
	I_4&=(\tg^{\al\be}-g^{\al\be})i\la_{\al\be}=-i\la_{\al\be}(g^{\al\mu}\de g_{\mu\nu}+\de g^{\al\mu}\de g_{\mu\nu})g^{\nu\be}\\
	&=2i\la^{\mu\nu}(\Re(\la_{\mu\nu}\bar{\om})-\nab_\mu U_\nu)-i\la^{\mu\nu}\d_\mu\de F \d_\nu\de F-i\la_\al^\nu \de g^{\al\mu}\de g_{\mu\nu}.
\end{align*}

Hence, from the above estimates, we obtain
\begin{align*}
	\d^B_t \om&=-iU^\si (\d^A_\si\psi-i\la_{\ga\si}V^\ga)+\tV^\ga (\d^A_\ga \om +U^\si \la_{\ga\si})
	-i\tg^{\al\be}\tGa_{\al\be}^\mu (\d^A_\mu \om+U^\ga \la_{\mu\ga})\\
	& +i\tg^{\al\be}\big(\d^A_\al\d^A_\be \om -\la_\al^\si \Re(\la_{\be\si}\bar{\om})\big)+i\tg^{\al\be}\big(2\nab_\al U^\ga \la_{\be\ga}+U^\ga \nab^A_\al \la_{\be\ga}+\Ga_{\al\be}^\si U^\ga \la_{\si\ga}\big)\\
	& +2i\la^{\mu\nu}(\Re(\la_{\mu\nu}\bar{\om})-\nab_\mu U_\nu)-i\la^{\mu\nu}\d_\mu\de F \d_\nu\de F-i\la_\al^\nu \de g^{\al\mu}\de g_{\mu\nu}+O(\d^2 F(\d\de F)^2)\\
	&=\tV^\ga \d^A_\ga \om+i\tg^{\al\be}\tilde\nab^A_\al \d^A_\be \om+i\la^{\mu\nu}\Re(\la_{\mu\nu}\bar{\om})-i(\tg^{\al\be}-g^{\al\be})\la_\al^\si\Re(\la_{\be\si}\bar{\om})\\
	& +i(\tg^{\al\be}-g^{\al\be})U^\ga \nab^A_\al \la_{\be\ga}+(\tV^\ga-V^\ga)U^\si \la_{\ga\si}-i\tg^{\al\be}(\tilde \Ga_{\al\be}^\mu-\Ga_{\al\be}^\mu)U^\ga\la_{\mu\ga}\\
	&+2i(\tg^{\al\be}-g^{\al\be})\nab_\al U^\ga \la_{\be\ga}+O(\d^2 F|\d\de F|_{\tg}^2).
\end{align*}
Hence the formula \eqref{dtom-Uniq} is obtained.
\end{proof}

\begin{proof}[Proof of Proposition~\ref{Prop-DiffBound}]
\ 

From the formula \eqref{dtom-Uniq} of $\om$ and \eqref{g_dt}, we derive
\begin{align*}
&\ \frac{1}{2}\frac{d}{dt}\|\om\|_{L^2(dvol_\tg)}^2=\Re \int \d^B_t \om\cdot \bar{\om}+|\om|^2 \frac{1}{4}\tg^{\al\be}\d_t \tg_{\al\be}\ dvol_\tg \\
&=\Re \int \big[ \tV^\ga \d^A_\ga \om+i\tg^{\al\be}\tilde\nab^A_\al \d^A_\be \om
\big] \bar{\om} +|\om|^2 \frac{1}{2}\tilde \nab_\al \tV^\al\ dvol_\tg
+\Re \int i\de g^{\al\be} U^\ga \nab^A_\al \la_{\be\ga}\bar{\om}\ dvol_\tg\\
&\quad +\Re \int \big[\de V^\ga U^\si \la_{\ga\si}+  i\la^{\mu\nu}\Re(\la_{\mu\nu}\bar{\om})-i\de g^{\al\be}\la_\al^\si\Re(\la_{\be\si}\bar{\om})\\
&\quad -i\tg^{\al\be}\de \Ga_{\al\be}^\mu U^\ga\la_{\mu\ga}+2i\de g^{\al\be} \nab_\al U^\ga \la_{\be\ga} +O(\d^2 F|\d \de F|_{\tg}^2) \big]\bar{\om}\ dvol_\tg \\
&=: I_1+I_2+I_3.
\end{align*}
Here the first term $I_1$ vanishes by integration by parts,
\begin{align*}
	I_1=\Re \int \tV \frac{1}{2}\d_\ga |\om|^2 +\frac{1}{2}\tilde \nab_\al \tV^\al|\om|^2 -i\tg^{\al\be} \d^A_\al \om \overline{\d^A_\be \om} \ dvol_\tg =0.
\end{align*}
The second term $I_2$ can also be estimated using integration by parts
\begin{align*}
    I_2&=-\int \tg^{\al\mu}\tg^{\be\nu}(\tg_{\mu\nu}-g_{\mu\nu})U^\ga \Im(\tilde\nab^A_\al \la_{\be\ga}\bar{\om})-(\tg^{\al\be}-g^{\al\be})U^\ga\Im((\Ga-\tGa)\la\bar{\om})\ dvol_\tg\\
    &=\Im\int \tg^{\al\mu}\tg^{\be\nu}\la_{\be\ga}\tilde\nab^A_\al\big((\tg_{\mu\nu}-g_{\mu\nu})U^\ga  \bar{\om}\big)\ dvol_\tg+O(\|U\|_{L^2}\|\om\|_{L^2})\\
    &=\Im\int \tg^{\al\mu}\tg^{\be\nu}\la_{\be\ga}\tilde\nab^A_\al\big(\d\de F\cdot(\d \tF+\d F)\de F\cdot \d F \de F\cdot m \big)\ dvol_\tg+O(\|U\|_{L^2}\|\om\|_{L^2})\\
    &=O\big(\int \d^2 F (\de F)^2+|\d\de F|_{\tg}^2 \de F \ dvol_{\tg}\big)+O(\|U\|_{L^2}\|\om\|_{L^2})\\
    &\leq C\|\de F\|_{L^2}^2.
\end{align*}
The last term $I_3$ is bounded by 
\begin{align*}
I_3 &\les \|\de V\|_{L^\infty}\|\la\|_{L^\infty}\|U\|_{L^2}\|\om\|_{L^2}+  \|(\tg,g)\|_{L^\infty}\|\la\|_{L^\infty}^2\|\om\|_{L^2}^2\\
&+\|\tg\|_{L^\infty}\|(\tGa,\Ga)\|_{L^\infty}\|\la\|_{L^\infty} \|U\|_{L^2}\|\om\|_{L^2}\\
&+\|(\d \de F\cdot \d F )\d( \de F\cdot \d F )\|_{L^2}\|\om\|_{L^2}\|\la\|_{L^\infty}+\|\d\de F\ \d\de F\|_{L^2}\|\om\|_{L^2}\\
&\leq C(\|U\|_{L^2}+\|\om\|_{L^2})^2+C(\|\d\de F\ \d\de F\|_{L^2}+\|\d\de F\ \de F\|_{L^2})\|\om\|_{L^2}\\
&\leq C\|\de F\|_{L^2}^2.
\end{align*}

Hence, we have
\begin{align*}
	\frac{1}{2}\frac{d}{dt}\|\om\|_{L^2(dvol_\tg)}^2\leq C\|\de F\|_{L^2}^2\leq C(\|\om\|_{L^2}^2+\|U\|_{L^2}^2)\leq C\|\om\|_{L^2}^2,
\end{align*}
where the last inequality is obtained using the property $|U|\les |\om|$ from Lemma~\ref{Lem4.4} on a short time interval,
and the constant $C$ only depends on $M$. Then the difference bound \eqref{DB-FtF} of (SMCF) follows by Gr\"onwall's inequality.
\end{proof}

\bigskip
\section{The initial data}   \label{Sec-Ini}
Our evolution begins at time $t = 0$, for which we must make a suitable gauge choice for the initial submanifold $\Sigma$. The original coordinates remain unchanged, which is sufficient for our purposes. The primary task is to select an orthonormal frame in $N\Sigma$ such that the bounds for $\lambda$ and $A$ are independent of the specific geometry of $\Sigma$. This issue reduces to the gauge choice on the background manifold $\Sigma_\bb$, where we will employ the modified Coulomb gauge.
Once this is done, we have the frame in the tangent space and the frame $(\nu_1,\nu_2)$ in the normal bundle. In turn, as described in Section \ref{Sec-gauge}, these generate the metric $g$, the second fundamental form $\la$ with trace $\psi$ and the connection $A$, all at the initial time $t = 0$. 

Here we will first carry out the construction of the orthonormal frame $\nu_\bb$ in $N\Sigma_\bb$, which is obtained using parallel transport method and the lifting criterion Proposition in \cite[p.61]{Hatcher}. 
Since $\Sigma$ is a small perturbation of $\Sigma_\bb$, we then use this to define the frame $\nu$ in $N\Sigma$.
Next, we prove bounds for the connections $A$ and the second fundamental form $\lambda$ that depend only on $M$.
The final objective of this section is to construct a family of regularized approximations to $\Sigma$. This allows us to estimate the norms of $(\lambda, g, A)$ in the function spaces $X^s$, $Y^{s+1}$, $Z^s$, respectively, and thus justify the initial data condition \eqref{psi-full-reg} for the Schr\"odinger-parabolic system \eqref{mdf-Shr-sys-2}-\eqref{par-syst}.

The main result of this section is stated below:

\begin{prop}[Initial data]\label{Global-harmonic}
Let $d\geq 2$, $s>\frac{d}{2}$ and $\si_d$ be given in \eqref{Constant}. Let
$F:(\R^d_x,g)\rightarrow (\R^{d+2},g_{\R^{d+2}})$
be an immersion with induced metric satisfying \eqref{MetBdd}. 
Assume that the metric $g$ and the mean curvature $\bH$ are finite, i.e.
\begin{equation*}
\| |D|^{\si_d}g\|_{H^{s+1-\si_d}}+\|\bH\|_{H^s}\leq M.
\end{equation*}

i) There exists a global orthonormal frame $\nu:=(\nu_1,\nu_2)$ on $\Sigma$ such that 
\begin{align}   \label{ini}
	\|\la\|_{H^s}\leq M,\qquad \| |D|^{\de_d} A \|_{H^{s-\de_d}} \lesssim M,\qquad \|\d\nu\|_{\dot H^{2\de_d}\cap \dot H^s}\les C(M).
\end{align}
    
ii) There exists a family of regularized submanifolds of $\Sigma$, denoted as $\Sigma^{(h)}$ with $h\in [h_0,\infty)$, such that the ellipticity and Sobolev embedding conditions are satisfied
\begin{gather}  \label{Reg-met}
	\frac{9}{10}c_0\leq (g^{(h)})\leq \frac{11}{10}c_0^{-1}I,\\ \label{SECs}
	|\Ric^{(h)}|\leq C(M),\qquad \inf_{x\in\Sigma^{(h)}} {\rm Vol}_{g^{(h)}}(B_x(e^{C(M)2^{-h_0}}))\geq v e^{-C(M)2^{-h_0}}.
\end{gather}
Moreover, we have the uniform bounds 
\begin{align}  \label{Reg-gAla}
\|g\|_{Y^{s+1}}+\|A\|_{Z^{s}}+\|\la\|_{X^{s}}\les C(M).
\end{align} 
\end{prop}

We remark that the bounds in \eqref{ini} are the only way the generalized Coulomb gauge condition at $t = 0$ enters this paper. 
Later, for the analysis of the Schr\"odinger-parabolic system \eqref{mdf-Shr-sys-2}-\eqref{par-syst} that follows, we instead assume the initial data $(\lambda, g, A)$ satisfies the conditions \eqref{Reg-gAla}, which provide uniform control of the Sobolev norms for this data.

\subsection{Global orthonormal frame and the initial data $(\la,g,A)$}
The $\la$ and $A$ are determined by the initial manifold $\Sigma$ given a gauge choice, which consists 
of choosing (i) a good set of coordinates
on $\Sigma$, where the original coordinates are used,
and (ii) a good orthonormal frame in $N \Sigma$,
where we will use the generalized Coulomb gauge.

In our previous article \cite{HT21,HT22}, the orthonormal frame in $N\Sigma$ was easily constructed due to the small data. However, this issue would be more complicated for large data, as the topology of submanifold must also be taken into account. To address this, we first construct a smooth modified Coulomb frame in the smooth normal bundle $N\Sigma_\bb$. This allow us to define $\lambda_\bb$ and $A_\bb$ and directly establish $H^s$ bounds for them. 
Then the manifold $\Sigma$ and its orthonormal frame are treated as the small perturbations of background manifold $\Sigma_\bb$ and $\nu_\bb$, respectively. 

Here we start with the following lemma: by choosing $N_1$ sufficiently large, we fix the background manifold $\Sigma_\bb$ and bound the differences $\d(F-F_\bb)$ and $g-g_\bb$ in $H^{s}$.

\begin{lemma} \label{d2F-lem}
Let $d\geq 2$, $s>\frac{d}{2}$ and $\si_d$ be given in \eqref{Constant}. 
Let $F:(\R^d,g)\rightarrow (\R^{d+2},g_{\R^{d+2}})$
be an immersion with metric $c_0I\leq g\leq c_0^{-1}I$, $\||D|^{\si_d}g\|_{H^{s+1-\si_d}}\leq M$ and mean curvature $\|\bH\|_{H^s}\leq M$ in some coordinates. Then we have
\begin{equation}     \label{d2F}
\|\d^2 F\|_{H^s}\lesssim C(M).
\end{equation}
Moreover, for background manifold $F_\bb=P_{\leq N_1}F$ and $\ep_0:=2^{-N_1}\ll 1$, then we have
\begin{align*}
\|\d(F-F_\bb)\|_{H^{s}}\les \ep_0,\qquad \|g-g_\bb\|_{H^{s}}\les \ep_0.
\end{align*}
\end{lemma}
\begin{proof}[Proof]
	By $c_0I\leq g\leq c_0^{-1}I$ and Sobolev embeddings, we have
	\begin{align*}
		\|\d^2 F\|_{L^2}^2&\les \int g^{\al\be}\d_\al \d F\cdot \d_\be \d F\ dx =\int -g^{\al\be}\d \d^2_{\al\be}F\cdot \d F-\d_\be g^{\al\be}\d_\al \d F \cdot \d F\ dx \\
		&\les \int g^{\al\be} \d^2_{\al\be}F\cdot \d^2 F+|\d g^{-1} \d^2 F \cdot \d F|\ dx\\
		&=\int (\bH+g^{\al\be}\Ga_{\al\be}^\ga \d_\ga F)\cdot \d^2 F+|\d g^{-1} \d^2 F \cdot \d F|\ dx\\
		&\leq (\|\bH\|_{L^2}+\|g^{-1}\|_{L^\infty}^2\|\d g\|_{L^2}\|\d F\|_{L^\infty}+\|\d g^{-1}\|_{L^2}\|\d F\|_{L^\infty})\|\d^2 F\|_{L^2}\\
		&\leq (M+M^4+M^2)\|\d^2 F\|_{L^2}
	\end{align*}
	Then we obtain $\|\d^2 F\|_{L^2}\les M^4$.
	For high regularity $\dot H^s$, we similarly  have
	\begin{align*}
		&\ \|\d^2 F\|_{\dot H^s}\lesssim  \int g^{\al\be}\d_\al \d^{s+1} F\cdot \d_\be \d^{s+1} F\ dx\\
		&=\int \d^s(\bH+g^{\al\be}\Ga_{\al\be}^\ga \d_\ga F)\cdot \d^{s+2} F+|\d g \d^{s+2} F \cdot \d^{s+1} F|+|[g,\d^{s+1}]\d^2 F\d^{s+1}F|\ dx\\
		&\les  (\|\bH\|_{H^s}+C(M)\|\d F\|_{L^\infty\cap \dot H^s})\|\d^2 F\|_{\dot H^s}+\|\d g\|_{H^s}\|\d^2 F\|_{H^s}\|\d^2 F\|_{\dot H^{s-1}}\\
		&\leq \big(\|\bH\|_{H^s}+C(M)(1+\|\d^2 F\|_{L^2}^{\frac{1}{s}}\|\d^2 F\|_{\dot H^{s}}^{\frac{s-1}{s}}\big)\|\d^2 F\|_{\dot H^s}\\
		&\ +\|\d g\|_{H^s}(C(M)+\|\d^2 F\|_{\dot H^s})\|\d^2 F\|_{L^2}^{\frac{1}{s}}\|\d^2 F\|_{\dot H^{s}}^{\frac{s-1}{s}}\\
		&\leq C(M)+\ep \|\d^2 F\|_{\dot H^s}^2.
	\end{align*}
	where the last term can be absorbed. Thus the bound \eqref{d2F} is obtained.
	
	From $F_\bb=P_{\leq N_1}F$ and the bound \eqref{d2F}, for any $0<\ep_0\ll 1$ we choose $2^{-N_1}\sim \ep_0$, then
	\begin{align*}
		\|\d (F-F_\bb)\|_{H^{s}}=\|\d P_{>N_1}F\|_{H^{s}}\les 2^{-N_1}\|\d^2 P_{>N_1}F\|_{H^s}\les \ep_0.
	\end{align*}
	Moreover, for the metric we have
	\begin{align*}
		\|g-g_\bb\|_{H^{s}}=&\ \|\d P_{>N_1}F \d F+\d P_{\leq N_1}F \d P_{>N_1}F\|_{H^{s}}\\
		\les &\ \|\d P_{>N_1}F\|_{H^{s}}\|\d F\|_{L^\infty\cap \dot H^{s}}\les C(M)\ep_0.
	\end{align*}
	This completes the proof of lemma.
\end{proof}

Now we construct the orthonormal frame in $N\Sigma_\bb$.
\begin{lemma}[Modified Coulomb gauge on $\Sigma_\bb$]\label{Lem-Gaugebb}
On the smooth background submanifold $\Sigma_\bb$ in $\R^{d+2}$, there exists a smooth orthonormal frame $\nu=(\nu_1,\nu_2)$  in $N \Sigma_\bb$ such that $\d \nu\in H^k$ for any $k\geq 0$.
Moreover, there exists a modified Coulomb gauge $\nu_\bb=(\nu_{\bb,1},\nu_{\bb,2})$ with $\d_\al A_{\bb,\al}=0$ by rotating the frame $\nu$. We then have the following bounds
\begin{equation} \label{Smnubb}
\|\d \nu_\bb\|_{L^{\frac{2d}{d-2\de_d}}}+\|\d \nu_\bb\|_{\dot H^{2\de_d}\cap\dot H^{\si}}+\||D|^{\de_d}A_{\bb}\|_{H^{\si+1-\de_d}}+\|\la_\bb\|_{H^{\si}}\les 2^{(\si-s)^+ N_1}C(M),\quad \si\geq s,
\end{equation}
where $\de_d$ is given in \eqref{Constant} and $(\si-s)^+=\max\{0,\si-s\}$.
\end{lemma}

\begin{rem}
(i) Note that the gauge condition $\partial_\alpha A_{\bb,\al}= 0$ depends strongly on the choice of coordinates. However, it ensures that the bounds for $\nu_\bb$ and $A_\bb$ are independent of the construction of $(\nu_1,\nu_2)$ and depend only on $M$. (ii) The bounds for $\nu_\mathbf{b}$ and $A_\mathbf{b}$ are worse in two dimensions because we have to solve $\Delta A_\mathbf{b} = \d(\lambda_\mathbf{b}^2)$. Furthermore, we must deal with their low-frequency part carefully.
\end{rem}

\begin{proof}
\emph{Step 1: We construct a normal frame $\nu^{(int)}$ on $F_\bb(B_{x_0}(R+1))$, which is a topologically trivial compact manifold with boundary.} 

Choose $x_0$ and a normal frame $\nu(x_0)=(\nu_1(x_0),\nu_2(x_0))$ at $F_\bb(x_0)$, extend the frame in all directions. Look on a ray 
$x = x_0 + h \omega$ and construct $\nu_1(x)$ by
\[
\frac{d}{dh} \nu_1(x) = v(x)
\]
so that $v(x)$ is tangent and $\partial_h (\nu_1 \cdot \d_\al F_\bb(x)) = 0$ for any $\al=1,\cdots ,d$. This gives
\[
\partial_h (\nu_1 \cdot \d_\al F_\bb(x)) = v \cdot \d_\al F_\bb + \nu_1 \cdot \partial_h \d_\al F_\bb=v \cdot \d_\al F_\bb + \nu_1 \cdot \omega^\ga  \d^2_{\al\ga} F_\bb=0.
\]
So we get 
\begin{align*}
v& = v^\al \d_\al F_\bb=g^{\al\be}(v\cdot \d_\be F_\bb)\d_\al F_\bb=-g^{\al\be}(\nu_1\cdot \omega^\ga \d^2_{\be\ga} F_\bb)\d_\al F_\bb\\
&=-g^{\al\be}\omega^\ga \d_\al F_\bb \d^2_{\be\ga}  F_\bb^T \nu_1=:G(x) \nu_1,
\end{align*}
where $G(x)\in \R^{(d+2)\times (d+2)}$ is a smooth matrix. Then we obtain a linear ODE
\[\frac{d}{dh}\nu_1(x)=G(x)\nu_1(x),\quad x=x_0+h\omega,\]
which has a unique solution along any ray for given initial data $\nu(x_0)$,
\[ \nu_1(x_0+h\omega)=e^{\int_0^h G(x_0+\tau\omega) d\tau}\nu_1(x_0).\] 

In a similar way, we construct $\nu_2(x)$ by
\[  \frac{d}{dh}\nu_2(x)=w(x),  \]
so that $w\in span\{\nu_1,\d_1 F_\bb,\cdots,\d_d F_\bb\}$ and $\nu_2 \perp \nu_1$, $\nu_2\perp \d_\al F_\bb$. Then we have
\[\frac{d}{dh}(\nu_2\cdot \nu_1)=w\cdot \nu_1+\nu_2 \cdot G(x)\nu_1=0,  \]
\[  \frac{d}{dh}(\nu_2\cdot \d_\al F_\bb)=w\cdot \d_\al F_\bb+\nu_2 \cdot \omega^\ga \d^2_{\al\ga}F_\bb=0.    \]
So we get
\begin{align*}  
w&=(w\cdot \nu_1)\nu_1+g^{\al\be}(w\cdot \d_\be F_\bb)\d_\al F_\bb= -\nu_1 \nu_1^T G^T \nu_2+G(x)\nu_2\\
&= (-\nu_1 \nu_1^T G^T \nu_2+G(x))\nu_2=:H(x)\nu_2,
\end{align*}
where $H(x)$ is a smooth matrix. Then $\nu_2$ is given by
\[  \nu_2(x_0+h\omega)=e^{\int_0^h H(x_0+\tau \omega)d\tau} \nu_2(x_0).   \]

Hence, we obtain $\nu^{(int)}:=(\nu_1,\nu_2)$ on $F_\bb(B_{x_0}(R+1))$. Moreover, since the matrices $G(x)$ and $H(x)$ are smooth, we also have the Sobolev bound 
\[\|\nu^{(int)} \|_{H^s(B_{x_0}(R+1))} \les C.\] 

\emph{Step 2. We construct a normal frame $\nu^{(ext)}$ on $F_\bb(\R^d\setminus B_{x_0}(R))$, where our manifold is almost flat.}

Since the vector $\d_x F_\bb(x)$ converges to $\d_x F_\bb(\infty)$ as $x\rightarrow \infty$, then there exists a large number $R$, such that
\[   |\d_x F_\bb(x)-\d_x F_\bb(\infty)|\leq \epsilon, \quad x\in B_{x_0}^c(R):=\R^d \setminus B_{x_0}(R).      \]
This means that $\d_x F_\bb$ has a small variation in $L^\infty$ on $\R^d \setminus B_{x_0}(R)$. So we can choose $\tilde\nu$ constant uniformly transversal to $T\Sigma(B_{x_0}^c(R))$ where $\Sigma(B_{x_0}^c(R))=F_\bb(B_{x_0}^c(R))$. Projecting $\tilde\nu$ on the normal bundle $N\Sigma(B_{x_0}^c(R))$ and normalizing we obtain a normalized section $\nu^{(ext)}_1$ of the normal bundle with the same regularity as $\d F_\bb$. Then we continuously choose $\nu^{(ext)}_2$ in $N\Sigma(B_{x_0}^c(R))$ perpendicular to $\nu^{(ext)}_1$. We obtain the orthonormal frame $\nu^{(ext)}=(\nu^{(ext)}_1,\nu^{(ext)}_2)$ in $N\Sigma(B_{x_0}^c(R))$, which again has the same regularity and bounds as $\d_x F_\bb$, namely 
\[  \|\d \nu^{(ext)}\|_{H^s(B^c_{x_0}(R))}\les M.  \]

\medskip
\emph{Step 3: Gluing the two normal vectors $\nu^{(int)}$ and $\nu^{(ext)}$ smoothly on the annulus $\{ y: R\leq |y-x_0|\leq R+1  \}$.}

In the annulus $A(R)=\{ y: R\leq |y-x_0|\leq R+1  \}$, outside the large ball $B_{x_0}(R)$, we have two frames:

\begin{itemize}
    \item $\nu^{(int)}=(\nu^{(int)}_1,\nu^{(int)}_2)$, which has large Sobolev norm;

    \item $\nu^{(ext)}=(\nu^{(ext)}_1,\nu^{(ext)}_2)$, which is almost constant ( $\epsilon$ close to constant)
\end{itemize}
Then we'd like to smoothly deform $\nu^{(int)}$ into $\nu^{(ext)}$, that is, enter the annulus with $\nu^{(int)}$ and exit with $\nu^{(ext)}$.

The relation between $\nu^{(int)}$ and $\nu^{(ext)}$ is given by
\[
\nu^{(int)}_1+i\nu^{(int)}_2 = e^{i\theta} (\nu^{(ext)}_1+i\nu^{(ext)}_2), \qquad \theta: A(R) \to S^1 = \R/2\pi Z.
\]
Here $\theta$ is smooth. We claim that: there exists a unique lifting to the universal covering smoothly 
\begin{align}  \label{lift}
\tilde\theta: A(R) \to \R    
\end{align}
such that $\theta=p\circ \tilde \theta$, where $p:\mathbb R\rightarrow S^1$ is the covering map.
Then we can obtain a global orthonormal frame $\nu$ on $N\Sigma_\bb$ by defining 
\[
\nu_1+i\nu_2 = e^{i \chi \theta} (\nu^{(ext)}_1+i\nu^{(ext)}_2)\,,
\]
where $\chi:\mathbb R^d\rightarrow [0,1]$ is a smooth function with $\chi= 1$ on inside sphere $B_{x_0}(R)$ and $\chi = 0$ on the outside $\mathbb R^d\setminus B_{x_0}(R+1)$.

Now we prove the existence of the lifting \eqref{lift}. Lifting is a topological problem. Since the fundamental group $\pi_1(\R)=0$ is trivial, we have $p_*(\pi_1(\mathbb R))=0$. 
By the lifting criterion, i.e. Proposition 1.33 in \cite[p.61]{Hatcher}, and since $\pi_1(A(R))\cong\pi_1(S^{d-1})$, a lift $\tilde\theta:A(R) \rightarrow \mathbb R$ of $\theta:A(R) \rightarrow S^1$ exists if and only if $\theta_*(\pi_1(S^{d-1}))\subset p_*(\pi_1(\mathbb R))=0$.
Then we consider the following two cases:

\emph{a) $d \geq 3$.} 
Here the homotopy group $\pi_{1}(S^{d-1})=0$ for $d\geq 3$, which is trivial, therefore we have $\theta_*(\pi_1(S^{d-1}))=0\subset p_*(\pi_1(\mathbb R))$. 

\emph{b) $d=2$.} Since the homotopy group $\pi_1(S^1)\cong\mathbb Z$, not all maps $\theta : S^1 \to S^1$ are topologically trivial
as characterized by the rotation number. Therefore, we need to prove that the frame $\nu^{(int)}=(\nu_1,\nu_2)$ on $B_{x_0}(R+1)$ is topologically trivial.
By the winding number formula, we have
\[
I(R)=\frac{1}{2\pi i}\int_{T_R} \frac{d(\nu_1+i\nu_2)}{\nu_1+i\nu_2} =\frac{1}{4\pi} \int_{T_R} -\d_x \nu_1\cdot \nu_2+\d_x \nu_2 \cdot \nu_1 dx=\frac{1}{2\pi} \int_{T_R}  \nu_1\cdot \d_x\nu_2\ dx,
\]
where $T_R:= \{y:|y-x_0|=R\}$. 
Now consider the same integral over smaller circles
\[
I(r) = \frac{1}{2\pi}\int_{T_r} \nu_1 \cdot \d_x \nu_2 dx,
\qquad r \in [0,R],
\]
which is continuous for $r\in [0,R+1]$ since the frame $\nu_1,\nu_2$ are constructed smoothly.
We know that
\[
I(0) = 0.
\]
Since $I(r)$ takes values in $\mathbb Z$, then $I(r) \equiv 0$, and hence $\nu^{(int)}$ is topologically trivial. 
Therefore, from the lifting criterion, there exists a unique lifting $\tilde\theta: A(R)\rightarrow \mathbb R$ of $\theta:A(R)\rightarrow \mathbb Z$ such that $\theta=p\circ \tilde \theta$ for all $d\geq 2$.

\medskip
\emph{Step 4: Constructing the Coulomb frame $\nu_\bb$ in $N \Sigma_\bb$ by rotating the frame $\nu$.}

The bound $\d \nu\in H^k$ in particular implies that the associated connection and the second fundamental form are also finite in $H^k$. However, these bounds would depend on the specific profile of $\Sigma_\bb$.
Hence we should rotate it to get a suitable frame $\nu_\bb=(\nu_{\bb,1},\nu_{\bb,2})$, i.e. we define
\[\nu_{\bb,1} + i \nu_{\bb,2} = e^{i\theta} (\nu_1+ i\nu_2),\]
where we impose the modified Coulomb gauge condition $\d_\al A_{\bb,\al}=0$.
Then we have $A_{\bb,\al} = A_\al -\partial_\al \theta$,
and the rotation angle $\theta$ must solve 
$\Delta \theta = \d_\alpha A_{\al}$.     
It directly follows that $\d\nu_\bb$ and $A_\bb$ are finite in $H^k$.

Next, we prove that $\nu_{\bb}$ and $A_\bb$ also satisfy the bounds \eqref{Smnubb}.
In the modified Coulomb gauge, the connection $A_\bb$ satisfies 
\[  \d_\al A_{\bb,\al}=0, \qquad \d_\al A_{\bb,\be}-\d_\be A_{\bb,\al}=\Im(\la_{\bb,\al\ga}\bar{\la}^\ga_{\bb,\be}).   \]
Using these equations we derive a second order elliptic equation for $A_\bb$, namely
\begin{equation} \label{Ellp-Ab}
    \Delta A_{\bb,\al}=\d_\be \Im(\la_{\bb,\be\si}\bar{\la}^\si_{\bb,\al}),
\end{equation}
where $\Delta$ is the standard Laplacian operator. 

From \eqref{Ellp-Ab} we have
\[  \||D|^{\delta_d} A_\bb\|_{H^{1-\de_d}}+\|A_\bb\|_{L^\infty}\les \|\la_\bb\|_{L^2\cap L^\infty}^2\les C(M),  \]
then we obtain the first bound in \eqref{Smnubb} for $\d \nu_\bb$
\begin{align*}
    &\ \|\d \nu_\bb\|_{L^{\frac{2d}{d-2\de_d}}\cap L^\infty}\les \|A_\bb\nu_\bb\|_{L^{\frac{2d}{d-2\de_d}}\cap L^\infty}+\|\la_\bb \d F_\bb\|_{L^{\frac{2d}{d-2\de_d}}\cap L^\infty}\\
    &\les \|A_\bb\|_{L^{\frac{2d}{d-2\de_d}}\cap L^\infty}+\|\la_\bb\|_{L^{\frac{2d}{d-2\de_d}}\cap L^\infty}\|\d F_\bb\|_{L^\infty}
    \les C(M)+\|\d^2 F_\bb\|_{L^{\frac{2d}{d-2\de_d}}\cap L^\infty}M^{1/2}\les C(M),
\end{align*}
and hence the bound for $\d\nu_\bb\in \dot H^{2\de_d}$
\begin{align*}
&\ \|\d\nu_\bb\|_{\dot H^{2\de_d}}\les \|A_\bb \nu_\bb\|_{\dot H^{2\de_d}}+\|\la_\bb \d F_\bb\|_{\dot H^{2\de_d}}=\|A_\bb \nu_\bb\|_{\dot H^{2\de_d}}+\|\d^2 F_\bb \nu_\bb \d F_\bb\|_{\dot H^{2\de_d}}\\
&\les (\|A_\bb\|_{\dot H^{2\de_d}\cap \dot H^{\de_d}}+\|\d^2 F_\bb \d F_\bb\|_{\dot H^{2\de_d}\cap \dot H^{\de_d}})(\|\nu_\bb\|_{L^\infty}+\|\d \nu_\bb\|_{L^{\frac{2d}{d-2\de_d}}})\\
&\les (C(M)+\|\d^2 F_\bb\|_{H^{2\de_d}}\|\d F_\bb\|_{L^\infty})C(M)
\les C(M).
\end{align*}
To bound the higher derivatives
of $A_\bb$ and $\nu_\bb$, for any $\si\geq s$ we have
\begin{align*}
    \|A_\bb\|_{\dot H^{\si+1}}\les \|\la_\bb^2\|_{\dot H^\si}\les \|\la_\bb\|_{\dot H^\si}\|\la_\bb\|_{L^\infty}.
\end{align*}
and by \eqref{ineqfg} we have
\begin{align*}
    &\ \|\la_\bb\|_{\dot H^\si}\les \|\d^2 F_\bb \cdot \nu_\bb\|_{\dot H^\si}\les \|\d^2 F_\bb\|_{ H^\si}\|\nu_\bb\|_{L^\infty}+\|\d^2 F_\bb\|_{L^\infty}\|P_{>0}\nu_\bb\|_{\dot H^\si}\\
    &\les 2^{(\si-s)^+N}C(M)+C(M)\|P_{>0} \nu_\bb\|_{L^2}^{\frac{1}{\si+1}}\|P_{>0} \nu_\bb\|_{\dot H^{\si+1}}^{\frac{\si}{\si+1}}\\
    &\les 2^{(\si-s)^+N}C(M)+C(M)\|P_{>0}\d \nu_\bb\|_{\dot H^{2\de_d}}^{\frac{1}{\si+1}}\|\d \nu_\bb\|_{\dot H^{\si}}^{\frac{\si}{\si+1}}.
\end{align*}
Then we have
\begin{align*}
    &\ \|\d \nu_\bb\|_{\dot H^\si}\les \|A_\bb\nu_\bb+\la_\bb \d F_\bb\|_{\dot H^\si}\\
    &\les \|A_\bb\|_{\dot H^\si}\|\nu_\bb\|_{L^\infty}+\|A_\bb\|_{\dot H^{\de_d}\cap L^\infty}(\|\d P_{\leq 0}\nu_\bb\|_{\dot H^{2\de_d}}+\|P_{>0}\nu_\bb\|_{\dot H^\si})\\
    &\quad\  +\|\la_\bb\|_{\dot H^\si}\|\d F_\bb\|_{L^\infty}+\|\la_\bb\|_{L^\infty}\|\d F_\bb\|_{\dot H^\si}\\
    &\les C(M)\|\la_\bb\|_{\dot H^{\si}}+C(M)2^{(\si-s)^+N}+C(M)\|\d \nu_\bb\|_{\dot H^{2\de_d}}^{\frac{1}{\si+1}}\|\d \nu_\bb\|_{\dot H^\si}^{\frac{\si}{\si+1}}\\
    &\leq C(M)2^{(\si-s)^+ N}+\frac{1}{2}\|\d \nu_\bb\|_{\dot H^\si},
\end{align*}
which yields the second bound in \eqref{Smnubb} for $\d \nu_\bb$. Combining with the previous estimates of $\la_\bb$ and $A_\bb$, we can also obtain the other two bounds in \eqref{Smnubb}.
This concludes the proof of Lemma~\ref{Lem-Gaugebb}.
\end{proof}

Now we construct the normal frame $(\nu_1,\nu_2)$ in $N\Sigma$ as the small perturbation of $(\nu_{\bb, 1},\nu_{\bb, 2})$ using projections and Schmidt orthogonalization, and then bound the $H^s$-norms for $\lambda$ and $A$.
Since the manifold $\Sigma$ is a perturbation of $\Sigma_\bb$, let
\begin{align*}
    \bar{\nu}_j:=\nu_{\bb,j}-g^{\al\be}\<\nu_{\bb,j},\d_\al (F-F_\bb)\>\d_\be F \in N\Sigma,
\end{align*}
which are  normal vectors in $N\Sigma$. Then by Schmidt orthogonalization, we can construct the orthonormal frame $(\nu_1,\nu_2)$ in $N\Sigma$ as 
\begin{align} \label{ConsNu}
    \nu_1=\frac{\bar{\nu}_1}{|\bar{\nu}_1|},\quad \nu_2=\frac{\bar{\bar{\nu}}_2}{|\bar{\bar{\nu}}_2|},\qquad \text{with  }\  \bar{\bar{\nu}}_2:=\bar{\nu}_2-\<\bar{\nu}_2, \nu_1\>\nu_1.
\end{align}

We have the following lemma:
\begin{lemma}
\begin{equation} \label{nu510}
    \||\bar{\nu}_1|^{-1}\|_{\dot H^{[s]+1}}+\||\bar{\bar{\nu}}_2|^{-1}\|_{\dot H^{[s]+1}}\les C(M).
\end{equation}
\end{lemma}
\begin{proof}
For any vector $|v|\sim 1$, by interpolation \eqref{HaInterpo} it holds
\begin{equation} \label{HNv}
\begin{aligned}
    &\ \||v|^{-1}\|_{\dot H^N}\les \sum_{1\leq j\leq N}\sum_{\substack{
         l_1+\cdots +l_j=N,  \\
          l_i\geq 1
    }}\||v|^{-2j-1}\d^{l_1}|v|^2\cdots \d^{l_j}|v|^2\|_{L^2}\\
    &\les \sum_{1\leq j\leq N}\sum_{\substack{
         l_1+\cdots +l_j=N,  \\
          l_i\geq 1
    }}\||v|^{-2j-1}\|_{L^\infty}\|\d^{l_1}|v|^2\|_{L^{\frac{2N}{l_1}}}\cdots \|\d^{l_j}|v|^2\|_{L^{\frac{2N}{l_j}}}\\
    &\les \|\d^N |v|^2\|_{L^2}\les \|v\|_{\dot H^N}.
\end{aligned}
\end{equation}
Then this can be used to bound 
\begin{align*}
\||\bar{\nu}_{1}|^{-1}\|_{\dot H^{[s]+1}}&
\les \| \bar{\nu}_{1}\|_{\dot H^{[s]+1}},\qquad \||\bar{\bar{\nu}}_{2}|^{-1}\|_{\dot H^{[s]+1}}\les \| \bar{\bar{\nu}}_{2}\|_{\dot H^{[s]+1}}.
\end{align*}
By the formula for $\bar{\bar{\nu}}_{2}$, we further have
\begin{align*}
    \|\bar{\bar{\nu}}_{2}\|_{\dot H^{[s]+1}}&\les \|\bar{\nu}_{2}\|_{\dot H^{[s]+1}}+\||\bar{\nu}_{1}|^{-2}\<\bar{\nu}_{2},\bar{\nu}_{1}\>\bar{\nu}_{1}\|_{\dot H^{[s]+1}}\\
    &\les \|\bar{\nu}_{2}\|_{\dot H^{[s]+1}}+\||\bar{\nu}_{1}|^{-1}\|_{\dot H^{[s]+1}}\||\bar{\nu}_{1}|^{-1}\bar{\nu}_{1}\<\bar{\nu}_{2},\bar{\nu}_{1}\>\bar{\nu}_{1}\|_{L^\infty}\\
    &+ \|\bar{\nu}_{2}\|_{\dot H^{[s]+1}}\||\bar{\nu}_{1}|^{-2}\bar{\nu}_{1}\bar{\nu}_{1}\|_{L^\infty}+\| \bar{\nu}_{1}\|_{\dot H^{[s]+1}}\||\bar{\nu}_{1}|^{-2}\bar{\nu}_{2}\bar{\nu}_{1}\|_{L^\infty}\\
    &\les \|\bar{\nu}\|_{\dot H^{[s]+1}}.
\end{align*}
Since
\begin{align*}
    &\ \|\bar{\nu}\|_{\dot H^{[s]+1}}\les \|\nu_\bb\|_{\dot H^{[s]+1}}+\|g_\bb \<\nu_\bb, \d(F- F_\bb)\>\d F_\bb \|_{\dot H^{[s]+1}}\\
    &\les \|\nu_\bb\|_{\dot H^{[s]+1}}+\|g_\bb\|_{\dot H^{[s]+1}}\|\<\nu_\bb, \d (F-F_\bb)\>\d F_\bb \|_{L^\infty}\\
    &\quad+\|\nu_\bb\|_{\dot H^{[s]+1}}\|g_\bb \d(F-F_\bb)\d F_\bb \|_{L^\infty}
    +\|\d(F-F_\bb) \|_{\dot H^{[s]+1}}\|g_\bb \nu_\bb \d F_\bb\|_{L^\infty}\\
    &\quad +\|\d F_\bb\|_{\dot H^{[s]+1}}\|g_\bb \nu_\bb \d(F-F_\bb)\|_{L^\infty}\\
    &\les \|\nu_\bb\|_{\dot H^{[s]+1}}+\|g_\bb\|_{\dot H^{[s]+1}}+\|\d F\|_{\dot H^{[s]+1}}
    \les C(M).
\end{align*}
Hence the estimates \eqref{nu510} are obtained.
\end{proof}

Then we can also obtain the estimate:
\begin{lemma}\label{Lemreg-tA}
    The connection coefficients $A_\mu=\d_\mu \nu_1\cdot \nu_2$ have the following properties:
    \begin{equation} \label{reg-tA}
        \| A-A_\bb\|_{H^s}\les_M \ep_0.
    \end{equation}
\end{lemma}

\begin{proof}
    Since $\bar{\nu}_1 \perp \nu_2$, we have $ A_j=\d_j   \nu_1\cdot   \nu_2=\frac{\d_j \bar{\nu}_1}{|\bar{\nu}_1|}\cdot   \nu_2$.
    We can rewrite the $  A-A_\bb$ as 
\begin{align*}
    &\   A_\mu-A_{\bb,\mu}=\d_\mu   \nu_1\cdot   \nu_2-\d_\mu \nu_{\bb,1}\cdot \nu_{\bb,2}\\
    &=\frac{1}{|\bar{\nu}_1||\bar{\bar{\nu}}_2|}[\d_j \nu_{\bb,1}-\d_j (g^{\al\be}\<\nu_{\bb,1},\d_\al (F-F_\bb)\>\d_\be F)]\\
    &\quad \cdot [\nu_{\bb,2}-g^{\al\be}\<\nu_{\bb,2},\d_\al (F-F_\bb)\>\d_\be F-\<\bar{\nu}_2,  \nu_1\>  \nu_1]-\d_\mu \nu_{\bb,1}\cdot \nu_{\bb,2}
\end{align*}
\emph{a) We estimate the term 
\begin{align*}
    \|\frac{1}{|\bar{\nu}_1||\bar{\bar{\nu}}_2|}\d_\mu \nu_{\bb,1}\cdot \nu_{\bb,2}-\d_\mu \nu_{\bb,1}\cdot \nu_{\bb,2}\|_{H^s}=\|A_{\bb,j}\frac{1-|\bar{\nu}_1|^2|\bar{\bar{\nu}}_2|^2}{|\bar{\nu}_1||\bar{\bar{\nu}}_2|(1+|\bar{\nu}_1||\bar{\bar{\nu}}_2|)}\|_{H^s}\les \ep_0^2.
\end{align*}}

Since $A_\bb\in L^\infty\cap \dot H^s$, $|\bar{\nu}_1|\sim | \bar{\bar{\nu}}_2|\sim 1$, and by \eqref{nu510} we have $P_{>0}|\bar{\nu}_1|^{-1},\ P_{>0}|\bar{\bar{\nu}}_2|^{-1},\ P_{>0}(1+|\bar{\nu}_1||\bar{\bar{\nu}}_2|)^{-1}\in \dot H^s$, they are all bounded by $C(M)$. Then by \eqref{ineqfg}, it suffices to bound $1-|\bar{\nu}_1|^2|\bar{\bar{\nu}}_2|^2$ in $H^s$. We denote 
\[X_j:=\<\nu_{\bb,j},\d(F-F_\bb)\>.\] 
From $\bar{\nu}_1$, $\bar{\bar{\nu}}_2$ and $\nu_{\bb,j}\perp \d F_\bb$, we have
\begin{align*}
    |\bar{\nu}_1|^2=1-|X_1|^2_g,\qquad 
    |\bar{\bar{\nu}}_2|^2=1-|X_2|_g^2-\<\bar{\nu}_2,  \nu_1\>^2=1-|X_2|_g^2-|\bar{\nu}_1|^{-2}\<X_1,X_2\>_g^2,
\end{align*}
which yields
\begin{align*}
    1-|\bar{\nu}_1|^2|\bar{\bar{\nu}}_2|^2=|X_1|_g^2+|X_2|_g^2+|\bar{\nu}_1|^{-2}\<X_1,X_2\>_g^2-|X_1|_g^2(|X_2|_g^2+|\bar{\nu}_1|^{-2}\<X_1,X_2\>_g^2).
\end{align*}
    Since 
    \begin{align*}
        &\|X_j\|_{H^s}\les (\|\nu_{\bb,j}\|_{L^\infty}+\|P_{>0}\nu_{\bb,j}\|_{\dot H^s})\|\d(F-F_\bb)\|_{H^s}\les_M \ep_0,\\
        &\|X_j\|_{L^\infty}\les \|\d(F-F_\bb)\|_{L^\infty}\les \ep_0,\\
        &\|\<\bar{\nu}_2, {\nu}_1\>\|_{H^s}=\||\bar{\nu}_1|^{-1}\<X_1,X_2\>_g\|_{H^s}\les_M \ep_0^2,
    \end{align*}
    then we obtain 
    \[\|1-|\bar{\nu}_1|^2|\bar{\bar{\nu}}_2|^2\|_{H^s}\les_M \ep_0^2.\]

    \emph{b) We estimate the term}
    \begin{align*}
        \|\d_\mu \nu_{\bb,1}\cdot [g^{\al\be}\<\nu_{\bb,2},\d_\al (F-F_\bb)\>\d_\be F+\<\bar{\nu}_2,  \nu_1\>  \nu_1]\|_{H^s}\les_M \ep_0.
    \end{align*}
    By \eqref{ineqfg}, $\d_j \nu_{\bb,1}\in L^\infty\cap \dot H^s$, and $\|\<\bar{\nu}_2,  \nu_1\>\|_{H^s}\les \ep_0$, it suffices to bound 
    \begin{align*}
        \|g^{\al\be}\<\nu_{\bb,2},\d_\al (F-F_\bb)\>\d_\be F\|_{H^s}\les \|g^{\al\be}\|_{L^\infty\cap \dot H^s}\|X_2\|_{H^s}\|\d_\be F\|_{L^\infty\cap \dot H^s}
        \les_M \ep_0 .
    \end{align*}

    \emph{c) We estimate the term}
    \begin{align*}
        \|\d_\mu (g^{\al\be}\<\nu_{\bb,1},\d_\al (F-F_\bb)\>\d_\be F) \cdot   \nu_2\|_{H^s} \les_M \ep_0.
    \end{align*}
    When $\d_\mu$ is applied to $g^{\al\be}\nu_{\bb,1} \d_\be F$, by $\|P_{>0}\nu_2\|_{\dot H^s}\les C(M)$ we have
    \begin{align*}
        &\ \|\d_\mu(g^{\al\be}\nu_{\bb,1}\d_\be F)\d_\al(F-F_\bb)\cdot   \nu_2 \|_{H^s}\\
        \les&\  \|\d_\mu(g^{\al\be}\nu_{\bb,1}\d_\be F)\|_{L^\infty\cap \dot{H}^s}\|\d_\al (F-F_\bb)\|_{H^s}\| ( \nu_2,P_{>0}\nu_2) \|_{L^\infty\times \dot{H}^s}\les_M \ep_0.
    \end{align*}
    When $\d_\mu$ is applied to $\d_\al(F-F_\bb)$, by $\|\tilde \nu_2-\nu_{\bb,2}\|_{H^s}\les \ep_0$, it suffices to bound
    \begin{align*}
        &\ \| g^{\al\be}\<\nu_{\bb,1},\d_\mu\d_\al (F-F_\bb)\>\d_\be F\cdot \nu_{\bb,2}\|_{H^s}\\
        =&\ \| g^{\al\be}\<\nu_{\bb,1},\d_\mu\d_\al (F-F_\bb)\>\<\d_\be (F-F_\bb)\cdot \nu_{\bb,2}\>\|_{H^s}\\
        \les &\ \|g^{\al\be}\|_{L^\infty\cap \dot H^s}\|(\nu_{\bb,1},P_{>0}\nu_{\bb,1})\|_{L^\infty\times \dot H^s}\|\d^2(F-F_\bb)\|_{H^s} \|X_2\|_{H^s}
        \les_M  \ep_0.
    \end{align*}
Therefore, we obtain the estimate \eqref{reg-tA}. 
\end{proof}

Now it directly follows from \eqref{Smnubb} and \eqref{reg-tA} that $A$ also satisfy the bound
\begin{equation*}
	\||D|^{\de_d} A\|_{H^{s-\de_d}}\les \||D|^{\de_d} A_\bb\|_{H^{s-\de_d}}+\| A-A_\bb\|_{H^s} \lesssim C(M).
\end{equation*}
Projecting the second fundamental form $\Lambda$ on the frame as in Section~\ref{sec-2.1} we obtain the complex second fundamental form $\lambda$. By \eqref{nu510} we have 
\begin{align*}
    \|\nu\|_{\dot H^{[s]+1}}\les \|(\bar{\nu}_1,\bar{\bar{\nu}}_2)\|_{\dot H^{[s]+1}}\les C(M).
\end{align*}
Then $\la$ has the same regularity as $\d^2 F$,
\begin{equation*}
\| \lambda \|_{H^s}\les \|\d^2 F\cdot \nu\|_{H^s}\les \|\d^2 F\|_{H^s}(\|\nu\|_{L^\infty}+\|P_{>0}\nu\|_{\dot H^s})\les C(M)(1+\|\nu\|_{\dot H^{[s]+1}})\les C(M).
\end{equation*}
Moreover, for the frame we have
\begin{align*}
    \|\d \nu\|_{L^{\frac{2d}{d-2\de_d}}}\les \|A\|_{L^{\frac{2d}{d-2\de_d}}}+\|\la\|_{L^{\frac{2d}{d-2\de_d}}}\|\d F\|_{L^\infty}\les C(M).
\end{align*}
This can be used to bound
\begin{align*}
    \|\d \nu\|_{\dot H^{2\de_d}}\les \|(A,\d^2 F \d F)\|_{\dot H^{2\de_d}\cap \dot H^{\de_d}}(\|\nu\|_{L^\infty}+\|\d\nu\|_{L^{\frac{2d}{d-2\de_d}}})\les C(M)
\end{align*}
and 
\begin{align*}
    \|\d\nu\|_{\dot H^s}&\les \|A\|_{\dot H^s}\|\nu\|_{L^\infty}+\|A\|_{\dot H^{\de_d}\cap L^\infty}(\|\d P_{\leq 0}\nu\|_{\dot H^{2\de_d}}+\|P_{>0}\nu\|_{\dot H^s})+\|\la\|_{H^s}\|\d F\|_{L^\infty\cap \dot H^s}\\
    &\les C(M)+C(M)\|P_{>0}\nu\|_{L^2}^{\frac{1}{s+1}}\|P_{>0}\nu\|_{\dot H^{s+1}}^{\frac{s}{s+1}}\\
    &\leq C(M)+\frac{1}{2}\|\d \nu\|_{\dot H^s}.
\end{align*}
Then we get $\|\d\nu\|_{\dot H^{2\de_d}\cap \dot H^s}\les C(M)$.
Thus the estimates in \eqref{ini} are obtained.

\subsection{Regularization of initial manifold $\Sigma$}\label{Reg-sec}
In the previous subsection, we have obtained a rough initial manifold
$\Sigma$ with gauge fixed, on which the data $g,\ A$ and $\la$ have finite Sobolev norms.
Our goal here is to construct a family of regularized initial manifolds and to show that the $X^s$-norm of $\lambda$, the $Y^{s+1}$-norm of $g$, and the $Z^s$-norm of $A$ are each equivalent to the standard Sobolev norms of these respective quantities.

Given an initial submanifold $\Sigma$ with an orthonormal frame $(\nu_1,\nu_2)$. By frequency projection, we regularize the manifold and denote its associated variables as
\begin{align} \label{RegMfd-0}
    (\Sigma^{(h)}:=F^{(h)}(\R^d), g^{(h)}, A^{(h)}, \la^{(h)}),\qquad F^{(h)}:=P_{<h}F,\ \ h\geq h_0.
\end{align}
where the coordinates remain fixed and are identical to those of $\Sigma$. The corresponding metric is given by
\begin{equation*}
    g^{(h)}_{\al\be}=\<\d_\al F^{(h)},\d_\be F^{(h)}\>.
\end{equation*}
To obtain the connection and second fundamental form, we first  regularize the orthonormal frame $(\nu_1,\nu_2)$ as 
\begin{align*}
    (\tilde \nu^{(h)}_{1},\tilde \nu^{(h)}_{2}):=(P_{<h}\nu_1,P_{<h}\nu_2),
\end{align*}
and obtain the normal vectors
\begin{align}  \label{nuba(h)}
    \bar{\nu}^{(h)}_{j}=\tilde \nu^{(h)}_{j}-g^{(h)\al\be}\<\tilde \nu^{(h)}_{j},\d_\al F^{(h)}\>\d_\be F^{(h)}. 
\end{align}
Then the orthonormal frame $(\nu^{(h)}_{1},\nu^{(h)}_{2})$ on $\Sigma^{(h)}$ is given by
\begin{align} \label{nuhini}
\nu^{(h)}_{1}=\frac{\bar{\nu}^{(h)}_{1}}{|\bar{\nu}^{(h)}_{1}|},\quad \nu^{(h)}_{2}=\frac{\bar{\bar{\nu}}^{(h)}_{2}}{|\bar{\bar{\nu}}^{(h)}_{2}|}, \qquad \text{with}\quad \bar{\bar{\nu}}^{(h)}_{2}=\bar{\nu}^{(h)}_{2}-\<\bar{\nu}^{(h)}_{2},\nu^{(h)}_{1}\>\nu^{(h)}_{1}.
\end{align}
Hence, the connection $A^{(h)}$ and the second fundamental form $\la^{(h)}$ are defined as
\begin{equation*}
    A^{(h)}_{\al}=\d_\al \nu^{(h)}_{1}\cdot \nu^{(h)}_{2},\quad \la^{(h)}_{\al\be}=\d^2_{\al\be}F^{(h)}\cdot m^{(h)},\quad m^{(h)}:=\nu^{(h)}_{1}+i\nu^{(h)}_{2}.
\end{equation*}

\begin{rem}
    Different from the construction \eqref{ConsNu} of the orthonormal frame on $\Sigma$, which relies on the smooth frame $\nu_\bb$, we begin instead with the frame $(\nu_1,\nu_2)$ defined on $\Sigma$. This frame is then regularized via frequency projection. Subsequently, by applying projection and Gram-Schmidt orthogonalization, we obtain an orthonormal frame on $\Sigma^{(h)}$. This procedure ensures the convergence of $g^{(h)}$, $A^{(h)}$, and $\lambda^{(h)}$ in suitable Sobolev spaces as $h \to \infty$.
\end{rem}

Next, we consider the proof of the  properties \eqref{Reg-met}, \eqref{SECs} and the bounds \eqref{Reg-gAla}.

\begin{proof}[Proof of \eqref{Reg-met}]
By the definition $g^{(h)}=\d P_{<h}F\cdot \d P_{<h}F$, we have
\begin{align*}
&\ \|g-g^{(h)}\|_{H^s}=\|\d F\cdot \d F-\d P_{<h}F\cdot \d P_{<h}F\|_{H^s}\\
&=\|\d P_{>h} F\cdot \d F+\d P_{<h}F\cdot \d P_{>h}F\|_{H^s} \les \|\d P_{>h} F\|_{H^s}\|\d F\|_{L^\infty\cap \dot H^s}\\
&\les 2^{-h}\|\d^2 P_{>h} F\|_{H^s} C(M)\les C(M)2^{-h}.
\end{align*}
Then for any vector $X$ we have
\begin{align*}
|(g^{(h)}_{\al\be}-g_{\al\be})X^\al X^\be|\les \|g^{(h)}-g\|_{L^\infty}|X|^2\les \|g^{(h)}-g\|_{H^s}|X|^2\les C(M)\ep_0 |X|^2.
\end{align*}
Hence, by $c_0I\leq g\leq c_0^{-1}I$ we obtain the ellipticity property  \eqref{Reg-met}.
\end{proof}

\begin{proof}[Proof of \eqref{SECs}]
By the definitions of $g^{(h)}$ and $\la^{(h)}$, we have
\begin{align*}
    \|\la^{(h)}\|_{L^\infty}\les \|\d^2 P_{<h}F\|_{L^\infty}\les \|\d^2 F\|_{H^s}\les C(M),\\
    \|g^{(h)}\|_{L^\infty}\les \|\d P_{<h}F\cdot \d P_{<h}F\|_{L^\infty}\les C(M).
\end{align*}
Then from the formula \eqref{R-la}, for any vectors $X$ we get the boundness of Ricci curvature
\begin{align*}
    |\Ric^{(h)}_{\al\be}X^\al X^\be|=|\Re(\la^{(h)}_{\al\be}\bar{\psi}^{(h)}-\la^{(h)}_{\al\si}\bar{\la}^{(h)\si}_\be)X^\al X^\be|\les \|\la^{(h)}\|_{\sfL^\infty}^2 |X|_{g^{(h)}}^2\les C(M)|X|_{g^{(h)}}^2.
\end{align*}

We now turn to the proof of the second bound in \eqref{SECs}. 
Here we first should consider the bound for $\d_h g^{(h)}$ with $h\geq h_0$:
\begin{align}   \label{dhg}
    &\ \int_{h}^\infty \|\d_h g^{(h)}\|_{L^\infty}dh\les\int_{h}^\infty \|P_h \d F\cdot P_{<h}\d F\|_{L^\infty}dh \\      \nonumber
    \les &\ \int_{h}^\infty 2^{-h}2^{(s+1)h}\|P_h \d F\|_{L^2}\|P_{<h}\d F\|_{L^\infty}dh\\   \nonumber
    \les &\  C(M) 2^{-h_0}\big(\int_{h}^\infty 2^{2(s+1)h}\|P_h \d F\|_{L^2}^2 dh\big)^{1/2}\les C(M)2^{-h_0}.
\end{align}
Then we claim that:
\begin{align}  \label{vol-reg}
    &e^{-C(M)2^{-h_0}}dvol_{g} \leq dvol_{g^{(h)}}  \leq e^{C(M)2^{-h_0}}dvol_{g},\\ \label{balls-reg}
    &B_x(r_0,\Sigma)\subset B_x(r_0 e^{C(M)2^{-h_0}},\Sigma^{(h)}).
\end{align}

\emph{a) Proof of claim \eqref{vol-reg}.} From the derivative of $\det g^{(h)}$, we know that
\begin{align*}
    |\d_h \sqrt{\det g^{(h)}}|=|\frac{1}{2}g^{(h)\al\be}\d_h g^{(h)}_{\al\be} \sqrt{\det g^{(h)}}| \leq \|\frac{1}{2}g^{(h)\al\be}\d_h g^{(h)}_{\al\be}\|_{L^\infty} \sqrt{\det g^{(h)}},
\end{align*}
which implies that
\begin{align*}
    |\d_h \ln \sqrt{\det g^{(h)}}|\leq \|\frac{1}{2}g^{(h)\al\be}\d_h g^{(h)}_{\al\be}\|_{L^\infty}.
\end{align*}
Integrating over $[h,\infty)$, this combined with \eqref{dhg} yields
\begin{align*}
    \sqrt{\det g}\  e^{-C(M)2^{-h_0}}\leq \sqrt{\det g^{(h)}}\leq \sqrt{\det g} \ e^{C(M)2^{-h_0}}
\end{align*}
Hence, by the volume form $dvol_{g^{(h)}}=\sqrt{\det g^{(h)}} dx$ we obtain the estimate \eqref{vol-reg}.

\smallskip 
\emph{b) Proof of claim \eqref{balls-reg}.} For any two points $F(x)$ and $F(y)$ in $\Sigma$, there exists a geodesic $\gamma:[0,1]\rightarrow \Sigma$ such that $\ga(0)=x$ and $\ga(1)=y$, whose distance is denoted as $l(\ga)$. Then we replace $F$ by $F^{(h)}$, and define the length of curve $\ga$ as
\begin{align*}
    l(\ga,h)=\int_0^1 |\dot{\gamma}(\tau)|_{g^{(h)}} d\tau=\int_0^1 \Big(g^{(h)}_{\al\be}\frac{\d \ga_\al}{\d\tau} \frac{\d \ga_\be}{\d\tau}\Big)^{1/2} d\tau.
\end{align*}
Since the metric $g^{(h)}$ varies with $h$, the length $l(\ga,h)$ would also change. Then we have
\begin{align*}
    \big|\frac{d}{dh}l(\ga,h)\big|=\big|\int_0^1 \frac{1}{2|\dot\ga|}\Big(\d_h g^{(h)}_{\al\be}\frac{\d \ga_\al}{\d\tau} \frac{\d \ga_\be}{\d\tau}\Big) d\tau\big|\leq \frac{1}{2}\|\d_h g^{(h)}\|_{L^\infty} l(\gamma,h),
\end{align*}
which yields
\begin{equation*}
    \big|\frac{d}{dh}\ln l(\ga,h) \big|\leq \frac{1}{2}\|\d_h g^{(h)}\|_{L^\infty} .
\end{equation*}
Integrating over $[h,\infty)$, this combined with \eqref{dhg} gives
\begin{align*}
    l(\ga)e^{-C(M)2^{-h_0}}\leq l(\ga,h)\leq l(\ga)e^{C(M)2^{-h_0}}.
\end{align*}
Hence, we obtain that the distance $d_h(x,y)$ between $F^{(h)}(x)$ and $F^{(h)}(y)$ for $h\in[h_0,\infty)$ satisfies the bound
\begin{equation*}
    d_h(x,y)\leq l(\ga,h) \leq l(\ga) e^{C(M)2^{-h_0}}=d(x,y) e^{C(M)2^{-h_0}}.
\end{equation*}
Hence the claim \eqref{balls-reg} follows.

With the two claims \eqref{vol-reg} and \eqref{balls-reg} at hand,  we  obtain
\begin{equation*}
    \begin{aligned}
    &{\rm Vol}_{g^{(h)}}(B_x(e^{C(M)2^{-h_0}}))= \int_{B_x(e^{C(M)2^{-h_0}},h)} 1\  dvol_{g^{(h)}}\geq \int_{B_x(1)}   e^{-C(M)2^{-h_0}}dvol_{g}\\
    =&\    e^{-C(M)2^{-h_0}} {\rm Vol}_{g}(B_x(1))
    \geq  e^{-C(M)2^{-h_0}} v.
\end{aligned}
\end{equation*}
Hence, the second bound in \eqref{SECs} also follows.
\end{proof}

\begin{proof}[Proof of the bound for the metric in \eqref{Reg-gAla}: $\|g\|_{Y^{s+1}}\les C(M)$.]
\

First, we consider the convergence of $g^{(h)}$. In a same way as the proof of \eqref{Reg-met}, we have
\begin{align} \label{g-gh}
    &\ \|g-g^{(h)}\|_{H^{s+1}}\les \|\d P_{>h}F(\d F+\d P_{<h}F)\|_{H^{s+1}}\\\nonumber
    &\les \|\d P_{>h}F\|_{H^{s+1}}\|\d F\|_{L^\infty}+\|\d P_{>h}F\|_{L^\infty}\|\d F\|_{\dot H^{s+1}}\les C(M)\|\d P_{>h}F\|_{H^{s+1}}.
\end{align}
In view of $\|\d^2 F\|_{H^s}\les C(M)$, this implies the convergence $\|g-g^{(h)}\|_{H^{s+1}}\rightarrow 0$ as $h\rightarrow\infty$. Hence, the family of regularization $[g^{(h)}]\in \mathrm{Reg}(g)$.

Next, we prove the bound $\tri[g^{(h)}]\tri_{s+1,g}\les C(M)$. By the estimate \eqref{g-gh}, we can bound the low-frequency part by
\begin{align*}
\||D|^{\si_d}g^{(h)}\|_{H^{s+1-\si_d}}\les \||D|^{\si_d}g\|_{H^{s+1-\si_d}}+\|g-g^{(h)}\|_{H^{s+1}}\les M.
\end{align*}
For the high derivatives $N\geq [s]+1$, we have
\begin{align*}
    &\ \int_{h_0}^\infty 2^{2(s-N)h} \| g^{(h)}\|_{\dot H^{N+1}}^2 dh=\int_{h_0}^\infty 2^{2h(s-N)} \|(P_{<h}\d F\cdot P_{<h}\d F)\|_{\dot H^{N+1}}^2 dh\\
    \les&\ \int_{h_0}^\infty 2^{2(s-N)h} \|P_{<h}\d F\|_{\dot H^{N+1}}^2 \|P_{<h}\d F\|_{L^\infty}^2 dh
    \les  C(M) \|\d F\|_{\dot H^{s+1}}\les C(M).
\end{align*}
This yields that for any $N\geq [s]+1$ 
\begin{align*}
    &\ \int_{h_0}^\infty 2^{2h(s-N)} \||D|^{\si_d}g^{(h)}\|_{H^{N+1-\si_d}}^2 dh
    \les  \int_{h_0}^\infty 2^{2h(s-N)} (\||D|^{\si_d}g^{(h)}\|_{L^2}^2+\|g^{(h)}\|_{\dot H^{N+1}}^2) dh \\
    &\les \int_{h_0}^\infty 2^{2h(s-N)}C(M) dh+C(M)\les  C(M).
\end{align*}

Finally, we bound the linearized part $\int_{h_0}^\infty 2^{2sh}\|\d_h g\|_{H^1}^2 dh$. 
Since
\begin{equation}\label{DB-011}
    \begin{aligned} 
    \|\d_h g^{(h)}\|_{H^1}&=\|\d_h (P_{<h}\d F\cdot P_{<h}\d F)\|_{H^1}\les \|P_{h}\d F\cdot P_{<h}\d F\|_{H^1}\\
    &\les \|\d P_h F\|_{H^1}\|\d P_{<h}F\|_{W^{1,\infty}}\les C(M)\|\d P_h F\|_{H^1},
\end{aligned}
\end{equation}
then we have 
\begin{align*}
    \int_{h_0}^\infty 2^{2hs}\|\d_h g^{(h)}\|_{H^1}^2 dh \les C(M)\int_{h_0}^\infty 2^{2hs}\|\d P_h F\|_{H^1}^2 dh\les C(M)\|P_{>h_0}\d F\|_{H^{s+1}}^2\les C(M).
\end{align*}
Thus, the term $\tri[g^{(h)}]\tri_{s+1,g}$, and hence the $Y^{s+1}$-norm of $g$, are bounded by $C(M)$.
\end{proof}

To bound $A\in Z^s$ and $\la\in X^s$,
we need the following estimates for $\bar{\nu}_1^{(h)}$ and $\bar{\bar{\nu}}_2^{(h)}$.
\begin{lemma}
Suppose $\|\d\nu\|_{\dot H^{2\de_d}\cap \dot H^s}\les C(M)$, then we have
\begin{align} \label{nuh519}
&\|\bar{\nu}_1^{(h)}\|_{\dot H^{[s]+1}}+\|\bar{\bar{\nu}}_2^{(h)}\|_{\dot H^{[s]+1}}\les C(M),\\ \label{nuh520}
&\|P_{>0}(|\bar{\nu}_1^{(h)}|^{-1},|\bar{\bar{\nu}}_2^{(h)}|^{-1},(1+|\bar{\nu}_1^{(h)}||\bar{\bar{\nu}}_2^{(h)}|)^{-1})\|_{\dot H^s}\les C(M),\\ \label{nuh522}
&\|1-|\bar{\nu}_1^{(h)}|^2|\bar{\bar{\nu}}_2^{(h)}|^2\|_{H^s}\les 2^{-h}C(M),\\ \label{nuh523}
&\|m^{(h)}-m\|_{H^s}\les 2^{-h}C(M).
\end{align}
\end{lemma}
\begin{proof}
By the same argument as \eqref{nu510}, and also by the bound \eqref{ini} for $\nu$, we get the estimate \eqref{nuh519} as follows
\begin{align*}
\|\bar{\nu}_1^{(h)}\|_{\dot H^{[s]+1}}+\|\bar{\bar{\nu}}_2^{(h)}\|_{\dot H^{[s]+1}}&\les \|(\tilde \nu^{(h)},g^{(h)},\d F^{(h)})\|_{\dot H^{[s]+1}}\\
&\les C(M)\|(P_{<h}\nu,P_{<h}\d F\|_{\dot H^{[s]+1}}\les C(M).
\end{align*}
Combined with \eqref{HNv}, this yields the second estimate \eqref{nuh520}.

Next, we prove the estimate \eqref{nuh522}. This term can be rewritten as
\begin{align*}
    1-|\bar{\nu}_1^{(h)}|^2|\bar{\bar{\nu}}_2^{(h)}|^2=&\ 1-|\tilde \nu_1^{(h)}|^2|\tilde \nu_2^{(h)}|^2+|\tilde \nu_1^{(h)}|^2|X_2|_{g^{(h)}}^2+|\tilde \nu_2^{(h)}|^2 |X_1|_{g^{(h)}}^2\\
    &+|\tilde \nu_1^{(h)}|^2|\bar{\nu}_1|^{-2}|\<\bar{\nu}_2^{(h)},\bar{\nu}_1^{(h)}\>|^2+|X_1|_{g^{(h)}}^2(|X_2|_{g^{(h)}}^2+|\bar{\nu}_1|^{-2}|\<\bar{\nu}_2^{(h)},\bar{\nu}_1^{(h)}\>|^2),
\end{align*}
where $X_j:=\<\tilde \nu_j^{(h)},\d F^{(h)}\>$. Since $|\nu_1|^2=|\nu_2|^2=1$, $\tilde \nu^{(h)}=P_{<h}\nu$ and $\|(\nu,P_{>0}\nu)\|_{L^\infty\times \dot H^s}\les C(M)$, the first term in the above is estimated by
\begin{align*}
    \|1-|\tilde \nu_1^{(h)}|^2|\tilde \nu_2^{(h)}|^2\|_{H^s}\les \|P_{>h}\nu (\tilde \nu^{(h)},\nu)\|_{H^s}\les \|P_{>h}\nu\|_{H^s} \|(\nu,P_{>0}\nu)\|_{L^\infty\times \dot H^s}\les 2^{-h}C(M).
\end{align*}
Since $\nu\perp \d F$, we can estimate $X_j$ by
\begin{align} \label{XX}
    \|X\|_{H^s}&=\|\<\tilde\nu^{(h)}-\nu+\nu,\d F^{(h)}\>\|_{H^s}\les \|\<P_{>h}\nu,\d F^{(h)}\>\|_{H^s}+\|\<\nu,\d(F^{(h)}-F)\>\|_{H^s}\\\nonumber
    &\les \|P_{>h}(\nu,\d F)\|_{H^s}(\|(\d F^{(h)},\nu)\|_{L^\infty}+\|P_{>0}(\d F^{(h)},\nu)\|_{\dot H^s})\les 2^{-h}C(M).
\end{align}
Then we can bound the following terms
\begin{align}\nonumber
    \||\tilde \nu^{(h)}|^2 |X|_{g^{(h)}}^2\|_{H^s}&\les \|X\|_{H^s}^2 \|g^{(h)}\|_{L^\infty\cap \dot H^s}\|(\tilde \nu^{(h)},P_{>0}\tilde \nu^{(h)})\|_{L^\infty\times \dot H^s}^2\les 2^{-2h}C(M),\\\label{nu2nu1}
    \|\<\bar{\nu}_2^{(h)},\bar{\nu}_1^{(h)}\>\|_{H^s}&=\|\tilde\nu_2^{(h)}\cdot \tilde\nu_1^{(h)}-\<X_1,X_2\>_{g^{(h)}}\|_{H^s}\\\nonumber
    &\les\|P_{>h}\nu\|_{H^s}\|(\nu,P_{>0}\nu)\|_{L^\infty\times \dot H^s}+2^{-2h}C(M)\les 2^{-h}C(M).
\end{align}
Hence, $\|1-|\bar{\nu}_1^{(h)}|^2|\bar{\bar{\nu}}_2^{(h)}|^2\|_{H^s}$ is bounded by $2^{-h}C(M)$.

Finally, we prove the estimate \eqref{nuh523}. The difference $m^{(h)}-m$ is expressed as
\begin{align*}
&\ m^{(h)}-m=\frac{\bar{\nu}_1^{(h)}}{|\bar{\nu}_1^{(h)}|}-\nu_1+i(\frac{\bar{\bar{\nu}}_2^{(h)}}{|\bar{\bar{\nu}}_2^{(h)}|}-\nu_2)\\
=&\ \frac{1-|\bar{\nu}_1^{(h)}|^2}{|\bar{\nu}_1^{(h)}|(1+|\bar{\nu}_1^{(h)}|)}\bar{\nu}_1^{(h)}+\bar{\nu}_1^{(h)}-\nu_1+i\big( \frac{1-|\bar{\bar{\nu}}_2^{(h)}|^2}{|\bar{\bar{\nu}}_2^{(h)}|(1+|\bar{\bar{\nu}}_2^{(h)}|)}\bar{\bar{\nu}}_2^{(h)}+\bar{\nu}_2^{(h)}-\nu_2 -\<\bar{\nu}_2^{(h)},\nu_1^{(h)}\>\nu_1^{(h)} \big)
\end{align*}
Similar to \eqref{nuh522}, we also have $\|1-|\bar{\nu}_1^{(h)}|^2\|_{H^s}+\|1-|\bar{\bar{\nu}}_2^{(h)}|^2\|_{H^s}\les 2^{-h}C(M)$, then by \eqref{ineqfg} and \eqref{nuh519}, we get 
\begin{align*}
    \|\frac{1-|\bar{\nu}_1^{(h)}|^2}{|\bar{\nu}_1^{(h)}|(1+|\bar{\nu}_1^{(h)}|)}\bar{\nu}_1^{(h)}\|_{H^s}+\|\frac{1-|\bar{\bar{\nu}}_2^{(h)}|^2}{|\bar{\bar{\nu}}_2^{(h)}|(1+|\bar{\bar{\nu}}_2^{(h)}|)}\bar{\bar{\nu}}_2^{(h)}\|_{H^s}\les 2^{-h}C(M).
\end{align*}
For the difference $\bar{\nu}_j^{(h)}-\nu_j$, by $\d\nu\in \dot H^s$ and \eqref{XX} we have
\begin{align*}
    &\ \|\bar{\nu}_j^{(h)}-\nu_j\|_{H^s}\les\|P_{>h}\nu_j\|_{H^s}+\|g^{(h)}X_j \d F^{(h)}\|_{H^s}\\
    \les&\  2^{-h}\|\d P_{>h}\nu_j\|_{H^s}+\|X_j\|_{H^s}\|g^{(h)}\d F^{(h)}\|_{L^\infty\cap \dot H^s}\les 2^{-h}C(M).
\end{align*}
The last term $\<\bar{\nu}_2^{(h)},\nu_1^{(h)}\>\nu_1^{(h)}$ is also bounded by $2^{-h}C(M)$ using \eqref{ineqfg}, \eqref{nuh519} and \eqref{nu2nu1}. Hence, the estimate \eqref{nuh523} follows.
\end{proof}

Now we continue our proof of \eqref{Reg-gAla} for the connection $A$ and the second fundamental form $\la$.
\begin{proof}[Proof of the connection bound in \eqref{Reg-gAla}: $\|A\|_{X^{s}_A}\les C(M)$.]

\ 
\medskip

\emph{Step 1. We show that}
\[
\||D|^{\de_d}A^{(h)}\|_{H^{s-\de_d}}\les C(M), \qquad h\geq h_0
\]
Since $\||D|^{\de_d}A\|_{H^{s-\de_d}}\les C(M)$, it suffices to show that 
\begin{equation}\label{diffAh-A}
    \|A^{(h)}-A\|_{H^s}\les C(M)\|P_{>h} \d\nu\|_{H^s}+2^{-h}C(M),
\end{equation}
whose proof is similar to Lemma~\ref{Lemreg-tA}.

For any $h\geq h_0$, by $\bar{\nu}^{(h)}_{1}\cdot \bar{\bar{\nu}}^{(h)}_{2}=0$, the term $A^{(h)}_\mu-A_\mu$ is expressed as 
\begin{align*}
&\ A^{(h)}_\mu-A_\mu =|\bar{\nu}_1^{(h)}|^{-1}|\bar{\bar{\nu}}_2^{(h)}|^{-1}\d_\mu \bar{\nu}_1^{(h)}\cdot \bar{\bar{\nu}}_2^{(h)}-\d_\mu \nu_1\cdot \nu_2\\
=&\ |\bar{\nu}_1^{(h)}|^{-1}|\bar{\bar{\nu}}_2^{(h)}|^{-1}[\d_\mu \tilde \nu_1^{(h)}-\d_\mu (g^{(h)\al\be}\<\tilde \nu_1^{(h)},\d_\al F^{(h)}\>\d_\be F^{(h)})]\\
&\ \cdot [\tilde \nu_2^{(h)}-g^{(h)\al\be}\<\tilde \nu_2^{(h)},\d_\al F^{(h)}\>\d_\be F^{(h)}-\<\bar{\nu}_2^{(h)},\nu_1^{(h)}\>\nu_1^{(h)}]-\d_\mu \nu_1\cdot \nu_2\\
=&\ (\d_\mu \tilde\nu_1^{(h)}\cdot \tilde \nu_2^{(h)}-\d_\mu \nu_1\cdot \nu_2)+ (|\bar{\nu}_1^{(h)}|^{-1}|\bar{\bar{\nu}}_2^{(h)}|^{-1}-1)\d_\mu \tilde \nu_1^{(h)}\cdot \tilde \nu_2^{(h)}\\
&-|\bar{\nu}_1^{(h)}|^{-1}|\bar{\bar{\nu}}_2^{(h)}|^{-1} \d_\mu \tilde \nu_1^{(h)}\cdot (g^{(h)\al\be}\<\tilde \nu_2^{(h)},\d_\al F^{(h)}\>\d_\be F^{(h)}+\<\bar{\nu}_2^{(h)},\nu_1^{(h)}\>\nu_1^{(h)})\\
&-|\bar{\nu}_1^{(h)}|^{-1}|\bar{\bar{\nu}}_2^{(h)}|^{-1} \d_\mu (g^{(h)\al\be}\<\tilde \nu_1^{(h)},\d_\al F^{(h)}\>\d_\be F^{(h)})\cdot \bar{\bar{\nu}}_2^{(h)}\\
=&:\ I_1+I_2+I_3+I_4.
\end{align*}
The terms on the right hand are estimated one by one.

\emph{a) Bound for $I_1$:}
\begin{equation}\label{I1nu}
\|I_1\|_{H^s}\les C(M)\|\d P_{>h}\nu\|_{H^s}\les C(M).
\end{equation}
This is obtained  using \eqref{ini} for $\nu$ and $h\geq h_0\gg 1$
\begin{align*}
    \|I_1\|_{H^s}&=\|\d P_{>h}\nu_1\cdot \tilde\nu_2^{(h)}\|_{H^s}+\|\d \nu_1\cdot P_{>h}\nu_2\|_{H^s}\\
    &\les \|\d P_{>h}\nu_1\|_{H^s}(\|\tilde\nu_2^{(h)}\|_{L^\infty}+\|P_{>0}\tilde\nu_2^{(h)}\|_{\dot H^s})+\|\d \nu_1\|_{L^\infty\cap \dot H^s}\|P_{>h}\nu_2\|_{H^s}\\
    &\les C(M)\|P_{>h}\d\nu\|_{H^s}.
\end{align*}

\emph{b) Bound for $I_2$:}
\begin{align*}
    \|I_2\|_{H^s}=\lV\frac{1-|\bar{\nu}_1^{(h)}|^2|\bar{\bar{\nu}}_2^{(h)}|^2}{|\bar{\nu}_1^{(h)}||\bar{\bar{\nu}}_2^{(h)}|(1+|\bar{\nu}_1^{(h)}||\bar{\bar{\nu}}_2^{(h)}|)}\d_\mu \tilde \nu_1^{(h)}\cdot \tilde \nu_2^{(h)}\rV_{H^s}\les C(M)2^{-h}.
\end{align*}
By $\d_\mu \tilde \nu_1^{(h)}\cdot \tilde\nu_2^{(h)}=A_\mu+I_1$, the estimate \eqref{ini} for $A$ and the estimate \eqref{I1nu}, we have $\|\d_\mu \tilde \nu_1^{(h)}\cdot \tilde\nu_2^{(h)}\|_{L^\infty\cap \dot H^s}\les C(M)$.
Moreover, we have $(|\bar{\nu}_1^{(h)}||\bar{\bar{\nu}}_2^{(h)}|(1+|\bar{\nu}_1^{(h)}||\bar{\bar{\nu}}_2^{(h)}|))^{-1}\in L^\infty$ and the estimate \eqref{nuh520} for high-frequency part. Hence, by \eqref{ineqfg} it suffices to consider the bound for $1-|\bar{\nu}_1^{(h)}|^2|\bar{\bar{\nu}}_2^{(h)}|^2$ in $H^s$, which has been provided by \eqref{nuh523}. This yields  the desired estimate for $I_2$.

\smallskip 
\emph{c) Bound for $I_3$: $\|I_3\|_{H^s}\les 2^{-h}C(M)$.}

By \eqref{ineqfg}, \eqref{nuh520}, \eqref{ini}, \eqref{XX} and \eqref{nu2nu1}, we have
\begin{align*}
    \|I_3\|_{H^s}\les&\  \big(\|(|\bar{\nu}_1^{(h)}|^{-1},|\bar{\bar{\nu}}_2^{(h)}|^{-1})\|_{L^\infty}+\|P_{>0}(|\bar{\nu}_1^{(h)}|^{-1},|\bar{\bar{\nu}}_2^{(h)}|^{-1})\|_{\dot H^s}\big)\|\d_\mu \tilde \nu_1^{(h)}\|_{L^\infty\cap \dot H^s}\\
    &\ \cdot \big(\|g^{(h)\al\be}X_2\d_\be F^{(h)}\|_{H^s}+\|\<\bar{\nu}_2^{(h)},\nu_1^{(h)}\>\nu_1^{(h)}\|_{H^s}\big)\\
    \les&\ C(M) (\|X_2\|_{H^s}+\|\<\bar{\nu}_2^{(h)},\bar{\nu}_1^{(h)}\>\|_{H^s})\les 2^{-h}C(M).
\end{align*}

\emph{d) Bound for $I_4$: $\|I_4\|_{H^s}\les 2^{-h}C(M)$.}

By \eqref{ineqfg} and \eqref{nuh520} we have
\begin{align*}
    &\ \|I_4\|_{H^s}\les C(M)\|\d_\mu (g^{(h)\al\be}X_1\d_\be F^{(h)})\cdot \bar{\bar{\nu}}_2^{(h)}\|_{H^s}\\
    \les&\  C(M)\big(\|\d_\mu (g^{(h)\al\be}X_1)\|_{L^\infty\cap \dot H^s} \|\d_\be F^{(h)}\cdot \bar{\bar{\nu}}_2^{(h)}\|_{H^s}\\
    &\ +\|g^{(h)}X_1\|_{H^s}\|\d^2 F^{(h)}\|_{H^s}\|(\bar{\bar{\nu}}_2^{(h)},P_{>0}\bar{\bar{\nu}}_2^{(h)})\|_{L^\infty\times \dot H^s}\big)\\
    \les &\ C(M)\|\d_\be F^{(h)}\cdot \bar{\bar{\nu}}_2^{(h)}\|_{H^s}+2^{-h}C(M),
\end{align*}
where the term $\d_\be F^{(h)}\cdot \bar{\bar{\nu}}_2^{(h)}$ can be estimated using the bound for  difference $\bar{\bar{\nu}}_2^{(h)}-\nu_2 \in H^s$,
\begin{align*}
    \|\bar{\bar{\nu}}_2^{(h)}-\nu_2\|_{H^s}&\les \|P_{>h}\nu_2\|_{H^s}+\|g^{(h)}X_2\d F^{(h)}\|_{H^s}+\|\<\bar{\nu}_2^{(h)},\bar{\nu}_1^{(h)}\>\bar{\nu}_1^{(h)}|\bar{\nu}_1^{(h)}|^{-2}\|_{H^s}\\
    &\les 2^{-h}C(M)\les 2^{-h}C(M),
\end{align*}
and 
\begin{align*}
    \|\d_\be F^{(h)}\cdot \nu_2\|_{H^s}=\|\d_\be (F^{(h)}-F)\cdot \nu_2\|_{H^s}\les \|P_{>h}\d F\|_{H^s}C(M)\les 2^{-h}C(M).
\end{align*}
Hence, the $H^s$-norm of $I_4$ is bounded by $2^{-h}C(M)$.

To conclude, the difference bound \eqref{diffAh-A} follows; thus we obtain $\||D|^{\de_d}A^{(h)}\|_{H^{s-\de_d}}\les C(M)$. Moreover, the estimate \eqref{diffAh-A} also implies the convergence $\lim_{h\rightarrow\infty} \|A^{(h)}-A\|_{H^s}=0$.

\smallskip 

\emph{Step 2. We prove that}
\[
\int_{h_0}^\infty 2^{2h(s-N)} \||D|^{\de_d}A^{(h)}\|_{H^{N-\de_d}}^2 dh\les C(M).
\]
Since $\||D|^{\de_d}A^{(h)}\|_{H^{s-\de_d}}\les C(M)$, it suffices to consider the term $\|P_{>0}A^{(h)}\|_{\dot H^N}$. For any integer $k\geq 1$ we have
\begin{align*}
    \|P_{>0}A^{(h)}\|_{\dot H^k}&=\|P_{>0}(\d \nu^{(h)}_{1}\cdot \nu^{(h)}_{2})\|_{\dot H^k}\les \|P_{>0} \nu^{(h)}\|_{\dot H^{k+1}}\|\nu^{(h)}\|_{L^\infty}
    \les  \| \nu^{(h)}\|_{\dot H^{k+1}}.
\end{align*}
By \eqref{HNv}, we further bound $\|\nu^{(h)}\|_{\dot H^{k+1}}$ by
\begin{align*}
    \|\nu^{(h)}_1\|_{\dot H^{k+1}}&=\|\frac{\bar{\nu}^{(h)}_{1}}{|\bar{\nu}^{(h)}_{1}|}\|_{\dot H^{k+1}}\les \|\bar\nu^{(h)}_{1}\|_{\dot H^{k+1}}\||\bar\nu^{(h)}_{1}|^{-1}\|_{L^\infty}+\|\bar\nu^{(h)}_{1}\|_{L^\infty}\||\bar\nu^{(h)}_{1}|^{-1}\|_{\dot H^{k+1}}\\
    &\les \|\bar\nu^{(h)}_{1}\|_{\dot H^{k+1}}
\end{align*}
and 
\begin{align*}
    \|\nu^{(h)}_{2}\|_{\dot H^{k+1}}&=\|\frac{\bar{\bar{\nu}}^{(h)}_{2}}{|\bar{\bar{\nu}}^{(h)}_{2}|}\|_{\dot H^{k+1}}
    \les \|\bar{\bar\nu}^{(h)}_{2}\|_{\dot H^{k+1}}\les \|\bar\nu^{(h)}_{2}\|_{\dot H^{k+1}}+\||\bar{\nu}^{(h)}_{1}|^{-2}\<\bar\nu^{(h)}_{2},\bar{\nu}^{(h)}_{1}\>\bar{\nu}^{(h)}_{1}\|_{\dot H^{k+1}}\\
    &\les \|\bar\nu^{(h)}_{2}\|_{\dot H^{k+1}}+\|\bar\nu^{(h)}_{1}\|_{\dot H^{k+1}}.
\end{align*}
From the formula of $\bar{\nu}^{(h)}$  in \eqref{nuba(h)}, we have
\begin{align*}
    &\ \|\bar{\nu}^{(h)}\|_{\dot H^{k+1}}\les \|\tilde \nu^{(h)}\|_{\dot H^{k+1}}+\|g^{(h)\al\be} \<\tilde \nu^{(h)}, \d_\al F^{(h)}\> \d_\be F^{(h)}\|_{\dot H^{k+1}}\\
    &\les \|\tilde \nu^{(h)}\|_{\dot H^{k+1}}+\|(g^{(h)})^{-1}\|_{\dot H^{k+1}}+\|\d F^{(h)}\|_{\dot H^{k+1}}\les \|P_{<h} \nu\|_{\dot H^{k+1}}+\|\d P_{<h} F\|_{\dot H^{k+1}},
\end{align*}
where $(g^{(h)})^{-1}$ is the inverse matrix of $g^{(h)}$, which is easily seen to satisfy the estimate $\|(g^{(h)})^{-1}\|_{\dot H^{k+1}}\les C(M)\|g^{(h)}\|_{\dot H^{k+1}}$.
Then we obtain for any integer $k\geq 1$,
\begin{align} \label{hfb-A0}
    \|P_{>0}A^{(h)}\|_{\dot H^k} \les \|\nu^{(h)}\|_{\dot H^{k+1}}\les \|P_{<h} \d\nu\|_{\dot H^{k}}+\| P_{<h}\d^2 F\|_{\dot H^k}.
\end{align}
Hence, we arrive at
\begin{align*}
    &\ \int_{h_0}^\infty 2^{2h(s-N)}\||D|^{\de_d}A^{(h)}\|_{H^{N-\de_d}}^2 dh\\
    \les&\ \int_{h_0}^\infty 2^{2h(s-N)} (\||D|^{\de_d}A^{(h)}\|_{H^{s-\de_d}}^2+\|P_{>0}A^{(h)}\|_{\dot H^N}^2) dh
    \les C(M).
\end{align*}

\medskip

\emph{Step 3. We prove 
that}
\[
\int_{h_0}^\infty 2^{2hs}\|\d_h A^{(h)}\|_{L^2}^2 dh\les C(M).
\]

By the formula $A^{(h)}=\d \nu^{(h)}_{1}\cdot \nu^{(h)}_{2}$, we have
\begin{align*}
    \|\d_h A^{(h)}\|_{L^2}&=\|\d_h(\d \nu^{(h)}_{1}\cdot \nu^{(h)}_{2})\|_{L^2}\les\|\d \d_h \nu^{(h)}_{1}\|_{L^2}\|\nu^{(h)}_{2}\|_{L^\infty}+\|\d \nu^{(h)}_{1}\|_{L^\infty}\| \d_h \nu^{(h)}_{2}\|_{L^2}\\
    &\les \|\d_h \nu^{(h)}_{1}\|_{H^1}+C(M)\|\d_h \nu^{(h)}_{2}\|_{L^2}.
\end{align*}
We estimate the first term by
\begin{align*}
    \|\d_h \nu^{(h)}_{1}\|_{H^1}=\|\d_h (\frac{\bar{\nu}^{(h)}_{1}}{|\bar{\nu}^{(h)}_{1}|})\|_{H^1}\les \|\frac{\d_h \bar{\nu}^{(h)}_{1}}{|\bar{\nu}^{(h)}_{1}|}\|_{L^2}+\|\frac{\d\d_h \bar{\nu}^{(h)}_{1}}{|\bar{\nu}^{(h)}_{1}|}\|_{L^2}+\|\frac{\d_h \bar{\nu}^{(h)}_{1} \d \bar{\nu}^{(h)}_{1}}{|\bar{\nu}^{(h)}_{1}|^2}\|_{L^2}\les \|\d_h \bar{\nu}^{(h)}_{1}\|_{H^1},
\end{align*}
and estimate the second term by
\begin{align*}
    &\ \|\d_h \nu^{(h)}_{2}\|_{L^2}=\|\d_h (\frac{\bar{\bar{\nu}}^{(h)}_{2}}{|\bar{\bar{\nu}}^{(h)}_{2}|})\|_{L^2}\les \|\frac{\d_h \bar{\bar{\nu}}^{(h)}_{2}}{|\bar{\bar{\nu}}^{(h)}_{2}|}\|_{L^2}\les \|\d_h \bar{\bar{\nu}}^{(h)}_{2}\|_{L^2}\\
    &\les \|\d_h \bar{\nu}^{(h)}_{2}\|_{L^2}+\|\frac{\d_h \bar{\nu}^{(h)}_{2}\cdot \bar{\nu}^{(h)}_{1}}{|\bar{\nu}^{(h)}_{1}|^2}\bar{\nu}^{(h)}_{1}\|_{L^2}+\|\frac{ \bar{\nu}^{(h)}_{2}\cdot \d_h\bar{\nu}^{(h)}_{1}}{|\bar{\nu}^{(h)}_{1}|^2}\bar{\nu}^{(h)}_{1}\|_{L^2}+\|\frac{\bar{\nu}^{(h)}_{2}\cdot \bar{\nu}^{(h)}_{1}}{|\bar{\nu}^{(h)}_{1}|^2}\d_h\bar{\nu}^{(h)}_{1}\|_{L^2}\\
    &\les \|\d_h \bar{\nu}^{(h)}_{2}\|_{L^2}+\|\d_h \bar{\nu}^{(h)}_{1}\|_{L^2}.
\end{align*}
By the formula \eqref{nuba(h)}, we further bound the $\d_h\bar{\nu}^{(h)}\in H^1$ by
\begin{align*}
    \|\d_h \bar{\nu}^{(h)}\|_{H^1}&\les \|\d_h \tilde \nu^{(h)}\|_{H^1}+\|\d_h g^{(h)}\|_{H^1}\|\tilde \nu^{(h)} \d F^{(h)} \d F^{(h)}\|_{W^{1,\infty}}\\
    &\quad 
    +\|\d_h \tilde \nu^{(h)}\|_{H^1}\|g^{(h)}\d F^{(h)}\d F^{(h)}\|_{W^{1,\infty}}
    +\|\d_h \d F^{(h)} \|_{H^1}\|g^{(h)}\tilde \nu^{(h)} \d F^{(h)}\|_{W^{1,\infty}}\\
    &\les \|\d_h \tilde \nu^{(h)}\|_{H^1}+\|\d_h g^{(h)}\|_{H^1}+\|\d_h \d F^{(h)}\|_{H^1}
    \les  \|\d_h \tilde \nu^{(h)}\|_{H^1}+\|\d_h \d F^{(h)}\|_{H^1}.
\end{align*}
Hence, we obtain
\begin{equation}\label{DB-012}
    \begin{aligned}
    \|\d_h A^{(h)}\|_{L^2}&\les \|\d_h \nu^{(h)}_{1}\|_{H^1}+\|\d_h \nu^{(h)}_{2}\|_{L^2}\les\|\d_h \tilde \nu^{(h)}\|_{H^1}+\|\d_h \d F^{(h)}\|_{H^1}\\
    &\les \|P_h\nu\|_{H^1}+\| \d P_h F\|_{H^1}.
\end{aligned}
\end{equation}
Since $h>h_0$ is positive, this also gives
\begin{align*}
    &\ \int_{h_0}^\infty 2^{2hs} \|\d_h A^{(h)}\|_{L^2}^2 dh\les \int_{h_0}^\infty 2^{2hs}(\|P_h\nu\|_{H^1}^2+\| \d P_h F\|_{H^1}^2 )dh\\
    &\les \int_{h_0}^\infty 2^{2hs}(\|\d P_h\nu\|_{L^2}^2+\| \d^2 P_h F\|_{L^2}^2 )dh
    \les \|\d \nu\|_{\dot H^s}^2+\|\d^2 F\|_{H^s}^2 \les C(M).
\end{align*}
This completes the proof of \eqref{Reg-gAla} for $A$.
\end{proof}

\begin{proof}[Proof of  the second fundamental form bound in \eqref{Reg-gAla}: $\|\la\|_{X^s}\les C(M)$.]

\ 

First we consider the convergence of $\la^{(h)}$ in $H^s$. By \eqref{nuh520}, \eqref{nuh519} and \eqref{nuh523}, the difference between $\la^{(h)}$ and $\la$ is bounded by
\begin{align*}
    &\ \|\la^{(h)}-\la\|_{H^s}\les \|\d^2 F^{(h)}\cdot m^{(h)}-\d^2 F\cdot m\|_{H^s}\\
    &\les \|\d^2 P_{>h}F\cdot m^{(h)}\|_{H^s}+\|\d^2 F\cdot (m^{(h)}-m)\|_{H^s}\\
    &\les \|\d^2 P_{>h}F\|_{H^s}\|(m^{(h)},P_{>0}m^{(h)})\|_{L^\infty\times \dot H^s}+\|\d^2 F\|_{H^s}\|m^{(h)}-m\|_{H^s}\\
    &\les \|\d^2 P_{>h}F\|_{H^s}C(M)+2^{-h}C(M).
\end{align*}
Hence, the $\la^{(h)}$ converges to $\la$ in $H^s$ as $h\rightarrow\infty$. This guarantees that $[\la^{(h)}]\in \mathrm{Reg}(\la)$.

Next, we prove the estimate $\|\la\|_{X^s}\les C(M)$. By the equivalence \eqref{Equ-la}, $[\la^{(h)}]\in \mathrm{Reg}(\la)$ and the definition of $X^s$, it suffices to prove the bound $\tri[\la^{(h)}]\tri_{s,{ext}}\les C(M)$.

For any $k\geq 1$ and $h\geq h_0$, by \eqref{hfb-A0} we have
\begin{align*}
\|\la^{(h)}\|_{H^k}&\les \|\d^2 F^{(h)}\cdot m^{(h)}\|_{H^k}\les \|\d^2 P_{<h}F\|_{H^k}+\|\d P_{<h}F\|_{L^\infty}\|m^{(h)}\|_{\dot H^{k+1}}\\
&\les \|\d^2 P_{<h}F\|_{H^k}+C(M)(\|P_{<h}\nu\|_{\dot H^{k+1}}+\|\d P_{<h}F\|_{\dot H^{k+1}})\\
&\les C(M)(\|\d^2 P_{<h}F\|_{H^k}+\|\d P_{<h}\nu\|_{\dot H^k}).
\end{align*}
Then, for low-frequency part $\la^{(h_0)}$ we get 
\begin{align*}
    2^{(s-[s])h_0}\|\la^{(h_0)}\|_{H^{[s]}}\les 2^{(s-[s])h_0} C(M)(\|\d^2 P_{<h_0}F\|_{H^{[s]}}+\|\d P_{<h_0}\nu\|_{\dot H^{[s]}})\les C(M),
\end{align*}
and 
\begin{align*}
    &\ 2^{(s-[s]-1)h_0}\|\la^{(h_0)}\|_{H^{[s]+1}}\les 2^{(s-[s]-1)h_0} C(M)(\|\d^2 P_{<h_0}F\|_{H^{[s]+1}}+\|\d P_{<h_0}\nu\|_{\dot H^{[s]+1}})\\
    \les&\  2^{(s-[s]-1)h_0} C(M)2^{([s]+1-s)h_0}(\|\d^2 P_{<h_0} F\|_{H^{s}}+\|\d P_{<h_0}\nu\|_{\dot H^s})
    \les C(M).
\end{align*}
For high frequency, we also have
\begin{align*}
    &\ \int_{h_0}^\infty 2^{2h(s-N)} \|\la^{(h)}\|_{H^N}^2 dh
    \les C(M)\int_{h_0}^\infty 2^{2h(s-N)}(\|P_{<h}\d^2 F\|_{H^N}^2+\|P_{<h}\d\nu\|_{\dot H^N}^2) dh\\
    &\les C(M)\int_{h_0}^\infty  2^{2h(s-N)} \sum_{l\leq [h]+1,l\in \N} 2^{2l(N-s)}(\|P_l\d^2 F\|_{H^s}^2+\|P_{l}|D|^{\de_d}\d\nu\|_{H^{s-2\de_d}}^2) dh\\
    &\les C(M) \sum_{l\in \N} (\|P_l\d^2 F\|_{H^s}^2+\|P_{l}|D|^{\de_d}\d\nu\|_{H^{s-2\de_d}}^2) \int_{h_0}^\infty 2^{2(N-s)(l-h)} {\bf 1}_{>l-1}(h) dh\\
    &\les C(M) (\|\d^2 F\|_{H^s}^2+\||D|^{\de_d}\d\nu\|_{H^{s-2\de_d}}^2)\les C(M).
\end{align*}
Finally, we consider the linearized part $\int_{h_0}^\infty 2^{2sh} \|\d_h\la^{(h)}\|_{L^2}^2 dh$.
Since $\la^{(h)}_{\al\be}=\d^2_{\al\be}P_{<h}F\cdot m^{(h)}$, then by \eqref{DB-012} we have
\begin{equation} \label{DB-013}
\begin{aligned}
&\ \|\d_h \la^{(h)}\|_{L^2}\les\|\d_h P_{<h}\d^2 F\|_{L^2}+\|\d^2 P_{<h}F\|_{L^\infty}\|\d_h m^{(h)}\|_{L^2}\\
\les&\ \|P_h\d^2 F\|_{L^2}+C(M)\|(\d_h\nu^{(h)}_{1},\d_h\nu^{(h)}_{2})\|_{L^2}
\les C(M)(\|P_h\nu\|_{H^1}+\| \d P_h F\|_{H^1}).
\end{aligned}
\end{equation}
This yields
\begin{align*}
    \int_{h_0}^\infty 2^{2sh}\|\d_h \la^{(h)}\|_{L^2}^2 dh &\les \int_{h_0}^\infty 2^{2sh} C(M)(\|P_h\nu\|_{H^1}+\| \d P_h F\|_{H^1})^2 dh\les C(M).
\end{align*}
Hence, the bound \eqref{Reg-gAla} for $\la$ follows.
\end{proof}

Finally, we need the following lemma about difference bounds and high frequency bounds for the regularized initial manifolds $\Sigma^{(h)}$.
\begin{lemma}
For the regularized manifolds $\Sigma^{(h)}$ in \eqref{RegMfd-0}, we have the following properties:

(i) Difference bounds: for any $j\geq h_0$
\begin{align}  \label{DB-01}
    \int_j^{j+1}\|\d_h g^{(h)}\|_{H^1} +\|\d_h A^{(h)}\|_{L^2} +\|\d_h \la^{(h)}\|_{L^2} dh\les_M 2^{-sj}c_j\,,\\ \label{DB-02}
    \|\d F^{(j+1)}-\d F^{(j)}\|_{L^2} +\|m^{(j+1)}-m^{(j)} \|_{L^2}\les_M 2^{-(s+1)j} c_j\,.
\end{align}

(ii) High frequency bounds: for any $N>s$ and any $j\geq h_0$:
\begin{align}\label{HFB-02}
    \|\d F^{(j)}\|_{\dot H^{N+1}\cap \dot H^N}+\|m^{(j)}\|_{\dot H^{N+1}\cap \dot H^N}\les_M 2^{(N-s)j}c_j,\\  \label{HFB-01}
    \|\d g^{(j)}\|_{H^{N}}+\||D|^{\de_d}A^{(j)}\|_{H^{N-\de_d}}+\|\la^{(j)}\|_{\sfH^N\cap H^N}\les_M 2^{(N-s)j}c_j.
\end{align}
\end{lemma}

\begin{proof}
(i) From the estimates \eqref{DB-011}, \eqref{DB-012} and \eqref{DB-013}, we have
\begin{align*}
    &\ \int_j^{j+1} \|\d_h g^{(h)}\|_{H^1}+\|\d_h A^{(h)}\|_{L^2}+\|\d_h \la^{(h)}\|_{L^2} dh\\
    &\les C(M)\int_j^{j+1} \|P_h\nu\|_{H^1}+\| \d P_h F\|_{H^1} dh\\
    &\les C(M) 2^{-sj} \big(\int_j^{j+1} 2^{2sh}(\|P_h\nu\|_{H^1}+\| \d P_h F\|_{H^1})^2 dh\big)^{1/2} \\
    &\les C(M)2^{-sj}c_j.
\end{align*}
The bound \eqref{DB-02} for $\d F^{(j+1)}-\d F^{(j)}$ follows from $F^{(j)}=P_{<j}F$. The second term in \eqref{DB-02} is obtained using \eqref{DB-012}
\begin{align*}
    &\ \|m^{(j+1)}-m^{(j)}\|_{L^2}\les \int_j^{j+1} \|\d_h m^{(h)}\|_{L^2} dh\les \int_j^{j+1} \|\d_h \nu^{(h)}_1\|_{L^2}+\|\d_h \nu^{(h)}_2\|_{L^2} dh\\
    &\les \int_j^{j+1} \|P_h \nu\|_{L^2}+\|\d P_h F \|_{L^2} dh\les  C(M)2^{-(s+1)j}c_j.
\end{align*}

(ii) We prove the high frequency bounds. By $F^{(j)}=P_{<j}F$ and \eqref{hfb-A0}, we easily have 
\[\|\d^2 F^{(j)}\|_{H^N}+\|\d F^{(j)}\|_{\dot H^{N+1}}+\|\d F^{(j)}\|_{\dot H^{N}}\les_M 2^{(N-s)j}c_j,\]
and 
\begin{align*}
    \|m^{(j)}\|_{\dot H^{N+1}\cap \dot H^N}\les \|\d P_{<h} \nu\|_{\dot H^{N}\cap \dot H^{N-1}}+\|\d^2 P_{<h} F\|_{H^{N}}
    \les 2^{(N-s)j}c_j.
\end{align*}
Thus the estimate \eqref{HFB-02} follows.

For the metric $\d g^{(j)}$, we have 
\begin{align*}
&\ \|\d g^{(j)}\|_{H^{N}}=\|\d(\d P_{<j}F\cdot \d P_{<j}F)\|_{H^{N}}\les \|\d^2 P_{<j}F\|_{H^N}\|\d P_{<j}F\|_{L^\infty}\\
&\les (\sum_{k\leq j} 2^{2(N-s)k} c_k^2)^{1/2}C(M)\les(\sum_{k\leq j} 2^{2(N-s)(k-j)} 2^{2(N-s)j} 2^{2\de(j-k)} c_j^2)^{1/2}C(M)\\
&\les  C(M) 2^{(N-s)j}c_j.
\end{align*}
Next, from \eqref{hfb-A0} the connection $A^{(j)}$ is estimated by
\begin{align*}
\||D|^{\de_d}A^{(j)}\|_{H^{N-\de_d}}\les \||D|^{\de_d}A^{(j)}\|_{H^{s-\de_d}}+\|P_{>0}A^{(j)}\|_{\dot H^{N}} \les_M 2^{(N-s)j}c_j,
\end{align*}
Finally, for the second fundamental form $\la^{(j)}$ in the extrinsic Sobolev spaces, we have
\begin{align*}
\|\la^{(j)}\|_{H^N}=\|\d^2 P_{<j} F\|_{H^N}\| m^{(j)}\|_{L^\infty}+\|\d P_{<j}F\|_{L^\infty}\|m^{(j)}\|_{\dot H^{N+1}}\les 2^{(N-s)j}c_j.
\end{align*}
Moreover, using the formula \eqref{nabl-Form}, we can bound the $\la^{(j)}$ in the intrinsic space $\sfH^N$ by
\begin{align*}
    \|\la^{(j)}\|_{\sfH^N}\les \|\la^{(j)}\|_{H^N}+C(M)(\|\d g^{(j)}\|_{H^{N-1}}+\||D|^{\de_d}A^{(j)}\|_{H^{N-1-\de_d}})\les C(M)2^{(N-s)j}c_j.
\end{align*}
This completes the proof of the lemma.
\end{proof}

\bigskip 
\section{Estimates for Parabolic equations}    \label{Sec-Para}
In this section, we consider the energy estimates for the parabolic system \eqref{par-syst}. For this purpose, we view $\la\in L^\infty X^s$
as a parameter and show the energy estimates for the solutions $(g, A)\in Y^{s+1}\times Z^s$ on $[0,T]$ for $T$ sufficiently small.

\begin{thm}\label{Para-thmS}
Let $d \geq 2$, $s > d/2$, and let $\si_d$ and $\de_d$ be given in \eqref{Constant}. 
Then the solutions $(g,A)$ of parabolic system \eqref{par-syst}-\eqref{ini-gA} have the following properties:

i) If $\||D|^{\si_d}g_0\|_{H^{s+1-\si_d}}+ \||D|^{\de_d}A_0\|_{H^{s-\de_d}}\leq M_1$ and $\|\la\|_{L_T^\infty H^s}\leq CM_1$ on $[0,T]$, then we have energy estimates on $[0,\min\{T,CM_1^{-6}\}]$:
\begin{align}   \label{para-bd0}
	&\||D|^{\si_d}g\|_{L^\infty H^{s+1-\si_d}}+ \| |D|^{1+\si_d} g \|_{L_T^2 H^{s+1-\si_d}}\leq 2M_1,\\ \label{para-bdA}
	&\| |D|^{\de_d}A\|_{L^\infty H^{s-\de_d}}+ \| |D|^{1+\de_d}A\|_{L^2_T H^{s-\de_d}}\leq 2M_1.
\end{align}
and the ellipticity bound
\begin{align}   \label{Para-met}
    &\frac{4}{5}c_0I\leq (g(t))\leq \frac{6}{5}c_0^{-1}I,\\ \label{Para-vol}
    &\inf_{x\in \Sigma}{\rm Vol}_{g(t)}(B_x(e^{tC_4 M_1^6}))\geq  e^{-tC_4 M_1^6}v,\quad |\Ric|\leq CM_1^2.
\end{align}

ii) Let $N=[2s]+1$. If $\tri[g_0^{(h)}]\tri_{s+1,g}+\tri[A_0^{(h)}]\tri_{s,A}\leq M_1$ and $\tri[\la^{(h)}]\tri_{s,int}\leq 8M_1$ on $[0,T]$, then we have energy estimates on $[0,\min\{T,CM_1^{-2N-8}\}]$:
\begin{equation}   \label{Para-S}
\tri [g^{(h)}] \tri_{s+1,g}\leq 8M_1,\qquad \tri [A^{(h)}] \tri_{s,A} \leq 8M_1\,.
\end{equation}
Moreover, we have the estimates for linearized terms and high-frequency terms
\begin{align} \label{DB-est1}
&\begin{aligned}
\frac{d}{dt}(\|\d_h g^{(h)}\|_{H^1}^2&+\|\d_h A^{(h)}\|_{L^2}^2)+c(\|\d\d_h g^{(h)}\|_{H^1}^2+\|\d\d_h A^{(h)}\|_{L^2}^2)\\ 
&\les C(M) (\|\d_h g^{(h)}\|_{H^1}^2+\|\d_h A^{(h)}\|_{L^2}^2+\|\d_h \la^{(h)}\|_{L^2}^2),
\end{aligned}\\ \label{HFB-est1}
&\begin{aligned}
\frac{d}{dt}\big(\|\d g^{(j)}\|_{H^N}^2&+\||D|^{\de_d}A^{(j)}\|_{H^{N-\de_d}}^2\big)\\
&\les C(M) 2^{2(N-s)j}c^2_j+C(M)\big(\|(\d g^{(j)},\la^{(j)})\|_{H^N}^2+\||D|^{\de_d}A^{(j)}\|_{H^{N-\de_d}}^2\big).
\end{aligned}
\end{align}
\end{thm}

\subsection{Energy estimates in Sobolev spaces} Here we prove the standard energy estimates \eqref{para-bd0} for parabolic equations \eqref{par-syst}. We start with the following bounds for the inverse $g^{-1}$.

\begin{lemma} 
Let $d\geq 2$, $s>d/2$ and $\si_d$ be given in \eqref{Constant}. Assume that $\|g-g_0\|_{H^{s}}\les \ep_0$ and $\||D|^{\si_d}g_0\|_{H^{s+1-\si_d}}\les M$. Then we have the bounds
\begin{align}   \label{Equi-h}
&\|g^{-1}-g_0^{-1}\|_{H^{s}} \les  \|g-g_0\|_{H^{s}},\qquad 
\|g^{-1}-g_0^{-1}\|_{H^{s+1}} \les  \|g-g_0\|_{H^{s+1}},
\end{align}
with implicit constants depending on  $M$.
\end{lemma}

\begin{proof}
Let $G_{\al\be}=g_{\al\be}-g_{0\al\be}$ and $G^{\al\be}=g^{\al\be}-g_0^{\al\be}$. Then the $G^{\al\be}$ and $G_{\al\be}$ satisfy the relation
\begin{align*}
\de^\al_\ga =g^{\al\be}g_{\be\ga}=(g_0^{\al\be}+G^{\al\be})(g_{0\be\ga}+G_{\be\ga})=\de^\al_\ga+g_0^{\al\be} G_{\be\ga}+G^{\al\be}g_{0\be\ga}+G^{\al\be}G_{\be\ga}.
\end{align*}
Multiplying $g_0^{\ga\si}$ yields 
\begin{align*}
G^{\al\si}=-g_0^{\al\be} G_{\be\ga}g_0^{\ga\si}-G^{\al\be}G_{\be\ga}g^{\ga\si}_0.
\end{align*}
Then the bounds in \eqref{Equi-h} are obtained by algebra property and the assumptions on $g-g_0$ and $g_0$.
\end{proof}

\begin{proof}[Proof of the bound \eqref{para-bd0} for the metric $g$]
\
	
We assume that $\||D|^{\si_d}g\|_{H^{s+1-\si_d}}\leq 2M_1$. It suffices to consider the general form:
\begin{align}    \label{Linz-h}
\d_t g-\d_\al(g^{\al\be} \d_\be g)=\la^2+(g^{-1})^2 \d g\d g+\d g^{-1} \d g=:N(g).
\end{align}	
Since $\||D|^{\si_d}g\|_{H^{s+1-\si_d}}\approx \|g\|_{\dot H^{\si_d}}+\|g\|_{\dot H^{s+1}}$, it suffices to consider the bound for $\||D|^{\si}g\|_{L^2}$ with $\si\in \{\si_d,s+1\}$. For the equation \eqref{Linz-h}, we derive
\begin{align}   \nonumber
&\ \frac{1}{2}\frac{d}{dt}\||D|^{\si}g \|_{L^2}^2= \int |D|^{\si} g \cdot |D|^{\si}  \d_t g \ dx
= \int |D|^{\si}  g \cdot |D|^{\si}  (\d_\al(g^{\al\be} \d_\be g )+N(g) )\ dx\\\nonumber
&=\int -g^{\al\be} \d_\al |D|^{\si}  g \cdot \d_\be |D|^{\si}  g 
-\d_\al |D|^{\si}  g \cdot  [|D|^{\si} ,g^{\al\be}] \d_\be g +|D|^{\si}g\cdot |D|^{\si} N(g) \ dx\\\label{comm-h}
&\leq - c\||D|^{\si}  g \|_{\dot H^1}^2+\||D|^{\si} g \|_{\dot H^1}\big(\|[|D|^{\si} ,g^{\al\be}] \d_\be g \|_{L^2}+\||D|^{\si-1}f \|_{L^2}\big).
\end{align}
Since $\si=\si_d$ or $s+1$, the commutator $[|D|^{\si},g^{\al\be}] \d_\be g $ in \eqref{comm-h} is bounded by 
\begin{align*}
    \|[|D|^{\si} ,g^{\al\be}] \d_\be g \|_{L^2}\les \||D|^{\si_d}g^{-1}\|_{H^{s+1-\si_d}}\||D|^{\si_d}g\|_{H^{s+1-\si_d}}\les \||D|^{\si_d}g\|_{H^{s+1-\si_d}}^2.
\end{align*}
Thus we obtain the energy estimate
\begin{align}\label{EEst-s}
\frac{1}{2}\frac{d}{dt}\||D|^{\si_d}g\|_{H^{s+1-\si_d}}^2+ c\||D|^{1+\si_d} g \|_{H^{s+1-\si_d}}^2
\les  \||D|^{\si_d}g\|_{H^{s+1-\si_d}}^4+\|N(g) \|_{H^{s}}^2.
\end{align}
The nonlinearities are bounded by 
\begin{align*}
	\|N(g)\|_{H^s}\les \|\la\|_{H^s}^2+\|g^{-1}\|_{L^\infty \cap \dot H^s}^2 \|\d g\|_{H^s}^2+\|\d g^{-1}\|_{H^s}\|\d g\|_{H^s}\leq CM_1^4
\end{align*}
Then from \eqref{EEst-s} we have
\begin{align}  \label{EE-glow}
	\frac{d}{dt}\||D|^{\si_d}g\|_{H^{s+1-\si_d}}^2+ 2c\||D|^{1+\si_d} g \|_{H^{s+1-\si_d}}^2
	\leq C_1 M_1^8.
\end{align}
This yields an improved bound on the time interval $t\in [0,\frac{5}{4C_1 M_1^6}]$
\begin{align*}
	\||D|^{\si_d}g\|_{H^{s+1-\si_d}}^2+ 2c\||D|^{1+\si_d} g \|_{L^2H^{s+1-\si_d}}^2\leq \||D|^{\si_d}g_0\|_{H^{s+1-\si_d}}^2+tC_1M_1^8\leq 4M_1^2.
\end{align*}
Hence, the estimate \eqref{para-bd0} follows.
\end{proof}

\begin{proof}[Proof of the bound \eqref{para-bdA} for connection $A$]
\ 

We assume that $\||D|^{\de_d}A\|_{H^{s-\de_d}}\leq 2M_1$. From \eqref{par-syst} and $\De_g=\nab^\mu\nab_\mu$, it suffices to consider the general form
\begin{align*}
\d_t A_\al-\d_\mu(g^{\mu\nu}\d_\nu A_\al)=\d_\mu(g^{-1}\Ga A)+\Ga(g^{-1}\nab A)+\nab(\la^2)+\la^2(A+V)=:N(A).
\end{align*}
The nonlinearity $N(A)$ is bounded by 
\begin{align*}
 \|N(A)\|_{H^{s-1}}&\les
 \|g^{-1}\|_{L^\infty \cap \dot H^s}^2\|\d g\|_{H^s}\||D|^{\de_d} A\|_{H^{s-\de_d}}\\
 &\quad +\|\la\|_{H^s}^2(1+\||D|^{\de_d} A\|_{H^{s-\de_d}}+\|g^{-1}\|_{L^\infty \cap \dot H^s}\|\d g\|_{H^s})
 \les  CM_1^4.
\end{align*}
Then similar to \eqref{EEst-s}, we obtain
\begin{align}  \label{EE-Alow}
&\ \frac{d}{dt}\||D|^{\de_d} A\|_{H^{s-\de_d}}^2+ 2c\|\d |D|^{\de_d}A\|_{H^{s-\de_d}}^2
\leq C_2 M_1^8.
\end{align}
Thus on the time interval $t\in [0, \frac{5}{4C_2M_1^6}]$, this yields the bound \eqref{para-bdA}.
 \end{proof}

\begin{proof}[Proof of \eqref{Para-met}]
    By \eqref{g_dt}, on $t\in[0,c_0(10C_3M_1^4)^{-1}]$ we have
\begin{align*}
    &|(g_{\al\be}(t)-g_{\al\be}(0))X^\al X^\be|\leq \int_0^t \|\d_\tau g_{\al\be}(\tau)\|_{L^\infty} d\tau |X|^2\\
    &\les\int_0^t\|\la\|_{H^s}^2+\|g^{-1}\|_{L^\infty}\|\d g\|_{H^{s+1}}+\|g^{-1}\|_{L^\infty}^2\|\d g\|_{H^s}^2\ d\tau |X|^2\\
    &\leq C( t M_1^4+\sqrt{t}M_1^2)|X|^2 \leq  t C_3 M_1^4\leq   \frac{c_0}{10} |X|^2.
\end{align*}
Then from $\frac{9}{10}c_0I\leq g(0)\leq \frac{11}{10}c_0^{-1}I$, we get 
\begin{align*}
    \frac{4}{5}c_0|X|^2\leq &\  g_{\al\be}X^\al X^\be=g_{\al\be}(0)X^\al X^\be+ (g_{\al\be}(t)-g_{\al\be}(0))X^\al X^\be
    \leq  \frac{6}{5}c_0^{-1}|X|^2.
\end{align*}
Thus the bound \eqref{Para-met} follows.
\end{proof}

\begin{proof}[Proof of \eqref{Para-vol}]
For the volume form, by \eqref{g_dt} and $V^\ga=g^{\al\be}\Ga_{\al\be}^\ga$ we have
\begin{align*}
    |\d_t\sqrt{\det g}|&=|\frac{1}{2}g^{\al\be}\d_t g_{\al\be}\sqrt{\det g}|=|\nab_\al V^\al \sqrt{\det g}| \\
    &\leq C(M_1^5+M_1^2\|\d g\|_{H^{s+1}}) \sqrt{\det g},
\end{align*}
Integrating over $[0,t]$, by \eqref{para-bd0} this yields
\begin{align*}
    e^{-tC_4 M_1^6} \sqrt{\det g(0)}\leq \sqrt{\det g(t)}\leq e^{t C_4 M_1^6} \sqrt{\det g(0)}.
\end{align*}
For any geodesic $\ga:[0,1]\rightarrow\Sigma_0$, we have
\begin{align*}
    \Big|\frac{d}{ds}l(\ga,s)\Big|\leq \|\d_s g\|_{L^\infty} l(\ga)\leq C(M_1^5+M_1^2\|\d g\|_{H^{s+1}})l(\ga)\,,
\end{align*}
which implies 
\begin{align*}
    d_s(x,y)\leq l(\ga,s)\leq e^{tC_4M_1^6} d_0(x,y).
\end{align*}
Then we obtain
\begin{align*}
    {\rm Vol}_{g(t)}(B_x(e^{tC_4 M_1^6}))=\int_{B_x(e^{tC_4 M_1^6},t)} 1\ dvol_{g(t)}\geq \int_{B_x(1,0)}  e^{-tC_4 M_1^6}dvol_{g(0)}
    = e^{-tC_4 M_1^6}v.
\end{align*}

In addition, by $\|\la\|_{L^\infty_T H^s}\les 2M_1$ and \eqref{Para-met}, we also have
\begin{align*}
    |\Ric_{\al\be}X^\al X^\be|\leq |\Ric|_g |X|_g^2\les \|\la\|_{L^\infty}^2 |X|_{g}^2\les CM_1^2 |X|_g^2.
\end{align*}
This completes the proof of \eqref{Para-vol}.
\end{proof}

\subsection{Energy estimates in the spaces $Y^{s+1}$ and $Z^s$} 
Here we focus on the energy estimates \eqref{Para-S} of parabolic system \eqref{par-syst} in our primary function spaces $Y^{s+1}$ and $Z^s$. By bootstrap argument, we assume that on some interval $[0,T_1]$ for $T_1\leq T$,
\begin{align*}
    \tri[g^{(h)}\tri_{s+1,g}\leq 8M_1,\quad \tri[A^{(h)}]\tri_{s,A}\leq 8M_1.
\end{align*}
Then by \eqref{Emb-la} we have
\begin{align*}
    \tri [\la^{(h)}]\tri_{s,ext}\leq 8C_{eq}M_1.
\end{align*}

\begin{proof}[Proof of the bound \eqref{Para-S} for metric $g$]
\ 

Since $\|\la^{(h)}\|_{L^\infty}\les\|\la^{(h)}\|_{H^s}\les \|\la\|_{s,\,int}$ and $\||D|^{\si_d}g^{(h)}\|_{H^{s+1-\si_d}}\leq M_1$, from \eqref{para-bd0} we have on the time interval $[0,CM_1^{-6}]$
\begin{align*}
    \||D|^{\si_d}g^{(h)}\|_{H^{s+1-\si_d}}+\||D|^{1+\si_d}g^{(h)}\|_{L^2H^{s+1-\si_d}}\leq 2M_1.
\end{align*}
Next, we bound the other terms respectively.

\emph{i) We bound the high frequency norm}
\begin{align}\label{EEhigh}
    \int_{h_0}^\infty 2^{2h(s-N)}\||D|^{\si_d}g\|_{H^{N+1-\si_d}}^2 dh+c\int_0^t\int_{h_0}^\infty 2^{2h(s-N)}\||D|^{1+\si_d}g\|_{H^{N+1-\si_d}}^2dhd\tau\leq 4M_1^2.
\end{align}

Here it suffices to bound the $\dot H^{N+1}$-norm of $g$
\begin{align*}
    &\ \frac{1}{2}\frac{d}{dt}\|g\|_{\dot H^{N+1}}^2=\int \d^{N+1}g\cdot \d^{N+1}(\d_\al(g^{\al\be}\d_\be g)+N(g))dx\\
    \leq &\ -c\|g\|_{\dot H^{N+2}}^2+\|g\|_{\dot H^{N+2}}(\|g\|_{\dot H^{N+1}}\|\d g\|_{L^\infty}+\|N(g)\|_{\dot H^N})\\
    \leq &\ -c\|g\|_{\dot H^{N+2}}^2+M_1^2\|g\|_{\dH^{N+1}}^2+\|N(g)\|_{\dot H^N}^2.
\end{align*}
The nonlinearity is bounded by
\begin{align*}
    \|N(g)\|_{\dot H^N}\les&\  \|\la\|_{\dH^N}\|\la\|_{L^\infty}+\|g^{-1}\|_{\dH^{N+1}}\|g^{-1}\|_{L^\infty}\|g\|_{L^\infty}\|\d g\|_{L^\infty}\\
    &+\|g^{-1}\|_{L^\infty}^2\|g\|_{\dH^{N+1}}\|\d g\|_{L^\infty}
    +\|\d g^{-1}\|_{L^\infty}\|g\|_{\dH^{N+1}}\\
    \leq  &\  C\|\la\|_{\dH^N}M_1+CM_1^{N+4}\|g\|_{\dH^{N+1}}.
\end{align*}
Then we obtain
\begin{align}  \label{EE-gHN}
	&\ \frac{1}{2}\frac{d}{dt}\|g\|_{\dH^{N+1}}^2 +c\|g\|_{\dH^{N+2}}^2
	\les M_1^{2N+8}\|g\|_{\dH^{N+1}}^2+\|\la\|_{\dH^{N}}^2 M_1^2.
\end{align}
Integrating over $[h_0,\infty)$, for $N>s+1$ this combined with \eqref{EE-glow} yields
\begin{align*}
    &\ \frac{1}{2}\frac{d}{dt}\int_{h_0}^\infty 2^{2h(s-N)}\||D|^{\si_d}g\|_{H^{N+1-\si_d}}^2 dh+c\int_{h_0}^\infty 2^{2h(s-N)}\||D|^{1+\si_d}g\|_{H^{N+1-\si_d}}^2dh\\
    \les &\ \int_{h_0}^\infty 2^{2h(s-N)}(C_1M_1^8+M_1^{2N+8}\|g\|_{\dH^{N+1}}^2+\|\la\|_{\dH^{N}}^2 M_1^2) dh\\
    \leq&\ C(M_1^{2N+10} +M_1^4)\leq C_5M_1^{2N+10}.
\end{align*}
Hence, the bound \eqref{EEhigh} follows on the time interval $t\in [0,3(2C_5M^{2N+8})^{-1}]$.

\emph{ii) We bound the linearized norm} 
\begin{align}  \label{EEg-Low}
    \int_{h_0}^\infty 2^{2sh} \|\d_h g(t)\|_{H^1}^2 dh+\int_0^t\int_{h_0}^\infty 2^{2sh} \|\d\d_h g\|_{H^1}^2 dhd\tau\leq 4M_1^2.
\end{align}

By the equations of $g$ and the nonlinearities in \eqref{par-syst}, we have
\begin{align*}
    &\ \frac{1}{2}\frac{d}{dt}\|\d_h g\|_{H^1}^2=\<\d_h g, \d_h \d_t g \>_{H^1}
    = \<\d_h g, \d_h (g^{\al\be}\d^2_{\al\be}g)\>_{H^1}+\<\d_h g, N(g))\>_{H^1}\\
    =&\ \<\d_h g, \d_h (g^{\al\be}\d^2_{\al\be}g)\>_{H^1}+\<\d_h g, \d_h(\la^2)\>_{H^1}\\
    &\ +\<\d_h g, \d_h((g^{-1})^2\d g\d g)\>_{H^1}+\<\d_h g, \d_h(\d g^{-1}\d g)\>_{H^1}\\
    =:&\  I_1+I_2+I_3+I_4.
\end{align*}

\emph{Estimates of $I_1$.}
We use integration by parts to rewrite the first term as 
\begin{align*}
    &I_1=\<\d_h g, \d_h (g^{\al\be}\d^2_{\al\be}g)\>_{L^2}+\<\d\d_h g, \d\d_h (g^{\al\be}\d^2_{\al\be}g)\>_{L^2}\\
    =&\ \<\d_h g, g^{\al\be}\d^2_{\al\be}\d_h g+\d_h g^{\al\be}\d^2_{\al\be}g\>_{L^2}+\<\d\d_h g, g^{\al\be}\d^2_{\al\be}\d\d_h g+\d g \d^2 \d_h g+\d (\d_h g^{\al\be}\d^2_{\al\be}g)\>_{L^2}\\
    \leq  &\ -\<\d_\al \d_h g,g^{\al\be}\d_\be \d_h g\>+|\<\d_h g, \d g^{-1}\d\d_h g+\d\d_h g^{-1}\d g\>|+|\<\d\d_h g,\d_h g^{-1}\d g\>|\\
    &\ -\<\d_\al \d \d_h g,g^{\al\be}\d_\be \d\d_h g\>-\<\d \d_h g,\d_\al g^{\al\be}\d_\be \d\d_h g\>\\
    &\ +\<\d \d_h g,\d g \d^2 \d_h g\>-\<\d^2 \d_h g,\d_h g^{\al\be}\d^2_{\al\be}g\>,
\end{align*}
By $(g^{\al\be})\geq cI$, this could be bounded by
\begin{align*}
    I_1\leq &\ -c\|\d\d_h g\|_{H^1}^2+\| \d_h g\|_{H^1}^2\|\d g\|_{L^\infty}
    +\|\d \d_h g\|_{H^1}\|\d_h g\|_{H^1}\|\d g\|_{H^s}\\
    \leq &\ -c\|\d \d_h g\|_{H^1}^2+\|\d_h g\|_{H^1}^2 \|\d g\|_{H^s}^2.
\end{align*}

\emph{Estimates of $I_2$.} We have
\begin{align*}
    I_2&=\<\d_h g, \d_h \la\cdot \la\>_{L^2}-\<\d^2 \d_h g,\d_h \la\cdot \la\>_{L^2}\\
    &\leq (\|\d_h g\|_{L^2}+\|\d^2 \d_h g\|_{L^2})\|\d_h \la\|_{L^2}\|\la\|_{L^\infty}\\
    &\leq \frac{1}{10}c\|\d \d_h g\|_{H^1}^2+\|\d_h g\|_{L^2}^2+C\|\d_h \la\|_{L^2}^2 \|\la\|_{H^s}^2.
\end{align*}

\emph{Estimates of $I_3$.} We have
\begin{align*}
	I_3&\leq \|\d_h( (g^{-1})^2\d g\d g)\|_{L^2}(\|\d_h g\|_{L^2}+\|\d^2 \d_h g\|_{L^2})\\
	& \leq (\|\d_h g^{-1}\|_{L^2}\|g^{-1}\d g\d g\|_{L^\infty}+\|(g^{-1})^2 \d g\|_{L^\infty}\|\d_h g\|_{H^1})(\|\d_h g\|_{L^2}+\|\d^2 \d_h g\|_{L^2})\\
	&\leq M_1^5\|\d_h g\|_{H^1}(\|\d_h g\|_{L^2}+\|\d^2 \d_h g\|_{L^2})\\
	&\leq \frac{1}{10}c\|\d \d_h g\|_{H^1}^2+M_1^{10}\|\d_h g\|_{H^1}^2 .
\end{align*}

\emph{Estimates of $I_4$.} We have
\begin{align*}
    I_4&=\<\d_h g, \d\d_h g \cdot \d g\>-\<\d^2 \d_h g,\d\d_h g\cdot \d g\>\\
    &\leq \|\d_h g\|_{H^1}^2 \|\d g\|_{H^s}+\frac{1}{10}c\|\d \d_h g\|_{H^1}^2+C\|\d \d_h g\|_{L^2}^2\|\d g\|_{H^s}^2\\
    &\leq \frac{1}{10}c\|\d \d_h g\|_{H^1}^2+CM_1^2\|\d_h g\|_{H^1}^2.
\end{align*}

From the above estimates, we obtain
\begin{align}   \label{EE-dhg}
    \frac{1}{2}\frac{d}{dt}\|\d_h g\|_{H^1}^2\leq &\ -\  \frac{1}{2}c\|\d \d_h g\|_{H^1}^2+CM_1^{10}\|\d_h g\|_{H^1}^2+CM_1^2\|\d_h \la\|_{L^2}^2.
\end{align}
Integrating over $[h_0,\infty)$ with respect to $h$, this also yields
\begin{align*}
    \frac{d}{dt}\int_{h_0}^\infty 2^{2hs}\|\d_h g\|_{H^1}^2 dh +c\int_{h_0}^\infty 2^{2hs} \|\d \d_h g\|_{H^1}^2   dh\leq C_6M_1^{12}.
\end{align*}
Hence, on the time interval $t\in[0,3(C_6M_1^{10})^{-1}]$ we obtain \eqref{EEg-Low}.

To conclude, on the time interval $t\in [0,CM_1^{-2N-8}]$, we obtain the improved bound $\|[g^{(h)}]\|_{s+1,g}\leq 6M_1$. Hence, the estimate \eqref{Para-S} for metric $g$ follows.
\end{proof}

\begin{proof}[Proof of the bound \eqref{Para-S} for connection $A$]

From \eqref{para-bdA}, we have on the time interval $[0,CM_1^{-6}]$
\begin{align*}
    \||D|^{\de_d}A\|_{H^{s-\de_d}}+c\||D|^{1+\de_d} A\|_{L^2 H^{s-\de_d}}\leq 2M_1.
\end{align*}
Next, it remains to bound the high frequency part and linearized part.

\emph{i) We bound the high frequency norm}
\begin{align} \label{EEA-Higher}
    \int_{h_0}^\infty 2^{2(s-N)h} \| |D|^{\de_d}A\|_{H^{N-\de_d}}^2 dh+\int_0^t\int_{h_0}^\infty 2^{2(s-N)h} \| |D|^{1+\de_d} A\|_{H^{N-\de_d}}^2 dhd\tau\leq 4M_1^2.
\end{align}

It suffices to consider the $\dH^N$-norm of $A$
\begin{align*}
&\frac{1}{2}\frac{d}{dt}\|A\|_{\dH^N}^2=\<A,\d_t A\>_{\dH^N}=\<A,\d_\mu(g^{\mu\nu}\d_\nu A)+N(A)\>_{\dH^N}\\
\leq &\ -c\| A\|_{\dH^{N+1}}^2+\|A\|_{\dH^{N+1}}(\|g^{-1}\|_{\dH^{N+1}}M_1+M_1^3\|A\|_{\dH^N}+\|N(A)\|_{\dH^{N-1}})\\
\leq &\ -c\| A\|_{\dH^{N+1}}^2+M_1^{2N+4}+M_1^6\|A\|_{\dH^N}^2+\|N(A)\|_{\dH^{N-1}}^2.
\end{align*}
The nonlinearities are bounded by 
\begin{align*}
	&\|N(A)\|_{\dH^{N-1}}\les \|g^{-1}\Ga A\|_{\dH^N}+\|\Ga g^{-1}\nab A\|_{\dH^{N-1}}+\|\la^2\|_{\dH^N}+\|\la^2(A+V)\|_{\dH^{N-1}}\\
	&\les \|(g^{-1},g)\|_{\dH^{N+1}}\|g\|_{L^\infty}\|A\|_{L^\infty}+\|g^{-1}\Ga\|_{L^\infty}\|A\|_{\dH^N}+\|\la\|_{\dH^N}M_1^3\\
	&\quad  +\|\la\|_{L^\infty}^2(\|A\|_{\dH^N}+\|\d A\|_{L^2}+\|g^{-1}\|_{\dH^{N+1}}\|g\|_{L^\infty}+\|g^{-1}\|_{L^\infty}\|g\|_{\dH^N}+\|g^{-1}\d g\|_{L^2})\\
	&\les M_1^{N+4}\|g\|_{\dH^{N+1}}+M_1^5\|A\|_{\dH^N}+M_1^3\|\la\|_{\dH^N}+M_1^4.
\end{align*}
Then we get
\begin{align}  \label{EE-AHN}
	\frac{d}{dt}\|A\|_{\dH^N}^2 +c\|A\|_{\dH^{N+1}}^2\leq M_1^{2N+4}+M_1^{2N+8}(\|g\|_{\dH^{N+1}}^2+\|(A,\la)\|_{\dH^N}^2).
\end{align}
Integrating over $[h_0,\infty)$, this combined with \eqref{EE-Alow} yields
\begin{align*}
	\frac{d}{dt}\int_{h_0}^\infty 2^{2h(s-N)}\| |D|^{\de_d}A\|_{H^{N-\de_d}}^2 dh+c\int_{h_0}^\infty 2^{2h(s-N)}\| |D|^{1+\de_d}A\|_{H^{N-\de_d}}^2 dh\leq C_7M_1^{2N+10}.
\end{align*}
Hence, we get the bound \eqref{EEA-Higher} on the interval $[0,3(C_7M_1^{2N+8})^{-1}]$.

\emph{ii) We bound the linearized norm}
\begin{align} \label{EEA}
    \int_{h_0}^\infty 2^{2sh} \|\d_h A\|_{L^2}^2 dh+\int_0^t\int_{h_0}^\infty 2^{2sh} \|\d\d_h A\|_{L^2}^2 dhd\tau\leq 4M_1^2.
\end{align}

By the equations of $A$ in \eqref{par-syst}, we have
\begin{align*}
    &\ \frac{1}{2}\frac{d}{dt}\|\d_h A\|_{L^2}^2=\int \d_h A\cdot \d_h (\d_\mu(g^{\mu\nu} \d_\nu  A)+\mathcal N(A)) dx \\
    =&\ -\int g^{\mu\nu} \d_\mu\d_h A \d_\nu \d_h A-\int \d_\mu\d_h A \d_h g^{\mu\nu}\d_\nu A +\int \d_h A\cdot \d_h N(A) dx \\
    \leq&\ -c\|\d \d_h A\|_{L^2}+\|\d\d_h A\|_{L^2}\|\d_h g^{-1}\|_{H^1}\||D|^{\de_d}A\|_{H^{s-\de_d}}+\int \d_h A\cdot \d_h N(A) dx
\end{align*}
The nonlinearity $N(A)$ is estimated by
\begin{align*}
    &\ |\int \d_h A\cdot \d_h \d(g^{-1}\Ga A+\la^2) dx |\\
    =&\ |\int \d\d_h A\cdot  (\d_h g^{-1}\Ga A+(g^{-1})^2 \d\d_h g A+g^{-1}\Ga \d_h A+\d_h \la \la)  dx|\\
    \les &\ \|\d\d_h A\|_{L^2}(\|\d_h g\|_{L^2}M_1^5+M_1^3\|\d\d_h g\|_{L^2}+M_1^3\|\d_h A\|_{L^2}+\|\d_h \la\|_{L^2}M_1)\\
    \leq &\ \frac{c}{10}\|\d\d_h A\|_{L^2}^2+CM_1^{10}(\|\d_h g\|_{H^1}+\|(\d_h A,\d_h \la)\|_{L^2})^2,
\end{align*}
and
\begin{align*}
    &\ \int \d_h A\cdot \d_h (\Ga g^{-1}\nab A+\la^2(A+V)) dx\\
    \les &\ \|\d_h A\|_{L^2}(\|\d_h g\|_{H^1}M_1^5+M_1^3\|\d \d_h A\|_{L^2}+\|\d_h \la\|_{L^2}M_1^5+M_1^2\|\d_h A\|_{L^2})\\
    \leq &\ \frac{c}{10}\|\d\d_h A\|_{L^2}^2+CM_1^{6}(\|\d_h g\|_{H^1}+\|(\d_h A,\d_h \la)\|_{L^2})^2.
\end{align*}
Then we obtain the estimates
\begin{align} \label{EE-dhA}
	\frac{d}{dt}\|\d_h A\|_{L^2}^2 + \|\d\d_h A\|_{L^2}^2
	\leq CM_1^{10} (\|\d_h g\|_{H^1}+\|(\d_h A,\d_h \la)\|_{L^2})^2.
\end{align}
Integrating over $[h_0,\infty)$, this yields
\begin{align*}
    &\ \frac{d}{dt}\int_{h_0}^\infty 2^{2hs} \|\d_h A\|_{L^2}^2 dh+\int_{h_0}^\infty 2^{2hs} \|\d\d_h A\|_{L^2}^2 dh\\
    &\leq CM^{10} \int_{h_0}^\infty 2^{2hs} (\|\d_h g\|_{H^1}+\|(\d_h A,\d_h \la)\|_{L^2})^2dh \leq C_8M_1^{12}.
\end{align*}
Hence, the bound \eqref{EEA} follows on the time interval $t\in[0, 3(C_8M_1^{10})^{-1}]$.

To conclude, on the time interval $[0,CM_1^{-2N-8}]$, we obtain the improved bound $\|[A^{(h)}]\|_{s,A}\leq 6M_1$. Hence, the estimate \eqref{Para-S} for connection $A$ follows.
\end{proof}

\begin{proof}[Proof of \eqref{DB-est1} and \eqref{HFB-est1}]
	The first bound \eqref{DB-est1} follows from \eqref{EE-dhg} and \eqref{EE-dhA}. The second bound \eqref{HFB-est1} is obtained by \eqref{EE-glow}, \eqref{EE-gHN}, \eqref{EE-Alow} and \eqref{EE-AHN}
	\begin{align*}
		&\ \frac{d}{dt}(\|\d g^{(j)}\|_{H^N}+\||D|^{\de_d}A^{(j)})\|_{H^{N-\de_d}}^2)\\
        &\leq M_1^{2N+4}+M_1^{2N+8}\big(\|(\d g^{(j)},\la^{(j)})\|_{H^N}^2+\||D|^{\de_d}A^{(j)}\|_{H^{N-\de_d}}^2\big)\\
		&\leq M_1^{2N+4}2^{2(N-s)j}c^2_j+M_1^{2N+8}\big(\|(\d g^{(j)},\la^{(j)})\|_{H^N}^2+\||D|^{\de_d}A^{(j)}\|_{H^{N-\de_d}}^2\big).
	\end{align*}
\end{proof}

\bigskip 
\section{Energy estimates for solutions}\label{Sec-EnEs}
This section is devoted to the energy estimates for the second fundamental form $\lambda$. More precisely, we aim to establish uniform control over the $X^s_{int}$ norm of the second fundamental form $\la$ by bootstrap argument. The key to this is to characterize these norms using intrinsic Sobolev norms with the natural metric as it evolves along the flow. In addition, we also prove the difference bounds and high frequency bounds for the regularized solutions, which will be used to establish the existence of rough solutions.

\subsection{Energy estimates for the second fundamental form $\la$}
Here we consider the quasilinear Schr\"odinger equation 
\begin{equation}\label{Lin-eq0}
\left\{
\begin{aligned}      
&\begin{aligned}
i(\d^B_t-V^\ga\nab^A_\ga )\la_{\al\be}+\De^A_g \la_{\al\be}
    =&\ i\la^{\ga}_{\al}\nab_{\be} V_{\ga}
+i\la_\be^\ga\nab_\al V_\ga
+\psi\Re(\la_{\al\de}\bar{\la}^\de_\be)\\
&\ -\Re(\la_{\si\de}\bar{\la}_{\al\be}-\la_{\si\be}\bar{\la}_{\al\de})\la^{\si\de} 
-\la_{\al\mu}\bar{\la}^\mu_\si\la^\si_\be\,,
\end{aligned}\\
&\la_{\al\be}(0)=\la_{\al\be,0}\,,
\end{aligned}\right.
\end{equation}
with the coefficients satisfying \eqref{par-syst}. Then under suitable assumptions on the coefficients,
we prove that the solution satisfies suitable energy bounds.

\begin{thm} \label{Thm-Es}
Let $N=[2s]+1$ and $M_1=C(M)$. Assume that the solutions $F^{(h)}$ of (SMCF) exist in some time interval $[0,T]$. 
If $\tri[g_0^{(h)}]\tri_{s+1,g}+\tri[A_0^{(h)}]\tri_{s,A}+\tri[\la^{(h)}_0]\tri_{s,ext}+\tri[\la^{(h)}_0]\tri_{s,int}\leq M_1$, then on the time interval $[0,\min\{T,CM_1^{-2N-8}\}]$ the solutions satisfy
\begin{equation}  \label{EE-H^s}
\tri [\la^{(h)}]\tri_{s,int}\leq 8M_1,\qquad \tri [\la^{(h)}]\tri_{s,ext}\leq 8C_{eq}M_1.
\end{equation}
Moreover, the solutions $\la^{(h)}$ and the orthonormal frames $m^{(h)}$ satisfy
\begin{align} \label{Hsla}
&\|\la^{(h)}\|_{H^s}\les C M_1,\\\label{mhs}
&\|\d^2 F^{(h)}\|_{H^s}+\|\d m^{(h)}\|_{\dot H^{2\de_d}\cap \dot H^s}\les C(M),\\\label{DB-est2}
&\begin{aligned}
\frac{d}{dt} \|\d_h \la^{(h)}\|_{L^2}^2&\leq  \ep \|(\d^2 \d_h g^{(h)},\d\d_h A^{(h)})\|_{L^2}^2 \\
&+ C(M)( \|\d_h g^{(h)}\|_{H^1}+\|(\d_h A^{(h)},\d_h \la^{(h)})\|_{L^2})^2.
\end{aligned}
\end{align}
\end{thm}


Here we start with the following energy estimates in intrinsic Sobolev spaces $\sfH^k$ \eqref{IntSob}. 

\begin{lemma}[Basic energy estimates]\label{lem-En}
Assume that the smooth solutions of (SMCF) exist in some time interval $[0,T]$. Then for each integer $k\geq 0$, there holds
\begin{align}  \label{Energy}
\frac{d}{dt}\|\la\|_{\sfH^k}^2\les \||\la|_g\|_{L^\infty}^2\|\la\|_{\sfH^k}^2\, .
\end{align}
\end{lemma}
\begin{proof}
The $\sfH^k$-norms are defined by the intrinsic Sobolev norm \eqref{IntSob}, which is independent on the choice of gauge. Hence, we can derive the energy estimates from \eqref{Lin-eq0} with the advection field $V=0$. Then the energy estimate \eqref{Energy} follows using \eqref{g_dt}, \eqref{comm-nab}, \eqref{comm-dt} and \eqref{HaInterpo}.
We can also refer to \cite[Lemma 2.7]{So19} for the proof.
\end{proof}

Next, we turn our attention to the proof of Theorem~\ref{Thm-Es}.

\begin{proof}[Proof of the energy estimate \eqref{EE-H^s}]
We assume that $\tri [\la^{(h)}]\tri_{s,int}\leq 8M_1$ on some interval $[0,T_1]$ with $T_1\leq T$, then we have the estimate \eqref{Para-S} on $[0,\min\{T_1,CM_1^{-2N-8}\}]$ for $g$ and $A$. Now it suffices to prove that on this time interval we have
\begin{align*}
    \tri[\la^{(h)}]\tri_{s,int}\leq 4M_1.
\end{align*}

By the energy estimates in Lemma~\ref{lem-En}, for any $h\geq h_0$ and $k\in \N$ we have
\begin{align*}
    \frac{d}{dt}\|\la^{(h)}(t)\|^2_{\sfH^{k}}\les \|\la^{(h)}\|_{L^\infty}^2 \|\la^{(h)}(t)\|^2_{\sfH^{k}}\les M_1^2 \|\la^{(h)}(t)\|^2_{\sfH^{k}}.
\end{align*}
Then we obtain
\begin{equation*}
\begin{aligned}
    &\ \frac{d}{dt}\big(\sum_{k=[s],[s]+1}2^{2(s-k)h_0}\|\la^{(h_0)}(t)\|^2_{\sfH^{k}}+\int_{h_0}^\infty 2^{2(s-N)h}\|\la^{(h)}(t)\|^2_{\sfH^{N}}dh\big)\\
    \leq&\  CM_1^2 \big(\sum_{k=[s],[s]+1}2^{2(s-k)h_0}\|\la^{(h_0)}\|^2_{\sfH^{k}}+\int_{h_0}^\infty 2^{2(s-N)h}\|\la^{(h)}\|^2_{\sfH^{N}}dh\big).
\end{aligned}  
\end{equation*}
Hence, on the time interval $[0,\frac{3}{2CM_1^2}]$, it holds
\begin{equation*}
    \sum_{k=[s],[s]+1}2^{2(s-k)h_0}\|\la^{(h_0)}(t)\|^2_{\sfH^{k}}+\int_{h_0}^\infty 2^{2(s-N)h}\|\la^{(h)}(t)\|^2_{\sfH^{N}}dh\leq 4M_1^2.
\end{equation*}

Next, we consider the estimates of $\int_{h_0}^\infty 2^{2sh}\|\d_h \la^{(h)}\|_{L^2}^2 dh$.
Formally, we define  
\[\mu^{(h)}_{\al\be}=\d_h \la^{(h)}_{\al\be},\qquad \mu^{(h)\al\be}=g^{(h)\al\si}g^{(h)\be\de}\mu^{(h)}_{\si\de}.\] 
For brevity, we omit the superscript $h$ of $\mu^{(h)}$, $\lambda^{(h)}$, $g^{(h)}$ for regularized manifold $\Sigma^{(h)}$.
Moreover, the metric and volume form satisfy
\begin{equation*} 
    \frac{4}{5}c_0\leq g^{(h)}\leq \frac{6}{5}c_0^{-1},\qquad  \sqrt{\det g^{(h)}}\sim 1.
\end{equation*}

Applying $\frac{d}{dt}$ to $\|\mu\|_{L^2}^2$, we have
\begin{align*}
    &\ \frac{1}{2}\frac{d}{dt}\int |\mu|^2_g\ dvol=\int \Re(\d_t^B \mu_{\al\be} \overline{\mu}^{\al\be})+\Re(\d_t g^{\al\si} \mu_{\al\be}\bar{\mu}_\si^\be)+\frac{1}{4}|\mu|^2_g  g^{\al\be} \d_t g_{\al\be}\ dvol\\
    &=\Re \int -i[i(\d_t^B-V^\ga \nab^A_\ga)+\nab^A_\si \nab^{A,\si}]\mu_{\al\be} \bar{\mu}^{\al\be}\ dvol\\
    & + \int \Re[(V^\ga \nab^A_\ga+i\nab^A_\si \nab^{A,\si})\mu_{\al\be} \bar{\mu}^{\al\be}]+\frac{1}{4}|\mu|^2_g  g^{\al\be} \d_t g_{\al\be}\ dvol
    +\int \Re(\d_t g^{\al\si} \mu_{\al\be}\bar{\mu}_\si^\be) \  dvol\\
    &=:I_1+I_2+I_3\,.
\end{align*}
From the $\la$-equation \eqref{Lin-eq0}, the first integral $I_1$ is rewritten as 
\begin{align*}
    I_1=&\ \Re \int ([\d_t^B-V^\ga\nab^A_\ga-i\nab^A_\si\nab^{A,\si},\d_h]\la_{\al\be})\bar{\mu}^{\al\be}\\
    &\ +2(\mu_{\al\ga}\nab_\be V^\ga+\la_{\al\ga}[\d_h,\nab_\be]V^\ga+\la_{\al\ga}\nab_\be \d_h V^\ga)\bar{\mu}^{\al\be}+\d_h(\la\ast\la\ast \la)_{\al\be}\bar{\mu}^{\al\be}\ dvol.
\end{align*}
The second integral $I_2$ vanishes since
\begin{equation*}
    I_2=\int -\frac{1}{2}\nab_\ga V^\ga |\mu|^2_g+\Re i|\nab^{A}\mu|^2+\frac{1}{2}|\mu|^2 \nab^\al V_\al\ dvol=0\,.
\end{equation*}
Using the formula \eqref{g_dt}, the sum of the last integral $I_3$ together with the term $2\mu_{\al\ga}\nab_\be V^\ga \bar{\mu}^{\al\be}$ in $I_1$ is  bounded by 
\begin{align*}
    I_3+2\Re \int \mu_{\al\ga}\nab_\be V^\ga \bar{\mu}^{\al\be}\ dvol
    \les \|\mu\|_{L^2}^2 \|\la\|^2_{L^\infty}.
\end{align*}
The other terms in $I_1$ are estimated as follows.

a) By $B=\nab^\al A_\al$, we estimate the term 
\begin{align*}
    &\ \Re \int [\d_t^B, \d_h ]\la \bar{\mu}\ dvol=\Re \int \d_h B\la \bar{\mu}\ dvol=\int \d_h (\nab A)\la \bar{\mu}\ dvol\\
    =&\ \int (\d_h \Ga  A+\nab^\al \d_h A_\al)\la\bar{\mu}\ dvol\les (\|\d_h \Ga\|_{L^2}\| A\|_{L^\infty}+\|\nab \d_h A\|_{L^2})\|\la\|_{L^\infty}\|\mu\|_{L^2}\\
    \les&\ (\|\d_h \Ga\|_{L^2}+\| \d_h A\|_{H^1})\|( A,\Ga)\|_{L^\infty}\|\la\|_{L^\infty}\|\mu\|_{L^2}.
\end{align*}

b) We estimate the term 
\begin{align*}
    &\ \Re \int [V^\ga \nab^A_\ga, \d_h ]\la_{\al\be} \bar{\mu}^{\al\be}\ dvol=\Re \int (\d_h V^\ga \nab^A_\ga+V^\ga \d_h \Ga+iV^\ga \d_h A)\la \bar{\mu}\ dvol\\
    \les&\  \big(\|\d_h V \nab^A \la\|_{L^2}+\|V\|_{L^\infty}(\|\d_h \Ga\|_{L^2}+\|\d_h A\|_{L^2})\|\la\|_{L^\infty}\big)\|\mu\|_{L^2}.
\end{align*}
Using Sobolev embeddings and $s>d/2$, this is bounded by
\begin{align*}
    &\  \big(\|\d_h V\|_{H^1}(\| \la\|_{H^s}+\|(\Ga+A) \la\|_{L^\infty})+\|V\|_{L^\infty}(\|\d_h \Ga\|_{L^2}+\|\d_h A\|_{L^2})\|\la\|_{L^\infty}\big)\|\mu\|_{L^2}\\
    \les &\ \big(  \|\d_h V\|_{H^1}+ \|\d_h \Ga\|_{L^2}+\|\d_h A\|_{L^2} \big) \|(\Ga,A,V)\|_{L^\infty}\|\la\|_{H^s}\|\mu\|_{L^2}.
\end{align*}

c) We estimate the term 
\begin{align*}
    &\ \Re \int [\nab^A_\si \nab^{A,\si}, \d_h ]\la_{\al\be} \bar{\mu}^{\al\be}\ dvol=\Re \int (\d_h (\Ga+A) \nab^A \la+ \nab_\si \d_h (\Ga+A)\la ) \bar{\mu}\ dvol\\
    \les &\ \big(\|\d_h \Ga+\d_h A\|_{H^1}(\| \la\|_{H^s}+\|(\Ga+A) \la\|_{L^\infty})+\|\nab \d_h (\Ga+A)\|_{L^2}\|\la\|_{L^\infty}\big)\|\mu\|_{L^2}\\
    \les&\ \big(\|\d_h \Ga\|_{H^1}+\|\d_h A\|_{H^1}\big)
    \|(\Ga,A)\|_{L^\infty}\|\la\|_{H^s}\|\mu\|_{L^2}
\end{align*}

d) We estimate the term 
\begin{align*}
    \int \la_{\al\ga}[\d_h,\nab_\beta]V^\ga \bar{\mu}^{\al\be}\ dvol=\int \la_{\al\ga} \d_h \Ga^\ga_{\be\si} V^\si \bar{\mu}^{\al\be}\ dvol
    \les  \|\la\|_{L^\infty}\|\d_h \Ga\|_{L^2}\|V\|_{L^\infty}\|\mu\|_{L^2}.
\end{align*}

e) We estimate the term 
\begin{align*}
    \int \la_{\al\ga}\nab_\beta \d_h V^\ga \bar{\mu}^{\al\be}\ dvol=&\ \int \la_{\al\ga}\nab_\beta \d_h (g^{\si\de}\Ga_{\si\de}^\ga) \bar{\mu}^{\al\be}\ dvol\les \|\la\|_{L^\infty}\|\nab\d_h V\|_{L^2}\|\mu\|_{L^2}\\
    \les&\  \|\d_h V\|_{H^1}(1+\|\Ga\|_{L^\infty})\|\la\|_{L^\infty}\|\mu\|_{L^2},\\
    \int \d_h (\la\ast\la\ast\la)\bar{\mu}^{\al\be}\ dvol\les&\  \|\mu\|_{L^2}^2\|\la\|_{L^\infty}^2+\|\d_h g\|_{L^2}\|\la\|_{L^\infty}^3\|\mu\|_{L^2}.
\end{align*}

Hence, from the above computations, we obtain 
\begin{align*}
    \frac{1}{2}\frac{d}{dt}\|\mu\|_{L^2}^2\les&\  \big(  \|\d_h V\|_{H^1}+ \|\d_h \Ga\|_{H^1}+\|\d_h A\|_{H^1} \big) \|(\Ga, A,V)\|_{L^\infty}\|\la\|_{H^s}\|\mu\|_{L^2}\\
    &\ +\|\mu\|_{L^2}^2\|\la\|_{L^\infty}^2+\|\d_h g\|_{L^2}\|\la\|_{L^\infty}^3\|\mu\|_{L^2}.
\end{align*}
By $V^\ga=g_{\al\be}\Ga_{\al\be}^\ga$ and Theorem~\ref{Para-thmS}, we further bound this by 
\begin{align} \label{mu-key}
\frac{d}{dt}\|\mu\|_{L^2}^2\les (\|\d_h g\|_{H^2}+\|\d_h A\|_{H^1})M_1^5 \|\mu\|_{L^2}+M_1^2\|\mu\|_{L^2}^2.
\end{align}
Integrating over $[h_0,\infty)$ and by \eqref{Para-S} and the bootstrap assumption, this yields
\begin{align*}
    \frac{d}{dt}\int_{h_0}^\infty 2^{2sh} \|\mu\|^2_{L^2} dh
    &\les M_1^5(\tri[g^{(h)}]\tri_{s+1,g}+\tri[A^{(h)}]\tri_{s,A}) \tri[\la^{(h)}]\tri_{s,ext}\\
    &+ M_1^5\big(\int_{h_0}^\infty 2^{2hs}\|(\d^2 \d_h g,\d\d_h A)\|_{L^2}^2 dh\big)^{1/2} \tri[\la^{(h)}]\tri_{s,ext} +M_1^2 \tri[\la^{(h)}]\tri_{s,ext}^2\\
    &\les M_1^7+M_1^6\big(\int_{h_0}^\infty 2^{2hs}\|(\d^2 \d_h g,\d\d_h A)\|_{L^2}^2 dh\big)^{1/2}.
\end{align*}
Integrating over $[0,t]$, we obtain on the time interval $[0,\frac{9}{CM_1^{10}}]$
\begin{align*}
    \int_{h_0}^\infty 2^{2sh} \|\mu(t)\|^2_{L^2} dh\leq \int_{h_0}^\infty 2^{2sh} \|\mu_0\|^2_{L^2} dh+C t M_1^7+\sqrt{t}M_1^7\leq M_1^2+C \sqrt{t} M_1^7\leq 4M_1^2.
\end{align*}
Hence, the estimate \eqref{EE-H^s} follows. 
\end{proof}

\begin{proof}[Proof of \eqref{Hsla}-\eqref{DB-est2}]
By \eqref{EE-H^s} and the embedding~\eqref{Emb-la}, we get the estimate \eqref{Hsla}.
From \eqref{mu-key} and H\"older's inequality we obtain
\begin{align*}
    \frac{d}{dt}\|\mu\|_{L^2}^2\les C(M)(\|\d_h g\|_{H^1}+\|\d_h A\|_{L^2}+\|\mu\|_{L^2})\|\mu\|_{L^2}+\ep\|(\d^2 \d_h g,\d\d_h A)\|_{L^2}^2,
\end{align*}
which gives the estimate \eqref{DB-est2}. 

Next, we prove the estimate \eqref{mhs}. By \eqref{strsys-cpf} we have
\begin{align*}
    &\|\d^2 F\|_{L^2\cap L^\infty}\les \|\Ga\|_{L^2\cap L^\infty}\|\d F\|_{L^\infty}+\|\la\|_{L^2\cap L^\infty}\|\nu\|_{L^\infty}\les C(M),\\
    &\|\d \nu\|_{L^{\frac{2d}{d-2\de_d}}}\les \|A\|_{L^{\frac{2d}{d-2\de_d}}}+\|\la\|_{L^{\frac{2d}{d-2\de_d}}}\les C(M),
\end{align*}
and 
\begin{align*}
    \|\d \nu\|_{\dot H^{2\de_d}}&\les \|A\|_{\dot H^{2\de_d}\cap \dot H^{\de_d}}(\|\nu\|_{L^\infty}+\|\d \nu\|_{L^{\frac{2d}{d-2\de_d}}})+\|\la\|_{\dot H^{2\de_d}\cap \dot H^{\de_d}}(\|\d F\|_{L^\infty}+\|\d^2 F\|_{L^{\frac{2d}{d-2\de_d}}})\\
    &\les C(M).
\end{align*}
Then we can bound the high frequency part by
\begin{align*}
    \|\d^2 F\|_{\dot H^s}&\les \|\Ga\|_{H^s}(\|\d F\|_{L^\infty}+\|P_{>0}\d F\|_{\dot H^s})+\|\la\|_{H^s}(\|\nu\|_{L^\infty}+\|P_{>0} \nu\|_{\dot H^s})\\
    &\les C(M)+C(M)\|P_{>0}\d F\|_{L^2}^{\frac{1}{s+1}}\|P_{>0}\d F\|_{\dot H^{s+1}}^{\frac{s}{s+1}}+C(M)\|P_{>0}\nu\|_{L^2}^{\frac{1}{s+1}}\|P_{>0}\nu\|_{\dot H^{s+1}}^{\frac{s}{s+1}}\\
    &\leq C(M)+\frac{1}{2}(\|\d^2 F\|_{\dot H^s}+\|\d \nu\|_{\dot H^s}).
\end{align*}
and 
\begin{align*}
    \|\d \nu\|_{\dot H^s}&\les \|A\|_{\dot H^s}\|\nu\|_{L^\infty}+\|A\|_{\dot H^{\de_d}\cap L^\infty}(\|P_{\leq 0}\d \nu\|_{\dot H^{2\de_d}}+\|P_{>0}\nu\|_{\dot H^s})\\
    &+\|\la\|_{H^s}(\|\d F\|_{L^\infty}+\|P_{>0} \d F\|_{\dot H^s})\\
    &\les C(M)+C(M)\|P_{>0}\nu\|_{L^2}^{\frac{1}{s+1}}\|P_{>0}\nu\|_{\dot H^{s+1}}^{\frac{s}{s+1}}+C(M)\|P_{>0}\d F\|_{L^2}^{\frac{1}{s+1}}\|P_{>0}\d F\|_{\dot H^{s+1}}^{\frac{s}{s+1}}\\
    &\leq C(M)+\frac{1}{2}(\|\d^2 F\|_{\dot H^s}+\|\d \nu\|_{\dot H^s}).
\end{align*}
This gives $\|(\d^2 F,\d \nu)\|_{\dot H^s}\les C(M)$. Hence the estimate \eqref{mhs} follows.
\end{proof}

\subsection{The bounds for the regularized solutions}

As a corollary of Theorem~\ref{Thm-Es}, we have the following bounds.
\begin{lemma}
The family of solutions $\Sigma^{(h)}$ given in Theorem~\ref{Thm-Es} satisfies the estimates
\begin{align} \label{DB-t}
\int_{j}^{j+1}\|\d_h g^{(h)}\|_{H^1}+\|\d_h A^{(h)}\|_{L^2}+\|\d_h\la^{(h)}\|_{L^2}dh\les C(M)2^{-sj}c_j\,,\\ \label{HFB-t}
\|\d g^{(j)}\|_{H^{N}}+\||D|^{\de_d}A^{(j)}\|_{H^{N-\de_d}}+\|\la^{(j)}\|_{\sfH^N\cap H^N}\les C(M)2^{(N-s)j}c_j\,.
\end{align}
\end{lemma}

\begin{proof}
From \eqref{DB-01} and the estimates \eqref{DB-est1} and \eqref{DB-est2}, we obtain
    \begin{align*}
    \int_j^{j+1} \|\d_h g\|_{H^1}+\|\d_h A\|_{L^2}+\|\d_h \la\|_{L^2}dh\les C(M) 2^{-sj}c_j\,.
    \end{align*}
The bound \eqref{HFB-t} is obtained immediately from \eqref{HFB-est1}, \eqref{Energy} and \eqref{HFB-01}.
\end{proof}

To gain the convergence of the solutions with regularized data in the strong topology, we will use the following lemma.
\begin{lemma}
For any $h\geq h_0$, the solutions $F^{(h)}$ and the orthonormal frame $m^{(h)}$ on $\Sigma^{(h)}$ satisfy
\begin{align} \label{DB-Fm}
    \|\d F^{(h+1)}-\d F^{(h)}\|_{H^1}+\|m^{(h+1)}-m^{(h)}\|_{H^1}&\les_M 2^{-sh}c_h\,,\\ \label{HFB-Fm}
    \|\d F^{(h)}\|_{\dot H^N}+\|m^{(h)}\|_{\dot H^N}+\|\d^2 F^{(h)}\|_{H^N}+\||D|^{2\de_d}\d m^{(h)}\|_{H^{N-2\de_d}}&\les_M 2^{(N-s)h}c_h\,.
\end{align}
\end{lemma}

\begin{proof}
\emph{i) We prove the estimate \eqref{DB-Fm}.} For simplicity, we denote $\de F=F^{(h+1)}-F^{(h)}$ and $\de m=m^{(h+1)}-m^{(h)}$. Then by \eqref{strsys-cpf}, \eqref{mo-frame} and \eqref{DB-t}
we have
\begin{equation} \label{816}
\begin{aligned}
\|\d^2 \de F\|_{L^2}&=\|(\Ga^{(h+1)}\d F^{(h+1)}+\la^{(h+1)}m^{(h+1)})-(\Ga^{(h)}\d F^{(h)}+\la^{(h)}m^{(h)})\|_{L^2}\\
&\les \|\de \Ga\|_{L^2}C(M)+C(M)\|\d\de F\|_{L^2}+\|\de \la\|_{L^2}+C(M)\|\de m\|_{L^2}\\
&\les C(M)2^{-sh}c_h+C(M)\|(\d \de F,\de m)\|_{L^2},
\end{aligned}
\end{equation}
\begin{equation}\label{817}
\begin{aligned}
\|\d \de m\|_{L^2}&=\|(Am+\la \d F)^{(h+1)}-(Am+\la \d F)^{(h)}\|_{L^2}\\
&\les \|\de A\|_{L^2}+C(M)\|\de m\|_{L^2}+\|\de \la\|_{L^2}C(M)+C(M)\|\d \de F\|_{L^2}\\
&\les C(M)2^{-sh}c_h+C(M)\|(\d \de F,\de m)\|_{L^2},
\end{aligned}
\end{equation}
and 
\begin{equation}\label{818}
\begin{aligned}
\|\d_t \de F\|_{L^2}&=\|(\psi^{(h+1)}m^{(h+1)}+V^{(h+1)}\d F^{(h+1)})-(\psi^{(h)}m^{(h)}+V^{(h)}\d F^{(h)})\|_{L^2}\\
&\les C(M)2^{-sh}c_h+C(M)\|(\d \de F,\de m)\|_{L^2}.
\end{aligned}
\end{equation}
Then by integration by parts, \eqref{816} and \eqref{818}, we get
\begin{align*}
&\ \frac{1}{2}\frac{d}{dt}\|\d\de F\|_{L^2}^2=\int \d\de F\cdot \d \d_t \de F dx\leq \|\d^2 \de F\|_{L^2}\|\d_t \de F\|_{L^2}\\
&\les C(M)2^{-2sh}c_h^2+C(M)\|(\d \de F,\de m)\|_{L^2}^2,
\end{align*}
and
\begin{align*}
&\ \frac{1}{2}\frac{d}{dt}\|\de m\|_{L^2}^2=\int \de m\cdot  \d_t \de m dx\\
&=\int \de m\cdot \big[ (B^{(h+1)}m^{(h+1)}+(\d^{A^{(h+1)}} \psi^{(h+1)}+\la^{(h+1)} V^{(h+1)})\d F^{(h+1)})\\
&\qquad -(B^{(h)}m^{(h)}+(\d^{A^{(h)}} \psi^{(h)}+\la^{(h)} V^{(h)})\d F^{(h)}) \big] dx \\
&=\int \de m\cdot (\d A^{(h+1)}m^{(h+1)}-\d A^{(h)} m^{(h)}) dx\\
&\quad + \int \de m\cdot (\d \psi^{(h+1)} \d F^{(h+1)}-\d \psi^{(h)}\d F^{(h)}) dx\\
&\quad +\int \de m\cdot \big[ (\Ga m+(A \psi +\la V)\d F)^{(h+1)}-(\Ga m+(A \psi +\la V)\d F)^{(h)}\big] dx\\
&=: I_1+I_2+I_3.
\end{align*}
By integration by parts, \eqref{DB-t} and \eqref{817}, $I_1$ is bounded by
\begin{align*}
    |I_1| &\leq \int |\d \de m\cdot (A^{(h+1)}m^{(h+1)}-A^{(h)}m^{(h)})|+|\de m\cdot (A^{(h+1)}\d m^{(h+1)}-A^{(h)}\d m^{(h)})| dx \\
    &\les \|\d \de m\|_{L^2}(\|\de A\|_{L^2}+C(M)\|\de m\|_{L^2})+\|\de m\|_{L^2}(\|\de A\|_{L^2}\|\d m\|_{L^\infty}+C(M)\|\d\de m \|_{L^2})\\
    &\les C(M)(2^{-2sh}c_h^2+\|(\d \de F,\de m)\|_{L^2}^2).
\end{align*}
And similarly, we have
\begin{align*}
    |I_2|+|I_3|\les C(M)2^{-2sh}c_h^2+\|(\d \de F,\de m)\|_{L^2}^2.
\end{align*}
Hence, we obtain
\begin{align*}
    \frac{d}{dt}\|(\d \de F,\de m)\|_{L^2}^2\leq C(M)2^{-2sh}c_h^2+C(M)\|(\d \de F,\de m)\|_{L^2}^2.
\end{align*}
By Gr\"onwall's inequality and \eqref{DB-02}, this yields the difference bound on the time interval $[0,\min\{T,CM_1^{-2N-8}\}]$
\begin{align*}
    \|(\d \de F,\de m)\|_{L^2}^2\leq C(M)2^{-2sh}c_h^2.
\end{align*}
Thus from \eqref{816} and \eqref{817} we also have
\begin{align*}
    \|\d^2 F^{(h+1)}-\d^2 F^{(h)}\|_{L^2}+\|\d m^{(h+1)}-\d m^{(h)}\|_{L^2}\les C(M)2^{-sh}c_h.
\end{align*}

\emph{ii) We prove the estimate \eqref{HFB-Fm}.} By \eqref{strsys-cpf}, \eqref{sys-cpf} and the estimate \eqref{HFB-t}, we have
\begin{equation} \label{hfb-k1}
    \begin{aligned}
    \|\d^2 F^{(j)}\|_{H^N}&=\|\Ga^{(j)}\d F^{(j)}+\la^{(j)}m^{(j)}\|_{H^N}\\
    &\les \|\Ga^{(j)}\|_{H^N}C(M)+C(M)\|\d F^{(j)}\|_{\dot H^N}+\|\la^{(j)}\|_{H^N}+C(M)\|m^{(j)}\|_{\dot H^N}\\
    &\les C(M)(2^{(N-s)j}c_j +\|(\d F^{(j)},m^{(j)})\|_{\dot H^N}).
\end{aligned}
\end{equation}
\begin{equation} \label{hfb-k2}
    \|\d m^{(j)}\|_{\dot H^N}=\|A^{(j)}m^{(j)}+\la^{(j)}\d F^{(j)}\|\les C(M)(2^{(N-s)j}c_j +\|(\d F^{(j)},m^{(j)})\|_{\dot H^N}).
\end{equation}
and 
\begin{align} \label{hfb-k3}
    \|\d_t F^{(j)}\|_{H^N}=\|\psi^{(j)}m^{(j)}+V^{(j)}\d F^{(j)}\|_{H^N}\les C(M)(2^{(N-s)j}c_j +\|(\d F^{(j)},m^{(j)})\|_{\dot H^N}).
\end{align}
Then by integration by parts, \eqref{hfb-k1} and \eqref{hfb-k3}, we get
\begin{align*}
    \frac{1}{2}\frac{d}{dt}\|\d F^{(j)}\|_{\dot H^N}^2\leq \|\d^2 F^{(j)}\|_{\dot H^N}\|\d_t F^{(j)}\|_{\dot H^N}\les C(M)(2^{2(N-s)j}c_j^2+\|(\d F^{(j)},m^{(j)})\|_{\dot H^N}^2),
\end{align*}
and by \eqref{mo-frame}, \eqref{hfb-k1} and \eqref{hfb-k2}, we arrive at
\begin{align*}
    &\ \frac{1}{2}\frac{d}{dt}\|m^{(j)}\|_{\dot H^N}^2\\
    &=\int \d^N m^{(j)} \d^N\big(\d A^{(j)} m^{(j)}+\d \psi^{(j)}\d F^{(j)}+\Ga^{(j)}A^{(j)}m^{(j)}+(A^{(j)}\psi^{(j)}+\la^{(j)}V^{(j)})\d F^{(j)}\big) dx\\
    &\leq \int \d^{N+1}m^{(j)} \d^N(A^{(j)}m^{(j)}+\psi^{(j)}\d F^{(j)})+\d^{N}m^{(j)} \d^N(A^{(j)}\d m^{(j)}+\psi^{(j)}\d^2 F^{(j)})dx\\
    &\quad +\|m^{(j)}\|_{\dot H^N}C(M)(2^{(N-s)j}c_j+\|(\d F^{(j)},m^{(j)})\|_{\dot H^N})\\
    &\les C(M)(2^{2(N-s)j}c_j^2+\|(\d F^{(j)},m^{(j)})\|_{\dot H^N}^2).
\end{align*}
The above two estimates together with \eqref{HFB-02} yield
\begin{align*}
    \|\d F^{(j)}\|_{\dot H^N}^2+\|m^{(j)}\|_{\dot H^N}^2\les C(M)2^{2(N-s)j}c_j^2.
\end{align*}
In view of \eqref{hfb-k1} and \eqref{hfb-k2}, this also gives
\begin{align*}
    \|\d^2 F^{(j)}\|_{H^N}+\|\d m^{(j)}\|_{\dot H^N}\les C(M)(2^{(N-s)j}c_j+\|(\d F^{(j)},m^{(j)})\|_{\dot H^N})\les C(M)2^{(N-s)j}c_j.
\end{align*}
These together with \eqref{mhs} yield the estimate \eqref{HFB-Fm}. We complete the proof of the lemma.
\end{proof}

\bigskip 
\section{Construction of regular solutions}     \label{Sec-RegSol}

In this section, we construct  regular solutions for the (SMCF) flow using an Euler type  time discretization method.
Since the (SMCF) flow is a quasilinear system, we will work in intrinsic Sobolev spaces in order to favourably propagate  the bounds for the  second fundamental form, and avoid the nontrapping condition.
Here we will work directly at the level of the manifold rather than at the level of the second fundamental form $\la$. This is because the second fundamental form $\la$ must satisfy compatibility conditions, and iterating it directly  over time steps would cause a loss of these constraints.

Let the initial manifold $(\Sigma_0,g(0))$ be a complete Riemannian manifold of dimension $d$ embedded in $\R^{d+2}$, with bounded second fundamental form
\begin{equation}  \label{RegData}
\|\Lambda_0\|_{\sfH^k(\Sigma_0)}\leq M,\qquad k>\frac{d}{2}+5,
\end{equation}
bounded Ricci curvature and bounded metric, i.e. 
\begin{equation}  \label{GeoCond}
    |\Ric(0)|\leq  C_0, \qquad \inf_{x\in \Sigma_0} {\rm Vol}_{g(0)}(B_x(1))\geq v,
\end{equation}  
for some $C_0>0$ and $v>0$,
where ${\rm Vol}_{g(0)}(B_x(1),\Sigma_0)$ stands for the volume of ball $B_x(1)$ on $\Sigma_0$ with respect to $g(0)$. 

We also assume that there exists a global $\R^d$ parametrization of $\Sigma_0$ so that we have the uniform bound 
\begin{equation}\label{good-param}
c I \leq g(0) \leq CI.    
\end{equation}
This, in turn, combined with the
bound \eqref{RegData}, implies that the parametrization can be in effect chosen so that
\begin{equation}\label{better-data}
\| \partial F_0\|_{H^{k+1}_{uloc}} \lesssim_M 1,
\end{equation}
where the uniform local norm is defined as 
\[
\| \partial F_0\|_{H^{k+1}_{uloc}}= \sup_{x \in \R^d} \| \partial F_0\|_{H^{k+1}(B_x(1))}.
\]

To see this we refer for instance  to Breuning~\cite{Bre}, which shows that locally the surface $\Sigma_0$ has $H^{k+2}$ regularity, i.e. there exists $r$ depending
only on  $\|\Lambda\|_{\sfH^k}$ so that 
for each $p \in \Sigma_0$, the set $\Sigma_0 \cap B(p,r)$ 
is the graph of a $H^{k+2}$ function, again with a bound depending only on  $\|\Lambda\|_{\sfH^k}$. After applying 
an Euclidean isometry, this implies that we can choose 
local coordinate functions 
\[
x^p : \Sigma_0 \cap B(p,r) \to \R^d
\]
with $H^{k+2}$ regularity
which match $F_0$ linearly at $F_0^{-1}(p)$, i.e.
\[
x^p(p) = F_0^{-1}(p), \qquad Dx^p(p) = (DF_0(F^{-1}(p))^{-1}.
\]
Then the local coordinate functions on nearby balls must be $C^2$ close. Then they can be easily assembled together using an appropriate partition of unity associated to the covering of $\Sigma_0$ with balls of radius $r$. This yields a global map 
\[
\Sigma_0 \ni p \to x(p) \in \R^d,
\]
By construction this map is $C^1$ close to $F_0^{-1}$, 
so it is a global diffeomorphism into $\R^d$. Inverting it yields the desired coordinates satisfying \eqref{better-data}.

\medskip

Consider a small time step $\ep>0$. Then our objective  will be to produce a discrete approximate solution $\Sigma^\ep(j\ep)=F^\ep(j\ep,\R^d)$ for any $j\ll_M \ep^{-1}$ with the following properties:

(a) Ricci curvature bound and volume of balls on $\Sigma^\ep(j\ep)$:
\begin{gather*}
	|\Ric^\ep(j\ep)|\leq C C_0\,,\quad \inf_{x\in \Sigma^\ep(j\ep)} {\rm Vol}_{g(j\ep)}(B_x(e^{C(M)j\ep},\Sigma^\ep(j\ep)))\geq e^{-C(M)j\ep}v\,.
\end{gather*}  

(b) Norm bound for second fundamental form $\Lambda^\ep(j\ep)$ on $\Sigma^\ep(j\ep)$:
\begin{align*}
    \|\La^\ep(j\ep)\|_{\mathsf H^k(\Sigma(j\ep))}\leq CM\,.
\end{align*}

(c) Approximate solution:
\begin{gather}  \label{appsol}
    \|F^\ep((j+1)\ep)-F^\ep(j\ep)+\ep J(F^\ep(j\ep))\bH(F^\ep(j\ep))\|_{L^2}\les \ep^{3/2}\, ,
\end{gather}

(d) Bounds for the metric and coordinate map: 
\begin{align}\nonumber   
c c_1 I\leq & \ g(j\ep)\leq CC_1 I,\  
\\\nonumber
    \|\d F^\ep((j+1)\ep)&-\d F^\ep(j\ep)\|_{L^\infty}\les \ep\,. 
\end{align}
 Here we remark that the bounds in part (a) and (b) are geometric bounds, independent of the choice 
 of the parametrization of the manifolds $\Sigma^\ep(j\ep)$. However, the bounds in (c), (d) are relative to a well chosen choice 
 of parametrization.

To obtain the above approximate solution, it suffices to carry out a single step:
\begin{thm} \label{thm-reg}
Let $(\Sigma_0,g_0$ be a complete Riemannian manifold of dimension $d$ 
satisfying \eqref{RegData}, \eqref{GeoCond} and \eqref{good-param}.
Let $\epsilon \ll 1$. Then there exists an approximate one step iterate $\Sigo=\Fo(\R^d)$ with the following properties:

(a) Ricci curvature bound and volume of balls on $\Sigo$:
\begin{gather*}
	|\Ric(\Sigo)|\leq C_0(1+C(M)\ep)\,,\qquad \inf_{x\in \Sigo} {\rm Vol}_{g}(B_x(r_0 e^{C(M)\ep}))\geq e^{-C(M)\ep} v\,,
\end{gather*}

(b) Norm bound for second fundamental form $\Lamo$:
\begin{align*}   
    \|\Lamo\|^2_{\mathsf H^k(\Sigo)}\leq (1+C(M)\ep)\|\Lambda_0\|^2_{\mathsf H^k(\Sigma_0)}\,.
\end{align*}
    
(c) Approximate solution:
\begin{align*}  
    \|\Fo- F_0+\ep\Im(\psi_0\bar{m}_0)\|_{L^2}\les \ep^{3/2}\,.
\end{align*}

(d) Bounds for the metric: 
\begin{gather*}
	(1-C(M)\ep)g_0 \leq \go \leq (1+C(M)\ep) g_0\,,\\
 \|\d\Fo- \d F_0\|_{L^\infty}\les \ep\,.
\end{gather*}
\end{thm}

Since a direct application of 
an Euler method looses derivatives, we instead construct our one step iterate $\Sigo$ in two steps: 
\begin{enumerate}[label= \roman*)]
\item  We use a Willmore-type flow to regularize the initial manifold $\Sigma_0$, in order to obtain a regularized manifold $\Sigma_\ep$, where we have good regularization estimates and norm bounds. This is a key step in order to deal with the derivative loss. 

\item  We use an Euler iteration, but starting with 
$\Sigma_\ep$ instead of $\Sigma_0$, in order  to construct the one step approximate solution $\Sigo$, where we also prove the properties in Theorem~\ref{thm-reg}.
\end{enumerate}

 To construct the one step iterate
we harmlessly initialize the coordinates on $\Sigma_0$  so that in our global $\R^d$ parametrization we have 
the optimal regularity \eqref{better-data}.
However this higher regularity is not uniformly propagated when iterating multiple steps. Instead, we obtain dynamical coordinates by simply propagating 
the initial choice of coordinates though each of the iterative steps. This will yield a short time solution $F$ in the temporal gauge, and with a loss of regularity.
This loss is rectified at the very end by switching to the heat gauge.

\subsection{Regularization of immersed submanifold}
Here we utilize a geometric Willmore-type flow in order to regularize the immersed manifold $\Sigma_0$. For an immersed submanifold $F:\Sigma\rightarrow \R^{d+2}$ we introduce the Willmore-type functional defined as
\begin{equation*}
    \mathcal W(F) = \int |\nabla^\perp \mathbf{H}|^2 dvol,
\end{equation*}
where $\mathbf H$ denotes the mean curvature, $\Lambda$ is the second fundamental form, $\nab^\perp$ is the covariant derivatives on normal bundle $\mathcal N\Sigma$ and $dvol$ is the induced volume form. The
associated Euler-Lagrange operator is as follows.
\begin{lemma}  \label{lem7.2}
The Euler-Lagrange operator of $\mathcal W(F)$ (or its variational derivative) is given by
\begin{align*}
W(F)=&\ -((\De^\perp)^2 \mathbf H+\Lambda^{\al\be}\<\Lambda_{\al\be},\De^\perp \bH\>-\Lambda^{\al\be}\<\nab^\perp_\al \bH,\nab^\perp_\be \bH\>+\frac{1}{2}\mathbf H |\nab^\perp\bH|^2\\
&\ +\nab^\perp_\si(\bH\<\Lambda^{\al\si},\nab^\perp_\al\bH\>+\nab^\perp_\al\bH\<\Lambda^{\al\si},\bH\>)).
\end{align*}
\end{lemma}
\begin{proof}
    Let $F: \R^d\times I\rightarrow \R^{d+2},\ I=(\tau_1,\tau_2)\ni 0$ be a smooth variation with normal velocity field $V=\d_\tau F\in \mathcal N\Sigma$. Then the following formulas hold
    \begin{align*}
    &\d_\tau g^{\al\be}=2\<\Lambda^{\al\be},V\>,\qquad\qquad\qquad 
    \d_\tau (dvol)=-\<\bH,V\>dvol,\\
    &\d_\tau^\perp \bH=\De V+\Lambda^{\al\be}\<\Lambda_{\al\be},V\>,\qquad 
    [\d_\tau^\perp,\nab_\al]\bH=\Lambda_{\al\si}\<\nab^\si V,\bH\>+\nab^\si V\<\Lambda_{\si\al},\bH\>.
    \end{align*}
    Thus we obtain 
    \begin{align*}
        \frac{d}{d\tau}\mathcal W(F)=&\ \int g^{\al\be}\nab_\al \d_\tau^\perp \bH \nab_\be \bH+ g^{\al\be}[\d_\tau^\perp,\nab_\al]\bH \nab_\be \bH\\
        &\ +\frac{1}{2}\d_\tau g^{\al\be} \nab_\al \bH \nab_\be\bH-\frac{1}{2}|\nab\bH|^2\<\bH,V\>\ dvol\\
        =&\ \int -(\De V+\Lambda^{\al\be}\<\Lambda_{\al\be},V\>)\De \bH+ (\Lambda_{\al\si}\<\nab^\si V,\bH\>+\nab^\si V\<\Lambda_{\si\al},\bH\>) \nab^\al \bH\\
        &\ +\<\Lambda^{\al\be},V\> \<\nab_\al \bH, \nab_\be\bH\>-\frac{1}{2}|\nab\bH|^2\<\bH,V\>\ dvol\\
        =&\ \int \< V, -\De^2 \bH -\Lambda^{\al\be}\<\Lambda_{\al\be},\De \bH\>-\nab^\si(\bH\<\Lambda_{\al\si},\nab^\al\bH\>+\nab^\al\bH\<\Lambda_{\al\si},\bH\>)\\
        &\ +\Lambda^{\al\be}\<\nab_\al\bH,\nab_\be\bH\>-\frac{1}{2}\bH|\nab\bH|^2 \> \ dvol.
    \end{align*}
    Hence, the Euler-Lagrange operator $W(F)$ is obtained.
\end{proof}

The Willmore flow is the gradient flow of the Willmore functional.
Given the form of the Euler-Lagrange operator of $\mathcal W(F)$ in Lemma~\ref{lem7.2}, we obtain the Willmore-type flow, where the map $F(s,x):[0,T]\times \R^d \rightarrow \R^{d+2}$ evolves via the evolution equation
\begin{equation}   \label{WillFlow}
\left\{\begin{aligned}
   &(\d_s F)^{\perp}=(\De^\perp)^2 \mathbf H+\Lambda^{\al\be}\<\Lambda_{\al\be},\De^\perp \bH\>-\Lambda^{\al\be}\<\nab^\perp_\al \bH,\nab^\perp_\be \bH\>+\frac{1}{2}\mathbf H |\nab^\perp\bH|^2\\
   &\qquad\ \ \    +\nab^\perp_\si(\bH\<\Lambda^{\al\si},\nab^\perp_\al\bH\>+\nab^\perp_\al\bH\<\Lambda^{\al\si},\bH\>)\,,\\
   &F(s,\cdot)\big|_{s=0}=F_0\,.
\end{aligned}\right.
\end{equation}
This is a quasilinear sixth order evolution equation of parabolic type, in a suitable gauge. The manifold $\Sigma_0$ is regularized by evolving along the above Willmore-type flow, for which all we need is local solvability. 

Similar to the mean curvature flow, it is easy to check that the system \eqref{WillFlow} is a degenerate parabolic system. 
To bypass this difficulty we can adapt the DeTurck trick as introduced by Hamilton \cite{Hamilton}.
The DeTurck trick is nothing but  
gauge fixing for the group of time dependent changes of coordinates.
In practice this involves adding a tangential term to the geometric flow in order to break the geometric invariance of the equation. 
The modified flow is then strongly parabolic and the almost standard parabolic theory can now be employed in order to insure the short time existence of solutions for \eqref{WillFlow}.

Modifying the flow by adding a tangential term, we obtain a Willmore-DeTurck type flow,
\begin{equation}   \label{WillFlow-Mf}
\left\{\begin{aligned}
&\d_s F=U^\ga \d_\ga F+(\De^\perp)^2 \mathbf H+\Lambda^{\al\be}\<\Lambda_{\al\be},\De^\perp \bH\>-\Lambda^{\al\be}\<\nab^\perp_\al \bH,\nab^\perp_\be \bH\>\\
&\qquad\ \ \    +\frac{1}{2}\mathbf H |\nab^\perp\bH|^2
+\nab^\perp_\si(\bH\<\Lambda^{\al\si},\nab^\perp_\al\bH\>+\nab^\perp_\al\bH\<\Lambda^{\al\si},\bH\>)\,,\\
&F(s,\cdot)\big|_{s=0}=F_0\,.
\end{aligned}\right.
\end{equation}
Our choice of the field $U^{\gamma}$  corresponds to introducing  \emph{generalized parabolic coordinates}, where we require the coordinate functions $x^\gamma$ to be global Lipschitz solutions of the heat equations
\[
(\partial_t-\Delta_g^3  - U^\si \partial_\si) x^\gamma=0.
\]
Then, for fixed $\gamma$, the functions $U^\ga$ are given by
\begin{equation}\label{deTurck-gauge}
U^\ga = \Delta^2 (g^{\al\be}\Ga_{\al\be}^\ga).
\end{equation}
Now we consider the local well-posedness question for 
the Willmore-DeTurck flow \eqref{WillFlow-Mf} with 
the gauge choice \eqref{deTurck-gauge}.

\begin{thm} 
Consider a smooth initial immersion $F_0:\R^d \rightarrow \R^{d+2}$ with second fundamental form $\La_0$, metric $g_0$ and volume of balls satisfying
\eqref{RegData}, \eqref{GeoCond} and \eqref{good-param}.
Then there exists $T>0$ depending only on $M$, $C_0$, $v$, $c$ and $C$ such that \eqref{WillFlow-Mf} with the gauge choice \eqref{deTurck-gauge} has a unique smooth solution $F$ in $[0,T]$ satisfying
\[\|\d^2 F\|_{L^\infty H^k_{uloc}}+\|\d^2 F\|_{L^2 H^{k+3}_{uloc}}\leq \|\d^2 F_0\|_{H^k_{uloc}}.\]
Moreover, the solution $F_\ep:=F(\ep^{3/2})$ satisfies the regularization estimate
\begin{align} \label{reg-Fep}
    \|\d^j \d F_\ep\|_{H^{k+1}_{uloc}}\les \ep^{-\frac{j}{4}}\|\d F_0\|_{H^{k+1}_{uloc}},
\end{align}
and there exists a normal frame $m_\ep$ which has the same regularity
\begin{align}  \label{reg-mep}
    \| \partial^j m_\epsilon \|_{H^{k+1}_{uloc}} \lesssim \epsilon^{-\frac{j}4} \| \partial F_0\|_{H^{k+1}_{uloc}}.
\end{align}
\end{thm}

\begin{proof}
Defining $G=F-F_0$, the function $G$ solves
\begin{equation}   \label{WillFlow-MfG}
\left\{\begin{aligned}
&\d_s G=\De_g^2 (\d_\al g^{\al\be}\d_\be )G+\sum_{j_1+\cdots+j_s\leq 5,\ 2\leq s\leq 5}(g^{-1})^5 (\d+\Ga)^{j_1}\d F\cdots (\d+\Ga)^{j_s}\d F\cdot (\d F)^2\\
&\quad\quad\ \   +\De_g^2 (\d_\al g^{\al\be}\d_\be )F_0\,,\\
&G(s,\cdot)\big|_{s=0}=0\,,
\end{aligned}\right.
\end{equation}
where the principal symbol of the leading term $-\De_g^2 (\d_\al g^{\al\be}\d_\be )$ is $(g^{\al\be}\xi_\al \xi_\be)^3$  and satisfies the property
\begin{align*}
    (g^{\al\be}\xi_\al \xi_\be)^3\geq C|\xi|^6.
\end{align*}
Hence the equation \eqref{WillFlow-MfG} is a sixth-order nondegenerate parabolic equation with lower order source terms.

Now we solve the equation \eqref{WillFlow-Mf} in the uniform local space $H^{k+2}_{uloc}$.
Using standard arguments involving Friedrichs smoothing
techniques and a bootstrap assumption 
\begin{align*}
    \|G\|_{L^\infty H^{k-3}_{uloc}}+\|G\|_{L^2 H^{k}_{uloc}}\leq C\epsilon,
\end{align*}
we deduce that for $\d F_0\in H^{k+1}_{uloc}$ with $k>\frac{d}{2}+5$, there exists a unique solution $G\in C([0,T_k],H^{k-3}_{uloc})\cap L^2([0,T_k],H^{k}_{uloc})$ for some sufficiently small $T_k>0$. From \eqref{WillFlow-Mf}, we can further improve the regularity of the map $F$ to
\begin{align*}
    \|\d F\|_{L^\infty H^{k+1}_{uloc}}+\|\d F\|_{L^2 H^{k+4}_{uloc}}\les \|\d F_0\|_{H^{k+1}_{uloc}}.
\end{align*}
\end{proof}

\begin{proof}[Proof of \eqref{reg-Fep}]
    Applying $s^j \d^{6j+1}$ for $j\geq 1$ to the equation of $F$, we get
    \begin{align*}  
    &\ (\d_s-(\d_\al g^{\al\be}\d_\be)^3)(s^j \d^{6j+1})F\\
    &=js^{j-1}\d^{6j+1}F-s^j[(\d_\al g^{\al\be}\d_\be)^3,\d^{6j+1}]F\\
    &+s^j \d^{6j+1}   \sum_{j_1+\cdots+j_s\leq 5,\ 2\leq s\leq 5}(g^{-1})^5 (\d+\Ga)^{j_1}\d F\cdots (\d+\Ga)^{j_s}\d F\cdot (\d F)^2  .
    \end{align*}
    Then by the partition of unity, energy estimates and interpolation inequality, we obtain 
    \begin{align*}
        \|s^j \d^{6j}\d F\|_{H^{k+1}_{uloc}}\les \|\d F_0\|_{H^{k+1}_{uloc}},
    \end{align*}
    which also implies for $s=\ep^{\frac{3}{2}}$
    \begin{align*}
        \|\d^{6j}\d F_\ep\|_{H^{k+1}_{uloc}}\les \ep^{-\frac{3j}{2}}\|\d F_0\|_{H^{k+1}_{uloc}}.
    \end{align*}
    Hence, the bound \eqref{reg-Fep} is obtained by interpolation.
\end{proof}

\begin{proof}[Proof of \eqref{reg-mep}]
Both $m_\epsilon$ and $\Lambda_\epsilon$ depend on the choice of the normal frame, but their product does not. So at the point where we obtain the $H^{k+1}_{uloc}$ regularity of $\partial F_\epsilon$ in local charts,
we should also point out that we can choose $m_\epsilon$ with the same regularity. 

Here we can construct the $m_\ep$ directly by $\d F_\ep$ as the graph case. For example, in dimensions $d=2$, we have $F=(F_1,\cdots,F_4):\R^2\rightarrow \R^4$ and
\[  \d_1 F=(\d_1 F_1,\cdots, \d_1 F_4),\qquad \d_2 F=(\d_2 F_1,\cdots, \d_2 F_4)    \]
Without loss of generality, denote $F'=(F_1,F_2,F_3)$ such that $\d_1 F'$ and $\d_2 F'$ are linearly independent. Using their cross product, we get
\[  \d_1 F'\times \d_2 F'  \perp \d_1 F',\ \d_2 F' \,,  \]
and obtain a normal vector
\[  ( \d_1 F'\times \d_2 F',0)  \perp \d_1 F,\ \d_2 F    \]
Thus one of the unit normal vectors is given by
\[ \nu_1:= \frac{( \d_1 F'\times \d_2 F',0)}{|\d_1 F'\times \d_2 F'|}\]
The other one unit normal vector $\nu_2$ can also be constructed using the generalized cross product. For general dimensions $d\geq 2$, we can construct them by generalized cross product directly.
Hence, the normal frame $m_\ep$ constructed as above has the same regularity and satisfies the estimate
\begin{align*}
    \|\d^j m_\ep\|_{H^{k+1}_{uloc}}\les \|\d^j F_\ep\|_{H^{k+1}_{uloc}}\les \ep^{-\frac{j}{4}}\|\d F_0\|_{H^{k+1}_{uloc}}.
\end{align*}
\end{proof}

The regularized manifold $\Sigma_\ep=F(\ep^{\frac 3 2},\R^d)$ is chosen by setting the Willmore time to be $s=\ep^{\frac 32}$. This time 
scale corresponds to a regularization on the $\epsilon^{-\frac14}$ spatial scale.

As in Section \ref{Sec-gauge} we define the complex orthonormal frame $m$, complex second fundamental form $\la$ and mean curvature $\psi$ as 
\[
m=\nu_1+i\nu_2,\qquad \la_{\al\be}=\Lambda_{\al\be}\cdot\nu_1+i\Lambda_{\al\be}\cdot\nu_2,\qquad \psi=g^{\al\be}\la_{\al\be},
\] 
with the same  gauge group given by the sections of an $SU(1)$ bundle.
Then we can do the following steps to rewrite the Willmore-type flow \eqref{WillFlow} in terms of these geometric parameters, in several steps:

\medskip

\emph{a) Rewrite the equation for the Willmore flow.} First, we derive the differential equations for second fundamental form. Since
\begin{align*}
    (\De^\perp )^2 \bH&=\Re((\De^A_g)^2\psi\bar{m}),\\
    \Lambda^{\al\be}\<\Lambda_{\al\be},\De^\perp \bH\>&=\Re(\la^{\al\be}\bar{m})\Re(\la_{\al\be}\overline{\De^A_g \psi}),\\
    -\Lambda^{\al\be}\<\nab^\perp_\al \bH,\nab^\perp_\be \bH\>&=-\Re(\la^{\al\be}\bar{m})\Re(\nab^A_\al\psi \overline{\nab^A_\be\psi}),\\
    \frac{1}{2}\bH|\nab^\perp\bH|^2&=\frac{1}{2}\Re(\psi\bar{m})|\nab^A\psi|^2,
\end{align*}
and 
\begin{align*}
    &\ \nab^\perp_\si(\bH\<\Lambda^{\al\si},\nab^\perp_\al\bH\>+\nab^\perp_\al\bH\<\Lambda^{\al\si},\bH\>)\\
    =&\ \Re(\nab^A_\si\psi\bar{m})\big(2\Re(\la^{\al\si}\overline{\nab^A_\al\psi})+\frac{1}{2}\nab^\si |\psi|^2 \big)+\Re(\psi\bar{m})\big(|\nab^A\psi|^2+\Re(\la^{\al\si}\overline{\nab^A_\si\nab^A_\al\psi})\big)\\
    &\ +\Re(\nab^A_\si\nab^A_\al\psi\bar{m})\Re(\la^{\al\si}\bar{\psi}).
\end{align*}
Then we obtain
\begin{equation} \label{dF}
    \d_s F=\Re (\mathcal L \bar{m})
\end{equation}
with $\mathcal L$ given by
\begin{equation*}
    \begin{aligned}
    \mathcal L=&\ (\De^A_g)^2\psi+\la^{\al\be}\Re(\la_{\al\be}\overline{\De^A_g \psi})-\la^{\al\be}\Re(\nab^A_\al\psi \overline{\nab^A_\be\psi})+\frac{3}{2}\psi|\nab^A\psi|^2\\
    &\ +\nab^A_\si\psi\big(2\Re(\la^{\al\si}\overline{\nab^A_\al\psi})+\frac{1}{2}\nab^\si |\psi|^2 \big)+\psi\Re(\la^{\al\si}\overline{\nab^A_\si\nab^A_\al\psi})+\nab^A_\si\nab^A_\al\psi\Re(\la^{\al\si}\bar{\psi}).
\end{aligned}
\end{equation*}

\medskip

\emph{b) The motion of the frame and commutators.}
Applying $\d_\al$ to formula \eqref{dF} and using the relation $m\perp \d_\al F$, we get
\begin{equation}   \label{Mot-fra}
	\left\{\begin{aligned}
		&\d_s F_\al 
		=\Re(\nab^A_\al \mathcal L \bar{m})-\Re(\mathcal L \overline{\la_\al^\ga })F_\ga\,,\\
		&\d_s^{A_0} m=-\nab^{A,\al} \mathcal L F_\al\,.
	\end{aligned}\right.
\end{equation}
This also gives the evolution equation of metric $g$
\begin{align}   \label{met-g}
    \d_s g_{\al\be}&= \d_s \<\d_\al F,\d_\be F\>=-2\Re(\mathcal L\overline{\la_{\al\be}}).
\end{align}

From the structure equations \eqref{strsys-cpf} and \eqref{Mot-fra}, we have
\begin{align*}
\d^{A_0}_s\d^A_{\al} m&=(-\d^{A_0}_s\la^{\ga}_{\al}+\la^{\si}_{\al}\Re(\mathcal L\bar{\la}^{\ga}_{\si}))F_{\ga}-\la^{\ga}_{\al} \Re(\nab^A_\ga \mathcal L \bar{m})\,,\\
\d^A_{\al}\d^{A_0}_s m&=-\nab^A_\al \nab^{A,\ga}\mathcal L F_\ga-\nab^{A,\ga}\mathcal L \Re(\la_{\al\ga}\bar{m})\,.
\end{align*}
By the above two formulas and the commutator
\begin{align*}
[\d_s^{A_0},\d^A_\al] m=i(\d_s A_\al-\d_\al A_0)m,
\end{align*}
equating the coefficients of the tangent vectors and
the normal vector $m$, we obtain the evolution equation for $\la$
\begin{equation}   \label{Evol-la}
\d^{A_0}_s\la^{\ga}_{\al}-\nab^A_{\al}\nab^{A,\ga}\mathcal L=\la^{\si}_{\al}\Re(\mathcal L\bar{\la}^{\ga}_{\si})
\end{equation}
and the compatibility conditions for the connection
\begin{align}    \label{conn-A}
    \d_s A_\al-\d_\al A_0=\Im\big(\overline{\la_\al^\ga} \nab^A_\ga\mathcal L\big),
\end{align}

In order to dynamically fix the gauge on the normal bundle along the Willmore-type flow, we will use the parallel transport relation $A_0=\d_s \nu_1\cdot\nu_2=0$, sometimes called the temporal gauge, which yields the main gauge condition
\begin{equation*}
    A_0=0.
\end{equation*}
Then we have the commutators
\begin{align*}
	[\d_s,\nab^A_\al]=[\d_s,\nab_\al]+i[\d_s,A_\al]=\nab^A\la\ast\mathcal L+\la\ast\nab^A\mathcal L.
\end{align*}

\medskip

\emph{c) The evolution equations of $\la$.} Using the compatibility conditions from  \eqref{la-commu} we have
\begin{align*}
    \nab^A_{\al}\nab^{A,\ga} \psi= \nab^A_\al \nab^{A,\si}\la_\si^\ga=\nab^{A,\si}\nab^A_\si \la_\al^\ga+[\nab^A_\al,\nab^{A,\si}]\la_\si^\ga,
\end{align*}
and 
\begin{align*}
    [\nab^A_\al,\nab^{A,\si}]\la_\si^\ga
    =&\ -\Re(\la_{\al\de}\Bar{\psi}-\la_{\al\mu}\Bar{\la}_{\de}^\mu)\la^{\de\ga}+\Re(\la_{\si\de}\bar{\la}_{\al}^\ga-\la_{\si}^\ga\bar{\la}_{\al\de})\la^{\si\de}+i\Im (\la_\al^\mu \overline{\la_{\si\mu}})\la^{\si\ga}.
\end{align*}
Then, under the gauge condition $A_0=0$, the evolution equations \eqref{Evol-la} for $\la$ are rewritten as 
\begin{equation}        \label{Eq-MCF}
\begin{aligned}
\d_s \la_\al^\ga-(\De^{A}_g)^3\la_\al^\ga= -(\De^A_g)^3\la_\al^\ga+\nab^A_{\al}\nab^{A,\ga}\mathcal L+\la^{\si}_{\al}\Re(\mathcal L\bar{\la}^{\ga}_{\si}):=\tilde{\mathcal L}\,,
\end{aligned}
\end{equation}
where the  nonlinearity $\tilde{\mathcal L}$ has the schematic form 
\begin{equation*}
    \tilde{\mathcal L}=\sum_{k_1+k_2+k_3=4} \nab^{A,k_1}\la\ast \nab^{A,k_2}\la\ast \nab^{A,k_3}\la +\la^4\ast \nab^{A,2}\la+\la^3\ast \nab^A\la\ast \nab^A\la,
\end{equation*} 
and the connection $A$ and the metric $g$ satisfy \eqref{conn-A} and \eqref{met-g}, respectively.

\medskip

Now we turn our attention to the regularized manifold. As the submanifold $\Sigma$ evolves along the Willmore-type flow \eqref{WillFlow}, the desired regularized manifold $\Sigma_\ep$ is obtained at the Willmore time $s=\ep^{3/2}$:
\begin{equation}   \label{F-ep}
\Sigma_\ep:=\Sigma(s)\big|_{s=\ep^{3/2}}=F_\ep(\R^d):=F(s,\R^d)\big|_{s=\ep^{3/2}}
\end{equation}
We use the following notations to denote the metric, Christoffel symbols, normal vectors, and second fundamental form on $\Sigma_\ep$,
\begin{align}     \label{frame-reg}
	g_\ep\,,\quad \Ga_{\ep;\al\be}^\ga\,,\quad (\nu^\ep_1,\ \nu^\ep_2)\,,\quad m_\ep=\nu^\ep_1+i\nu^\ep_2\,,\quad \Lambda_\ep\, ,\quad \la_\ep=\Lambda_\ep \cdot m_\ep\,.
\end{align}
Then compared with the initial manifold $\Sigma_0$, we have the following properties.

\begin{prop}[Bounds for the regularized submanifold $\Sigma_\ep$]\label{Reg-Prop}
Let $(\Sigma_0,g(0))$ be a complete Riemannian manifold of dimension $d$ satisfying the assumptions \eqref{RegData}, \eqref{GeoCond}, \eqref{good-param}, with the initial choice of coordinates as in \eqref{better-data}. We regularize the initial manifold $\Sigma_0$ as $\Sigma_\ep$ in \eqref{F-ep}.
Denote the second fundamental forms of $\Sigma_0$ and $\Sigma_\ep$ as $\La_0$ and $\La_\ep$, respectively. Then we have the following properties:
    
(a) Ricci curvature bound and volume of balls:
\begin{gather}   \label{RegEst-Ricci}
	|\Ric_\ep|\leq (1+C(M)\ep)C_0\,,\qquad \inf_{x\in \Sigma_\ep} {\rm Vol}_{g_\ep}(B_x(e^{C(M)\ep}r_0))\geq e^{-C(M)\ep}v\,,\\   \label{RegEst-g}
	(1-C(M)\ep^{3/2})g_0\leq g_{\ep}\leq (1+C(M)\ep^{3/2})g_0\,.
\end{gather}

(b) Energy bound
\begin{align}    \label{RegEst-IM}
	\|\Lambda_\ep\|^2_{\mathsf H^k}\leq (1+C(M)\ep^{3/2})\|\Lambda_0\|^2_{\mathsf H^k}\,.
\end{align}
    
(c) Regularization:
\begin{align}    \label{RegEs-HiReg}
\|\Lambda_\ep\|_{\mathsf H^{k+m}}\les \ep^{-m/4}\|\Lambda_0\|_{\mathsf H^k}\,.
\end{align}
    
(d) Approximate solution:
\begin{align}   \label{RegEst-AppSol}
    \|\Lambda_\ep-\Lambda_0\|_{L^2}\les \ep^{3/2}\,,\qquad 
    \|\d F_\ep-\d F_0\|_{{L^\infty_{uloc}}}\leq \ep^{3/2}.
\end{align}
\end{prop}

The rest of this subsection is devoted to the proof of the above theorem. 
\medskip

 We remark that, while parts (a,b,c) are covariant, the last part (d) depends on using the flow induced coordinates on $\Sigma_\epsilon$, which in turn depends 
 on the choice of the initial coordinates. Here we assume the improved regularity for the initial coordinates as in \eqref{better-data}, and as a consequence we also obtain the $\partial F_\epsilon$ regularity in the same local charts:
\begin{align*}
    \|\partial^j\d F_\ep\|_{H^{k+1}_{uloc}}\les \ep^{-\frac{j}4}.
\end{align*}
This is important as we will also use the same coordinates for the Euler step. However, we carefully 
note that we will not directly propagate this higher
regularity across iteration steps; instead, we reinitialize the coordinates to satisfy \eqref{better-data} at the beginning of each step. 

\

First we verify the conditions in \eqref{RegEst-Ricci} about Sobolev embeddings on the time interval $[0,\ep^{3/2}]$. These are proved using bootstrap argument and energy estimates.

\begin{proof}[Proof of \eqref{RegEst-Ricci}]

To prove the bound~\eqref{RegEst-Ricci} it is convenient to make the following bootstrap assumptions on the time interval $J=[0,T]\subset [0,\ep^{3/2}]$, 
\[|\Ric_s(X,X)|\leq (1+C(M)\ep) C_0|X|_{g_s}^2,\quad \inf_{x\in \Sigma_s} {\rm Vol}_{g_s}(B_x(r_0 e^{C(M)s}))\geq e^{-C(M)s}v.\] 
By \eqref{VolComp}, we still can bound the volume of balls
\begin{align*}
    {\rm Vol}_{g_s}(B_x(r_0))&\geq e^{-\sqrt{(d-1)2C_0}r_0 e^{C(M)s}} e^{-C(M)sd} {\rm Vol}_{g_s}(B_x(r_0 e^{C(M)s}))\\
    &\geq \exp\big\{-\sqrt{(d-1)2C_0}r_0 e^{C(M)s}-C(M)s(d+1)\big\} v.
\end{align*}
Hence, the Sobolev embeddings on $\Sigma_s$ still hold. This can be used to prove the energy estimates for $\la$, and then we close the bootstrap argument.
\medskip

\emph{(i)  Energy estimates for $\lambda$.} We claim that
\begin{equation}    \label{pre-1}
        \|\la\|_{L^\infty_s (J;\mathsf H^k)}+\|\la\|_{L^2_s (J;\mathsf H^{k+3})}\les \|\la_0\|_{\mathsf H^k}.
\end{equation}
Applying $\frac{d}{ds}$ to $\|\la\|_{\mathsf H^k}^2$, we have
\begin{align*}
&\ \frac{1}{2}\frac{d}{ds}\|\la\|_{\mathsf H^k}^2=
\frac{1}{2}\frac{d}{ds}\int |\nab^{A,k}\la|_g^2\ d\mu\\
&= \int_{\mathbb R^d} \Re \< \d_s\nabla^{A,k}\la, \overline{\nabla^{A,k}\la}\>_g + (\nabla_s g) (\nabla^{A,k}{\la}, \overline{\nabla^{A,k}\la}) +|\nab^{A,k}\la|^2 g^{\al\be}\d_s g_{\al\be}\ d\mu\\
&\les \int_{\mathbb R^d} \Re \< \d_s\nabla^{A,k}\la, \overline{\nabla^{A,k}\la}\>_g + \mathcal L \ast \la \ast \nabla^{A,k}{\la} \ast \overline{\nabla^{A,k}\la}.
\end{align*}
So by \eqref{Eq-MCF}, the  first term in the right-hand side reduces to
\begin{align*}
&\ \int_{\mathbb R^d} \Re g( \d_s\nabla^{A,k}\la, \overline{\nabla^{A,k}\la})\ d\mu\nonumber\\
=&\ \int_{\mathbb R^d}  \Re g((\De^A_g)^3 \nab^{A,k} \la, \overline{\nabla^{A,k}\la})+ \Re g((\partial_s-(\De^A_g)^3)\nabla^{A,k} \la, \overline{\nabla^{A,k}\la})\ d\mu\\
=&\ \int_{\mathbb R^d} - |\nab^{A,k+3}\la|^2+\Re g(\nabla^{A,k}\tilde{\mathcal L},\overline{\nab^{A,k}\la})+\Re g([\partial_s-(\De^A_g)^3,\nabla^{A,k}] \la, \overline{\nabla^{A,k}\la})\ d\mu\\
\leq &\ -\int_{\mathbb R^d} |\nab^{A,k+3}\la|^2 d\mu+\int \Re g(\nabla^{A,k}\tilde{\mathcal L},\overline{\nab^{A,k}\la})d\mu\\
&\ +\sum_{k_1+k_2+k_3=k+4}\int |\nab^{A,k_1}\la\ast \nab^{A,k_2}\la \ast \nab^{A,k_3}\la\ast \nab^{A,k}\la|\ d\mu\\
&\ +\sum_{k_1+\cdots+k_5=k+2}\int |\nab^{A,k_1}\la\ast \cdots\ast \nab^{A,k_5}\la\ast \nab^{A,k}\la|\ d\mu\,.
\end{align*}
We bound the worst term by
\begin{align*}
    &\ \int (\nab^A)^{k+4}\la\ast \la^2 \ast (\nab^A)^k \la + \nab^{A,k+3}\la\ast \nab^A\la\ast\la\ast \nab^{A,k}\la d\mu\\
    =&\ \int (\nab^A)^{k+3}\la\ast \big(\nab^A \la\ast \la \ast (\nab^A)^k \la+\la^2\ast (\nab^A)^{k+1}\la\big) d\mu\\
    \les &\ \|\nab^{A,k+3}\la\|_{L^2}\big( \|\nab^A\la\|_{L^{2(k+1)}}\|\la\|_{L^\infty}\|\nab^{A,k}\la\|_{L^{\frac{2(k+1)}{k}}}+\|\la\|_{L^\infty}^2\|\la\|_{H^k}^{2/3}\|\nab^{A,3}\la\|_{H^k}^{1/3}   \big)\\
    \les &\ \|\nab^{A,k+3}\la\|_{L^2}\big( \|\la\|_{L^\infty}^{\frac{k}{k+1}}\|\nab^{A,k+1}\la\|_{L^2}^{\frac{1}{k+1}}\|\la\|_{L^\infty}\|\la\|_{L^\infty}^{\frac{1}{k+1}}\|\nab^{A,k+1}\la\|_{L^2}^{\frac{k}{k+1}}\\
    &\ +\|\la\|_{L^\infty}^2\|\la\|_{H^k}^{2/3}\|\nab^{A,3}\la\|_{H^k}^{1/3}   \big)\\
    \leq &\ \de \|\nab^{A,3}\la\|_{H^k}^2+C_\de \|\la\|_{L^\infty}^6\|\la\|_{H^k}^2\,.
\end{align*}
For the other terms, by interpolation inequalities \eqref{HaInterpo} we have
\begin{align*}
	&\ \sum_{k_1+k_2+k_3=k+4;\max k_i\leq k+2}\int |\nab^{A,k_1}\la\ast \nab^{A,k_2}\la \ast \nab^{A,k_3}\la\ast \nab^{A,k}\la|\ d\mu\\
	\les &\ \|\nab^{A,k_1}\la\|_{L^{\frac{2(k+2)}{k_1}}}\|\nab^{A,k_2}\la\|_{L^{\frac{2(k+2)}{k_2}}}\|\nab^{A,k_3}\la\|_{L^{\frac{2(k+2)}{k_3}}}\|\nab^{A,k}\la\|_{L^{\frac{2(k+2)}{k}}}\\
	\les &\ \|\la\|_{L^\infty}^{2}\|\nab^{A,k+2}\la\|_{L^2}^2\\
 	\les &\ \|\la\|_{L^\infty}^{2}\|\nab^{A,k}\la\|_{L^2}^{2/3}\|\nab^{A,k+3}\la\|_{L^2}^{4/3}\\
	\leq &\ \de \|\nab^{A,3}\la\|_{H^k}^2+C_\de \|\la\|_{L^\infty}^{6}\|\la\|_{H^k}^2\,.
\end{align*}
and 
\begin{align*}
\sum_{k_1+\cdots+k_5=k+2}&\int |\nab^{A,k_1}\la\ast \cdots\ast \nab^{A,k_5}\la\ast \nab^{A,k}\la|\ d\mu
\les \prod_{j=1}^5 \|\nab^{A,k_j}\la\|_{L^{\frac{2(k+2)}{k_j}}}\|\la\|_{H^k}\\
&\les  \|\la\|_{L^\infty}^4 \|\nab^{A,k+2}\la\|_{L^2}\|\la\|_{H^k}
\les\|\la\|_{L^\infty}^4\|\nab^{A,k+3}\la\|_{L^2}^{2/3} \|\la\|_{H^k}^{4/3}\\
&\leq\de\|\nab^{A,k+3}\la\|_{L^2}^2+C_\de\|\la\|_{L^\infty}^6 \|\la\|_{H^k}^2\,.
\end{align*}
Hence, by Sobolev embedding we obtain the energy estimates
\begin{align}  \label{Eng-Willmore}
    \frac{1}{2}\frac{d}{ds}\|\la\|_{\mathsf H^k}^2+\|\nab^{A,3}\la\|_{\mathsf{H}^k}^2
    \leq &\ C_\de \|\la\|_{L^\infty}^6\|\la\|_{H^k}^2\les \|\la\|_{\mathsf H^{k_0}}^6 \|\la\|_{\mathsf H^k}^2.
\end{align}
Then we obtain the energy bound \eqref{pre-1}.

\medskip

\emph{(ii) Prove the improved bound for Ricci curvature.} For any $X$ and $s\in[0,\ep^{3/2}]$, we have
\begin{align}   \label{dsg}
	\Big|\frac{d}{ds}|X|_{g_s}^2\Big|=|\d_s g_{s,\al\be}X^\al X^\be|=|2\Re(\mathcal L \bar{\la}_{\al\be})X^\al X^\be|\leq C\|\la\|_{\sfH^{k_0+4}}^4 |X|^2_{g_s}.
\end{align}
This implies the equivalence $e^{-C(M)\ep^{3/2}} |X|_{g(0)}^2\leq |X|_{g(s)}^2\leq e^{C(M)\ep^{3/2}} |X|_{g(0)}^2$. Then we obtain
\begin{align*}
    |\Ric_{\al\be}(t)X^\al X^\be|&\leq |\Ric_{\al\be}(0)X^\al X^\be|+|\int_0^t \d_s \Ric_{\al\be}(s) X^\al X^\be ds|\\
    &\leq C_0 |X|_{g(0)}^2+t \|\la\|_{L^\infty} \|\d_s\la\|_{L^\infty} |X|_{g(s)}^2\\
    &\leq e^{C(M)\ep^{3/2}} C_0 |X|_{g(t)}^2+t C|X|_{g(t)}^2\\
    &\leq (1+\frac{C(M)}{2}\ep^{3/2})C_0|X|_{g(t)}^2.
\end{align*}

\emph{(iii) Prove the improved bound for the volume of balls.}

To bound the volume of a ball from below, we begin with the following two claims: 
\begin{align}  \label{vol}
    &e^{-sC(M)}dvol_{g(0)} \leq dvol_{g(s)}  \leq e^{sC(M)}dvol_{g(0)},\\ \label{balls}
    &B_x(r_0,\Sigma_0)\subset B_x(r_0 e^{C(M)s},\Sigma_s).
\end{align}

For the first claim \eqref{vol}, from Sobolev embedding and energy estimates, we know that
\begin{align*}
    |\d_s \sqrt{\det g}|=|-2\Re(\mathcal L\bar{\psi}) \sqrt{\det g}| \leq C(M) \sqrt{\det g},
\end{align*}
which implies that
\begin{align*}
    \sqrt{\det g(0)} e^{-sC(M)}\leq \sqrt{\det g(s)}\leq \sqrt{\det g(0)} e^{sC(M)}.
\end{align*}
Hence, by the volume form $dvol_{g(s)}=\sqrt{\det g(s)} dx$ we obtain \eqref{vol}.

We then prove the second claim \eqref{balls}. For any two points $x$ and $y$ in $\Sigma_0$, there exists a geodesic $\gamma:[0,1]\rightarrow \Sigma_0$ such that $\ga(0)=x,\ \ga(1)=y$. Then
\begin{align*}
    d(x,y)=l(\ga)=\int_0^1 |\dot{\gamma}(\tau)| d\tau=\int_0^1 \Big(g_{\al\be}\frac{\d \ga_\al}{\d\tau} \frac{\d \ga_\be}{\d\tau}\Big)^{1/2} d\tau.
\end{align*}
Since the metric $g_{\al\be}$ evolves along the mean curvature flow, then length of curve $\ga$ also change. Hence we have
\begin{align*}
    \frac{d}{ds}l(\ga,s)=\int_0^1 \frac{1}{2|\dot\ga|}\Big(\d_s g_{\al\be}\frac{\d \ga_\al}{\d\tau} \frac{\d \ga_\be}{\d\tau}\Big) d\tau=-\int_0^1 \frac{1}{|\dot\ga|}\Big(\Re(\mathcal{L}\bar{\la}_{\al\be})\frac{\d \ga_\al}{\d\tau} \frac{\d \ga_\be}{\d\tau}\Big) d\tau.
\end{align*}
which yields
\begin{equation*}
    \big|\frac{d}{ds}l(\ga,s) \big|\leq \|\mathcal L\|_{L^\infty}\|\la\|_{L^\infty}\int_0^1 |\dot \ga|^{-1} |\dot\ga|^2 d\tau \leq C(M) l(\ga).
\end{equation*}
Hence, we obtain that the distance between $x$ and $y$ at $s\in[0,\ep^{3/2}]$ have the bound
\begin{equation*}
    d_s(x,y)\leq l(\ga,s) \leq l(\ga,0) e^{C(M)s}=d_0(x,y) e^{C(M)s},
\end{equation*}
which implies the claim \eqref{balls}.

With the above two claims in hand,  we  obtain
\begin{equation*}
\begin{aligned}
    &{\rm Vol}_{g(s)}(B_x(r_0 e^{C(M)s}))= \int_{B_x(e^{C(M)s},s)} 1\  dvol_{g(s)}\geq \int_{B_x(r_0,0)}   e^{-sC(M)}dvol_{g(0)}\\
    =&\    e^{-sC(M)} {\rm Vol}_{g(0)}(B_x(r_0,0))
    \geq  e^{-sC(M)} v,
\end{aligned}
\end{equation*}
which for $s=\ep^{3/2}$ gives 
\begin{align*}
    {\rm Vol}_{g(\ep^{3/2})}(B_x(r_0 ))\geq (1- C(M)\ep^{3/2})v.
\end{align*}

Therefore, the Ricci curvature and volume of ball admit the improved bounds. This closes the bootstap argument, and hence the bounds in \eqref{RegEst-Ricci} are obtained.
\end{proof}

\begin{proof}[Proof of \eqref{RegEst-IM} and \eqref{RegEst-g}]
From \eqref{Eng-Willmore} we have
\begin{align*}
   \frac{d}{ds}\|\la\|_{\sfH^k}\leq C\|\la\|_{L^\infty}^6\|\la\|_{\sfH^k}\leq C(M) \|\la_0\|_{\sfH^k}.
\end{align*}
Integrating over the time interval $J=[0,\ep^{3/2}]$, we obtain that for any $s\in J$
\begin{align*}
    \|\la(\ep^{3/2})\|_{\mathsf{H}^k}\leq  (1+C(M)\ep^{3/2})\|\la(0)\|_{\mathsf{H}^k},
\end{align*}
where $\ep>0$ depending on initial data $\la(0)$ is sufficiently small. Moreover, 
the estimate \eqref{dsg} implies that
\begin{align*}
e^{-CM^4 s}|X|_{g(0)}^2\leq |X|_{g_s}^2\leq e^{CM^4 s}|X|_{g(0)}^2.
\end{align*}
Thus the bound \eqref{RegEst-g} is obtained when $s=\ep^{3/2}$.
\end{proof}

\begin{proof}[Proof of \eqref{RegEs-HiReg}]
First, we prove that for $j \geq 0$ we have the estimate
\begin{equation} \label{pre-2}
\|s^j\nab^{A,6j}\la\|_{L^\infty_s (J;\sfH^k)}+\|s^j\nab^{A,6j}\la\|_{L^2_s (J;\sfH^{k+3})}
\les_j \|\la_0\|_{\sfH^k},
\end{equation}  
which for $j=0$ is nothing but \eqref{pre-1}.
To prove \eqref{pre-2} for $j  \geq 1$, we need the commutator
\begin{align*}
&\ [\d_s,(-\De^A_g)^{3j}]\la= \sum_{k=0}^{6j-1} (\nab^A_g)^k[\d_s,\nab^A_g](\nab^A_g)^{6j-1-k}\la\\
=&\ \sum_{k=0}^{6j-1} (\nab^A_g)^k\big((\nab^A\la\ast \mathcal L+\la\ast \nab^A \mathcal L)\ast (\nab^A_g)^{6j-1-k}\la\big)\\
=&\ \la^2\ast (\nab^A)^{6j+4}\la+\sum_{k_1+k_2+k_3=6j+4;k_i\leq 6j+3} (\nab^A_g)^{k_1} \la \ast (\nab^A)^{k_2}\la\ast (\nab^A)^{k_3}\la\\
&\ +\sum_{k_1+\cdots+k_5=6j+2} (\nab^A_g)^{k_1} \la \ast\cdots\ast (\nab^A)^{k_5}\la
\end{align*}
Then we obtain
\begin{align}\nonumber
    &\ (\d_s-(\De^A_g)^3) (s^j\nab^{A,6j}\la)= [\d_s-(\De^A_g)^3,s^j\nab^{A,6j}]\la+s^j\nab^{A,6j} \tilde{\mathcal L}\\\nonumber
    =&\ js^{j-1}\nab^{A,6j}\la+s^j [\d_s-(\De^A_g)^3,\nab^{A,6j}]\la+s^j\nab^{A,6j} \tilde{\mathcal L}\\ \label{NonL}
    =&\ js^{j-1}\nab^{A,6j}\la+s^j\la^2\ast (\nab^A)^{6j+4}\la\\\nonumber
    &\ +s^j\sum_{k_1+k_2+k_3=6j+4;k_i\leq 6j+3} (\nab^A_g)^{k_1} \la \ast (\nab^A)^{k_2}\la\ast (\nab^A)^{k_3}\la\\\nonumber
    &\ +s^j\sum_{k_1+\cdots+k_5=6j+2} (\nab^A_g)^{k_1} \la \ast\cdots\ast (\nab^A)^{k_5}\la\,.
\end{align}
For a small number $\de>0$ to be chosen later, by \eqref{L2interpolation} we estimate the first term on the right as follows:
\begin{align*}
    &\ \|js^{j-1}\nab^{A,6j} \la\|_{L^2(J;\sfH^{k-3})}\les  j\|s^{j-1}\nab^{A,6(j-1)} \la\|_{L^2(J;\sfH^{k+3})} \\
    \les&\  j \|\la\|_{L^2 (J;\sfH^{k+3})}^{\frac{1}{j}}\|s^{j}\nab^{A,6j} \la\|_{L^2(J;\sfH^{k+3})}^{\frac{j-1}{j}}\\
    \les&\ \de^{-(j-1)}\|\la\|_{L^2(J;\sfH^{k+3})}+(j-1)\de \|s^{j}\nab^{A,6j} \la\|_{L^2(J;\sfH^{k+3})}.
\end{align*}
The other three terms in \eqref{NonL} are handled similarly to the nonlinear estimates in the proof of energy estimate \eqref{pre-1}. We apply \eqref{pre-1} to yield
\begin{align*}
    &\ \|s^j\nab^{A,6j}\la\|_{L^\infty_s(J;\sfH^k)}+\|s^{j}\nab^{A,6j} \la\|_{L^2(J;\sfH^{k+3})}\\
    \leq&\  C\de^{-(j-1)}\|\la\|_{L^2(J;\sfH^{k+3})}+C(j-1)\de \|s^{j}\nab^{A,6j} \la\|_{L^2(J;\sfH^{k+3})}\\
    &\ +C_\de \ep^{3/2} \|\la\|_{L^\infty L^\infty}^3\|s^j\nab^{A,6j}\la\|_{L^\infty_s \sfH^k}.
\end{align*}
If $\ep$ is small and $\la\in L^\infty L^\infty$ are finite, then the last term in the above is also absorbed. We obtain the bound \eqref{pre-2}.

Next, we turn to the proof of estimate \eqref{RegEs-HiReg}. By \eqref{pre-2}, we have
\begin{align*}
    \|s^j\nab^{A,6j}\la\|_{L^\infty_s(J;\sfH^k)}\leq&\ C\|\la_0\|_{\sfH^k}
\end{align*}
This implies that
\begin{align*}
    \|\nab^{A,6j}\la_\ep\|_{\sfH^k}\leq &\ C\ep^{-3j/2}\|\la_0\|_{\sfH^k}.
\end{align*}
By the interpolation inequality \eqref{L2interpolation} we obtain
\begin{align*}
    \|\nab^{A,m}\la\|_{\sfH^k}\les \|\la\|_{\sfH^k}^{\frac{6j-m}{6j}}\|\nab^{A,6j}\la_\ep\|_{\sfH^k}^{\frac{m}{6j}}\leq &\ C\ep^{-m/4}\|\la_0\|_{\sfH^k}.
\end{align*}
Hence the bound \eqref{RegEs-HiReg} follows.
\end{proof}

\begin{proof}[Proof of \eqref{RegEst-AppSol}]
For the first estimate in \eqref{RegEst-AppSol}, by \eqref{RegEst-g}, \eqref{Eq-MCF} and Sobolev embeddings we have
\begin{align*}
\|\la(\ep^{3/2})-\la(0)\|_{L^2}=&\ \|\int_0^{\ep^{3/2}} \d_s \la(s)ds\|_{L^2}\leq \ep^{3/2} \|(\De^A_g)^3 \la+\tilde{\mathcal L}\|_{L^\infty L^2}\\
\leq&\  \ep^{3/2}(\|\la\|_{\sfH^6}+\|\la\|_{L^\infty}^2\|\la\|_{\sfH^4})
\leq  C(M)\ep^{3/2}.
\end{align*}
By the equivalence \eqref{RegEst-g}, we easily have
\begin{align*}
    &\ \|\d_\al F_\ep-\d_\al F_0\|_{L^\infty}\leq \int_0^{\ep^{3/2}} \|\d_\al \d_s F\|_{L^\infty} ds\leq \int_0^{\ep^{3/2}} \|\nab^A_\al \mathcal L \bar{m}-\mathcal L \bar{\la}_\al^\ga \d_\ga F\|_{L^\infty} ds\\
    \les &\ \ep^{3/2}(\|\nab^A\mathcal L\|_{L^\infty}+\|\mathcal L\|_{L^\infty}\|\la\|_{L^\infty})\les \ep^{3/2} \|\la\|_{\sfH^{k_0+5}}\|\la\|_{\sfH^{k_0+3}}^3\les C(M)\ep^{3/2}.
\end{align*}
\end{proof}

\subsection{The Euler iteration}
We recall the formula \eqref{sys-cpf} derived from the original equation \eqref{Main-Sys},
\begin{equation*}          \label{sys-cpfre}
\d_t F=J(F)\mathbf{H}(F)=-\Im (\psi\bar{m})\, .
\end{equation*}
We will use this formula to construct the approximate solution $\Sigo=\Fo(\R^d)$ starting at the regularized manifold $\Sigma_\ep=F_\ep(\R^d)$.

Since the bound for second fundamental form is independent on the coordinates and gauge, we could work in a special gauge with the advection field $V=0$. Then the immersed submanifold and the associated immersed map at time $t=\ep$ are given by 
\begin{align}   \label{EulerIte}
\Sigo=\Fo(\R^d)\,,\quad\qquad  \Fo=F_\ep-\ep \Im (\psi_\ep \overline{m_\ep})\,,
\end{align}
where $F_\ep,\ \psi_\ep$ and $m_\ep$ are given in \eqref{F-ep} and \eqref{frame-reg} with respect to regularized manifold.
On the manifold $\Sigo$, we denote the metric as
\begin{align*}
	\go_{\al\be}=\<\d_\al \Fo,\d_\be \Fo\>\,,
\end{align*}
the normal vectors and the associated metric on normal bundle are denoted as
\begin{align*}
(\nu^{\mathbf{1}}_1,\ \nu^\mathbf{1}_2)\,,\quad \quad \go_{ij}=\<\nuo_i,\nuo_j\>\,,
\end{align*}
we denote the second fundamental form as
\begin{align*}
	\Lamo_{\al\be}=\Lamo(e_\al,e_\be)\,,\qquad \Lamo_{\al\be,j}=\<\Lamo(e_\al,e_\be),\nuo_j\>=\<\d^2_{\al\be}\Fo,\nuo_j\>\,.
\end{align*}

Compared with the initial manifold $\Sigma_0$, we have the properties.
\begin{prop} \label{EuIt-p}
For the approximate submanifold $\Sigo=\Fo(\R^d)$ given by \eqref{EulerIte}, we have the following properties:

(a) Ricci curvature, volume of balls and ellipticity: :
\begin{gather}   \label{AppM-Ricci}
	|\Ric(\Sigo)|\leq (1+C(M)\ep)C_0\,,\qquad \inf_{x\in \Sigo} {\rm Vol}(B_x(e^{C(M)\ep},\Sigo))\geq e^{-C(M)\ep}v\,,\\  \label{Ellpt-g}
	(1-C(M)\ep) g_0 \leq \go \leq (1+C(M)\ep) g_0. 
\end{gather}

(b) Norm bound:
\begin{equation}   \label{UB-EuIt}
\|\Lamo\|_{\sfH^k(\Sigo)}^2\leq (1+C(M)\ep)\|\La_0\|_{\sfH^k(\Sigma_0)}^2.
\end{equation}

(c) Approximate solution:
\begin{equation}   \label{Appsol-1}
\|\Fo-F_0+\ep \Im(\psi_0\bar{m}_0)\|_{L^2_{uloc}}\les \ep^{3/2}\,,\qquad 
\|\d \Fo-\d F_0\|_{L^\infty_{uloc}} \les \ep\,.
\end{equation}

\end{prop}

\medskip 
As before we remark that parts (a),(b) are covariant.
On the other hand part (c) depends on the coordinate 
flow map between $\Sigma_0$ and $\Sigo$, though  not
the chosen coordinates on $\Sigma_0$.

Before proceeding to the proof of the Proposition, we  begin by computing some geometric variables on $\Sigo$, which are the perturbations of those on $\Sigma_\ep$.
We recall the structure equations on $\Sigma_\ep$
\begin{align*}
&\d^2_{\al\be}F_\ep=(\Ga_\ep)_{\al\be}^\mu \d_\mu F_\ep+\Re\big((\la_\ep)_{\al\be}\overline{m_\ep}\big)\,,\\
&\d^{A_\ep}_\al m_\ep=-(\la_\ep)_{\al}^{\si} \d_\si F_\ep\,.
\end{align*}

\emph{i) Metrics and normal vectors.} Applying $\d_\al$ to the map $\Fo$, we have 
\begin{align}   \label{dF1}
\d_\al \Fo=\d_\al F_\ep+\ep\Im\big(\psi_\ep \overline{(\la_\ep)_{\al}^{\mu}}\big)\d_\mu F_\ep-\ep \Im(\d^{A_\ep}_\al \psi_\ep \overline{m_\ep})\,.
\end{align}
Then the metric on the manifold $\Sigo=\Fo(\R^d)$ is given by
\begin{align*}
&\go_{\al\be}=g_{\ep,\al\be}+2\ep\Im\big(\psi_\ep\overline{(\la_\ep)_{\al\be}}\big)+\ep^2 \Im\big(\psi_\ep\overline{(\la_\ep)_{\al}^\si}\big)\Im\big(\psi_\ep\overline{(\la_\ep)_{\be\si}}\big)+\ep^2 \Re(\d^{A_\ep}_\al\psi_\ep \overline{\d^{A_\ep}_\be\psi_\ep}).
\end{align*}
Since the manifold $\Sigo$ is a perturbation of the regularized manifold $\Sigma_\ep$, we would like to construct the normal vectors $(\nuoo,\nut)$ on $\Sigo$ by the unit normal vectors $(\nu^\ep_1,\nu^\ep_2)$ on $\Sigma_\ep$. By \eqref{dF1} and $\nu^\ep_j \perp \d_\al F_\ep$, the projections of $\nu^\ep_1,\ \nu^\ep_2$ on tangent vectors $\d_\al \Fo$ are given by
\begin{align*}
     \<\nu^\ep_1, \d_\al \Fo\>
     =-\ep \Im (\d^{A_\ep}_\al\psi_\ep),\qquad 
     \<\nu^\ep_2, \d_\al\Fo\>
     =\ep \Re (\d^{A_\ep}_\al\psi_\ep).
\end{align*}
Then the normal vectors $\nuoo$ and $\nut$ on $\Sigo$ can be constructed as
\begin{align*}\label{nuo}
	\nuoo=\nu^\ep_1+ \ep g^{\mathbf{1},\al\be} \Im (\d^{A_\ep}_\al\psi_\ep)\d_\be \Fo\,,\qquad 
	\nut=\nu^\ep_2-\ep  g^{\mathbf{1},\al\be} \Re (\d^{A_\ep}_\al\psi_\ep)\d_\be \Fo\,,
\end{align*}
which are almost orthonormal vectors. We can also obtain the metric $\go_{ij}=\<\nuo_i,\nuo_j\>$ on normal bundle $N \Sigo$
\begin{equation*}
\begin{aligned}
	\go_{11}&=1-\ep^2 g^{\mathbf{1},\al\be} \Im(\d^{A_\ep}_\al\psi_\ep)\Im (\d^{A_\ep}_\be\psi_\ep)\,,\\
	\go_{22}&=1-\ep^2g^{\mathbf{1},\al\be}\Re(\d^{A_\ep}_\al\psi_\ep)\Re(\d^{A_\ep}_\be\psi_\ep)\,,\\
	\go_{12}&=\ep^2g^{\mathbf{1},\al\be}\Re(\d^{A_\ep}_\al\psi_\ep)\Im(\d^{A_\ep}_\be\psi_\ep)\,.
\end{aligned}
\end{equation*}
Hence, the metric $(g^{\mathbf{1}}_{ij})$ has the form $I_2+\ep^2 O(\d^A \psi_\ep)^2$.

\emph{ii) Second fundamental form $\Lamo$.} Since 
\begin{align*}
    &\ \quad \nab^\ep_\al \d_\be \Fo=  \nab^\ep_\al \d_\be(F_\ep-\ep\Im(\psi_\ep \bar{m}_\ep))\\
    &=\Re\big((\la_\ep)_{\al\be}\bar{m}_\ep\big)-\ep\Im(\nab^{A_\ep}_\al\nab^{A_\ep}_\be \psi_\ep \bar{m}_\ep-\nab^{A_\ep}_\be \psi_\ep \overline{(\la_\ep)_{\al}^\si} \d_\si F_\ep-\nab^{A_\ep}_\al \psi_\ep \overline{(\la_\ep)_{\be}^\si} \d_\si F_\ep\\
    &\quad -\psi_\ep \overline{\nab^{A_\ep}_\al (\la_\ep)_{\be}^\si}\d_\si F_\ep-\psi_\ep \overline{(\la_\ep)_{\be}^\si}\Re\big((\la_\ep)_{\al\si}\bar{m}_\ep)\big)
\end{align*}
and 
\begin{align*}
    &\nuo_1=\nu^\ep_{1}+\ep g^{\mathbf{1},\mu\nu}\Im(\nab^{A_\ep}_\mu \psi_\ep) (\d_\nu F_\ep-\ep\Im(\nab^{A_\ep}_\nu\psi_\ep \bar{m}_\ep-\psi_\ep \overline{(\la_\ep)_{\nu}^\de} \d_\de F_\ep)),\\
    &\nut=\nu^\ep_2-\ep  g^{\mathbf{1},\mu\nu} \Re (\nab^{A_\ep}_\mu\psi_\ep)(\d_\nu F_\ep-\ep\Im(\nab^{A_\ep}_\nu\psi_\ep \bar{m}_\ep-\psi_\ep \overline{(\la_\ep)_{\nu}^\de} \d_\de F_\ep))\,.
\end{align*}
Then we obtain the second fundamental form
\begin{align*}
  &\Lamo_{\al\be,1}=\<\Lamo(e_\al,e_\be),\nuoo\>=\<\nab^\ep_\al \d_\be \Fo +\big((\Ga_\ep)_{\al\be}^\ga-(\Ga_\mathbf{1})_{\al\be}^\ga\big)\d_\ga\Fo,\nuoo\>\\
  &=\Re\big((\la_\ep)_{\al\be}+i\ep\nab^{A_\ep}_\al\nab^{A_\ep}_\be \psi_\ep\big)+\ep\Im(\psi_\ep \overline{(\la_\ep)_{\be}^\si})\Re(\la_\ep)_{\al\si}+\Im(\nab^{A_\ep}_\mu\psi_\ep) \widehat{T}_{\al\be}^{\ \ \mu},
\end{align*}
and 
\begin{align*}
  &\Lamo_{\al\be,2}=\<\Lamo(e_\al,e_\be),\nuoo\>=\<\nab^\ep_\al \d_\be \Fo +\big((\Ga_\ep)_{\al\be}^\ga-(\Ga_\mathbf{1})_{\al\be}^\ga\big)\d_\ga\Fo,\nuoo\>\\
  &= \Im((\la_\ep)_{\al\be}+i\ep\nab^{A_\ep}_\al\nab^{A_\ep}_\be \psi_\ep)+\ep\Im(\psi_\ep \overline{(\la_\ep)_{\be}^\si})\Im(\la_\ep)_{\al\si}-\Re(\nab^{A_\ep}_\mu\psi_\ep) \widehat{T}_{\al\be}^{\ \ \mu},
\end{align*}
where $\widehat{T}_{\al\be}^{\ \ \mu}$ are denoted as
\begin{align*}
	&\widehat{T}_{\al\be}^{\ \ \mu}:=\ep^2\Im(\nab^{A_\ep}_\be \psi_\ep \overline{(\la_\ep)_{\al\nu}} +\nab^{A_\ep}_\al \psi_\ep \overline{(\la_\ep)_{\be\nu}} +\psi_\ep \overline{\nab^{A_\ep}_\al (\la_\ep)_{\be\nu}})\big(g^{\mathbf{1},\mu\nu}+\ep g^{\mathbf{1},\mu\de}\Im(\psi_\ep \overline{(\la_\ep)^\nu_{\de}})\big)\\
	&+ \ep^2 g^{\mathbf{1},\mu\nu}\big(\Im((\la_\ep)_{\al\be}\overline{\nab^{A_\ep}_\nu\psi_\ep})+\ep \Im(\psi_\ep \overline{(\la_\ep)_\be^\si})\Im((\la_\ep)_{\al\si}\overline{\nab^{A_\ep}_\nu\psi_\ep})+\ep\Re(\nab^{A_\ep}_\al \nab^{A_\ep}_\be \psi_\ep \overline{\nab^{A_\ep}_\nu}\psi_\ep) \big).
\end{align*}
Note that the leading order terms are as below
\begin{align*}
   \Lamo_{\al\be,1}\sim \Re((\la_\ep)_{\al\be}+i\ep \nab^{A_\ep}_\al\d^{A_\ep}_\be\psi_\ep)\,,\qquad \Lamo_{\al\be,2}\sim \Im((\la_\ep)_{\al\be}+i\ep \nab^{A_\ep}_\al\d^{A_\ep}_\be\psi_{\ep})\,.
\end{align*}
Then the $L^2$- norm of second fundamental form $\Lamo$ is exactly a perturbation of that for $\la_\ep$
\begin{align*}
&\ g^{\mathbf{1};\al\mu}g^{\mathbf{1};\be\nu} g^{\mathbf{1};ij}\Lamo_{\al\be,i}\Lamo_{\mu\nu,j}
\sim g_\ep^{\al\mu}g_\ep^{\be\nu} \de^{ij}\Lamo_{\al\be,i}\Lamo_{\mu\nu,j}\\
=&\ g_\ep^{\al\mu}g_\ep^{\be\nu} \big(\Re ((\la_\ep)_{\al\be}+i\ep \nab^{A_\ep}_\al\d^{A_\ep}_\be\psi_\ep)\Re ((\la_\ep)_{\mu\nu}+i\ep \nab^{A_\ep}_\mu\d^{A_\ep}_\nu\psi_\ep)\\
&\  +\Im ((\la_\ep)_{\al\be}+i\ep \nab^{A_\ep}_\al\d^{A_\ep}_\be\psi_\ep)\Im ((\la_\ep)_{\mu\nu}+i\ep \nab^{A_\ep}_\mu\d^{A_\ep}_\nu\psi_\ep)\big)\\
=&\ g_\ep^{\al\mu}g_\ep^{\be\nu} ((\la_\ep)_{\al\be}+i\ep \nab^{A_\ep}_\al\d^{A_\ep}_\be\psi_\ep) \overline{((\la_\ep)_{\mu\nu}+i\ep \nab^{A_\ep}_\mu\d^{A_\ep}_\nu\psi_\ep)}\\
=&\ ((\la_\ep)_{\al}^{\ \be}+i\ep \nab^{A_\ep}_\al\nab^{A_\ep,\be}\psi_\ep) \overline{((\la_\ep)_{\be}^{\al}+i\ep \nab^{A_\ep,\al}\nab^{A_\ep}_\be\psi_\ep)}\,.
\end{align*}
However, the higher-order norms of $\Lamo$ is more complicated, we should compute more carefully.

\smallskip 
Now we start our proof of Proposition \ref{EuIt-p}. For convenience, we will use the linear flow $\Sigma_s=F_s(\R^d)$ with $ F_s=F_\ep-s\Im(\psi_\ep \bar{m}_\ep)$.
Then the associated geometric variables on $\Sigma_s$ are simply given by those variables on $\Sigo$ with coefficient $\ep$ replaced by $s$.  These geometric variables on $\Sigma_s$ (for instance metric, covariant derivatives, Christoffel symbols, connection coefficients, second fundamenatal form, Ricci curvature and so on) are denoted as \begin{equation*}
g(s),\quad  \nab(s),\quad   \Ga_{\al\be}^\ga(s),\quad    A_{\al j}^i(s),\quad   \Lambda(s),\quad \Ric(s),\qquad s\in[0,\ep].
\end{equation*}

\begin{lemma} Let $s\in [0,\ep]$ with $\ep \les  M^{-2}$. Then we have
\begin{align}   \label{BD-dsg}
	\|\d_s g(s)\|_{L^\infty}&\les \|\la_0\|_{\sfH^k}^2,\qquad   
	\|\d_s g_{ij}(s)\|_{L^\infty}\les \ep^{1/2}\|\la_0\|_{\sfH^k}^2,\\ \label{BD-Lambda}
	\| \Lambda(s)\|_{L^\infty}&\les \|\la_0\|_{\sfH^{k}},\qquad \|\d_s \Lambda(s)\|_{L^\infty}\les \|\la_0\|_{\sfH^{k}}^3.
\end{align}
\end{lemma}
\begin{proof}
For the first bound \eqref{BD-dsg}, by Sobolev embeddings, \eqref{RegEst-IM}, \eqref{RegEs-HiReg} and $s\in [0,\ep]$, we have
\begin{align*}
\|\d_s g(s)\|_{L^\infty}&\les \|\psi_\ep \bar{\la}_{\ep}+s\psi_\ep\overline{\la_{\ep}}\psi_\ep\overline{\la_{\ep}}+s\d^{A_\ep}\psi_\ep \overline{\d^{A_\ep}\psi_\ep}\|_{L^\infty}\\
	&\les \|\la_0\|_{\sfH^k}^2+\ep \|\la_0\|_{\sfH^k}^4+\ep \ep^{-1/2}\|\la_0\|_{\sfH^k}^2\les \|\la_0\|_{\sfH^k}^2.
\end{align*}
We also have 
\begin{align*}
\|\d_s g_{ij}(s)\|_{L^\infty}&\les \|(\ep g^{\al\be}(s)+\ep^2 \d_s g^{\al\be}(s)) \d^{A_\ep}_\al\psi_\ep \d^{A_\ep}_\be\psi_\ep\|_{L^\infty}\\
&\les (\ep+\ep^2\|\la_0\|_{\sfH^k}^2)\|\nab^{A_\ep}\psi_\ep\|_{\sfH^k}^2 \les \ep^{1/2}\|\la_0\|_{\sfH^k}^2.
\end{align*}

By Sobolev embedding on $\Sigma_\ep$ and \eqref{RegEs-HiReg}, we have
\begin{align*}
\|\widehat{T}\|_{L^\infty}&\les \ep^2 \|\nab^{A_\ep} \la_\ep\|_{\sfH^{k_0}}\| \la_\ep\|_{\sfH^{k_0}}(1+\ep \| \la_\ep\|_{\sfH^{k_0}}^2)+\ep^3 \|(\nab^{A_\ep} )^2\la_\ep\|_{\sfH^{k_0}}\|\nab^{A_\ep} \la_\ep\|_{\sfH^{k_0}}\\
&\les \ep^{7/4}\| \la_0\|_{\sfH^{k_0}}\| \la_0\|_{\sfH^{k_0}}(1+\ep \| \la_0\|_{\sfH^{k_0}}^2),
\end{align*}
Then we obtain
\begin{align*}
\|\Lambda(s)\|_{L^\infty}&\les \|\la_\ep\|_{\sfH^{k_0}}+\ep \|(\nab^{A_\ep})^2\psi_\ep\|_{\sfH^{k_0}}+\ep \|\la_\ep\|_{\sfH^{k_0}}^3+\|\nab^{A_\ep} \psi_\ep\|_{\sfH^{k_0}} \ep^{7/4}\| \la_0\|_{\sfH^{k_0}}^2\les \|\la_0\|_{\sfH^{k_0}}
\end{align*}
and 
\begin{align*}
	\|\d_s \Lambda(s)\|_{L^\infty}&\les  \|(\nab^{A_\ep})^2\psi_\ep\|_{\sfH^{k_0}}+ \|\la_\ep\|_{\sfH^{k_0}}^3+\|\nab^{A_\ep} \psi_\ep\|_{\sfH^{k_0}} \ep^{3/4}\| \la_0\|_{\sfH^{k_0}}\| \la_0\|_{\sfH^{k_0}}\\
	&\les \|\la_0\|_{\sfH^{k_0+2}}(1+\| \la_0\|_{\sfH^{k_0}}^2).
\end{align*}
\end{proof}

First we prove the ellipticity condition\eqref{Ellpt-g}.
\begin{proof}[Proof of \eqref{Ellpt-g}]
By \eqref{BD-dsg}, we have for any $X$
\begin{align*}
\d_s |X|_{g(s)}^2 &\les |\d_s g_{\al\be}(s)X^\al X^\be|\les \|\la_0\|_{\sfH^{k_0}}^2|X|_{g(s)}^2
\leq CM^2 |X|_{g(s)}^2\,,
\end{align*}
which implies that
\begin{align*}
	&e^{-CM^2 \ep}|X|_{g_\ep}^2\leq |X|_{\go}^2\leq e^{CM^2 \ep}|X|_{g_\ep}^2\,.
\end{align*}
This together with \eqref{RegEst-g} yields the bound \eqref{Ellpt-g}.
\end{proof}

Next, we bound the Ricci curvature and volume of balls \eqref{AppM-Ricci}. 
\begin{proof}[Proof of \eqref{AppM-Ricci}]
By the standard computations, we have the curvature and Ricci curvature on manifold $\widetilde\Sigma_s$
\begin{align*}
&R_{\si\ga\al\be}(s)=\Lambda^{\ \ j}_{\be\ga}(s)\Lambda_{\al\si,j}(s)-\Lambda^{\ \ j}_{\al\ga}(s)\Lambda_{\be\si,j}(s)\,,\\
&\Ric_{\ga\be}(s)=\Lambda^{\ \ j}_{\be\ga}(s)\Lambda_{\ \al,j}^\al(s)-\Lambda^{\ \ j}_{\al\ga}(s)\Lambda_{\ \be,j}^\al(s)\,,
\end{align*}
where $\Lambda_{\al\be}^{\ \ j}(s)=\Lambda_{\al\be,k}(s) g^{jk}(s)$.
Then by \eqref{BD-Lambda} we have
\begin{align*}
    |\d_s \Ric_{\ga\be}(s) X^\ga X^\be|
    &\leq \|\d_s \Lambda(s)\|_{L^\infty} \|\Lambda(s)\|_{L^\infty} |X|^2_{ g(s)}\les M^4 |X|^2_{g(s)}.
\end{align*}
This, combined with $(1-CM^2 \ep)|X|_{g_\ep}^2\leq |X|_{\go}^2\leq (1+CM^2 \ep)|X|_{g_\ep}^2$, further implies that 
\begin{align*}
    &\ |\Ric^{\bf{1}}_{\ga\be} X^\ga X^\be|\leq |\Ric_{\ep,\ga\be} X^\ga X^\be|+ |\int_0^\ep \d_s \Ric_{\ga\be}(s) X^\ga X^\be ds|\\
    &\leq (1+C(M)\ep) C_0 |X|^2_{g_\ep}+ C\ep M^4 |X|^2_{g(s)}
    \leq (1+\tilde C(M)\ep)C_0|X|_{\go}^2.
\end{align*}

For the volume element, by \eqref{BD-dsg} we have
\begin{align*}
    |\d_s \sqrt{\det g(s)}|&=|g^{\al\be}(s)\d_s g_{\al\be}(s)|\sqrt{\det g(s)}
    \leq  C M^2 \sqrt{\det g(s)},
\end{align*}
which implies that
\begin{align*}
    &\sqrt{\det g_0} e^{-C(M)\ep^{3/2}}\leq \sqrt{\det g_\ep} e^{-C(M)\ep}\leq \sqrt{\det \go}\leq \sqrt{\det g_\ep} e^{C(M)\ep}\leq \sqrt{\det g_0} e^{C(M)\ep^{3/2}}.
\end{align*}
Moreover, the lengths of any curves evolving along the manifold $\Sigma_s$ would only change slightly
\begin{align*}
    \frac{d}{ds}l(\ga,s)=&\ \int_0^1 (2|\dot{\ga}|)^{-1} \d_s g_{\al\be}(s) \d_\tau \ga_\al \d_\tau \ga_\be\ d\tau \leq \|\d_s g(s)\|_{L^\infty(g(s))} l(\ga,s)
    \leq  CM^2 l(\ga,s),
\end{align*}
which implies that $B_x(r,\Sigma_\ep)\subset B_x(re^{\ep C(M)},\Sigo)$ for any $r$. Then we can bound the volume of balls from below
\begin{align*}
    &{\rm Vol}_{\go}(B_x(r_0 e^{C(M)\ep}))= \int_{B_x(e^{C(M)\ep},\Sigo)} 1\  dvol_{\go}\geq \int_{B_x(r_0,\Sigma_\ep)}   e^{-\ep C(M)}dvol_{g_\ep}\\
    =&\    e^{-\ep C(M)} {\rm Vol}_{g_\ep}(B_x(r_0,\Sigma_\ep))
    \geq  e^{-\ep C(M)} v,
\end{align*}
Hence, the estimates in \eqref{AppM-Ricci} are obtained.
\end{proof}

Next, we prove the norm bound \eqref{UB-EuIt}.

\begin{proof}[Proof of norm bound \eqref{UB-EuIt}]
We consider the linear flow $F_s=F_\ep-s \Im(\psi_\ep \bar{m}_\ep)$. This can be expressed as 
\begin{align*}
    \d_s F_s=-\Im(\psi_\ep \bar{m}_\ep)=-\Im(\psi_s \bar{m}_s)+\int_0^s \d_\tau \Im (\psi_\tau \bar{m}_\tau) d\tau=:-\Im(\psi_s \bar{m}_s)+G_s, 
\end{align*}
where $\psi_s$ and $m_s$ are the complex mean curvature and normal frame on $\Sigma_s=F_s(\R^d)$, respectively. Compared with original flow \eqref{sys-cpf}, here we add a source term $G_s$. 
Then using an argument similar to the one  in Section~\ref{Sec-gauge}, we can also derive an equation for $\la$
\begin{align*}
    i\d_s^B \la_\al^\si + \De_g^A \la_\al^\si =&\ i\nab_\al^A \<m,\d^\si G_s\>-i\la_\al^\mu \<\d_\mu G_s,\d^\si F\>-i\la_\al^\ga \Im(\psi\bar{\la}_\ga^\si)+\Re(\psi\bar{\la}_{\de\al})\la^{\si\de}\\
    &-\la_{\al}^\mu \bar{\la}_{\de\mu}\la^{\si\de}-\Re(\la_{\de\mu}\bar{\la}_\al^\si-\la^\si_\mu\bar{\la}_{\de\al})\la^{\de\mu}.
\end{align*}
This we can use in order to prove energy estimates,
\begin{align*}
     &\ \frac{1}{2}\frac{d}{ds}\|\la\|_{\sfH^k(\Sigma_s)}^2\\
    &=   \sum_{j\leq k}\int \Re\<\nab^{A,j}\d^B_s \la_\al^\si,\nab^{A,j}\la^\al_\si\> +\Re\< [\d^B_s,\nab^{A,j}]\la_\al^\si,\nab^{A,j}\la^\al_\si\>
    +|\nab^{A,j}\la|^2 \frac{1}{4}g^{\al\be}\d_s g_{\al\be}\ dvol\\
    &\lesssim  \|\<\d G_s,m\>\|_{\sfH^{k+1}}\|\la\|_{\sfH^k}+(\|\la\|_{L^\infty}+\|\<\d G_s,\d F\>\|_{L^\infty})^2\|\la\|_{\sfH^k}^2\\
    &\quad +(\|\<\d G_s,\d F\>\|_{L^\infty}+\|\<\d G_s,m\>\|_{L^\infty}+\|\la\|_{L^\infty})^2(\|\<\d G_s,\d F\>\|_{\sfH^k}+\|\<\d G_s,m\>\|_{\sfH^k})\|\la\|_{\sfH^k} .
\end{align*}
Then we'd like to bound $\<\d G_s,m\>$ and $\<\d G_s,\d F\>$ by $\|\d F_0\|_{H^{k+1}_{uloc}}$. 

In local charts, by \eqref{reg-Fep} and \eqref{reg-mep} we have 
\begin{align*}
\| \partial^j (\psi_\epsilon m_\ep) \|_{H^{k}_{uloc}}=\| \partial^j \big(g_\ep^{\al\be}(\d^2_{\al\be}F_\ep-\Ga_{\al\be}^\ga \d_\ga F_\ep)\cdot m_\ep\ m_\ep\big) \|_{H^{k}_{uloc}} \lesssim \epsilon^{-\frac{j}4} \| \partial F_0\|_{H^{k+1}_{uloc}}.  
\end{align*}
Then by the Euler iteration $F_s=F_\ep-s\Im(\psi_\ep\bar{m}_\ep)$ for $s\leq \ep$, we have
\begin{equation*}
\begin{aligned}
    &\ \|\partial^j \d F_s\|_{H^{k+1}_{uloc}}\leq \|\partial^j \d F_\ep\|_{H^{k+1}_{uloc}}+\ep \|\partial^j \d (\psi_\ep m_\ep)\|_{H^{k+1}_{uloc}} \\
    &\lesssim \epsilon^{-\frac{j}4} \| \partial F_0\|_{H^{k+1}_{uloc}}+\ep^{1-\frac{j+2}{4}}\| \partial F_0\|_{H^{k+1}_{uloc}}\les \epsilon^{-\frac{j}4} \| \partial F_0\|_{H^{k+1}_{uloc}},
\end{aligned}
\end{equation*}
and 
\begin{align*}
    \|\d^j\d_s F_s\|_{H^{k}_{uloc}}= \|\d^j(\psi_\ep m_\ep)\|_{H^{k}_{uloc}}\les \epsilon^{-\frac{j}4} \| \partial F_0\|_{H^{k+1}_{uloc}}.
\end{align*}
A construction similar to \eqref{reg-mep} yields
\begin{align*}
    \|\d^j  m_s\|_{H^{k+1}_{uloc}}\les \|\d^j \d F_s\|_{H^{k+1}_{uloc}}\les \epsilon^{-\frac{j}4} \| \partial F_0\|_{H^{k+1}_{uloc}},
\end{align*}
and 
\begin{align*}
    \|\d^j  \d_s m_s\|_{H^{k-1}_{uloc}}\les \|\d^j \d\d_s F_s\|_{H^{k-1}_{uloc}}\les \epsilon^{-\frac{j}4} \| \partial F_0\|_{H^{k+1}_{uloc}}.
\end{align*}
Then we get the estimate 
\[
\| \partial^j \partial_\tau (\psi_\tau m_\tau)\|_{  H^{k-2}}   \lesssim \epsilon^{-\frac{j}4} \| \partial F_0\|_{H^{k+1}_{uloc}} \|\Lambda_0\|_{\sfH^k},
\]
where the $\ell^2$ summation with respect to local charts 
comes from the similar property for $\Lambda$ and thus for $\psi$. This implies 
\[
\|\partial^j \partial G_s\|_{H^{k-3}}=\|\d^j \d \int_0^s \d_\tau (\psi_\tau m_\tau) d\tau\|_{H^{k-3}}\les \ep \| \partial^j \partial_\tau (\psi_\tau m_\tau)\|_{H^{k-2}}    \lesssim \epsilon^{1-\frac{j}4} \| \partial F_0\|_{H^{k+1}_{uloc}} \|\Lambda_0\|_{\sfH^k} ,
\]
and in particular
\[
\| \partial G_s\|_{H^{k+1}}   \lesssim  \| \partial F_0\|_{H^{k+1}_{uloc}} \|\Lambda_0\|_{\sfH^k} .
\]

This in turn yields
\begin{align*}
    \|\<\d G_s,m\>\|_{L^\infty}+\|\<\d G_s,\d F\>\|_{L^\infty}\les \|\d G_s\|_{H^{k+1}}\les \| \partial F_0\|_{H^{k+1}_{uloc}}\|\Lambda_0\|_{\sfH^k} .
\end{align*}
\begin{align*}
    \|\<\d G_s,m\>\|_{\sfH^{k+1}}+\|\<\d G_s,\d F\>\|_{\sfH^{k+1}}\les \|\<\d G_s,m\>\|_{ H^{k+1}}\les \|\d F_0\|_{H^{k+1}_{uloc}}\|\Lambda_0\|_{\sfH^k}.
\end{align*}
Using this in the energy estimates above, we obtain
\begin{align*}
    \frac{d}{ds}\|\la\|_{\sfH^k(\Sigma_s)}^2\les C(M) \|\la\|_{\sfH^k(\Sigma_s)}^2+C(M)\|\la\|_{\sfH^k(\Sigma_s)}.
\end{align*}
This implies the norm bound \eqref{UB-EuIt} for $\Lamo$.
\end{proof}

\begin{proof}[Proof of \eqref{Appsol-1}]
By \eqref{EulerIte}, it suffices to show that
\begin{align*}
    &\ \|F_\ep-\ep\Im(\psi_\ep\bar{m}_\ep)-F_0+\ep\Im(\psi_0\bar{m}_0)\|_{L^2}\\
    \leq &\ \|\int_0^{\ep^{3/2}}\d_s F ds\|_{L^2}+\ep \|\Im ((\psi_\ep-\psi_0)\bar{m}_\ep)\|_{L^2}+\ep \|\Im(\psi_0(\overline{m_\ep-m_0}))\|_{L^2}\\
    \leq &\ \ep^{3/2}\|\nab \mathcal L\|_{L^2}+\ep \|\psi_\ep-\psi_0\|_{L^2}+\ep\|\psi_0\|_{L^\infty} \int_0^{3/2} \|\d_s m\|_{L^2} ds\\
    \leq &\ C(M) \ep^{3/2}+C(M) \ep^{5/2}+\ep^{5/2} \|\nab \mathcal L\|_{L^2}\\
    \leq &\ C(M) \ep^{3/2}\,.
\end{align*}
By the equivalence \eqref{Ellpt-g}, we can bound the difference of $\d F$ by
\begin{align*}
    &\ \|\d_\al \Fo-\d_\al F_\ep\|_{L^\infty}\leq \ep \|\psi_\ep \bar{\la}_{\ep,\al}^\mu \d_\mu F_\ep-\nab^{A_\ep}_\al\psi_\ep \bar{m}_\ep\|_{L^\infty} \\
    \les &\ \ep(\|\psi_\ep\|_{L^\infty}\|\la_\ep\|_{L^\infty}+\|\nab^A\psi_\ep\|_{L^\infty})\les \ep \|\la_0\|_{\sfH^{k_0+1}}^2.
\end{align*}
Hence, the bounds in \eqref{Appsol-1} are obtained.
\end{proof}

\subsection{Construction of regular exact solutions}
Here we use the approximate solutions above. Given initial manifold $\Sigma_0=F_0(\R^d)$ with the map $F_0: \R^d\rightarrow \R^{d+2}$ so that
\begin{align*}
    \|\Lambda_0\|_{\sfH^k}\leq M,\quad |\Ric(0)|\leq C,\quad \inf_{x\in\Sigma_0} {\rm Vol}_{g(0)}(B_x(1))\geq v,
\end{align*}
applying the successive iterations above we obtain approximate solutions $\Sigma^\ep(t)=F^\ep(t,\R^d)$ with $t\in\ep \mathbb N\cap [0,T(M)]$ defined
at $\ep$ steps, so that
\begin{align*}
&\|\Lambda^\ep((j+1)\ep)\|_{\sfH^k}\leq (1+C(M)\ep)\|\Lambda^\ep(j\ep)\|_{\sfH^k},\\
&|\Ric^\ep((j+1)\ep)|\leq (1+C(M)\ep)^{j+1} C_0,\\
&\inf_{x\in\Sigma_{j+1}} {\rm Vol}_{g(j+1)}(B_x(e^{C(M)(j+1)\ep}))\geq e^{-C(M)(j+1)\ep} v,
\end{align*}
In addition, choosing the coordinates on $\Sigma^\ep$ induced by our, single step construction,  we also have the relations
\[
(1-C(M)\ep)g(j\ep) \leq g((j+1)\ep) \leq (1+C(M)\ep) g(j\ep),
\]
\[
 \|\d F^\ep((j+1)\ep)- \d F^{\ep}(j\ep)\|_{L^\infty}\les \ep\,.
\]

By the discrete Gr\"onwall's inequality, it follows that these approximate solutions are defined
uniformly up to a time $T=T(M)$, with uniform bounds
\begin{align*}
    \|\Lambda^\ep(j\ep)\|_{\sfH^k}\leq (1+C(M)\ep)^{j}\|\la_0\|_{\sfH^k}\les_M 1.
\end{align*}
as well as 
\begin{equation*}
c(M) \leq g(j\ep) \leq C(M),
\end{equation*}
By Sobolev embeddings, the $\Lambda$ bound implies
\[
\|\Lambda \|_{L^\infty} \lesssim 1
\]
and a similar bound for $\psi$, which in turn by \eqref{appsol} shows that 
\[
\| F^{\ep}((j+1)\ep)- F^{\ep}(j\ep)\|_{L^\infty}\les \ep\,.
\]
Thus the functions $F^{\ep}$ are Lipschitz in time with values in $C^1$, 
uniformly in $\ep$. By Arzela-Ascoli applied on compact sets, this yields
a subsequence which converges uniformly on compact sets,
\[
F^\ep \to F.
\]

We now need to examine more closely the regularity of $F$, and in particular to show that $F$ solves the SMCF flow. This is more easily done locally, in cartesian coordinates. Near some point $p$ on $F_0$, we represent $\Sigma_0$ as a graph in a local cartesian frame, say, after a rotation,
\[
\Sigma_0 = \{ y'' = \FF_0(y')\}, \qquad y' = (y_1,\cdots, y_d),  \quad y''= (y_{d+1},y_{d+2}).
\]
Then for small $t$ we can represent our approximate solutions in the same frame,
\[
\Sigma_t^\epsilon = \{ y'' = \FF^\epsilon
(t,y')\}.
\]
By the above Lipschitz property of $F$, the time dependent change of coordinate map $x \to y=F^\ep(x)'$ is bilipschitz, 
and also Lipschitz in $t$. This in particular implies that the functions $\FF^\epsilon$ are also Lipschitz in $t$ and $y'$.
The advantage in using the extrinsic local coordinates is that the covariant $H^k$ bound on the second fundamental form
implies that we have the uniform local regularity
\[
\FF^\ep \in L^\infty_t H^{k+2}_y, \qquad \psi^\ep m^\ep \in L^\infty_t H^k_y.
\]
Using Sobolev embeddings and interpolating with the Lipschitz bound, for large enough $k$ we also get 
\[
\FF^\ep \in C^\frac12_t C^3_y,
\]
which in turn implies that 
\[
\psi^\ep m^\ep  \in C^\frac12_t C^1_y.
\]
This property we can return to the $(x,t)$ coordinates,
\[
\psi^\ep m^\ep  \in C^\frac12_t C^1_x.
\]
This in turn shows that 
\[
\FF^\ep \in C^{\frac32}_t C^0_x.
\]

Passing to the limit, we obtain that $F$ is bilipschitz, and so is the corresponding local representation $\FF$, 
with $\FF^\ep \to \FF$ uniformly on a subsequence. Taking weak limits in the extrinsic coordinates, all the above regularity properties transfer to $F$ and $\FF$. This allows us to upgrade the convergence to all weaker norms.
In particular we get on a subsequence
\[
\psi^\ep m^\ep  \to \psi m \in C^{\frac12-}_t C^1_y.
\]
Then we can pass to the limit in the relation 
\[
F^\ep((j+1)\ep) = F^\ep(j\ep) -\ep\Im(\psi^\ep \overline{m^\ep}) (\epsilon j) + O(\ep^\frac32)
\]
to obtain that 
\[
\partial_t F = -\Im(\psi \bar{m}),
\]
i.e. $F$ solves the SMCF equation. We can further upgrade the regularity of $F$ in the extrinsic coordinates. There, 
by a direct application of chain rule on $F=(y',\FF(t,y'))$, the SMCF equation is rewritten as 
\[
\partial_t \FF + W^j \partial_j \FF = -(\Im(\psi \bar{m}))'', \qquad W^j = -(\Im(\psi \bar{m}))_j,\ \text{for}\ j=1,\cdots,d,
\]
where $(\Im(\psi\bar{m}))''$ is the last two components of vector $\Im(\psi\bar{m})$, and we have the local regularity
\[
\FF \in L^\infty H^{k+2}, \qquad \Im(\psi \bar{m}) \in L^\infty H^k.
\]
This in particular shows that 
\[
\partial_t \FF \in L^\infty H^k, \qquad \partial_t^2 \FF \in L^\infty H^{k-2}.
\]

To return to the $x$ coordinates, we need to track $F'$, via the nonlinear ode 
\[
\partial_t F' = W(F'), \qquad F'(0,x) = F'_0(x),
\]
where the remaining component of $F$ is given by $ F'' = \FF(F')$.
Here the initial data $F_0$ has maximal local regularity
$\partial F_0 \in H^{k+1}$ but $W$ is less regular so we 
only obtain dynamically 
\[
\partial F \in L^\infty  H^{k-1}, \partial_t F  \in L^\infty  H^{k}.
\]
Hence we have produced a solution to the equation \eqref{Main-Sys}, which is unique by Theorem~\ref{Uniqueness-thm}. This solution has an apparent loss of regularity,
which is expected since our solution is constructed as a solution to \eqref{Main-Sys-re} in the temporal gauge 
$V=0$.
\bigskip

The remaining step of our construction is to move the
 solution to \eqref{Main-Sys-re}  constructed above to the heat gauge $V^\ga=g^{\al\be}\Ga_{\al\be}^\ga$. This corresponds to a change of coordinates $x\rightarrow y(t,x)$, where, defining $F(t,x)=\tF(t,y(t,x))$, $\tF$ is a solution of \eqref{Main-Sys-re} in the heat gauge.  Here we have
\begin{align*}
    \d_t F=\d_t \tF(t,y)+\d_t \varphi^k \d_k \tF(t,y)=J(\tF)\bH(\tF)+\tilde{V}^\ga \d_\ga \tF(t,y)+\d_t \varphi^k \d_k \tF(t,y).
\end{align*}
Since $\d_t F=J(F)\bH(F)$, this requires that $y(t,x)$ satisfies 
\begin{align*}
    \d_t y^\ga=-\tilde V^\ga(t,y)=-\tg^{\al\be}\tGa_{\al\be}^\ga (t,y)=\tg^{\al\be}(\d^2_{y_{\al}y_{\be}}-\tGa_{\al\be}^\si \d_{y_\si})y^\ga =\De_{\tg(y)}y^\ga,
\end{align*}
This can be rewritten as a linear parabolic equation
\begin{align*}
    \d_t y^\ga-\De_g y^\ga=0,
\end{align*}
with initial data $y(0,x)=x$, which is solvable in a short time and $y(t,x)$ is a diffeomorphism. Hence we obtain the regular solution of \eqref{Main-Sys-re} in the heat gauge.

\

\section{Rough solutions}\label{Sec-rough}
In this section we aim to construct rough solutions as limits of smooth solutions, and
conclude the proof of Theorem~\ref{LWP-thm}. In terms of a general outline, the argument here is
relatively standard, and involves the following steps:
(1) We regularize the initial manifold $\Sigma_0=F_0(\R^d)$.
(2) We prove uniform bounds for the regularized solutions.
(3) We prove convergence of the regularized solutions in a weaker topology.
(4) We prove convergence in the strong topology by combining the weak difference
bounds with the uniform bounds in a frequency envelope fashion.

\subsection{Regularization of initial data}
Given a rough initial submanifold $\Sigma_0=F_0(\R^d)$ satisfying \eqref{MetBdd} and \eqref{small-datad3}, 
then from Proposition~\ref{Global-harmonic}, there exist a gauge \eqref{Coulomb} in $N\Sigma_0$ such that 
\begin{align*}
    \|\la_0\|_{H^s}+ \||D|^{\de_d}A_0\|_{H^{s-\de_d}}+\||D|^{\si_d}g_0\|_{H^{s+1-\si_d}}\leq M_1.
\end{align*}

As in Section~\ref{Reg-sec}, under the assumptions \eqref{MetBdd} and \eqref{small-datad3}, we can construct an appropriate family of regularized data, depending smoothly on the regularization parameter $h$, as
\begin{align*}
    (\Sigma_0^{(h)}:=F^{(h)}_0(\R^d), g_0^{(h)}, A^{(h)}_0,\la^{(h)}_0),\qquad F^{(h)}_0=P_{<h}F_0,\ h\geq h_0,
\end{align*}
with the following properties:
\begin{align}  \label{Cond1}
&\|\la_0^{(h)}\|_{\sfH^k(\Sigma_0^{(h)})}
\leq C(M)2^{(k-s)h},\qquad k>\frac{d}{2}+5,\\ \label{Cond2}
&|\Ric^{(h)}_0|\leq CM_1^2,\quad \inf_{x\in\Sigma_0^{(h)}} {\rm Vol}_{g_0^{(h)}}(B_x(e^{C(M)2^{-h_0}}))\geq v e^{-C(M)2^{-h_0}},\\ \label{RegIni-4}
&\frac{9}{10}c_0\leq (g_0^{(h)})\leq \frac{11}{10}c_0^{-1}.
\end{align}
where the constant $M_1=C(M,c_0)$ depends on $M$ and $c_0$. 
Then by Theorem~\ref{RegSol-thm}, we obtain the regular solutions $F^{(h)}(t)$ for all $h\geq h_0$ on some time interval $[0,T(M,h)]$ depending on the $M$, $c_0$ and $h$.

\subsection{Uniform bounds and the lifespan of regular solutions}

Once we have the regularized manifold $\Sigma_0^{(h)}=F_0^{(h)}(\R^d)$ for $h\in[h_0,\infty)$ large, we consider the corresponding smooth solutions $\Sigma^{(h)}$ generated by the smooth data $\Sigma_0^{(h)}$. A priori these solutions exist on a time interval that depends on the $\sfH^k$-norm of the second fundamental form $\la^{(h)}_0$ and the Sobolev embeddings on $\Sigma_0^{(h)}$, and hence depends on $h$ and $M$. Instead, here we would like to have a lifespan bound which is independent of $h$.

We remark that the bound $\|\la\|_{\sfH^k}$ does not directly propagate unless $k>\frac{d}{2}$ is an integer. 
Indeed, in that case one could immediately close the bootstrap at the level of the $\sfH^k$-norm using the Standard Sobolev embedding and the equivalence $\|\la\|_{\sfH^k}\approx \|\la\|_{H^k}$. 
The goal of the argument that follows is to establish the $X^s\subset H^s$ bound for any $s>\frac{d}{2}$, by working only with energy estimates for integer indices.

From the construction in Section~\ref{Reg-sec}, the manifolds $\Sigma^{(h_1)}_0$ with $h_1\in [h_0,h]$ can also be seen as one of the regularizations of $\Sigma^{(h)}_0$. 
Moreover, by Proposition~\ref{Global-harmonic} ii), the smooth initial manifold $\Sigma^{(h)}_0$ for any $h\in[h_0,\infty)$ satisfies
\begin{align}\label{Cond3}
&\||D|^{\si_d}g_0^{(h)}\|_{H^{s+1-\si_d}}+\||D|^{\de_d}A_0^{(h)}\|_{H^{s-\de_d}}+\|\la_0^{(h)}\|_{H^s}\leq M_1,\\ \label{Cond4}
&\tri[\la_0^{(h)}]\tri_{s,int}+\tri[g_0^{(h)}]\tri_{s+1,g}+\tri[A_0^{(h)}]\tri_{s,A}\leq M_1,
\end{align}
We then prove that the lifespan of (SMCF) only depends on $M$ and $c_0$.

\begin{prop} \label{RegSol-prop}
    Assume that the smooth initial manifold $\Sigma^{(h)}_0$ satisfies 
    \eqref{Cond1}, \eqref{Cond2}, \eqref{Cond3}, \eqref{Cond4} and \eqref{RegIni-4},
    then the lifespan $[0,CM_1^{-2N-8}]$ of $\Sigma^{(h)}(t)$ evolving along skew mean curvature flow only depends on $M$ and $c_0$. 
\end{prop}

\begin{proof}
Since the (SMCF) with initial manifold $\Sigma_0^{(h)}$ has a unique smooth solution $\Sigma^{(h)}$ on $[0,T(M,h)]$, then under the conditions \eqref{Cond3}, \eqref{Cond4}, \eqref{RegIni-4} and by Theorem~\ref{Para-thmS} and~\ref{Thm-Es}, the solution $\Sigma^{(h)}$ on the time interval $[0,\min\{T(M,h),CM_1^{-2N-8}\}]$ satisfies
\begin{align*}
    \|[\la^{(h)}]\|_{X^s_{int}}\leq 8M_1,\qquad \|[\la^{(h)}]\|_{X^s_{ext}}\leq 8C_{eq}M_1,
\end{align*}
and 
\begin{gather}\nonumber
    \frac{4}{5}c_0I\leq (g(t))\leq \frac{6}{5}c_0^{-1}I,\\\label{RegCondre}
    {\rm Vol}_{g(t)}(B_x(e^{tC_4 M_1^6}))\geq  e^{-tC_4 M_1^6}v,\quad |\Ric|\leq CM_1^2.
\end{gather}
Thus $\|\la^{(h)}\|_{L^\infty}\les\|\la\|_{X^s_{ext}}\leq CC_{eq}M_1$, then by \eqref{Energy} the $\la^{(h)}$ for any $h\geq h_0$ is bounded by
\begin{align*}
    \|\la^{(h)}\|_{\sfH^k}\leq \|\la^{(h)}_0\|_{\sfH^k} e^{CC_{eq}M_1 t}\les 2^{(k-s)h}c_h,
\end{align*}
which means that $\|\la^{(h)}\|_{\sfH^k}$ is still bounded on $[0,T(M,h)]$ if $T(M,h)< CM_1^{-2N-8}$. From Theorem~\ref{RegSol-thm} and \eqref{RegCondre}, the solution $\Sigma^{(h)}$ can be extended to the time interval $[0,CM_1^{-2N-8}]$. Hence, the lifespan of the SMCF depends only on $M$ and $c_0$.
\end{proof}

\subsection{The limiting solution}

Our goal in this section is to construct rough solutions as limits of smooth solutions. Here we show that the limit 
\begin{align*}
    F=\lim_{h\rightarrow \infty} F^{(h)}
\end{align*}
exists, first in a weaker topology and then in the strong topology, where $F^{(h)}$ are the solutions of (SMCF) with initial data $F^{(h)}_0=P_{<h}F_0$ on a uniform time interval $[0,T(M)]$.

\begin{prop} \label{prop8.2}
    The smooth solutions $F^{(h)}$ for $h\geq h_0$ are convergent in $L^2$ as $h\rightarrow +\infty$. Moreover, the limiting solution $F=\lim\limits_{h\rightarrow +\infty} F^{(h)}$ satisfies
    \begin{align}\label{Fconv}
    \lim_{h\rightarrow\infty}\|F^{(h)}-F\|_{H^{s+2}}=0,\qquad \|\d^2 F\|_{H^s}\les C(M).
    \end{align}
and the orthonormal frame $m^{(h)}$ satisfies
\begin{align}\label{mconv}
    \lim_{h\rightarrow\infty}\|m^{(h)}-m\|_{H^{s+1}}=0,\qquad \|\d m\|_{\dot H^{2\de_d}\cap \dot H^s}\les C(M).
\end{align}
\end{prop}

To prove the proposition, we consider the normal component and tangent component, respectively
\[
\omega^{(h)}: = \Xi^{(h)} \cdot m^{(h)},\qquad  \d_h F^{(h)}=\Xi^{(h)}+U^{(h)\ga}\d_\ga F^{(h)}.
\]
Then $\om^{(h)}$ and $U^{(h)}_\ga$ satisfy the two formulas \eqref{LinEq} and \eqref{dt-Tangent}.

\begin{lemma} 
On $[0,T(M)]$, the normal component $\om$ and tangent component $U$ satisfy the estimates
\begin{align} \label{om-Energy2}
    \int_{h_1}^\infty \|\om^{(h)}\|_{\sfH^1} dh&\les C(M) 2^{-(s+1)h_1},\\ \label{U-Energy}
    \int_{h_1}^\infty \|U^{(h)}\|_{L^2} dh  &\les C(M) 2^{-sh_1}. 
\end{align}
\end{lemma}

\begin{proof}
Since $\om^{(h)}(0)=\d_h F_0^{(h)}\cdot m_0^{(h)}$ and $\d^{A^{(h)}_\al}\om^{(h)}(0)=\d_\al\d_h F_0^{(h)}\cdot m_0^{(h)}-\la_\al^{(h)\si}(0)\d_h F_0^{(h)}\cdot \d_\si F_0^{(h)}$ for $F_0^{(h)}=P_{<h}F_0$ and $m_0^{(h)}$ given in \eqref{nuhini}, then by \eqref{om-Energy} we arrive at
\begin{align*}
    &\ \int_{h_1}^\infty \|\om(t)\|_{\sfH^1} dh\les \int_{h_1}^\infty \|\om(0)\|_{\sfH^1} dh\les \int_{h_1}^\infty \|P_h F_0\|_{H^1}(1+\|\la_0^{(h)}\|_{\infty}\|\d F_0\|_{L^\infty}) dh\\
    &\les C(M) \big(\int_{h_1}^\infty 2^{2sh}\|\d^2 P_h F_0\|_{L^2}^2 dh\big)^{1/2} \big(\int_{h_1}^\infty 2^{-2(s+1)h} dh\big)^{1/2}
    \les C(M) 2^{-(s+1)h_1}.
\end{align*}
Thus the first bound \eqref{om-Energy2} follows.

By Gr\"onwall's inequality on $[0,T(M)]$, the estimate \eqref{U-Energy0} implies
\begin{align*}
    \|U^{(h)}(t)\|_{L^2}&\les C(M)\big(\|U^{(h)}(0)\|_{L^2}+ \int_0^t\|\d_h g^{(h)}\|_{H^1}+\|\om^{(h)}\|_{\sfH^1} ds\big). 
\end{align*}
Then integrating over $[h_1,\infty)$ with respect to $h$ and using $\int_{h_0}^\infty 2^{2sh}\|\d_h g^{(h)}\|_{H^1}^2 dh \les C(M)$ and the bound \eqref{om-Energy2}, this yields
\begin{align*}
    &\ \int_{h_1}^\infty \|U^{(h)}(t)\|_{L^2} dh \les C(M)\big(\int_{h_1}^\infty\|U^{(h)}(0)\|_{L^2}dh+\int_{h_1}^\infty\int_0^t (\|\d_h g^{(h)}\|_{H^1}+\|\om^{(h)}\|_{\sfH^1}) d\tau dh\big)\\
    & \leq C(M)\big(\int_{h_1}^\infty\|\d_h P_{<h}F_0\|_{L^2}\|\d F_0\|_{L^\infty}dh+t\sup_{s\in[0,t]} \int_{h_1}^\infty \|\d_h g^{(h)}(s)\|_{H^1}+\|\om^{(h)}(s)\|_{\sfH^1}  dh\big)\\
    &\leq C(M)2^{-(s+2)h_1}+C(M)(2^{-sh_1}+2^{-(s+1)h_1})
    \leq C(M)2^{-sh_1}.
\end{align*}
We obtain the bound \eqref{U-Energy}.
\end{proof}

Using the above two lemmas, we then finish the proof of Proposition \ref{prop8.2}.
\begin{proof}[Proof of Proposition \ref{prop8.2}]
\

\emph{i) We prove \eqref{Fconv}.}
    From \eqref{om-Energy2} and \eqref{U-Energy}, we obtain the uniform bound on $[0,T(M)]$ for any $h_2>h_1\geq h_0$
    \begin{align*}
        \|F^{(h_2)}-F^{(h_1)}\|_{L^2}&\leq \int_{h_1}^{h_2} \|\d_h F^{(h)}\|_{L^2} dh\leq \int_{h_1}^{h_2} \|\om^{(h)}\|_{L^2}+\|U^{(h)}\|_{L^2}\|g^{(h)}\d F^{(h)}\|_{L^\infty} dh\\
        &\leq C(M)2^{-(s+1)h_1}+C(M)2^{-sh_1}\leq C(M)2^{-sh_1}.
    \end{align*}
This means that $F^{(h)}-F^{(h_0)}\in L^2$ is a Cauchy sequence, and therefore it is convergent.

Since the $H^s$-norm of $\d^2 F^{(h)}$ is uniformly bounded,
\[  \|\d^2 F^{(h)}\|_{H^s}=\|\Ga^{(h)} \d F^{(h)}+\la^{(h)} m^{(h)}\|_{H^s}\les \|\Ga^{(h)}\|_{L^\infty}\|\d F^{(h)}\|_{L^\infty}+\|\la^{(h)}\|_{L^\infty}\les C(M),  \]
we obtain that the similar norm of the limiting solution $F=\lim_{h\rightarrow \infty} F^{(h)}$ is also bounded by $C(M)$. Moreover, by interpolation we have the convergence in $H^{\si+2}$ for any $\si<s$ as $h\rightarrow \infty$
\begin{align*}
    \|F^{(h)}-F\|_{H^{\si+2}}\les \|F^{(h)}-F\|^{\frac{s-\si}{s+2}}_{L^2}\|F^{(h)}-F\|^{\frac{\si+2}{s+2}}_{H^{s+2}}\les 2^{-\frac{s-\si}{s+2}sh}C(M)\rightarrow 0.
\end{align*}

Since 
\begin{align*}
    \d^2 F=\d^2 F^{(h)}+\sum_{j=h}^\infty (\d^2 F^{(j+1)}-\d^2 F^{(j)}),
\end{align*}
then by \eqref{DB-Fm} and \eqref{HFB-Fm} we obtain
\begin{equation}\label{ConF}
\begin{aligned}
&\ \|\d^2 F-\d^2 F^{(h)}\|_{H^s}^2\\
&\les \sum_{j=h}^\infty 2^{2sj}\|\d^2 F^{(j+1)}-\d^2 F^{(j)}\|_{L^2}^2+2^{2(s-N)j}\|\d^2 F^{(j+1)}-\d^2 F^{(j)}\|_{H^N}^2\\
&\les \sum_{j=h}^\infty c_j^2= \|c_{>h}\|_{\ell^2}^2 \rightarrow 0,\qquad {\rm as}\ h\rightarrow\infty 
\end{aligned}  
\end{equation}
Thus the solution $\d^2 F^{(h)}$ also converge to $\d^2 F$ in $H^s$ strongly.

\emph{ii) We prove \eqref{mconv}.}
From \eqref{DB-Fm}, we have for any $h_2>h_1\geq h_0$
\begin{align*}
    \|m^{(h_1)}-m^{(h_2)}\|_{L^2}\les \sum_{h_1\leq j\leq h_2} 2^{-sj}c_j\les \sum_{h_1\leq j\leq h_2}c_j^2\rightarrow 0,\qquad as\ h_1\rightarrow \infty.
\end{align*}
Then the limit $\lim_{h\rightarrow\infty}(m^{(h)}-m^{(h_0)})$ exists in $L^2$, and hence we obtain the frame $m$. By \eqref{DB-Fm} and \eqref{HFB-Fm}, we also have
\begin{align*}
    \|\d m-\d m^{(h)}\|_{H^{s+1}}^2&\les \sum_{j\geq h}2^{2sj}\big(\|\d m^{(j+1)}-\d m^{(j)}\|_{L^2}^2+2^{2(s-N)j}\|\d m^{(j+1)}-\d m^{(j)}\|_{H^N}^2\big)\\
    &\les \sum_{j\geq h} c_j^2\rightarrow 0,\quad as\ h\rightarrow\infty.
\end{align*}
Moreover, 
\begin{align*}
    &\ \|\d F\cdot m\|_{L^2}=\|\d F\cdot m-\d F^{(h)}\cdot m^{(h)}\|_{L^2}\\
    &\les \|\d F-\d F^{(h)}\|_{L^2}+\|\d F\|_{L^\infty}\|m-m^{(h)}\|_{L^2}\rightarrow 0,\quad h\rightarrow\infty.
\end{align*}
Hence, we obtain the orthonormal frame $m$ as the limit of $m^{(h)}$ in $\dot H^{1+2\de_d}\cap \dot H^{s+1}$. 
\end{proof}

Now we show that the limiting map $F$ is a solution of (SMCF). It suffices to check that for any $v\in C_0^\infty ([0,T(M)]\times \R^d)$, it holds
\begin{align}  \label{Aim-sol}
    \int_0^T\int \<\d_t F, v\> dxdt=\int_0^T \int \<J(F)H(F)+V^\ga \d_\ga F,v\> dxdt.
\end{align}
Since $F^{(h)}$ is the solution of (SMCF) with initial data $F^{(h)}_0$, then the above equality holds when replacing $F$ by $F^{(h)}$. Moreover, by $\lim_{h\rightarrow \infty}\|F-F^{(h)}\|_{L^\infty L^2}= 0$, we have
\begin{align*}
    &\ \int_0^T \int \<\d_t F^{(h)},v\>dxdt\\
    &=-\int_0^T \int \<F^{(h)},\d_t v\>dxdt+\int \<F^{(h)}(T),v(T)\>dx-\int \<F^{(h)}(0),v(0)\>dx\\
    &\longrightarrow -\int_0^T \int \<F,\d_t v\>dxdt+\int \<F(T),v(T)\>dx-\int \<F(0),v(0)\>dx=\int_0^T \int \<\d_t F,v\>dxdt
\end{align*}
For the source term, by \eqref{DB-t} and \eqref{DB-Fm} we have
\begin{align*}
    &\ \int_0^T \int  \<J(F)H(F)+V^\ga \d_\ga F,v\> -\<J(F^{(h)})H(F^{(h)})+V^{(h)\ga} \d_\ga F^{(h)},v\> dxdt\\
    &=\int_0^T \int  \<-\Im(\psi \bar{m}-\psi^{(h)}\bar{m}^{(h)})+V^\ga \d_\ga F-V^{(h)\ga} \d_\ga F^{(h)},v\>dxdt\\
    &\les \big(\|\psi-\psi^{(h)}\|_{L^\infty L^2}+C(M)\|m-m^{(h)}\|_{L^\infty L^2}+\|V-V^{(h)}\|_{L^\infty L^2}C(M)\\
    &\quad +C(M)\|\d F-\d F^{(h)}\|_{L^2}\big)\|v\|_{L^1 L^2}\\
    &\les C(M)2^{-sh}\|v\|_{L^1L^2}
    \rightarrow 0,\qquad h \rightarrow \infty.
\end{align*}
Then the equality \eqref{Aim-sol} holds. Thus $F$ is the solution of (SMCF).

In addition, as a consequence of Proposition~\ref{prop8.2}, we get the convergence of metric $g^{(h)}$, connection $A^{(h)}$ and the second fundamental form $\la^{(h)}$:
\begin{align*}
    \lim_{h\rightarrow\infty}\big(\|g-g^{(h)}\|_{H^{s+1}}+\|A-A^{(h)}\|_{H^s}+\|\la-\la^{(h)}\|_{H^s}\big)=0.
\end{align*}
This means that the solutions $\Sigma^{(h)}$ for $h\geq h_0$ are a family of regularizations of $\Sigma$ on the time interval $[0,CM_1^{-2N-8}]$. Hence, the rough solution $\Sigma(t)$ exists on $[0,CM_1^{-2N-8}]$, and from Theorem~\ref{Para-thmS} and ~\ref{Thm-Es} it satisfies the energy estimates 
\begin{align*}
\|g\|_{Y^{s+1}}+\|A\|_{Z^s}+\|\la\|_{X^s}\les C(M).
\end{align*}

\subsection{Continuous dependence}
Suppose there is a sequence of $F_{0n}$ converge to $F_0$ in $H^{s+2}$ with metric and mean curvature satisfying \eqref{MetBdd} and \eqref{small-datad3}. 
The difference of the corresponding solutions can be rewritten as 
\begin{align}\label{ConDep-1}
    \|\d^2 (F_n-F)\|_{H^s}&\les \|\d^2 (F_n- F^{(h)}_n)\|_{H^s}+\|\d^2( F^{(h)}_n-F^{(h)})\|_{H^s}+\|\d^2 (F^{(h)}- F)\|_{H^s},
\end{align}
where $F^{(h)}_n$ and $F^{(h)}$ are the solutions of (SMCF) with initial data $P_{<h}F_{0n}$ and $P_{<h}F_0$, respectively.

The convergence $F_{0n}\rightarrow F_0$ in $H^{s+2}$ implies that the sequence of corresponding frequency envelopes may be chosen so that it is   convergent in $l^2$, $c^{(n)}\rightarrow c$. Then we have 
\begin{align*}
    \lim_{n\rightarrow\infty} c^{(n)}_{\geq h}=c_{\geq h}.
\end{align*}
Hence, using the estimate \eqref{ConF}, for any $\ep>0$ there exists $n_\ep$ and $h_{\ep}$ such that for any $n>n_\ep$ and $h>h_\ep$, it holds
\begin{align*}
    \|\d^2 (F_n^{(h)}-F_n)\|_{H^s}\les c_{\geq h}^{(n)}\leq \ep/3,\qquad  \|\d^2 (F^{(h)}-F)\|_{H^s}\les c_{\geq h}\leq \ep/3.
\end{align*}
Now it remains to bound the second term $\|\d^2( F^{(h)}_n-F^{(h)})\|_{H^s}$ in \eqref{ConDep-1}.
Fix the $h\geq h_\ep$, we have uniform $H^N$ bounds for the sequences $F^{(h)}_n$, $n>n_\ep$ and $F^{(h)}$ by \eqref{HFB-Fm}. Moreover, we denote $F^{(h)}_0(s)=P_{<h}F_0+s(P_{<h}F_{0n}-P_{<h}F_0)$, by \eqref{om-Energy} and \eqref{U-Energy0} we obtain the convergence in $L^2$
\begin{align*}
    &\|F^{(h)}_n-F^{(h)}\|_{L^2}\leq \int_0^1 \|\d_s F^{(h)}(s)\|_{L^2} ds\les \sup_s \|\om(t;s)\|_{L^2}+\|U(t;s)\|_{L^2}\\
    &\les \sup_s \|\om(0;s)\|_{\sfH^1}+\|U(0;s)\|_{L^2}
    \les \|\d_s F^{(h)}(0;s)\|_{H^1}\les \|P_{<h}(F_n(0)-F(0))\|_{H^1}\rightarrow 0.
\end{align*}
Then there exists $\tilde n_\ep>n_\ep$ such that for any $n>\tilde n_\ep$ we arrive at
\begin{align*}
\|\d^2( F^{(h)}_n-F^{(h)})\|_{H^s}&\leq \|\d^2( F^{(h)}_n-F^{(h)})\|_{H^N}^{\frac{s+2}{N+2}}\|F^{(h)}_n-F^{(h)}\|_{L^2}^{\frac{N-s}{N+2}}\\
&\les (2^{(N-s)h}c_h)^{\frac{s+2}{N+2}}\|F^{(h)}_n-F^{(h)}\|_{L^2}^{\frac{N-s}{N+2}}\leq \frac{\ep}{3}.
\end{align*}
Hence, for any $\ep>0$, there exists $\tilde n_\ep$ such that for any $n>\tilde n_\ep$ it holds
$\|\d^2 (F_n-F)\|_{H^s}\leq \ep$.
This completes the proof of continuous dependence.

\

\noindent {\bf Acknowledgements.}
J. Huang is supported by National Key Research and Development Program of China (Grant No. 2024YFA1015300), National Natural Science Foundation of China (Grant No. 12301293), Guangdong Basic and Applied Basic Research Foundation (Grant No. 2025A1515011984) and Beijing Institute of Technology Research Fund Program for Young Scholars. 
D. Tataru was supported by the NSF grant DMS-2054975. by a Simons Investigator grant from the Simons Foundation, as well as by a Simons Fellowship.

\end{document}